\documentclass[12pt,a4paper]{article}

\usepackage{color}

\usepackage{amsmath,amsthm}
\usepackage{amsfonts}
\usepackage{latexsym}
\usepackage{indentfirst}
\usepackage{psfrag,epsf}
\usepackage{pstool}
\usepackage{hyperref}
\usepackage{epsfig, subfigure}
\usepackage{graphicx}
\usepackage{epstopdf}
\usepackage{enumerate}
\usepackage[subnum]{cases}
\usepackage{geometry}
\usepackage{fullpage}
\usepackage{tikz}

\usetikzlibrary{shapes,arrows,calc}
\usepackage{multicol}  

\usepackage{bm}

\newtheorem{theorem}{Theorem} \rm
\newtheorem{lemma}[theorem]{Lemma}

\newtheorem{corollary}[theorem]{Corollary}
\newtheorem{conjecture}[theorem]{Conjecture}

\newtheorem{remark}[theorem]{Remark}
\newtheorem{question}[theorem]{Question}
\theoremstyle{plain}

\title{The square of a subcubic planar graph without a 5-cycle is 7-choosable}
\author
{
Seog-Jin Kim$^{\rm a}$,
Xiaopan Lian$^{\rm b}$,
Atsuhiro Nakamoto$^{\rm c}$,
Kenta Ozeki$^{\rm c}$,
\\
{\footnotesize$^{\rm a}$Department of Mathematics Education, Konkuk University, Korea}\\
{\footnotesize$^{\rm a}$Center for Combinatorics and LPMC, Nankai University, China}\\
{\footnotesize$^{\rm c}$Graduate school of Environmental and Information Science, Yokohama National University, Japan}\\
}
\begin{document}

\maketitle

\begin{abstract}
The square of a graph $G$, denoted $G^2$, has the
same vertex set as $G$ and has an edge between two vertices if the distance between
them in $G$ is at most $2$.
Thomassen \cite{Thomassen} showed that $\chi(G^2) \leq 7$ if $G$ is a subcubic planar graph.
A natural question is whether $\chi_{\ell}(G^2) \leq 7$ or not if $G$ is a subcubic planar graph.  Recently Kim and Lian \cite{KL23} showed that $\chi_{\ell}(G^2) \leq 7$ if $G$ is a subcubic planar graph of girth  at least 6. And Jin, Kang, and Kim \cite{JKK} showed that $\chi_{\ell}(G^2) \leq 7$ if $G$ is a subcubic planar graph without 4-cycles and 5-cycles.
In this paper, we show that the square of a subcubic planar graph without 5-cycles is 7-choosable, which improves the results of \cite{JKK} and \cite{KL23}.
\end{abstract}

\tableofcontents

\section{Introduction}
The {\em square} of a graph $G$, denoted $G^2$, has the
same vertex set as $G$ and has an edge between two vertices if the distance between
them in $G$ is at most $2$.
We say a graph $G$ is  {\em subcubic} if $\Delta(G) \leq 3$, where $\Delta(G)$ is the maximum degree in $G$.  The  {\em girth} of $G$, denoted $g(G)$, is the size of smallest cycle in $G$.
Let $\chi(G)$ be the  chromatic number of a graph $G$.

Wegner \cite{Wegner} posed the following conjecture.

\begin{conjecture}\cite{Wegner} \label{conj-Wegner}
Let $G$ be a planar graph. The chromatic number
$\chi(G^2)$ of $G^2$ is at most 7 if $\Delta(G) = 3$,
at most $\Delta(G)+5$ if $4 \leq \Delta(G) \leq 7$, and at most $\lfloor \frac{3 \Delta(G)}{2} \rfloor$ if $\Delta(G) \geq 8$.
\end{conjecture}

Conjecture \ref{conj-Wegner} is still wide open. The only case for which the answer is known is when $\Delta(G) = 3$.
Thomassen \cite{Thomassen} showed that $\chi(G^2) \leq 7$ if $G$ is a planar graph with $\Delta(G)  = 3$, which implies that Conjecture \ref{conj-Wegner} is true for $\Delta(G)  = 3$. Conjecture \ref{conj-Wegner} for $\Delta(G)  = 3$ is also confirmed by
Hartke, Jahanbekam and Thomas \cite{Hartke}.
Many results were obtained with conditions on $\Delta(G)$.
One may see  a detailed story on the study of Wegner's conjecture in \cite{Cranston22}.

A list assignment for a graph    $G$   is a function $L$ that assigns each vertex a list of
available colors. The graph is {\em $L$-colorable} if it has a proper coloring $f$ such that
$f (v) \in L(v)$ for all $v$.  If $G$ is $L$-colorable whenever
all lists have size $k$, then it is called {\em $k$-choosable}.  The {\em list chromatic number $\chi_{\ell}(G)$} is the minimum $k$ such that $G$
is $k$-choosable.

Since it was known in \cite{Thomassen} that $\chi(G^2) \leq 7$ if $G$ is a subcubic planar graph, the following natural question was raised in \cite{CK} and \cite {Havet}, independently.

\begin{question} \cite{CK,Havet} \label{CK-question}
Is it true that  $\chi_{\ell} (G^2) \leq 7$ if $G$ is a subcubic planar graph?
\end{question}

Considering the Thomassen's proof in \cite{Thomassen},
it seems difficult  to answer Question \ref{CK-question} completely if
the answer of the question is positive.

For general upper bound on $\chi_{\ell}(G^2)$ for a subcubic graph $G$,
Cranston and Kim \cite{CK} proved that $\chi_{\ell} (G^2) \leq 8$ if $G$ is a connected graph (not necessarily planar) with $\Delta(G) = 3$ and if $G$ is not the Petersen graph. To the direction of Question \ref{CK-question}, Cranston and Kim \cite{CK} proved that $\chi_{\ell} (G^2) \leq 7$ if $G$ is a  subcubic planar graph with $g(G) \geq 7$.  Recently, Kim and Lian \cite{KL23} made an interesting progress by showing that
that $\chi_{\ell} (G^2) \leq 7$ if $G$ is a  subcubic planar graph with $g(G) \geq 6$. And Jin, Kang, and Kim \cite{JKK} improved the result further by showing that  $\chi_{\ell} (G^2) \leq 7$ if $G$ is a  subcubic planar graph with 4-cycles and 5-cycles.

\medskip
In this paper, we make a big improment of previous results by showing
the following main theorem.

\begin{theorem} \label{main-thm}
If $G$ is a subcubic planar graph without 5-cycles, then $\chi_{\ell}(G^2) \leq 7$.
\end{theorem}

Theorem \ref{main-thm} improves the result of \cite{JKK} and \cite{KL23} since it forbids only 5-cycles. In  \cite{KL23}, $k$-cycles are forbidden for $k \in \{3, 4, 5\}$,  and in \cite{JKK}
$4$-cycles and $5$-cycles are forbidden.  But, in our paper, only $5$-cycles are forbidden.

On the other hand, it was asked in \cite{Hartke} whether $\chi(G^2) \leq 6$ if $G$ is a subcubic planar graph drawn without 5-faces, since every example having $\chi(G^2) = 7$ has a 5-cycle.  As a weaker version, it was conjectured in \cite{DST08, FHS} that $\chi(G^2) \leq 6$ when $G$ is a cubic bipartite planar graph.  As a natural question, one may ask the following question.

\begin{question}
Is it true that $\chi_{\ell}(G^2) \leq 6$ if $G$ is a cubic bipartite planar graph?
\end{question}
But, it is not known whether the square of a cubic bipartite planar graph is 7-choosable or not.
In this direction, our result provides an interesting upper bound in a  more general setting.

\begin{corollary} \label{cor1}
If $G$ is a subcubic bipartite planar graph, then $\chi_{\ell}(G^2) \leq 7$.
\end{corollary}

This paper is organized as follows. In Section 2, we introduce Combinatorial Nullstellensatz, which is an important tool for list coloring. In Section \ref{section-reducible}, we summarize the list of subgraphs which do not appear in a minimal counterexample to Theorem \ref{main-thm}.
In Section \ref{section-discharging}, we prove Theorem \ref{main-thm} by discharging argument.
In Section \ref{proof-reducible-config}, we provide the proofs of reducible configurations, which completes the proof of Theorem \ref{main-thm}.

\section{Preliminary}
Let $G$ be a graph and let `$<$' be an arbitrary fixed ordering of the vertices of $G$.
The \emph{graph polynomial} of $G$ is defined as
 $$P_{G}(\bm{x})=\prod_{u\sim v,u<v}(x_u-x_v),$$
 where $u\sim v$ means that $u$ and $v$ are adjacent, and $\bm{x}=(x_v)_{v\in V(G)}$ is a vector of $|V(G)|$ variables indexed by the vertices of $G$.
It is easy to see that a mapping $c:V(G) \to \mathbb{N}$ is a proper coloring of $G$ if and only if $P_{G}(\bm{c}) \neq 0$, where $\bm{c} = \big(c(v) \big)_{v \in V(G)}$.
Therefore, to find a proper coloring of $G$ is equivalent to find an assignment of $\bm{x}$ so that $P_{G}(\bm{x}) \neq 0$.
The following theorem, which was proved by Alon and Tarsi, gives sufficient conditions for the existence of such assignments as above.

\begin{theorem}[\cite{Alon1999}]\label{cnull}(Combinatorial Nullstellensatz) Let $\mathbb{F}$ be an arbitrary field and let $f=f(x_1,x_2,\ldots,x_n)$ be a polynomial in $\mathbb{F}[x_1,x_2,\ldots,x_n]$. Suppose that the degree $\deg(f)$ of $f$ is $\sum_{i=1}^n t_i$ where each $t_i$ is a nonnegative integer, and suppose that the coefficient of $\prod_{i=1}^n x_i^{t_i}$ of $f$ is nonzero. Then if $S_1,S_2,\ldots,S_n$ are subsets of $\mathbb{F}$ with $|S_i|\ge t_i+1$, then there are $s_1\in S_1$,$s_2\in S_2$,\ldots,$s_n\in S_n$ so that $f(s_1,s_2,\ldots,s_n)\neq 0$.
\end{theorem}

In particular, a graph polynomial $P_{G}(\bm{x})$ is a homogeneous polynomial and $\deg(P_{G})$ is equal to $|E(G)|$.

Let $G$ be a graph and let $L: V(G) \to 2^{\mathbb{N}}$ be a list.
Then $\bm{c}$ is an $L$-coloring of $G$
if and only if
$P_{G}(\bm{c}) \neq 0$,
where $\bm{c} = ( c(v_1), c(v_2), \dots, c(v_n))$.
Therefore, if there exists a monomial $\alpha \prod_{v \in V(G)} {x_{v}}^{t_{v}}$ in the expansion of $P_{G}$ so that $\alpha \neq 0$ and $t_{v} < k$ for each $v \in V(G)$, then $G$ is $k$-choosable.

\section{Summary of reducible configurations} \label{section-reducible}

In this section,
let $G$ be a minimal counterexample to Theorem \ref{main-thm}.
We will study structural properties of a minimal counterexample to Theorem \ref{main-thm}.
A configuration is {\em reducible} if a planar graph containing it cannot be a minimal counterexample.

 A vertex of degree $k$ is called a  $k$-vertex.  A cycle of size $k$ is called a $k$-cycle, and a cycle of size at least $k$ (resp.~at most $k$) is called a $k^+$-cycle (resp.~a $k^-$-cycle).

We summarize the key reducible configurations which will be used in the discharging part.
In Section \ref{proof-reducible-config}, we will prove that the following subgraphs do not appear in $G$.

 \medskip
First, we list the reducible configurations related with a 3-cycle
(see Figure \ref{key configuration-C3}).

\begin{enumerate}[(1)]

\item Subgraph $F_1$,
which consists of a $3$-cycle adjacent to a $4^-$-cycle.  (Lemma \ref{C3-C6} (a))

\item Subgraph $F_2$,
which consists of a $3$-cycle and a $4$-cycle whose distance is at most 1.
(Lemma \ref{reducible-F2})

\item Subgraph $F_3$,
which consists of a $3$-cycle adjacent to a $6$-cycle. (Lemma \ref{C3-C6} (b))

\item Subgraph $F_4$,
which consists of a $8^-$-cycle $F$ adjacent to a $3$-cycle and a $4$-cycle
such that the distance between the 3-cycle and the 4-cycle is 2.
 (Lemma \ref{reducible-F4})

\begin{figure}[htbp]
  \begin{center}
\begin{tikzpicture}[
  v2/.style={fill=black,minimum size=4pt,ellipse,inner sep=1pt},
  node distance=1.5cm,scale=0.6
]

\node[v2] (F1_1) at (0,2) {};
\node[v2] (F1_2) at (1.5,1) {};
\node[v2] (F1_4) at (0,0) {};
\node[v2] (F1_5) at (-1.5,1) {};

\draw (F1_1) -- (F1_2) --(F1_4)--(F1_1);
\draw (F1_1) -- (F1_5) -- (F1_4);

\node[font=\scriptsize\bfseries] at (0,-1.4) {$F_1$};
\node[font=\scriptsize] at (0.6,1) {$C_3$};
\node[font=\scriptsize] at (-0.4,1) {$C_3$};
\end{tikzpicture} \hspace{1.1cm}
\begin{tikzpicture}[
  v2/.style={fill=black,minimum size=4pt,ellipse,inner sep=1pt},
  node distance=1.5cm,scale=0.6
]

 \node[v2] (F2_1) at (0, 1) {};
    \node[v2] (F2_2) at (-1, 2) {};
    \node[v2] (F2_3) at (-1, 0) {};

    \node[v2] (F2_4) at (2, 1) {};
    \node[v2] (F2_5) at (3, 2) {};
    \node[v2] (F2_6) at (4, 1) {};
    \node[v2] (F2_7) at (3, 0) {};

    \draw (F2_1) -- (F2_2) -- (F2_3) -- (F2_1);

    \draw (F2_4) -- (F2_5) -- (F2_6) -- (F2_7) -- (F2_4);

    \draw (F2_1) -- (F2_4);

    \node[font=\scriptsize] at (-0.6, 1) {$C_3$};
    \node[font=\scriptsize] at (3, 1) {$C_4$};
    \node[font=\scriptsize\bfseries] at (1.2, -1.4) {$F_2$};
\end{tikzpicture}\hspace{1.1cm}
\begin{tikzpicture}[
  v2/.style={fill=black,minimum size=4pt,ellipse,inner sep=1pt},
  node distance=1.5cm,scale=0.45
]
 \node[v2] (F3_1) at (0, 0) {};
    \node[v2] (F3_2) at (0, 2) {};

    \node[v2] (F3_4) at (1.732, 3) {};
    \node[v2] (F3_5) at (3.464, 2) {};
    \node[v2] (F3_6) at (3.464, 0) {};
    \node[v2] (F3_7) at (1.732, -1) {};
    \node[v2] (F3_3) at (-1.5, 1) {};

    \draw (F3_1) -- (F3_2);
    \draw (F3_1) -- (F3_3) -- (F3_2);
    \draw (F3_1) -- (F3_7) -- (F3_6) -- (F3_5) -- (F3_4) -- (F3_2);

    \node[font=\scriptsize] at (-0.5, 1) {$C_3$};
    \node[font=\scriptsize] at (1.5, 1) {$C_6$};
    \node[font=\scriptsize\bfseries] at (1.5, -2) {$F_3$};
\end{tikzpicture}\hspace{1.1cm}
\begin{tikzpicture}[
  v2/.style={fill=black,minimum size=4pt,ellipse,inner sep=1pt},invis/.style={circle,draw=none,fill=none,inner sep=0pt,minimum size=0pt},
  node distance=1.5cm,scale=0.6
]

 \node[v2] (F2_1) at (1, 1) {};
    \node[v2] (F2_2) at (-1, 2) {};
    \node[v2] (F2_3) at (-1, 0) {};

    \node[v2] (F2_4) at (2, 1) {};
    \node[v2] (F2_5) at (3, 2) {};
    \node[v2] (F2_6) at (4, 1) {};
    \node[v2] (F2_7) at (3, 0) {};
    \node[v2] (F2_8) at (0, 1) {};
    \node[invis] (F2_9) at (1, 2) {};
     \node[v2] (F2_10) at (-0, -1) {};
    \node[v2] (F2_11) at (2, -1) {};

    \draw (F2_8) -- (F2_2) -- (F2_3) -- (F2_8);

    \draw (F2_4) -- (F2_5) -- (F2_6) -- (F2_7) -- (F2_4);

    \draw (F2_1) -- (F2_4);
    \draw (F2_1) -- (F2_8);
    \draw (F2_1) -- (F2_9);
     \draw (F2_3) -- (F2_10)  -- (F2_11) -- (F2_7);

    \node[font=\scriptsize] at (-0.6, 1) {$C_3$};
    \node[font=\scriptsize] at (3, 1) {$C_4$};
    \node[font=\scriptsize] at (1, -0.2) {$F$};
    \node[font=\scriptsize\bfseries] at (1, -2) {$F_4$};
\end{tikzpicture}

  \end{center}
  \caption{Subgraphs $F_1, F_2, F_3, F_4$}\label{key configuration-C3}
\end{figure}
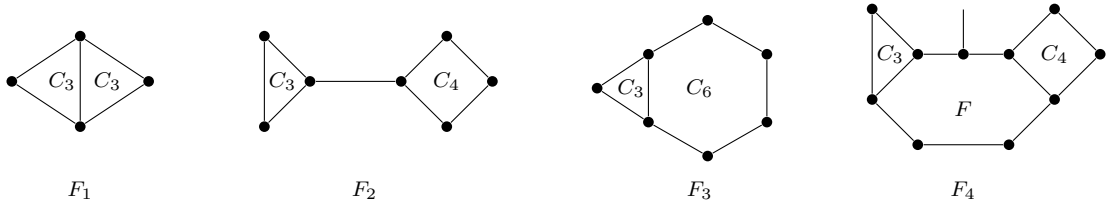

\begin{figure}[htbp]
\begin{center}
 \begin{tikzpicture}[
    v2/.style={fill=black,minimum size=4pt,ellipse,inner sep=1pt},
    scale=0.4
]

    \node[v2] (H1_1) at (0, 2) {};
    \node[v2] (H1_2) at (2, 2) {};
    \node[v2] (H1_3) at (4, 2) {};

    \node[v2] (H1_4) at (0, 0) {};
    \node[v2] (H1_5) at (2, 0) {};
    \node[v2] (H1_6) at (4, 0) {};

    \draw (H1_1) -- (H1_2) -- (H1_3);
    \draw (H1_4) -- (H1_5) -- (H1_6);

    \draw (H1_1) -- (H1_4);
    \draw (H1_2) -- (H1_5);
    \draw (H1_3) -- (H1_6);

    \node[font=\scriptsize] at (1, 1) {$C_4$};
    \node[font=\scriptsize] at (3, 1) {$C_4$};
    \node[font=\scriptsize\bfseries] at (2, -2.5) {$H_1$};
\end{tikzpicture}\hspace{1.1cm}
\begin{tikzpicture}[
    v2/.style={fill=black,minimum size=4pt,ellipse,inner sep=1pt},
    scale=0.4
]

    \node[v2] (H2_1) at (0, 0) {};
    \node[v2] (H2_2) at (0, 2) {};
    \node[v2] (H2_3) at (1.732, 3) {};
    \node[v2] (H2_4) at (3.464, 2) {};
    \node[v2] (H2_5) at (3.464, 0) {};
    \node[v2] (H2_6) at (1.732, -1) {};

    \node[v2] (H2_7) at (-2, 0) {};
    \node[v2] (H2_8) at (-2, 2) {};
    \node[v2] (H2_9) at ({1.732 + 1}, {3 + 1.732}) {};
    \node[v2] (H2_10) at ({3.464 + 1}, {2 + 1.732}) {};

    \draw (H2_1) -- (H2_2) -- (H2_3) -- (H2_4) -- (H2_5) -- (H2_6) -- (H2_1);

    \draw (H2_1) -- (H2_7) -- (H2_8) -- (H2_2);

    \draw (H2_3) -- (H2_9) -- (H2_10) -- (H2_4);

    \node[font=\scriptsize] at (-1, 1) {$C_4$};
    \node[font=\scriptsize] at (1.732, 1) {$C_6$};
    \node[font=\scriptsize] at (3.1, 3.2) {$C_4$};

    \node[font=\scriptsize\bfseries] at (1.732, -3) {$H_2$};

\end{tikzpicture} \hspace{1.1cm}
\begin{tikzpicture}[
    v2/.style={fill=black,minimum size=4pt,ellipse,inner sep=1pt},
    scale=0.4
]
     \node[v2] (H3_1) at (0, 0) {};
    \node[v2] (H3_2) at (0, 2) {};
    \node[v2] (H3_3) at (1.5, 3.4) {};
    \node[v2] (H3_4) at (3, 2) {};
    \node[v2] (H3_5) at (3, 0) {};
    \node[v2] (H3_6) at (1.5, -1.4) {};
    \node[v2] (H3_7) at (-2, 0) {};
    \node[v2] (H3_8) at (-2, 2) {};
    \node[v2] (H3_9)  at (0, 4.8) {};
    \node[v2] (H3_10) at (-2,4.8) {};
    \node[v2] (H3_11) at (-3.5, 3.4){};

\draw(H3_1)--(H3_2)--(H3_3)--(H3_4)--(H3_5)--(H3_6)--(H3_1);
\draw (H3_2)--(H3_8)--(H3_7)--(H3_1);
        \draw (H3_3) -- (H3_9) -- (H3_10) -- (H3_11) -- (H3_8);

        \node[font=\scriptsize] at (-1, 1) {$C_4$};
    \node[font=\scriptsize] at (1.732, 1) {$C_6$};
    \node[font=\scriptsize] at (-1, 3.2) {$C_6$};
          \node[font=\scriptsize\bfseries] at (0.6, -2.4) {$H_3$};
\end{tikzpicture}\hspace{1.1cm}\begin{tikzpicture}[
    v2/.style={fill=black,minimum size=4pt,ellipse,inner sep=1pt},
    scale=0.4
]
\node[v2] (H4_1) at (0, 0){};
    \node[v2] (H4_2) at (0, 2){};
    \node[v2] (H4_3) at (-1.732,-1){};
 \node[v2] (H4_4) at (-2*1.732,0){};
 \node[v2] (H4_5) at (-2*1.732,2){};
 \node[v2] (H4_6) at (-1.732,3){};
 \node[v2] (H4_7) at (1.732,-1){};
 \node[v2] (H4_8) at (2*1.732,0){};
 \node[v2] (H4_9) at (2*1.732,2){};
 \node[v2] (H4_10) at (1.732,3){};
 \node[v2] (H4_11) at (-4.464,-1.732){};
 \node[v2] (H4_12) at (-2.732,-2.732){};
 \node[v2] (H4_13) at (4.464,-1.732){};
 \node[v2] (H4_14) at (2.732,-2.732){};
 \node[v2] (H4_15) at (-2,-4.5){};
 \node[v2] (H4_16) at (0,-5.5){};
 \node[v2] (H4_17) at (2,-4.5){};
   \draw (H4_1) -- (H4_2) -- (H4_6) -- (H4_5) -- (H4_4) -- (H4_3)-- (H4_1);
   \draw (H4_2) -- (H4_10) -- (H4_9) -- (H4_8) -- (H4_7) -- (H4_1);
   \draw (H4_4) -- (H4_11) -- (H4_12) -- (H4_3);
   \draw (H4_8) -- (H4_13) -- (H4_14) -- (H4_7);
   \draw (H4_12) -- (H4_15) -- (H4_16) -- (H4_17) -- (H4_14);

    \node[font=\scriptsize] at (-0, -3) {$C_8$};
    \node[font=\scriptsize] at (1.732, 1) {$C_6$};
    \node[font=\scriptsize] at (-1.732, 1) {$C_6$};
 \node[font=\scriptsize] at (3.1, -1.5) {$C_4$};
 \node[font=\scriptsize] at (-3.1, -1.5) {$C_4$};
 \node[font=\scriptsize\bfseries] at (0, -6.5) {$H_4$};
\end{tikzpicture}
\end{center}
\caption{Subgraphs $H_1, H_2, H_3, H_4$} \label{key configuration}
\end{figure}
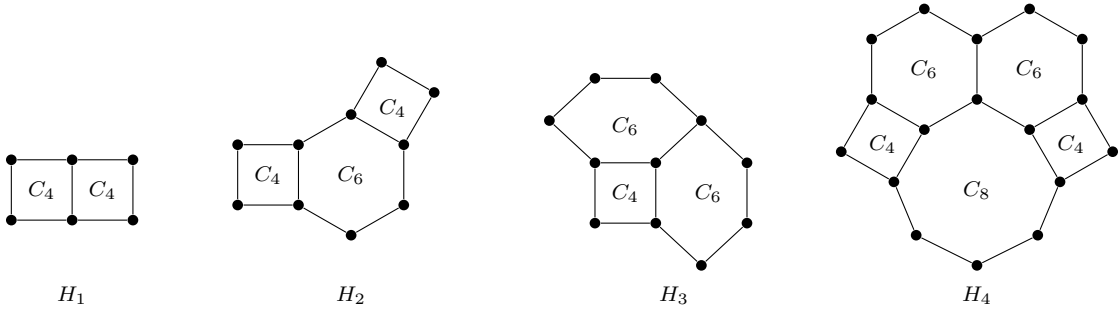

\begin{figure}[htbp]
  \begin{center}
\begin{tikzpicture}[
  v2/.style={fill=black,minimum size=4pt,ellipse,inner sep=1pt},
  node distance=1.5cm,scale=0.4]

      \node[v2] (H5_1) at (0, 0){};
      \node[v2] (H5_2) at (0,-2){};
      \node[v2] (H5_3) at (1,-3.5){};
      \node[v2] (H5_4) at (2.8,-3.5){};
      \node[v2] (H5_5) at (4,-2){};
      \node[v2] (H5_6) at (4, 0){};
      \node[v2] (H5_7) at (2, 1.2){};
      \node[v2] (H5_9) at (-2, 0){};
      \node[v2] (H5_10) at (-2,-2){};
      \node[v2] (H5_11) at (6, 0){};
      \node[v2] (H5_12) at  (6,-2){};

    \node[font=\scriptsize] at (-1, -1) {$C_4$};
    \node[font=\scriptsize] at (5, -1) {$C_4$};
    \node[font=\scriptsize] at (2, -1.2) {$C_7$};

   \draw (H5_1)--(H5_2)--(H5_3)--(H5_4)--(H5_5)--(H5_6)--(H5_7)--(H5_1);
   \draw (H5_1)--(H5_9)--(H5_10)--(H5_2);
   \draw (H5_5)--(H5_12)--(H5_11)--(H5_6);

  \node[font=\scriptsize]  at (2,-5.5) {Graph $H_5$};
  \end{tikzpicture}\hspace{3cm}
\begin{tikzpicture}[
  v2/.style={fill=black,minimum size=4pt,ellipse,inner sep=1pt},
  node distance=1.5cm,scale=0.4]

      \node[v2] (H6_1) at (0, 0){};
      \node[v2] (H6_2) at (0,-2){};
      \node[v2] (H6_3) at (1,-3.5){};
      \node[v2] (H6_4) at (2.8,-3.5){};
      \node[v2] (H6_5) at (4,-2){};
      \node[v2] (H6_6) at (4, 0){};
      \node[v2] (H6_7) at (2, 1.2){};
      \node[v2] (H6_9) at (-2, 0){};
      \node[v2] (H6_10) at (-2,-2){};
      \node[v2] (H6_11) at  (3.2,3.2){};
      \node[v2] (H6_12) at  (5.2,2){};

     \node[font=\scriptsize] at (-1, -1) {$C_4$};
    \node[font=\scriptsize] at (3.6, 1.6) {$C_4$};
    \node[font=\scriptsize] at (2, -1.2) {$C_7$};

   \draw (H6_1)--(H6_2)--(H6_3)--(H6_4)--(H6_5)--(H6_6)--(H6_7)--(H6_1);
   \draw (H6_1)--(H6_9)--(H6_10)--(H6_2);
   \draw (H6_6)--(H6_12)--(H6_11)--(H6_7);

  \node[font=\scriptsize]  at (2,-5.5) {Graph $H_6$};
\end{tikzpicture}
  \end{center}
  \caption{Subgraphs $H_5$ and $H_6$} \label{Conf-H56}
\end{figure}
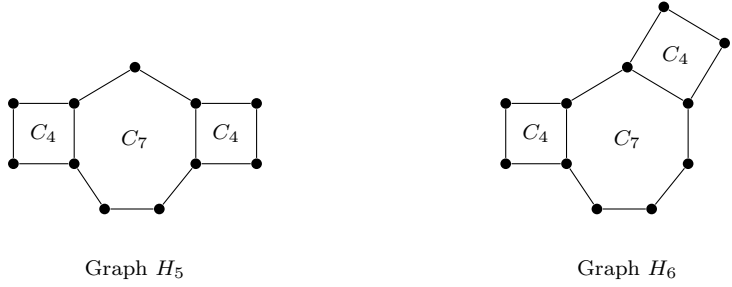


Next, we list the reducible configurations related with a 4-cycle,
 (see Figure \ref{key configuration}).

\item Subgraph $H_1$,
which consists of two adjacet 4-cycles.  (Lemma \ref{reducible-H0})

\item Subgraph $H_2$,
which consists of a 6-cycle adjacent to two 4-cycles
such that the distance between the two $4$-cycles is 1.
(Lemma \ref{reducible-H2})

\item Subgraph $H_3$,
which consists of a 4-cycle adjacent to two 6-cycles consecutively.
(Lemma \ref{reducible-H3})

\item Subgraph $H_4$,
which consists of a 8-cycle adjacent to four faces $f_1, f_2, f_3, f_4$ in order
such that $f_1$ and $f_4$ are 4-cycles and $f_2$ and $f_3$ are 6-cycles.  (Lemma \ref{reducible-H4})

\item
Subgraphs $H_5$ and $H_6$,
both of which consist of a 7-cycle adjacent to two 4-cycles. (Lemmas \ref{reducible-H6} and \ref{reducible-H7})

\end{enumerate}

Let $G$ be a plane graph drawn on the plane without crossing edges. A face of size $k$ is called a $k$-face, and a face  of size at least $k$ (resp.~at most $k$) is called a $k^+$-face  (resp.~a $k^-$-face).

\medskip
In addtion, we use the following properties for a minimal counterexample $G$.

\newtheorem*{lemma:C3-C6}{Lemma \ref{C3-C6}}
\begin{lemma:C3-C6}(c)
The distance between any two 3-cycles is at least 3 in $G$.
\end{lemma:C3-C6}

\newtheorem*{lemma:no-2-vertex}{Lemma \ref{no-2-vertex}}
\begin{lemma:no-2-vertex}
A minimal counterexample $G$ to Theorem \ref{main-thm} is a cubic planar graph.
\end{lemma:no-2-vertex}

\newtheorem*{cor-reducible-F4}{Corollary \ref{cor-reducible-F4}}
\begin{cor-reducible-F4}
If a $7$-face is adjacent to a $3$-face  in $G$,
then it is not adjacent to a $4$-face.
And if a $8$-face is adjacent to a $3$-face, then it cannot be adjacent to two $4$-faces.
\end{cor-reducible-F4}

We have the following by
Lemma \ref{C3-C6} (c) and Corollary \ref{cor-reducible-F4},
together with the fact that $F_2$ and $H_1$ do not exist in $G$:

\begin{itemize}
\item[(A)]
If a $7$-face is adjacent to a $3$-face,
then it is not adjacent to a $4$-face and no other 3-faces.
\item[(B)]
If an $8$-face is adjacent to a $3$-face,
then it is adjacent to at most one another $4^-$-face.
\item[(C)]
If a $9$-face is adjacent to a $3$-face,
then it is adjacent to at most two other $4^-$-faces.
\item[(D)]
If a $10$-face is adjacent to a $3$-face,
then it is adjacent to at most three other $4^-$-faces.
\item[(E)]
 An $11^+$-face $f$ is adjacent to
at most $\lfloor\frac{1}{2} d(f)\rfloor$ $4^-$-faces, where $d(f)$ is the length of $f$.
\end{itemize}

\section{Proof of Theorem \ref{main-thm}} \label{section-discharging}
In this section, we prove Theorem \ref{main-thm},
assuming reducible configurations and properties of a minimal counterexample introduced in the previous section.
Let $G$ be a minimal counterexample to Theorem \ref{main-thm} and 
let $G$ be a plane graph drawn on the plane without crossing edges.  Let $F(G)$ be the set of faces of $G$. For a face $f \in F(G)$, let $d(f)$ be the length of $f$.

We assign $2d(x)-6$ to each vertex $x \in V(G)$ and $d(x) - 6$ for each face $x \in F(G)$ as an original charge function $\omega(x)$ of $x$.
According to Euler's formula $|V(G)| - |E(G)| + |F(G)| = 2$,
we have
\begin{equation*} \label{Eulereq}
\sum_{v\in V(G)}(2d(v)-6)+\sum_{f\in F(G)}(d(f)-6) =  -12. 
\end{equation*}

We next design some discharging rules to redistribute charges along the graph with conservation of the total charge. Let $\omega'(x)$ be the charge of $x \in V(G)\cup F(G)$ after the discharge procedure that we will later explain. Note that ${\displaystyle \sum_{x\in V(G)\cup F(G)}\omega(x)=\sum_{x\in V(G)\cup F(G)}\omega'(x)}$. Next, we will show that $\omega'(x)\ge 0$ for all $x\in V(G)\cup F(G)$, which leads to the following contradiction.
\[
0\le \sum_{x\in V(G)\cup F(G)}\omega'(x)=\sum_{x\in V(G)\cup F(G)}\omega(x) = \sum_{v\in V(G)}(2d(v)-6)+\sum_{f\in F(G)}(d(f)-6) = -12.
\]

\medskip
We now discharge by the following rules.

\medskip
\noindent {\bf The discharging rule:} \\
For each edge $e$ between a 
$7^{+}$-face
$f$ and a $4$-face $f'$,
let $f_1$ and $f_2$ be the faces
containing one of the end vertices of $e$,
and then
\begin{enumerate}[(R1)]
\item $f$ sends $1$ to $f'$ if both $f_1$ and $f_2$ are $6$-faces,
\item $f$ sends $\frac{3}{4}$ to $f'$ if one of $f_1$ and $f_2$ is a $6$-face and the other is a $7^{+}$-face,
\item $f$ sends $\frac{1}{2}$ to $f'$ if both $f_1$ and $f_2$ are $7^{+}$-faces.\end{enumerate}

\begin{figure}[htbp]
  \begin{center}
\begin{tikzpicture}[
    v2/.style={fill=black,minimum size=4pt,ellipse,inner sep=1pt},invis/.style={circle,draw=none,fill=none,inner sep=0pt,minimum size=0pt},
    scale=0.45]

    \node[invis] (R1_1) at (0, 0) {};
    \node[invis] (R1_2) at (0, 2) {};
    \node[invis] (R1_3) at (1.732, 3) {};
    \node[invis] (R1_4) at (3.464, 2) {};
    \node[invis] (R1_5) at (3.464, 0) {};
    \node[invis] (R1_6) at (1.732, -1) {};
    \node[invis] (R1_7) at (-2, 0) {};
    \node[invis] (R1_8) at (-2, 2) {};
     \node[invis] (R1_9) at (-3.732, 3) {};
    \node[invis] (R1_10) at (-5.464, 2) {};
    \node[invis] (R1_11) at (-5.464, 0) {};
    \node[invis] (R1_12) at (-3.732, -1) {};
  \node[invis] (R1_13) at (-3.732, -2) {};
   \node[invis] (R1_14) at (1.732, -2) {};
   \node[invis] (R1_15) at (-1, -1) {};
   \node[invis] (R1_16) at (-1, 0.5) {};

    \draw (R1_1) -- (R1_2) -- (R1_3) -- (R1_4) -- (R1_5) -- (R1_6) -- (R1_1);

    \draw (R1_1) -- (R1_7) -- (R1_8) -- (R1_2);

    \draw (R1_7) -- (R1_12) -- (R1_11) -- (R1_10) -- (R1_9)-- (R1_8);
    \draw (R1_6) -- (R1_14);
     \draw (R1_12) -- (R1_13);
      \draw[->, >=to, ultra thick]  (R1_15) -- (R1_16);

    \node[font=\scriptsize] at (-1, 1) {$f'$};
     \node[font=\scriptsize] at (-3.6, 1) {$f_1$: 6-face};
     \node[font=\scriptsize] at (1.6, 1) {$f_2$: 6-face};
     \node[font=\scriptsize] at (-0.6, -0.4) {1};
      \node[font=\scriptsize] at (-1, -1.4) {$f$: $7^+$-face};
      \node[font=\scriptsize]  at  (-5, 3) {(R1)};
      \end{tikzpicture}
\hspace{0.8cm}
\begin{tikzpicture}[
    v2/.style={fill=black,minimum size=4pt,ellipse,inner sep=1pt},invis/.style={circle,draw=none,fill=none,inner sep=0pt,minimum size=0pt},scale=0.45]

    \node[invis] (R1_1) at (0, 0) {};
    \node[invis] (R1_2) at (0, 2) {};
    \node[invis] (R1_3) at (1.732, 3) {};
    \node[invis] (R1_4) at (3.464, 2) {};
    \node[invis] (R1_5) at (3.464, 0) {};
    \node[invis] (R1_6) at (1.732, -1) {};
    \node[invis] (R1_7) at (-2, 0) {};
    \node[invis] (R1_8) at (-2, 2) {};
     \node[invis] (R1_9) at (-2.866, 2.5) {};
    \node[invis] (R1_11) at (-4.732, -1) {};
    \node[invis] (R1_12) at (-3.732, -1) {};
  \node[invis] (R1_13) at (-3.732, -2) {};
   \node[invis] (R1_14) at (1.732, -2) {};
   \node[invis] (R1_15) at (-1, -1) {};
   \node[invis] (R1_16) at (-1, 0.5) {};

    \draw (R1_1) -- (R1_2) -- (R1_3)-- (R1_4)-- (R1_5) -- (R1_6) -- (R1_1);

    \draw (R1_1) -- (R1_7) -- (R1_8) -- (R1_2);

    \draw (R1_7) -- (R1_12) -- (R1_11);
    \draw(R1_9)-- (R1_8);
    \draw (R1_6) -- (R1_14);
     \draw (R1_12) -- (R1_13);
      \draw[->, >=to, ultra thick]  (R1_15) -- (R1_16);

    \node[font=\scriptsize] at (-1, 1) {$f'$};
     \node[font=\scriptsize] at (-3.8, 1) {$f_1$: $7^+$-face};
     \node[font=\scriptsize] at (1.6, 1) {$f_2$: 6-face};
     \node[font=\scriptsize] at (-0.6, -0.45) {$\frac{3}{4}$};
      \node[font=\scriptsize] at (-1, -1.4) {$f$: $7^+$-face};
      \node[font=\scriptsize]  at  (-4.5, 3) {(R2)};
      \end{tikzpicture}
\hspace{0.8cm}
            \begin{tikzpicture}[
    v2/.style={fill=black,minimum size=4pt,ellipse,inner sep=1pt},invis/.style={circle,draw=none,fill=none,inner sep=0pt,minimum size=0pt},
    scale=0.45
]

    \node[invis] (R1_1) at (0, 0) {};
    \node[invis] (R1_2) at (0, 2) {};
    \node[invis] (R1_3) at (0.85, 2.5) {};
    \node[invis] (R1_4) at (3.464, 2) {};
    \node[invis] (R1_5) at (2.732, -1) {};
    \node[invis] (R1_6) at (1.732, -1) {};
    \node[invis] (R1_7) at (-2, 0) {};
    \node[invis] (R1_8) at (-2, 2) {};
     \node[invis] (R1_9) at (-2.866, 2.5) {};
    \node[invis] (R1_11) at (-4.732, -1) {};
    \node[invis] (R1_12) at (-3.732, -1) {};
  \node[invis] (R1_13) at (-3.732, -2) {};
   \node[invis] (R1_14) at (1.732, -2) {};
   \node[invis] (R1_15) at (-1, -1) {};
   \node[invis] (R1_16) at (-1, 0.5) {};

    \draw (R1_1) -- (R1_2) -- (R1_3);
    \draw  (R1_5) -- (R1_6) -- (R1_1);

    \draw (R1_1) -- (R1_7) -- (R1_8) -- (R1_2);

    \draw (R1_7) -- (R1_12) -- (R1_11);
    \draw(R1_9)-- (R1_8);
    \draw (R1_6) -- (R1_14);
     \draw (R1_12) -- (R1_13);
      \draw[->, >=to, ultra thick]  (R1_15) -- (R1_16);

    \node[font=\scriptsize] at (-1, 1) {$f'$};
     \node[font=\scriptsize] at (-3.8, 1) {$f_1$: $7^+$-face};
     \node[font=\scriptsize] at (1.8, 1) {$f_2$: $7^+$-face};
     \node[font=\scriptsize] at (-0.6, -0.45) {$\frac{1}{2}$};
      \node[font=\scriptsize] at (-1, -1.4) {$f$: $7^+$-face};
      \node[font=\scriptsize]  at  (-4.5, 3) {(R3)};
      \end{tikzpicture}
  \end{center}
  \caption{Discharging rules} \label{rulefig}
\end{figure}
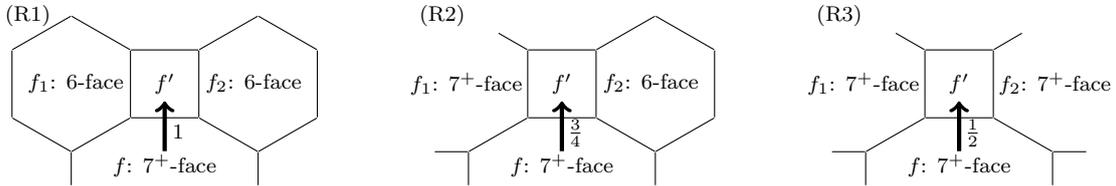

%

In addition,
we give one more discharging rule:
\begin{enumerate}
\item[(R4)] If a face $f$ of size at least 7 is adjacent to a 3-face $f'$, then $f$ sends $1$ to $f'$.
\end{enumerate}
\medskip

We now show that $\omega'(x)\ge 0$ for all $x\in V(G)\cup F(G)$.
Note that $G$ is a cubic graph by Lemma \ref{no-2-vertex}.
In addition,
$G$ has no subgraphs $F_1$--$F_4$
and $H_1$--$H_6$,
and satisfies (A)--(E)
as in Section \ref{section-reducible}.

\medskip

First,
for each vertex $v$,
$v$ is a $3$-vertex by Lemma \ref{no-2-vertex},
and hence $\omega'(v) = \omega(v) = 0$.

Next,
we will show $\omega'(f) \geq 0$ for each face $f$,
depending on the value of $d(f)$.
Note that only $7^{+}$-faces may send charge to adjacent 3-faces or 4-faces.

\medskip

\noindent  (1) The case $d(f) = 3$. \\
Note that $\omega(f) = 3-6 = -3$.
Let $f_1, f_2, f_3$ be the faces adjacent to $f$.
Since $F_1$ and $F_3$ do not exist,
$d(f_i) \geq 7$ for $1 \leq i \leq 3$.
So, by (R4), each of $f_i$ sends $1$ to $f$.  Hence
$\omega'(f) = -3 + 3 \times 1 = 0$.

\medskip
\noindent  (2) The case $d(f) = 4$. \\
Note that $\omega(f) = 4-6 = -2$.
Let $f_1, f_2, f_3, f_4$ be the faces adjacent to $f$ in the clockwise order.
Since $F_1$ and $H_1$ do not exist in $G$,
$d(f_i) \geq 6$ for $1 \leq i \leq 4$.
\begin{itemize}
\item
If $f$ is adjacent to no $6$-face,
then it follows from (R3) 
that each face adjacent to $f$
sends $\frac{1}{2}$ to $f$,
and hence $\omega'(f) = -2 + \frac{1}{2} \times 4 = 0$.
\item
Suppose that $f$ is adjacent to a $6$-face, say $f_1$.
Since $H_3$ does not exist in $G$,
both $f_2$ and $f_4$ are $7^{+}$-faces.
If $f_3$ is a $6$-face,
then both $f_2$ and $f_4$ send $1$ to $f$ by (R1), 
and hence $\omega'(f) = -2 + 1 \times 2 = 0$.
Otherwise,
that is,
if $f_3$ is a $7^{+}$-face,
then both $f_2$ and $f_4$ send $\frac{3}{4}$ to $f$
by (R2),
and $f_3$ sends $\frac{1}{2}$ to $f$ by (R3).
Thus, $\omega'(f) = -2 + \frac{3}{4} \times 2 + \frac{1}{2} = 0$.
\end{itemize}

\bigskip
\noindent  (3) The case  $d(f) = 6$. \\
In this case, $f$ does not send nor receive any charge. Hence $\omega'(f) = \omega(f) = d(f) -6=0$.

\medskip


\noindent  (4) The case $d(f) = 7$. \\
Note that $\omega(f) = 7-6 = 1$.
We consider the following two cases.

\begin{itemize}
\item
Suppose that $f$ is adjacent to a $3$-face.
By (A), $f$ is not adjacent to  another $4^-$-face.
Then by (R1)--(R4),
$\omega'(f) =  \omega(f) - 1 = 0$.

\item
Suppose that $f$ is not adjacent to a $3$-face.
Since neither $H_5$ nor $H_6$  exists  in $G$,
$f$ is adjacent to at most one 4-face.
Then by (R1)--(R4),
$\omega'(f) \geq  \omega(f) - 1 = 0$.
\end{itemize}

\medskip
\noindent  (5) The case $d(f) = 8$. \\
Note that $\omega(f) = 8-6 = 2$.
If $f$ is adjacent to a $3$-face,
then $f$ is adjacent to  at most one another $4^-$-face  by (B).
Hence $\omega'(f) \geq 2 - 1 \times 2 = 0$.
Thus, we may assume that $f$ is not adjacent to a $3$-face.
\medskip

Since $H_1$ does not exist in $G$,
$f$ is adjacent to at most four $4$-faces,
and no two $4$-faces are adjacent.
Let $f_1, f_2, \dots, f_8$ be the faces adjacent to $f$
in the clockwise order.
\begin{itemize}
\item
If $f$ is adjacent to at most two $4$-faces,
then it follows from (R1)--(R3) 
that $f$ sends at most $1$ to
each $4$-face adjacent to $f$,
and hence $\omega'(f) \geq 2 - 1 \times 2 = 0$.
\item
Suppose that $f$ is adjacent to exactly three $4$-faces.
By symmetry,
we have the following two cases;

\begin{itemize}
\item
Suppose that $f_1, f_3$ and $f_5$ are $4$-faces.
Since $H_2$ does not exist in $G$,
both $f_2$ and $f_4$ are $7^{+}$-faces.
Thus,
$f$ sends $\frac{1}{2}$ to $f_3$ by (R3),
and at most $\frac{3}{4}$ to $f_1$ and $f_5$ by (R2) and (R3).
Therefore, $\omega'(f) \geq 2 - \frac{1}{2} - \frac{3}{4} \times 2 = 0$.

\item
Suppose that $f_1, f_3$ and $f_6$ are $4$-faces.
Since $H_2$ does not exist in $G$,
$f_2$ is a $7^{+}$-face.
\begin{itemize}
\item
If both $f_4$ and $f_8$ are $7^{+}$-faces,
then $f$ sends $\frac{1}{2}$ to $f_1$ and $f_3$ by (R3),
and at most $1$ to $f_6$ by (R1)--(R3),
and hence $\omega'(f) \geq 2 - \frac{1}{2} \times 2 -1 = 0$.
\item
Suppose that $f_4$ is a $7^{+}$-face and $f_8$ is a $6$-face.
Note that $f$ sends $\frac{3}{4}$ to $f_1$ by (R2),
and $\frac{1}{2}$ to $f_3$ by (R3).
Since $H_4$ does not exist in $G$,
$f_7$ is a $7^{+}$-face,
and hence $f$ sends at most $\frac{3}{4}$ to $f_6$ by (R2) and (R3).
Thus, $\omega'(f) \geq 2 - \frac{3}{4} \times 2 - \frac{1}{2} = 0$.
\item
If $f_4$ is a $6$-face and $f_8$ is a $7^{+}$-face,
then by the symmetry to the previous case,
$\omega'(f) \geq 2 - \frac{3}{4} \times 2 - \frac{1}{2} = 0$.
\item
Suppose that both $f_4$ and $f_8$ are $6$-faces.
Note that $f$ sends $\frac{3}{4}$ to $f_1$ and $f_3$ by (R2).
Since $H_4$ does not exist in $G$,
both $f_5$ and $f_7$ are $7^{+}$-faces,
and hence $f$ sends $\frac{1}{2}$ to $f_6$ by (R3).
Thus, $\omega'(f) \geq 2 - \frac{3}{4} \times 2 - \frac{1}{2} = 0$.
\end{itemize}
\end{itemize}

\item
Suppose that $f$ is adjacent to exactly four $4$-faces.
Since $H_1$ does not exists in $G$,
it follows from the symmetry that $f_1, f_3, f_5$ and $f_7$ are $4$-faces.
Since $H_2$ does not exists in $G$,
all of $f_2, f_4, f_6$ and $f_8$ are $7^{+}$-faces.
Then it follows from (R1) that
$f$ sends $\frac{1}{2}$ to $f_1, f_3, f_5$ and $f_7$,
and hence $\omega'(f) = 2 - \frac{1}{2} \times 4 = 0$.
\end{itemize}

\medskip

\noindent  (6) The case $d(f) = 9$. \\
Note that $\omega(f) = 9-6 = 3$.  We consider the following two cases.

\begin{itemize}
\item
If $f$ is adjacent to a $3$-face,
by (C),
$f$ is adjacent to at most two other $4^-$-faces,
and hence $\omega'(f) \geq 3 - 1 \times 3 = 0$.

\item
Suppoce that $f$ is not adjacent to a $3$-face.
Since $H_1$ does not exist in $G$,
$f$ is adjacent to at most four $4$-faces.
\begin{itemize}
\item
If $f$ is adjacent to at most three $4$-faces,
then
it follows from (R1)--(R3) that
$f$ sends at most $1$ to each $4$-face adjacent to $f$,
and hence $\omega'(f) \geq 3 - 1 \times 3 = 0$.
\item
Suppose that $f$ is adjacent to exactly four $4$-faces.
Let $f_1, f_2, \dots, f_9$ be the faces adjacent to $f$
in the clockwise order.
Since $H_1$ does not exist in $G$,
we may assume that $f_1, f_3, f_5, f_7$ are the 4-faces.
Since $H_2$ does not exist in $G$,
$f_2, f_4, f_6$ are $7^{+}$-faces.
By (R2) and (R3), $f$ sends $\frac{1}{2}$ to $f_3$ and $f_5$, respectively,
and sends at most $\frac{3}{4}$ to $f_1$ and $f_7$, respectively.
Thus
$\omega'(f) \geq 3 - \frac{1}{2} \times 2 - \frac{3}{4} \times 2  > 0$.
\end{itemize}
\end{itemize}


\medskip
\noindent  (7) The case $d(f) = 10$. \\
Note that $\omega(f) = 10-6 = 4$.  We consider the following two cases.

\begin{itemize}
\item
If $f$ is adjacent to a $3$-face, then
it follows from (D) that $f$ is adjacent to at most three other $4^-$-faces.
So, $\omega'(f) \geq 4 - 1 \times 4 = 0$.

\item
Suppose that $f$ is not adjacent to a $3$-face.
Since $H_1$ does not exist in $G$,
$f$ is adjacent to at most five $4$-faces.
If $f$ is adjacent to at most four $4$-faces,
then
it follows from (R1)--(R3) that
$f$ sends at most $1$ to each $4$-face adjacent to $f$,
and hence $\omega'(f) \geq 4 - 1 \times 4 = 0$.
Thus,
we may assume that $f$ is adjacent to exactly five $4$-faces.
Since neither $H_1$ nor $H_2$  exists  in $G$,
the $4$-faces appear every other along the clockwise order of the faces adjacent to $f$
and the other faces are $7^{+}$-faces.
Thus,
by (R3),
$\omega'(f) = 4 - \frac{1}{2} \times 5 > 0$.
\end{itemize}

\noindent (8) The case $d(f) \geq 11$. \\
By (E), the number of $4^-$-faces adjacent to $f$ is at most $\lfloor\frac{1}{2} d(f)\rfloor$.
We have $\omega'(f) \geq (d(f)-6) - 1 \times \lfloor\frac{1}{2} d(f)\rfloor = \lceil\frac{1}{2} d(f)\rceil  -6\geq 0$ by (R1)--(R4).

\bigskip
Therefore,
$\omega'(x)\ge 0$ for all $x\in V(G)\cup F(G)$,
which is a contradiction.
This completes the proof of Theorem \ref{main-thm}.
\qed

\section{Proofs of reducible configurations} \label{proof-reducible-config}

In this section, we study the properties of a minimal counterexample  to  Theorem \ref{main-thm}, and show that
then subgraphs  $F_1$--$F_4$
and subgraphs $H_1$--$H_6$ do not appear.
 Throughout this section, let $G$ be a minimal counterexample to Theorem \ref{main-thm}.

\subsection{Basic configurations} \label{subsec-basic}
In this subsection, we prove several important reducible configurations which will be used in the proofs of important lemmas.
First, we will list reducible configurations related with a $3$-face,
where (a), (b) and (c) show that
subgraphs $F_1$ and $F_3$ in Figure \ref{key configuration-C3}
and subgraph $T$ in Figure \ref{distance-triangle}
are all reducible configurations.

\begin{lemma} \label{C3-C6} We have the following properties.
\begin{enumerate}[(a)]
\item A $3$-cycle do not share an edge with a $4^-$-cycle.

\item A $3$-cycle do not share an edge with a $6$-cycle.

\item The distance between any two 3-cycles is at least 3 in $G$.

\end{enumerate}
\end{lemma}
\begin{proof}
(b) follows from \cite[Lemma 8]{JKK}.
If a $3$-cycle share an edge with a $4$-cycle,
then there is a $5$-cycle, a contradiction.
Thus, in (a) it suffices to show the case when
two $3$-cycles share an edge.
This can be shown similarly to \cite[Lemma 8]{JKK} or the proof of the next lemma.
(We leave the detail to the readers.)
(c) is obtained from \cite[Lemma 9]{JKK}, together with (a).
\end{proof}

\begin{figure}[htbp]
\begin{center}
\begin{tikzpicture}[u/.style={fill=black,minimum size =4pt,ellipse,inner sep=1pt},invisible/.style={circle,draw=none,fill=none,inner sep=0pt,minimum size=0pt},node distance=1.5cm,scale=0.9]

\node[u] (v1) at (0,0){};
\node[u] (v2) at (-1,0.7){};
\node[u] (v3) at (-1,-0.7){};
\node[u] (v5) at (2.4,0){};
\node[u] (v6) at (3.4,0.7) {};
\node[u] (v7) at (3.4,-0.7) {};
\node[u] (v8) at (1.2,0) {};

\draw  (v1) to (v2){};
\draw  (v1) to (v3){};
\draw  (v1) to (v5){};
\draw  (v2) to (v3){};
\draw  (v5) to (v6){};
\draw  (v5) to (v7){};
\draw  (v6) to (v7){};

\node[above] at (v1) {$v_3$};
\node[below] at (v3) {$v_1$};
\node[above] at (v2) {$v_2$};
\node[above] at (v5) {$v_4$};
\node[above] at (v6) {$v_5$};
\node[below] at (v7) {$v_6$};
\node[above] at (v8) {$v_7$};
\node[below=0.8cm,left=0.5cm] at (v5) {}; 
\end{tikzpicture}
\end{center}
\caption{Graph $T$} \label{distance-triangle}
\end{figure}


Next, we will show that subgraph $H_1$ does not appear in $G$.

\begin{lemma} \label{reducible-H0}
The subgraph $H_1$ in Figure \ref{key configuration} does not appear in $G$.
\end{lemma}
\begin{proof}
Suppose that
$G$ has $H_1$ as a subgraph, and denote $V(H_1) = \{v_1, v_2, v_3, v_4, v_5, v_6\}$ (see Figure \ref{H1fig}).
Let $L$ be a list assignment with lists of size 7 for each vertex in $G$.
We will show that $G^2$ has a proper coloring from the list $L$, which is a contradiction for the fact that $G$ is a counterexample to the theorem.
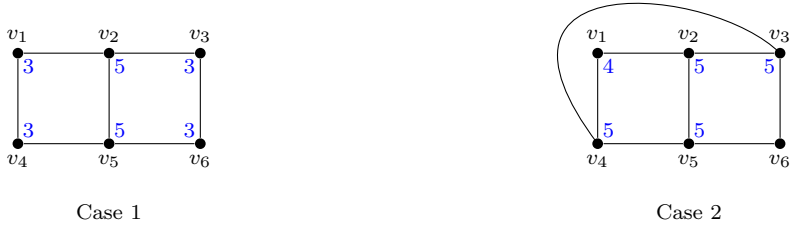
\begin{figure}[htbp]
  \begin{center}
 \begin{tikzpicture}[
    v2/.style={fill=black,minimum size=4pt,ellipse,inner sep=1pt},
    scale=0.6
]

    \node[v2] (H1_1) at (0, 2) {};
    \node[v2] (H1_2) at (2, 2) {};
    \node[v2] (H1_3) at (4, 2) {};

    \node[v2] (H1_4) at (0, 0) {};
    \node[v2] (H1_5) at (2, 0) {};
    \node[v2] (H1_6) at (4, 0) {};

    \draw (H1_1) -- (H1_2) -- (H1_3);
    \draw (H1_4) -- (H1_5) -- (H1_6);

    \draw (H1_1) -- (H1_4);
    \draw (H1_2) -- (H1_5);
    \draw (H1_3) -- (H1_6);
    \node[font=\scriptsize,above] at (H1_1){$v_1$};
     \node[font=\scriptsize,yshift=-5pt,xshift=4pt] at (H1_1){{\color{blue}3}};
    \node[font=\scriptsize,above] at (H1_2){$v_2$};
    \node[font=\scriptsize,yshift=-5pt,xshift=4pt] at (H1_2){{\color{blue}5}};
    \node[font=\scriptsize,above] at (H1_3){$v_3$};
     \node[font=\scriptsize,yshift=-5pt,xshift=-4pt] at (H1_3){{\color{blue}3}};
    \node[font=\scriptsize,below] at (H1_4){$v_4$};
    \node[font=\scriptsize,yshift=5pt,xshift=4pt] at (H1_4){{\color{blue}3}};
    \node[font=\scriptsize,below] at (H1_5){$v_5$};
     \node[font=\scriptsize,yshift=5pt,xshift=4pt] at (H1_5){{\color{blue}5}};
    \node[font=\scriptsize,below] at (H1_6){$v_6$};
\node[font=\scriptsize,yshift=5pt,xshift=-4pt] at (H1_6){{\color{blue}3}};

    \node[font=\scriptsize] at (2, -1.5) {Case 1};
\end{tikzpicture}\hspace{3cm}
 \begin{tikzpicture}[
    v2/.style={fill=black,minimum size=4pt,ellipse,inner sep=1pt},
    scale=0.6
]

    \node[v2] (H1_1) at (0, 2) {};
    \node[v2] (H1_2) at (2, 2) {};
    \node[v2] (H1_3) at (4, 2) {};

    \node[v2] (H1_4) at (0, 0) {};
    \node[v2] (H1_5) at (2, 0) {};
    \node[v2] (H1_6) at (4, 0) {};

    \draw (H1_1) -- (H1_2) -- (H1_3);
    \draw (H1_4) -- (H1_5) -- (H1_6);

    \draw (H1_1) -- (H1_4);
    \draw (H1_2) -- (H1_5);
    \draw (H1_3) -- (H1_6);
    \draw
  (0,0)
    .. controls (-3,4) and (2.5,3.5) ..
  (4,2);
  \node[font=\scriptsize,above] at (H1_1){$v_1$};
     \node[font=\scriptsize,yshift=-5pt,xshift=4pt] at (H1_1){{\color{blue}4}};
    \node[font=\scriptsize,above] at (H1_2){$v_2$};
    \node[font=\scriptsize,yshift=-5pt,xshift=4pt] at (H1_2){{\color{blue}5}};
    \node[font=\scriptsize,above] at (H1_3){$v_3$};
     \node[font=\scriptsize,yshift=-5pt,xshift=-4pt] at (H1_3){{\color{blue}5}};
    \node[font=\scriptsize,below] at (H1_4){$v_4$};
    \node[font=\scriptsize,yshift=5pt,xshift=4pt] at (H1_4){{\color{blue}5}};
    \node[font=\scriptsize,below] at (H1_5){$v_5$};
     \node[font=\scriptsize,yshift=5pt,xshift=4pt] at (H1_5){{\color{blue}5}};
    \node[font=\scriptsize,below] at (H1_6){$v_6$};

    \node[font=\scriptsize] at (2, -1.5) {Case 2};
\end{tikzpicture}
  \end{center}
  \caption{Cases 1 and 2 of $H_1$, respectively. The numbers at vertices are the number of available colors.} \label{H1fig}
\end{figure}

\medskip

\noindent {\bf Case 1:} $H_1^2$ is an induced subgraph of $G^2$.
\\
Let $G' = G - V(H_1)$.
Then $G'$ is also a subcubic planar graph and $|V(G')| < |V(G)|$.   Since $G$ is a minimal counterexample to Theorem \ref{main-thm},
the square of $G'$ has a proper coloring $\phi$ such that $\phi(v) \in L(v)$ for each vertex $v \in V(G')$.

Now, for each $v_i \in V(H_1)$, we define \[
L_{H_1}(v_i) = L(v_i) \setminus \{\phi(x) : xv_i \in E(G^2) \mbox{ and } x \notin V(H_1)\}.
\]
Then, we have the following (see Case 1 in Figure \ref{H1fig}).
$$
|L_{H_1}(v_i)| \geq
\begin{cases}
3 & i=1,3,4,6, \\
5 & i=2,5.
\end{cases}
$$
Now, we will show that
$H_1^2$ admits a proper coloring from the list $L_{H_1}$.

\medskip

\noindent {\bf Subcase 1.1:} $L_{H_1}(v_1) \cap L_{H_1}(v_6) \neq \emptyset$.

Color $v_1$ and $v_6$ by a color $c \in L_{H_1}(v_1) \cap L_{H_1}(v_6)$, and
greedily color $v_3, v_4, v_2, v_5$ in order.
Then $H_1^2$ admits a proper coloring from its list.  This gives an $L$-coloring for $G^2$.

\medskip

\noindent  {\bf Subcase 1.2:} $L_{H_1}(v_1) \cap L_{H_1}(v_6) = \emptyset$.

Note that in this case
$|L_{H_1}(v_1) \cup L_{H_1}(v_6)| \geq 6$ and $|L_{H_1}(v_5)| \geq 5$.  So,
we can color $v_1$ by $c_1 \in L_{H_1}(v_1)$ and $v_6$ by $c_6 \in L_{H_1}(v_6)$
so that $|L_{H_1}(v_5) \setminus \{c_1, c_6 \}| \geq 5$.
Then greedily color $v_3, v_4, v_2, v_5$ in order.
Then $H_1^2$ admits a proper coloring from its list $L_{H_1}$.  This gives an $L$-coloring for $G^2$.

\medskip
\noindent {\bf Case 2:} $H_1^2$ is not an induced subgraph in $G^2$.
\\
In this case, we need to consider the case when $v_1$ and $v_6$ are adjacent or $v_3$ and $v_4$ are adjacent in $G^2$. By symmetry, we may assume that $v_3$ and $v_4$ are adjacent in $G^2$.
Since $G$ has no $5$-cycle, $v_3$ and $v_4$ cannot have a common neighbor in $G$,
and hence $v_3$ and $v_4$ are adjacent in $G$.
By the planarity,
$v_1$ and $v_6$ are not adjacent in $G^2$.

Let $G' = G - \{v_1, v_2, v_3, v_4, v_5\}$.
Then $G'$ is also a subcubic planar graph and $|V(G')| < |V(G)|$.   Since $G$ is a minimal counterexample to Theorem \ref{main-thm},
the square of $G'$ has a proper coloring $\phi$ such that $\phi(v) \in L(v)$ for each vertex $v \in V(G')$.
Then $v_1, v_2, v_3, v_4, v_5$ induce a $K_5$ and the number of available colors at vertices are like Case 2 in Figure \ref{H1fig}.  Then $v_1, v_2, v_3, v_4, v_5$ are colored properly from the list.
Thus, $G^2$ has an $L$-coloring.
This is a contradiction for the fact that $G$ is a counterexample.
So, $H_1$ does not appear in $G$.
\end{proof}

Next, we will prove that a $6$-cycle has no $2$-vertex.

\begin{lemma} (\cite[Lemma 5]{KL23}) \label{6-face}
$G$ has no 6-cycle which contains a 2-vertex.
\end{lemma}
\begin{proof}
Suppose that  $G$ has a $6$-cycle $C$ such that $V(C) = \{v_1, v_2, v_3, v_4, v_5, v_6\}$ and $v_6$ is a $2$-vertex (see Figure \ref{6-cycle-2-vertex}).
Note that for $1 \leq i \leq 3$, $v_i$ and $v_{i+3}$ are not adjacent in $G$ by Lemma \ref{reducible-H0}, and $v_i$ and $v_{i+3}$ have no common neighbor since $G$ has no $5$-cycle.
Thus $C^2$ is an induced subgraph of $G^2$.  But in this case, subgraph $C$ cannot appear in $G$
by \cite[Lemma 5]{KL23}.
So, $G$ has no 6-cycle which contains a 2-vertex.
\end{proof}

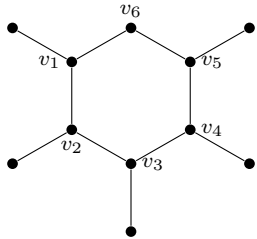
\begin{figure}[htbp]
  \begin{center}
 \begin{tikzpicture}[
  v2/.style={fill=black,minimum size=4pt,ellipse,inner sep=1pt},
  node distance=1.5cm,scale=0.45
]
 \node[v2] (F3_1) at (0, 0) {};
 \node[v2] (v1) at (-1.732, -1.0){};
    \node[v2] (F3_2) at (0, 2) {};
      \node[v2] (v2) at (-1.732, 3.0) {};

    \node[v2] (F3_4) at (1.732, 3) {};

    \node[v2] (F3_5) at (3.464, 2) {};
      \node[v2] (v5) at (5.196, 3.0) {};
    \node[v2] (F3_6) at (3.464, 0) {};
    \node[v2] (v6) at (5.196, -1.0){};
    \node[v2] (F3_7) at (1.732, -1) {};
 \node[v2] (v7) at (1.732, -3) {};

    \draw (F3_1) -- (F3_2);
    \draw (F3_1) -- (F3_7) -- (F3_6) -- (F3_5) -- (F3_4) -- (F3_2);
    \draw (F3_7) -- (v7);
     \draw (F3_5) -- (v5);
       \draw (F3_1) -- (v1);
     \draw (F3_2) -- (v2);
\draw (F3_6) -- (v6);
\node[font=\scriptsize,below] at (F3_1) {$v_2$};
\node[font=\scriptsize,left] at (F3_2) {$v_1$};
\node[font=\scriptsize,above] at (F3_4) {$v_6$};
\node[font=\scriptsize,right] at (F3_5) {$v_5$};
\node[font=\scriptsize,right] at (F3_6) {$v_4$};
\node[font=\scriptsize,right] at (F3_7) {$v_3$};
\end{tikzpicture}
  \end{center}
\caption{A 6-cycle $C$ which has a $2$-vertex $v_6$.} \label{6-cycle-2-vertex}
\end{figure}

Next, we show that if $G$ is a minimal counterexample, then $\delta(G) \geq 3$.
\begin{lemma} \label{no-2-vertex}
$G$ is a cubic planar graph.
\end{lemma}
\begin{proof}
It is easily checked that $G$ has no $1$-vertex. Also,
it is easily verified that a $3$-cycle has no 2-vertex, and a $4$-cycle has no $2$-vertex.
A $6$-cycle has no $2$-vertex by Lemma \ref{6-face}.
Thus, if a minimal counterexample $G$ of has a 2-vertex $w$ with $N_G(w) = \{x, y\}$, then there exist no cycle of length at most $6$ that passes through $w$.
We modify $G$ with $H = G - w + xy$.
Then,
we know that the resulting graph $H$ is still  a planar graph without 5-cycles.
Now since $|V(H)| < |V(G)|$,
$H^2$ has  an  $L$-coloring $\phi$ by induction hypothesis.
And then we can color $w$ by a color $c \in L(w)$
since $w$ has at most six neighbors in $G^2$.
Thus, $G^2$ has a proper $L$-coloring, which is a contradiction.  So, $G$ has no 2-vertex.  Thus $\delta(G) \geq 3$ and $G$ is a cubic planar graph since $\Delta(G) \leq 3$.
\end{proof}

Lemma \ref{no-2-vertex} shows the next lemma.

\begin{lemma} \label{reducible-H0_2}
$G$ has no two 4-faces which are adjacent.
\end{lemma}
\begin{proof}
By Lemma \ref{reducible-H0}, $G$ has no two $4$-faces sharing exactly one edge.
If there are two $4$-faces sharing two edges,
then $G$ must have a vertex of degree $2$, contradicting  to  Lemma \ref{no-2-vertex}.
\end{proof}

Next, we prove that subgraph $F_2$ in Figure \ref{key configuration-C3} does not appear in $G$.

\begin{lemma} \label{reducible-F2}
The subgraph $F_2$ in Figure \ref{key configuration-C3} does not appear in $G$.
\end{lemma}
\begin{proof}
Suppose that $G$ has $F_2$ as a subgraph, where the distance between a $3$-face and a $4$-face is 1.
Let $H$ be the subgraph induced by $\{v_1, v_2, \ldots, v_{7}\}$ in $G$ (see Figure \ref{subgraph-M1}).
Let $G' = G - V(H)$.
Then $G'$ is also a subcubic planar graph and $|V(G')| < |V(G)|$.   Since $G$ is a minimal counterexample to Theorem \ref{main-thm},
the square of $G'$ has a proper coloring $\phi$ such that $\phi(v) \in L(v)$ for each vertex $v \in V(G')$.

Now, for each $v_i \in V(H)$, we define \[
L_{H}(v_i) = L(v_i) \setminus \{\phi(x) : xv_i \in E(G^2) \mbox{ and } x \notin V(H)\}.
\]
Then, we have the following (see Figure \ref{subgraph-M1}).
$$
|L_{H}(v_i)| \geq
\begin{cases}
2 & i=6,\\
3 & i=1, 2, 5, 7, \\
5 & i=3, 4.
\end{cases}
$$

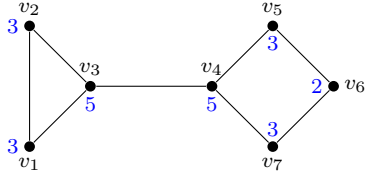
\begin{figure}[htbp]
  \begin{center}
\begin{tikzpicture}[
  v2/.style={fill=black,minimum size=4pt,ellipse,inner sep=1pt},
  node distance=1.5cm,scale=0.8]

 \node[v2] (F2_1) at (0, 1) {};
    \node[v2] (F2_2) at (-1, 2) {};
    \node[v2] (F2_3) at (-1, 0) {};

    \node[v2] (F2_4) at (2, 1) {};
    \node[v2] (F2_5) at (3, 2) {};
    \node[v2] (F2_6) at (4, 1) {};
    \node[v2] (F2_7) at (3, 0) {};

    \draw (F2_1) -- (F2_2) -- (F2_3) -- (F2_1);

    \draw (F2_4) -- (F2_5) -- (F2_6) -- (F2_7) -- (F2_4);

    \draw (F2_1) -- (F2_4);

    \node[font=\scriptsize,above] at (F2_1) {$v_3$};
    \node[font=\scriptsize,below] at (F2_1) {\color{blue}5};
    \node[font=\scriptsize,above] at (F2_4) {$v_4$};
    \node[font=\scriptsize,below] at (F2_4) {\color{blue}5};
    \node[font=\scriptsize,above] at (F2_2) {$v_2$};
    \node[font=\scriptsize,left] at (F2_2) {\color{blue}3};
     \node[font=\scriptsize,below] at (F2_3) {$v_1$};
    \node[font=\scriptsize,left] at (F2_3) {\color{blue}3};
    \node[font=\scriptsize,above] at (F2_5) {$v_5$};
    \node[font=\scriptsize,below] at (F2_5) {\color{blue}3};
    \node[font=\scriptsize,right] at (F2_6) {$v_6$};
    \node[font=\scriptsize,left] at (F2_6) {\color{blue}2};
      \node[font=\scriptsize,below] at (F2_7) {$v_7$};
    \node[font=\scriptsize,above] at (F2_7) {\color{blue}3};
\end{tikzpicture}
  \end{center}
  \caption{The distance between $3$-face and $4$-face is 1. The numbers at vertices are the number of available colors.} \label{subgraph-M1}
\end{figure}

If $v_1$ and $v_5$ are adjacent in $G^2$,
then $G$ has a $3$-cycle adjacent to a $4$-cycle or a $5$-cycle,
contradicting Lemma \ref{C3-C6} (a) or the assumption.
Thus, we can assume that $v_1$ and $v_5$ are not adjacent in $G^2$.
Similarly, we see that
$v_i$ and $v_j$ are not adjacent in $G^2$
for any $i \in \{1,2\}$ and $j \in \{5,6,7\}$.
\\

\noindent
Case 1: $L_{H}(v_1) \cap L_{H}(v_5) \neq \emptyset$.

Color $v_1$ and $v_5$ by a color $c \in L(v_1) \cap L_{H}(v_5)$, and then greedily color $v_6, v_7, v_4, v_2, v_3$ in order.

\medskip

\noindent
Case 2: $L_{H}(v_1) \cap L_{H}(v_5) = \emptyset$.

Since $|L_{H}(v_1) \cup L_{H}(v_5)| \geq 6$ and $|L_{H}(v_3)| \geq 5$, we can color $v_1$ by $c_1 \in L_{H}(v_1)$ and $v_5$ by $c_5 \in L_{H}(v_5)$ so that $|L_{H}(v_3) \setminus \{c_1, c_5\}| \geq 4$.  And then
 greedily color $v_6, v_7, v_4, v_2, v_3$ in order.

So, $H^2$ admits an $L$-coloring from the list $L_H(v)$.
Thus $G^2$ admits an $L$-coloring from the list $L(v)$, which is a contradition.
\end{proof}

Next lemma is for the property of subgraph $F_4$ in Figure \ref{key configuration-C3}.

\begin{lemma} \label{reducible-F4}
If a face $F$ is adjacent to a $3$-face and a $4$-face such that the distance between the $3$-face and the $4$-face is $2$, then the size of $F$ is at least 9.
\end{lemma}
\begin{proof}
Suppose that $G$ has $F_4$ as a subgraph with the size of the face $F$ is at most 8 (see Figure \ref{7cycle-C34-color}).   Then $F$ is either a $7$-face or $8$-face,
see subgraphs $W_1$ and $W_2$ in Figure \ref{7cycle-C34-color}, respectively.

\medskip

We have the following two cases.

\medskip
\noindent {\bf Case 1:} $F$ is a $7$-face.

Let ${W_1}$ be the subgraph induced by $\{v_1, v_2, \ldots, v_{10}\}$ in $G$.
Let $G' = G - V({W_1})$.
Then $G'$ is also a subcubic planar graph and $|V(G')| < |V(G)|$.   Since $G$ is a minimal counterexample to Theorem \ref{main-thm},
the square of $G'$ has a proper coloring $\phi$ such that $\phi(v) \in L(v)$ for each vertex $v \in V(G')$.

\medskip
\noindent
{\bf Subcase 1.1:} $W_1^2$ is an induced subgraph of $G^2$.

Now, for each $v_i \in V({W_1})$, we define \[
L_{W_1}(v_i) = L(v_i) \setminus \{\phi(x) : xv_i \in E(G^2) \mbox{ and } x \notin V({W_1})\}.
\]
Then, we have the following (see $W_1$ in Figure \ref{7cycle-C34-color}).
$$
|L_{W_1}(v_i)| \geq
\begin{cases}
3 & i=6, 7, 9, 10, \\
4 & i=2,4,\\
5 & i=1, 3, 5, 8.
\end{cases}
$$
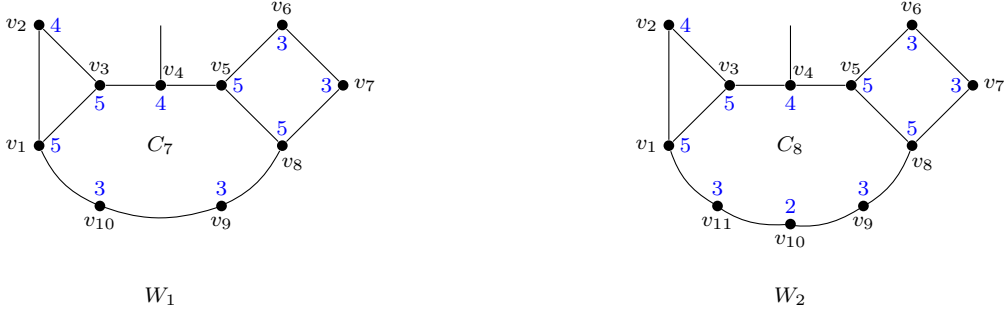
\begin{figure}[htbp]
  \begin{center}
\begin{tikzpicture}[
  v2/.style={fill=black,minimum size=4pt,ellipse,inner sep=1pt},invis/.style={circle,draw=none,fill=none,inner sep=0pt,minimum size=0pt},
  node distance=1.5cm,scale=0.8
]

 \node[v2] (F2_1) at (1, 1) {};
    \node[v2] (F2_2) at (-1, 2) {};
    \node[v2] (F2_3) at (-1, 0) {};

    \node[v2] (F2_4) at (2, 1) {};
    \node[v2] (F2_5) at (3, 2) {};
    \node[v2] (F2_6) at (4, 1) {};
    \node[v2] (F2_7) at (3, 0) {};
    \node[v2] (F2_8) at (0, 1) {};
    \node[invis] (F2_9) at (1, 2) {};
     \node[v2] (F2_10) at (-0, -1) {};
    \node[v2] (F2_11) at (2, -1) {};

    \draw (F2_8) -- (F2_2) -- (F2_3) -- (F2_8);

    \draw (F2_4) -- (F2_5) -- (F2_6) -- (F2_7) -- (F2_4);

    \draw (F2_1) -- (F2_4);
    \draw (F2_1) -- (F2_8);
    \draw (F2_1) -- (F2_9);
     \draw (F2_3) to [bend right=20]  (F2_10) to[bend right=20]  (F2_11)to[bend right=20](F2_7);

    \node[font=\scriptsize, xshift=5pt,yshift=5pt] at (F2_1) {$v_4$};
    \node[font=\scriptsize, below] at (F2_1) {\color{blue}4};
    \node[font=\scriptsize, left] at (F2_2) {$v_2$};
    \node[font=\scriptsize, right] at (F2_2) {\color{blue}4};
     \node[font=\scriptsize, left] at (F2_3) {$v_1$};
    \node[font=\scriptsize, right] at (F2_3) {\color{blue}5};
     \node[font=\scriptsize, above] at (F2_4) {$v_5$};
    \node[font=\scriptsize, right] at (F2_4) {\color{blue}5};
     \node[font=\scriptsize, above] at (F2_5) {$v_6$};
    \node[font=\scriptsize, below] at (F2_5) {\color{blue}3};
    \node[font=\scriptsize, right] at (F2_6) {$v_7$};
    \node[font=\scriptsize, left] at (F2_6) {\color{blue}3};
        \node[font=\scriptsize, xshift=4pt,yshift=-7pt] at (F2_7) {$v_8$};
    \node[font=\scriptsize, above] at (F2_7) {\color{blue}5};
\node[font=\scriptsize, above] at (F2_8) {$v_3$};
    \node[font=\scriptsize, below] at (F2_8) {\color{blue}5};
    \node[font=\scriptsize, below] at (F2_10) {$v_{10}$};
    \node[font=\scriptsize, above] at (F2_10) {\color{blue}3};
    \node[font=\scriptsize, below] at (F2_11) {$v_{9}$};
    \node[font=\scriptsize, above] at (F2_11) {\color{blue}3};
        \node[font=\scriptsize\bfseries] at (1,0) {$C_7$};
       \node[font=\scriptsize] at (1, -2.5) {$W_1$};
\end{tikzpicture}\hspace{3cm}
\begin{tikzpicture}[
  v2/.style={fill=black,minimum size=4pt,ellipse,inner sep=1pt},invis/.style={circle,draw=none,fill=none,inner sep=0pt,minimum size=0pt},
  node distance=1.5cm,scale=0.8
]

 \node[v2] (F2_1) at (1, 1) {};
    \node[v2] (F2_2) at (-1, 2) {};
    \node[v2] (F2_3) at (-1, 0) {};

    \node[v2] (F2_4) at (2, 1) {};
    \node[v2] (F2_5) at (3, 2) {};
    \node[v2] (F2_6) at (4, 1) {};
    \node[v2] (F2_7) at (3, 0) {};
    \node[v2] (F2_8) at (0, 1) {};
    \node[invis] (F2_9) at (1, 2) {};
     \node[v2] (F2_10) at (1, -1.3) {};
    \node[v2] (F2_11) at (2.2, -1) {};
    \node[v2] (F2_12) at (-0.2, -1) {};

    \draw (F2_8) -- (F2_2) -- (F2_3) -- (F2_8);

    \draw (F2_4) -- (F2_5) -- (F2_6) -- (F2_7) -- (F2_4);

    \draw (F2_1) -- (F2_4);
    \draw (F2_1) -- (F2_8);
    \draw (F2_1) -- (F2_9);
     \draw (F2_3) to [bend right=20]  (F2_12) to[bend right=20]  (F2_10)to[bend right=20](F2_11)to[bend right=20](F2_7);

    \node[font=\scriptsize, xshift=5pt,yshift=5pt] at (F2_1) {$v_4$};
    \node[font=\scriptsize, below] at (F2_1) {\color{blue}4};
    \node[font=\scriptsize, left] at (F2_2) {$v_2$};
    \node[font=\scriptsize, right] at (F2_2) {\color{blue}4};
     \node[font=\scriptsize, left] at (F2_3) {$v_1$};
    \node[font=\scriptsize, right] at (F2_3) {\color{blue}5};
     \node[font=\scriptsize, above] at (F2_4) {$v_5$};
    \node[font=\scriptsize, right] at (F2_4) {\color{blue}5};
     \node[font=\scriptsize, above] at (F2_5) {$v_6$};
    \node[font=\scriptsize, below] at (F2_5) {\color{blue}3};
    \node[font=\scriptsize, right] at (F2_6) {$v_7$};
    \node[font=\scriptsize, left] at (F2_6) {\color{blue}3};
        \node[font=\scriptsize, xshift=4pt,yshift=-7pt] at (F2_7) {$v_8$};
    \node[font=\scriptsize, above] at (F2_7) {\color{blue}5};
\node[font=\scriptsize, above] at (F2_8) {$v_3$};
    \node[font=\scriptsize, below] at (F2_8) {\color{blue}5};
    \node[font=\scriptsize, below] at (F2_10) {$v_{10}$};
    \node[font=\scriptsize, above] at (F2_10) {\color{blue}2};
    \node[font=\scriptsize, below] at (F2_11) {$v_{9}$};
    \node[font=\scriptsize, above] at (F2_11) {\color{blue}3};
    \node[font=\scriptsize, below] at (F2_12) {$v_{11}$};
    \node[font=\scriptsize, above] at (F2_12) {\color{blue}3};

       \node[font=\scriptsize] at (1, -2.5) {$W_2$};
   \node[font=\scriptsize\bfseries] at (1,0) {$C_8$};
\end{tikzpicture}
  \end{center}
  \caption{
Face $F$ is adjacent to a $3$-face and a $4$-face such that the distance between the $3$-face and the $4$-face is 2.
$W_1$ and $W_2$ are for the case when $F$ is a 7-face and a 8-face, respectively.
The numbers at vertices are the number of available colors.} \label{7cycle-C34-color}
 \end{figure}

If $|L_{W_1}(v_2)| \geq 5$, then  greedily color $v_{10}, v_9, v_8, v_7, v_6, v_5, v_4, v_3, v_1, v_2$ in order.  This coloring is possible since $|L_{W_1}(v_2)| \geq 5$ and $v_2$ has only four neighbors in $G^2$.

If $|L_{W_1}(v_2)| = 4$, then there exists a color $c_1 \in L_{W_1}(v_1) \setminus L_{W_1}(v_2)$ since $|L_{W_1}(v_1)| \geq 5$.  Color $v_{10}$ by a color $c_{10} \in L_{W_1}(v_{10}) \setminus \{c_1\}$, and then color $v_{9}$ by a color $c_{9} \in L_{W_1}(v_{9}) \setminus \{c_1, c_{10}\}$.  Next, greedily color $v_8, v_7, v_6, v_5, v_4$ in order. Let $\phi(v_i)$ be the color that is colored at $v_i$ for $i \in \{4, 5, 6, 7, 8\}$.  Let
\begin{eqnarray*}
&& L_{W_1}'(v_1)  =  L_{W_1}(v_1) \setminus \{c_{9}, c_{10}, \phi(v_4)\},
~~~L_{W_1}'(v_2) = L_{W_1}(v_2) \setminus \{c_{10}, \phi(v_4)\}, \\
&& L_{W_1}'(v_3)  =  L_{W_1}(v_3) \setminus \{c_{10}, \phi(v_4), \phi(v_5)\}.
\end{eqnarray*}
Then we have that $|L_{W_1}'(v_i)| \geq 2$ for $i \in \{1, 2, 3\}$.  Here if $\phi(v_4) \neq c_1$, then $c_1 \in L_{W_1}'(v_1) \setminus L_{W_1}'(v_2)$.  So, $L_{W_1}'(v_1) \neq L_{W_1}'(v_2)$.
Hence, we can color $v_1, v_2, v_3$ from the list $L_{W_1}'(v_i)$ for $i \in \{1, 2, 3\}$.

Next, if $\phi(v_4) = c_1$, then $|L_{W_1}'(v_2)| \geq 3$ since $c_1 \in L_{W_1}(v_1) \setminus L_{W_1}(v_2)$. So, we can color $v_1, v_2, v_3$ from the list
$L_{W_1}'(v_i)$ for $i \in \{1, 2, 3\}$.
Hence $W_1^2$ admits an $L$-coloring from the list $L_{W_1}(v)$.

\medskip

\noindent
{\bf Subcase 1.2:} $W_1^2$ is not an induced subgraph of $G^2$.

\noindent {\bf Simplifying cases:}
\begin{itemize}
\item
The vertices in each of the following pairs are nonadjacent since it makes a $5$-cycle:
$\{v_2,v_6\}$, $\{v_2,v_{9}\}$, $\{v_4,v_6\}$,
$\{v_4,v_9\}$, $\{v_4,v_{10}\}$, $\{v_6, v_{10}\}$, and $\{v_7, v_{9}\}$.

\item
The vertices in each of the following pairs are nonadjacent since it makes $F_1$
in Figure \ref{key configuration-C3},
which does not exist by Lemma \ref{C3-C6}(a):
$\{v_2,v_4\}$, and $\{v_2,v_{10}\}$.

\item
The vertices in each of the following pairs are nonadjacent since it makes $F_2$
in Figure \ref{key configuration-C3},
which does not exist by Lemma \ref{reducible-F2}:
$\{v_4,v_7\}$.

\item
The vertices in each of the following pairs are nonadjacent since it makes $F_3$
in Figure \ref{key configuration-C3},
which does not exist by Lemma \ref{C3-C6}(b):
$\{v_2,v_7\}$.

\item
The vertices in each of the following pairs are nonadjacent since it makes $H_1$
in Figure \ref{key configuration},
which does not exist by Lemma \ref{reducible-H0}:
$\{v_7,v_{10}\}$.

\item
The vertices in each of the following pairs have no common neighbor outside $W_1$,
since it makes a $5$-cycle:
$\{v_2,v_4\}$, $\{v_2,v_9\}$, $\{v_2,v_{10}\}$, $\{v_4,v_{7}\}$,
$\{v_4,v_{9}\}$, $\{v_4,v_{10}\}$, $\{v_6,v_{7}\}$, $\{v_6,v_{9}\}$,
and $\{v_7,v_{10}\}$.

\item
The vertices in each of the following pairs have no common neighbor outside $W_1$,
since it makes $F_2$ in Figure \ref{key configuration-C3},
which does not exist by Lemma \ref{reducible-F2}:
$\{v_9,v_{10}\}$.

\item
The vertices in each of the following pairs have no common neighbor outside $W_1$,
since it makes $F_3$ in Figure \ref{key configuration-C3},
which does not exist by Lemma \ref{C3-C6}(b):
$\{v_2,v_6\}$.

\item
The vertices in each of the following pairs have no common neighbor outside $W_1$,
since it makes $H_1$ in Figure \ref{key configuration},
which does not exist by Lemma \ref{reducible-H0}:
$\{v_4,v_6\}$, and $\{v_7,v_{9}\}$.

\end{itemize}

Considering these,
for edges in $G^2$ but not in $W_1^2$,
we only need to consider the following cases.

\medskip \noindent
{\bf Subcase 1.2.1:} $v_2$ and $v_7$ have a common neighbor outside $W_1$.

In this case, the number of available colors at vertices of $V(W_1)$ is Subcase 1.2.1 in Figure \ref{7cycle-C34-neighbor}.
If $|L_{W_1}(v_2)| \geq 6$, then greedily color $v_{10}, v_9, v_8, v_7, v_6, v_5, v_4, v_3, v_1, v_2$ in order as in Subcase 1.1.
Next, if $|L_{W_1}(v_2)| = 5$, then we color $v_7$ first, and then color  $v_{10}, v_{9}, v_8, v_6, v_5, v_4$ in order  by the same procedure as Subcase 1.1.
Then  we can show that
the vertices in $W_1$ can be colored from the list $L_{W_1}$
so that we obtain an $L$-coloring in $G^2$.

\medskip \noindent
{\bf Subcase 1.2.2:} $v_6$ and $v_{10}$ have a common neighbor outside $W_1$,
or $v_6$ and $v_{9}$ are adjacent in $G$.

In these case, we can follow the same procedure as Subcase 1.1.  Then can show that
the vertices in $W_1$ can be colored from the list $L_{W_1}$
so that we obtain an $L$-coloring in $G^2$.

This completes the proof of Case 1.
\\

\noindent {\bf Case 2:} $F$ is a $8$-face. \\
Let $W_2$ be the subgraph induced by $\{v_1, \ldots, v_{11}\}$ in $G$
(see Figure \ref{7cycle-C34-color}).

\medskip
\noindent
{\bf Subcase 2.1:} $W_2^2$ is an induced subgraph of $G^2$.

Similarly to Subcase 1.1,
the square of $G - V({W_2})$ has a proper coloring $\phi$ such that $\phi(v) \in L(v)$ for each vertex $v \in V(G) - V({W_2})$.
For each $v_i \in V({W_2})$, we define \[
L_{W_2}(v_i) = L(v_i) \setminus \{\phi(x) : xv_i \in E(G^2) \mbox{ and } x \notin V({W_2})\},
\]
and then we have the following (see $W_2$ in Figure \ref{7cycle-C34-color}).
$$
|L_{W_2}(v_i)| \geq
\begin{cases}
2 & i=10, \\
3 & i=6, 7, 9, 11, \\
4 & i=2,4, \\
5 & i=1, 3, 5, 8.
\end{cases}
$$

If $|L_{W_2}(v_2)| \geq 5$, then  greedily color $v_{11}, v_{10}, v_9, v_8, v_7, v_6, v_5, v_4, v_3, v_1, v_2$ in order.  This is possible since $|L_{W_2}(v_2)| \geq 5$ and $v_2$ has only four neighbors in $W_2^2$.

If $|L_{W_2}(v_2)| = 4$, then there exists a color $c_1 \in L_{W_2}(v_1) \setminus L_{W_2}(v_2)$ since $|L_{W_2}(v_1)| \geq 5$.  Color $v_{10}$ by a color $c_{10} \in L_{W_2}(v_{10}) \setminus \{c_1\}$, and then color $v_{11}$ by a color $c_{11} \in L_{W_2}(v_{11}) \setminus \{c_1, c_{10}\}$. Next, greedily color $v_9, v_8, v_7, v_6, v_5, v_4$ in order, and then color $v_1, v_2, v_3$
by the same argument as Case 1.1 using $c_1 \in L_{W_2}(v_1) \setminus L_{W_2}(v_2)$.
So, $W_2^2$ admits an $L$-coloring from the list $L_{W_2}(v)$.
\\

\begin{figure}[htbp]
  \begin{center}
\begin{tikzpicture}[
  v2/.style={fill=black,minimum size=4pt,ellipse,inner sep=1pt},invis/.style={circle,draw=none,fill=none,inner sep=0pt,minimum size=0pt},
  node distance=1.5cm,scale=0.8
]

 \node[v2] (F2_1) at (1, 1) {};
    \node[v2] (F2_2) at (-1, 2) {};
    \node[v2] (F2_3) at (-1, 0) {};
\node[v2] (F2_w) at (1.5, 3.2) {};
    \node[v2] (F2_4) at (2, 1) {};
    \node[v2] (F2_5) at (3, 2) {};
    \node[v2] (F2_6) at (4, 1) {};
    \node[v2] (F2_7) at (3, 0) {};
    \node[v2] (F2_8) at (0, 1) {};
    \node[invis] (F2_9) at (1, 2) {};
     \node[v2] (F2_10) at (-0, -1) {};
    \node[v2] (F2_11) at (2, -1) {};

    \draw (F2_8) -- (F2_2) -- (F2_3) -- (F2_8);

    \draw (F2_4) -- (F2_5) -- (F2_6) -- (F2_7) -- (F2_4);

    \draw (F2_1) -- (F2_4);
     \draw (F2_2) to[bend left=20] (F2_w)to[bend left=40](F2_6);
    \draw (F2_1) -- (F2_8);
    \draw (F2_1) -- (F2_9);
     \draw (F2_3) to [bend right=20]  (F2_10) to[bend right=20]  (F2_11)to[bend right=20](F2_7);

    \node[font=\scriptsize, xshift=5pt,yshift=5pt] at (F2_1) {$v_4$};
    \node[font=\scriptsize, below] at (F2_1) {\color{blue}4};
    \node[font=\scriptsize, left] at (F2_2) {$v_2$};
    \node[font=\scriptsize, right] at (F2_2) {\color{blue}5};
     \node[font=\scriptsize, left] at (F2_3) {$v_1$};
    \node[font=\scriptsize, right] at (F2_3) {\color{blue}5};
     \node[font=\scriptsize, above] at (F2_4) {$v_5$};
    \node[font=\scriptsize, right] at (F2_4) {\color{blue}5};
     \node[font=\scriptsize, above] at (F2_5) {$v_6$};
    \node[font=\scriptsize, below] at (F2_5) {\color{blue}3};
    \node[font=\scriptsize, right] at (F2_6) {$v_7$};
    \node[font=\scriptsize, left] at (F2_6) {\color{blue}4};
        \node[font=\scriptsize, xshift=4pt,yshift=-6pt] at (F2_7) {$v_8$};
    \node[font=\scriptsize, above] at (F2_7) {\color{blue}5};
\node[font=\scriptsize, above] at (F2_8) {$v_3$};
    \node[font=\scriptsize, below] at (F2_8) {\color{blue}5};
    \node[font=\scriptsize, below] at (F2_10) {$v_{10}$};
    \node[font=\scriptsize, above] at (F2_10) {\color{blue}3};
    \node[font=\scriptsize, below] at (F2_11) {$v_{9}$};
    \node[font=\scriptsize, above] at (F2_11) {\color{blue}3};
       \node[font=\scriptsize] at (1, -2.5) {$W_1$ in (i) in Subcase  1.2.1};
       \node[font=\scriptsize,above] at (F2_w) {$w$};
          \node[font=\scriptsize\bfseries] at (1,0) {$C_7$};
\end{tikzpicture}\hspace{3cm}
\begin{tikzpicture}[
  v2/.style={fill=black,minimum size=4pt,ellipse,inner sep=1pt},invis/.style={circle,draw=none,fill=none,inner sep=0pt,minimum size=0pt},
  node distance=1.5cm,scale=0.8
]

 \node[v2] (F2_1) at (1, 1) {};
    \node[v2] (F2_2) at (-1, 2) {};
    \node[v2] (F2_3) at (-1, 0) {};
\node[v2] (F2_w) at (1.5, 3.2) {};
    \node[v2] (F2_4) at (2, 1) {};
    \node[v2] (F2_5) at (3, 2) {};
    \node[v2] (F2_6) at (4, 1) {};
    \node[v2] (F2_7) at (3, 0) {};
    \node[v2] (F2_8) at (0, 1) {};
    \node[invis] (F2_9) at (1, 2) {};
     \node[v2] (F2_10) at (1, -1.3) {};
    \node[v2] (F2_11) at (2.2, -1) {};
    \node[v2] (F2_12) at (-0.2, -1) {};

    \draw (F2_8) -- (F2_2) -- (F2_3) -- (F2_8);

    \draw (F2_4) -- (F2_5) -- (F2_6) -- (F2_7) -- (F2_4);
\draw (F2_2) to[bend left=20] (F2_w)to[bend left=40](F2_6);
    \draw (F2_1) -- (F2_4);
    \draw (F2_1) -- (F2_8);
    \draw (F2_1) -- (F2_9);

     \draw (F2_3) to [bend right=20]  (F2_12) to[bend right=20]  (F2_10)to[bend right=20](F2_11)to[bend right=20](F2_7);

    \node[font=\scriptsize, xshift=5pt,yshift=5pt] at (F2_1) {$v_4$};
    \node[font=\scriptsize, below] at (F2_1) {\color{blue}4};
    \node[font=\scriptsize, left] at (F2_2) {$v_2$};
    \node[font=\scriptsize, right] at (F2_2) {\color{blue}5};
     \node[font=\scriptsize, left] at (F2_3) {$v_1$};
    \node[font=\scriptsize, right] at (F2_3) {\color{blue}5};
     \node[font=\scriptsize, above] at (F2_4) {$v_5$};
    \node[font=\scriptsize, right] at (F2_4) {\color{blue}5};
     \node[font=\scriptsize, above] at (F2_5) {$v_6$};
    \node[font=\scriptsize, below] at (F2_5) {\color{blue}3};
    \node[font=\scriptsize, right] at (F2_6) {$v_7$};
    \node[font=\scriptsize, left] at (F2_6) {\color{blue}4};
        \node[font=\scriptsize,xshift=4pt,yshift=-6pt] at (F2_7) {$v_8$};
    \node[font=\scriptsize, above] at (F2_7) {\color{blue}5};
\node[font=\scriptsize, above] at (F2_8) {$v_3$};
    \node[font=\scriptsize, below] at (F2_8) {\color{blue}5};
    \node[font=\scriptsize, below] at (F2_10) {$v_{10}$};
    \node[font=\scriptsize, above] at (F2_10) {\color{blue}2};
    \node[font=\scriptsize, below] at (F2_11) {$v_{9}$};
    \node[font=\scriptsize, above] at (F2_11) {\color{blue}3};
    \node[font=\scriptsize, below] at (F2_12) {$v_{11}$};
    \node[font=\scriptsize, above] at (F2_12) {\color{blue}3};
    \node[font=\scriptsize,above] at (F2_w) {$w$};
       \node[font=\scriptsize] at (1.5, -2.5) {$W_2$ in (i) in Subcase  2.2.1 };
   \node[font=\scriptsize\bfseries] at (1,0) {$C_8$};
\end{tikzpicture}
  \end{center}
  \caption{
Subcase 1.2 and Subcase 2.2.
The numbers at vertices are the number of available colors.} \label{7cycle-C34-neighbor}
\end{figure}

\noindent
{\bf Subcase 2.2:} $W_2^2$ is not an induced subgraph of $G^2$.

\noindent {\bf Simplifying cases:}
\begin{itemize}
\item
The vertices in each of the following pairs are nonadjacent since it makes a $5$-cycle:
$\{v_2,v_6\}$, $\{v_2,v_{9}\}$, $\{v_2,v_{10}\}$, $\{v_4,v_6\}$,
$\{v_4,v_{10}\}$, $\{v_4,v_{11}\}$, $\{v_6, v_{10}\}$, $\{v_7, v_{9}\}$, and $\{v_7, v_{11}\}$.

\item
The vertices in each of the following pairs are nonadjacent since it makes $F_1$
in Figure \ref{key configuration-C3},
which does not exist by Lemma \ref{C3-C6}(a):
$\{v_2,v_4\}$, and $\{v_2,v_{11}\}$.

\item
The vertices in each of the following pairs are nonadjacent since it makes $F_2$
in Figure \ref{key configuration-C3},
which does not exist by Lemma \ref{reducible-F2}:
$\{v_4,v_7\}$, and $\{v_9, v_{11}\}$.

\item
The vertices in each of the following pairs are nonadjacent since it makes $F_3$
in Figure \ref{key configuration-C3},
which does not exist by Lemma \ref{C3-C6}(b):
$\{v_2,v_7\}$, $\{v_4,v_9\}$, and $\{v_6,v_{11}\}$.

\item
The vertices in each of the following pairs are nonadjacent since it makes $H_1$
in Figure \ref{key configuration},
which does not exist by Lemma \ref{reducible-H0}:
$\{v_7,v_{10}\}$.

\item
The vertices in each of the following pairs have no common neighbor outside $W_2$,
since it makes a $5$-cycle:
$\{v_2,v_4\}$, $\{v_2,v_{10}\}$, $\{v_2,v_{11}\}$, $\{v_4,v_{7}\}$,
$\{v_4,v_{9}\}$, $\{v_4,v_{11}\}$, $\{v_6,v_{7}\}$, $\{v_6,v_{9}\}$,
and $\{v_7,v_{10}\}$.

\item
The vertices in each of the following pairs have no common neighbor outside $W_2$,
since it makes $F_2$ in Figure \ref{key configuration-C3},
which does not exist by Lemma \ref{reducible-F2}:
$\{v_9,v_{10}\}$, and $\{v_9,v_{11}\}$.

\item
The vertices in each of the following pairs have no common neighbor outside $W_2$,
since it makes $F_3$ in Figure \ref{key configuration-C3},
which does not exist by Lemma \ref{C3-C6}(b):
$\{v_2,v_6\}$, $\{v_2,v_9\}$, $\{v_4,v_{10}\}$.

\item
The vertices in each of the following pairs have no common neighbor outside $W_2$,
since it makes two $3$-cycles of distance $1$,
which does not exist by Lemma \ref{C3-C6}(c):
$\{v_{10},v_{11}\}$.

\item
The vertices in each of the following pairs have no common neighbor outside $W_2$,
since it makes $H_1$ in Figure \ref{key configuration},
which does not exist by Lemma \ref{reducible-H0}:
$\{v_4,v_6\}$, and $\{v_7,v_{9}\}$.

\end{itemize}

Considering these,
we only need to consider the following five cases (i)--(v):

\begin{enumerate}[(i)]
\item $v_2$ and $v_7$ have a common neighbor outside $W_2$ in $G$.

\item
$v_6$ and $v_{10}$ have a common neighbor outside $W_2$ in $G$.

\item
$v_6$ and $v_{11}$ have a common neighbor outside $W_2$ in $G$.

\item
$v_7$ and $v_{11}$ have a common neighbor outside $W_2$ in $G$.

\item
$v_6$ is adjacent to $v_{9}$ in $G$.
\end{enumerate}

\medskip \noindent
{\bf Subcase 2.2.1:} Case $(i)$: $v_2$ and $v_7$ have a common neighbor outside $W_2$ in $G$.

The number of available colors at vertices of $V(W_1)$ is
presented in Subcase 2.2.1 in Figure \ref{7cycle-C34-neighbor}.
We first color $v_{7}$ by a color $c_{7} \in L_{W_2}(v_{7})$
so that $|L_{W_2}(v_{9}) \setminus \{c_7\}| \geq 3$.
If $|L_{W_1}(v_2) \setminus \{c_7\}| \geq 5$, then  greedily color $v_{11}, v_{10}, v_9, v_8, v_6, v_5, v_4, v_3, v_1, v_2$ in order.
If $|L_{W_1}(v_2) \setminus \{c_7\}| = 4$, then choose a color $c_1 \in L_{W_1}(v_1) \setminus (L_{W_1}(v_2) \setminus \{c_7\})$,
and follow the same procedure as Subcase 2.1.
Then,
the vertices in $W_2$ can be colored from the list $L_{W_2}$
so that we obtain an $L$-coloring in $G^2$.  Note that it is possible that $c_1 = c_7$, but we can still follow the same procedure as Subcase 2.1 since we colored $v_7$ first.

\medskip \noindent
{\bf Subcase 2.2.2:} Case $(ii)$--$(v)$.

We follow the same arugment as Subcase 2.1.  Then, we can show that
the vertices in $W_2$ can be colored from the list $L_{W_2}$
so that we obtain an $L$-coloring in $G^2$.
\medskip

This completes the proof of Case 2.
Thus, in either case, $G^2$ admits an $L$-coloring, which is a contradition.
This completes the proof of Lemma \ref{reducible-F4}.
\end{proof}


\begin{corollary}
\label{cor-reducible-F4}
If a $7$-face is adjacent to a $3$-face,
then it is not adjacent to a $4$-face.
And if a $8$-face is adjacent to a $3$-face, then it cannot be adjacent to two $4$-faces.
\end{corollary}
\begin{proof}
If a $7$-face $C_7$ is adjacent to a $3$-face $f_1$ and a $4$-face $f_2$, then the distance between $f_1$ and $f_2$ is at most 2.
However, this contradicts Lemma \ref{reducible-F2} or \ref{reducible-F4}.
So, a $7$-face cannot be adjacent to both of  $3$-face and $4$-face.

If a $8$-face $C_8$ is adjacent to a $3$-face $f_3$ and two $4$-faces, then there exists a $4$-face $f_4$ such that the distance between $f_3$ and $f_4$ is at most 2.  Again,
this contradicts Lemma \ref{reducible-F2} or \ref{reducible-F4}.
So, if a $8$-face is adjacent to a $3$-face, then it cannot be adjacent to two $4$-faces.
\end{proof}

\subsection{Special configurations}
In this subsection, we prove three configurations which are of independent interest.

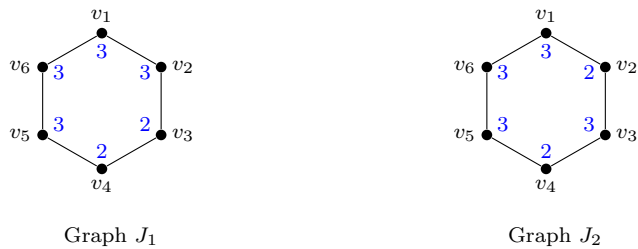
\begin{figure}[htbp]
  \begin{center}
\begin{tikzpicture}[
  v2/.style={fill=black,minimum size=4pt,ellipse,inner sep=1pt},
  node distance=1.5cm,scale=0.45
]
 \node[v2] (F3_1) at (0, 0) {};
    \node[v2] (F3_2) at (0, 2) {};
\node[v2] (F3_4) at (1.732, 3) {};
\node[v2] (F3_5) at (3.464, 2) {};
    \node[v2] (F3_6) at (3.464, 0) {};
    \node[v2] (F3_7) at (1.732, -1) {};

    \draw (F3_1) -- (F3_2);
    \draw (F3_1) -- (F3_7) -- (F3_6) -- (F3_5) -- (F3_4) -- (F3_2);

\node[font=\scriptsize,left] at (F3_1) {$v_5$};
\node[font=\scriptsize,xshift=6pt,yshift=4pt] at (F3_1) {\color{blue}3};
\node[font=\scriptsize,left] at (F3_2) {$v_6$};
\node[font=\scriptsize,xshift=6pt,yshift=-2pt] at (F3_2) {\color{blue}3};
\node[font=\scriptsize,above] at (F3_4) {$v_1$};
\node[font=\scriptsize,below] at (F3_4) {\color{blue}3};
\node[font=\scriptsize,right] at (F3_5) {$v_2$};
\node[font=\scriptsize,xshift=-6pt,yshift=-2pt] at (F3_5) {\color{blue}3};
\node[font=\scriptsize,right] at (F3_6) {$v_3$};
\node[font=\scriptsize,xshift=-6pt,yshift=4pt] at (F3_6) {\color{blue}2};
\node[font=\scriptsize,below] at (F3_7) {$v_4$};
\node[font=\scriptsize,above] at (F3_7) {\color{blue}2};
\node[font=\scriptsize] at (2,-3) { Graph $J_1$};
\end{tikzpicture}\hspace{3cm}
\begin{tikzpicture}[
  v2/.style={fill=black,minimum size=4pt,ellipse,inner sep=1pt},
  node distance=1.5cm,scale=0.45
]
 \node[v2] (F3_1) at (0, 0) {};
    \node[v2] (F3_2) at (0, 2) {};
\node[v2] (F3_4) at (1.732, 3) {};
\node[v2] (F3_5) at (3.464, 2) {};
    \node[v2] (F3_6) at (3.464, 0) {};
    \node[v2] (F3_7) at (1.732, -1) {};

    \draw (F3_1) -- (F3_2);
    \draw (F3_1) -- (F3_7) -- (F3_6) -- (F3_5) -- (F3_4) -- (F3_2);

\node[font=\scriptsize,left] at (F3_1) {$v_5$};
\node[font=\scriptsize,xshift=6pt,yshift=4pt] at (F3_1) {\color{blue}3};
\node[font=\scriptsize,left] at (F3_2) {$v_6$};
\node[font=\scriptsize,xshift=6pt,yshift=-2pt] at (F3_2) {\color{blue}3};
\node[font=\scriptsize,above] at (F3_4) {$v_1$};
\node[font=\scriptsize,below] at (F3_4) {\color{blue}3};
\node[font=\scriptsize,right] at (F3_5) {$v_2$};
\node[font=\scriptsize,xshift=-6pt,yshift=-2pt] at (F3_5) {\color{blue}2};
\node[font=\scriptsize,right] at (F3_6) {$v_3$};
\node[font=\scriptsize,xshift=-6pt,yshift=4pt] at (F3_6) {\color{blue}3};
\node[font=\scriptsize,below] at (F3_7) {$v_4$};
\node[font=\scriptsize,above] at (F3_7) {\color{blue}2};
\node[font=\scriptsize] at (2,-3) { Graph $J_2$};
\end{tikzpicture}
  \end{center}
\caption{
The size of list $L(v_3)$ and $L(v_4)$ are 2, and the other vertices have lists of size 3 in $J_1$.
The size of list $L(v_2)$ and $L(v_4)$ are 2, and the other vertices have lists of size 3 in $J_2$.
}\label{6-cycle-list}
\end{figure}

\begin{lemma} \label{cycle-six-original}
For a 6-cycle $J_1$ with $V(J_1) = \{v_1, v_2, v_3, v_4, v_5, v_6\}$,
if each vertex $v_i$ has a list $L(v_i)$ with $|L(v_i)| = 3$ for $i = 1, 2, 5, 6$, $|L(v_3)| = |L(v_4)|=2$ and $L(v_3) \neq L(v_4)$ (see Figure \ref{6-cycle-list}), then $J_1^2$ has a proper coloring from the list.
\end{lemma}
\begin{proof}
Let $J_1$ be a 6-cycle with $V(J_1) = \{v_1, v_2, v_3, v_4, v_5, v_6\}$ and
suppose that the list $L$ satisfies $|L(v_i)| = 3$ for $i = 1, 2, 5, 6$, $|L(v_3)| = |L(v_4)|=2$ and $L(v_3) \neq L(v_4)$.
We will show that $J_1^2$ has a proper coloring from the list.

\medskip
\noindent {\bf Case 1:} $L(v_1) \cap L(v_4) \neq \emptyset$.

Color $v_1$ and $v_4$ by a color $c \in L(v_1) \cap L(v_4)$.
Let $L'(v_i) = L(v_i) \setminus \{c\}$ for $i = 2, 3, 5, 6$. Then $|L'(v_2)| \geq 2$, $|L'(v_3)| \geq 1$, $|L'(v_5)| \geq 2$, and  $|L'(v_6)| \geq 2$.  In this case, if $|L'(v_3)| \geq 2$, then $v_2, v_3, v_5, v_6$  can be colored properly from the list since  $v_2, v_3, v_5, v_6$ form a 4-cycle in $J_1^2$ and every vertex has a list of size at least 2.  So, we can assume that
$|L'(v_3)| = 1$.
In this case, the only case when $v_2, v_3, v_5, v_6$  cannot be colored properly from the list is essentially as follows for some colors $1,2,3$:
\[
L'(v_2) = \{1, 2\}, L'(v_3) = \{1\}, L'(v_5) = \{1, 3\}, L'(v_6) = \{2, 3\}.
\]
So, the original list of $v_i$ for $i = 1, 2, 3, 4, 5, 6$ is as follows
for some colors $a,b$ and $d$.
\begin{eqnarray*}
& L(v_1) = \{a, b, c\},  & L(v_2) = \{1, 2, c\},  L(v_3) = \{1, c\}, \\
& L(v_4) = \{c, d\}, & L(v_5) = \{1, 3, c\},  L(v_6) = \{2, 3, c\}.
\end{eqnarray*}
Note that $d \neq 1$ since $L(v_3) \neq L(v_4)$.
By the symmetry between $a$ and $b$, we may assume that $a \neq 1$.
Then we can color $v_2, v_3, v_5, v_6$  properly from the list by $\phi$ as \[
\phi(v_1) = a, \phi(v_2) = 1, \phi(v_3) = c, \phi(v_4) = d, \phi(v_5) = 1, \phi(v_6) = c.
\]
This is a proper coloring since $a \neq 1$ and $d \neq 1$.

\medskip
\noindent {\bf Case 2:} $L(v_3) \cap L(v_6) \neq \emptyset$.

By the same augument, we can show that $J_1^2$ has a proper coloring from the list.
\\

\noindent {\bf Case 3:} $L(v_1) \cap L(v_4) = L(v_3) \cap L(v_6) = \emptyset$, and $L(v_2) \cap L(v_5) \neq \emptyset$.

Color $v_2$ and $v_5$ by a color $c \in L(v_2) \cap L(v_5)$.  Let $L'(v_i) = L(v_i) \setminus \{c\}$ for $i = 1, 3, 4, 6$. Then since  $L(v_1) \cap L(v_4) = \emptyset$ and
$L(v_3) \cap L(v_6) = \emptyset$, we have that $|L'(v_1)| = 3$ if $|L'(v_4)| = 1$ or
$|L'(v_1)| \geq 2$ if $|L'(v_4)| = 2$.
Similarly, $|L'(v_6)| = 3$ if $|L'(v_3)| = 1$ or $|L'(v_6)| \geq 2$ if $|L'(v_3)| = 2$.
So,
the sizes of $L'(v)$ are as follows.

\begin{itemize}
\item $|L'(v_1)| = 3$, $|L'(v_3)| = 1$, $|L'(v_4)| = 1$, $|L'(v_6)| = 3$, or
\item $|L'(v_1)| \geq 2$, $|L'(v_3)| = 1$, $|L'(v_4)| = 2$, $|L'(v_6)| = 3$, or \item $|L'(v_1)| = 3$, $|L'(v_3)| = 2$, $|L'(v_4)| = 1$, $|L'(v_6)| \geq 2$, or \item $|L'(v_1)| \geq 2$, $|L'(v_3)| = 2$, $|L'(v_4)| = 2$, $|L'(v_6)| \geq 2$.
\end{itemize}
At each case, $v_2, v_3, v_5, v_6$ are colorable from the list $L'$
since $v_2, v_3, v_5, v_6$ form a 4-cycle in $J_1^2$
and $L'(v_3) \neq L'(v_4)$ in the first case.  This completes the proof Case 3.
\\

\noindent {\bf Case 4:} $L(v_1) \cap L(v_4) = L(v_3) \cap L(v_6) = L(v_2) \cap L(v_5)= \emptyset$.

Let $X = \{v_1, v_2, v_3, v_4, v_5, v_6\}$, and
we define  a  bipartite graph $W$ as follows.
\begin{itemize}
\item $V(W) = X \cup Y$ where $Y = \bigcup_{v_i \in X} L(v_i)$.

\item For each  $v_i$, $v_i \alpha \in E(W)$ if $\alpha \in L(v_i)$.
\end{itemize}
Note that $|N_W(v_1) \cup N_W(v_4)| = |N_W(v_3) \cup N_W(v_6)| = 5$ and $|N_W(v_2) \cup N_W(v_5)| = 6$ since $L(v_j) \cap L(v_{j+3}) = \emptyset$ for each $j = 1, 2, 3$.  Thus, we can easily show that $|N_W(S)| \geq |S|$ for every $S \subset \{v_1, v_2, v_3, v_4, v_5, v_6\}$.  Hence by Hall's theorem,
the bipartite graph $W$ has a matching $M$ that contains all vertices of $\{v_1, v_2, v_3, v_4, v_5, v_6\}$.  So, we can choose six different colors $\{c_1, c_2, c_3, c_4, c_5, c_6\}$ such that $c_i \in L(v_i)$ for each $1 \leq i \leq 6$.  Hence $J_1^2$ can be colored properly from the list.
\end{proof}


Next we prove Lemma \ref{cycle-six-original-second} which will be used for the proof of the reducibility of $H_4$.

\begin{lemma} \label{cycle-six-original-second}
For a 6-cycle $J_2$ with $V(J_2) = \{v_1, v_2, v_3, v_4, v_5, v_6\}$, if each vertex $v_i$ has a list $L(v_i)$ with $|L(v_i)| = 3$ for $i = 1, 3, 5, 6$, $|L(v_2)| = |L(v_4)|=2$ and $L(v_2) \neq L(v_4)$ (see Figure \ref{6-cycle-list}), then $J_2^2$ has a proper coloring from the list.
\end{lemma}
\begin{proof}
Let $J_2$ be a 6-cycle with $V(J_2) = \{v_1, v_2, v_3, v_4, v_5, v_6\}$ and suppose that the list $L$ satisfies $|L(v_i)| = 3$ for $i = 1, 3, 5, 6$, $|L(v_2)| = |L(v_4)|=2$ and $L(v_2) \neq L(v_4)$.

Then we can easily check that $J_2^2$ has a proper coloring from the list if
$L(v_1) \cap L(v_4) \neq \emptyset$ or $L(v_2) \cap L(v_5) \neq \emptyset$ by the same arugment as Case 1 in the proof of Lemma \ref{cycle-six-original}.
And by the same argument as Case 3,
we can show the case $L(v_1) \cap L(v_4) = L(v_2) \cap L(v_5) = \emptyset$, and $L(v_3) \cap L(v_6) \neq \emptyset$. Then, we can show that $J_2^2$ has a proper coloring from the list by the same arugment as Case 4 in Lemma \ref{cycle-six-original}.
\end{proof}

Next, we will prove an important lemma which will be used in the proof for the reducibility of $H_3$.

\begin{lemma} \label{cycle-six}
For a 6-cycle $H$ with $V(H) = \{v_1, v_2, v_3, v_4, v_5, v_6\}$,
if each vertex $v_i$ has a list $L(v_i)$ with $|L(v_i)| = 3$ for $i = 1, 2, 3, 5, 6$ and $ |L(v_4)|=2$ (see Figure \ref{6-cycle-list-two}), then $H^2$ has a proper coloring from the list.
\end{lemma}
\begin{proof}
Since $|L(v_3)|  = 3$ and $|L(v_4)| = 2$, we can remove a color $c$ from $L(v_3)$ so that $L(v_3) \setminus \{c\} \neq L(v_4)$.  So, it holds
by Lemma \ref{cycle-six-original}.
\end{proof}

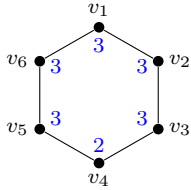
\begin{figure}[htbp]
  \begin{center}
\begin{tikzpicture}[
  v2/.style={fill=black,minimum size=4pt,ellipse,inner sep=1pt},
  node distance=1.5cm,scale=0.45
]
 \node[v2] (F3_1) at (0, 0) {};
    \node[v2] (F3_2) at (0, 2) {};
\node[v2] (F3_4) at (1.732, 3) {};
\node[v2] (F3_5) at (3.464, 2) {};
    \node[v2] (F3_6) at (3.464, 0) {};
    \node[v2] (F3_7) at (1.732, -1) {};

    \draw (F3_1) -- (F3_2);
    \draw (F3_1) -- (F3_7) -- (F3_6) -- (F3_5) -- (F3_4) -- (F3_2);

\node[font=\scriptsize,left] at (F3_1) {$v_5$};
\node[font=\scriptsize,xshift=6pt,yshift=4pt] at (F3_1) {\color{blue}3};
\node[font=\scriptsize,left] at (F3_2) {$v_6$};
\node[font=\scriptsize,xshift=6pt,yshift=-2pt] at (F3_2) {\color{blue}3};
\node[font=\scriptsize,above] at (F3_4) {$v_1$};
\node[font=\scriptsize,below] at (F3_4) {\color{blue}3};
\node[font=\scriptsize,right] at (F3_5) {$v_2$};
\node[font=\scriptsize,xshift=-6pt,yshift=-2pt] at (F3_5) {\color{blue}3};
\node[font=\scriptsize,right] at (F3_6) {$v_3$};
\node[font=\scriptsize,xshift=-6pt,yshift=4pt] at (F3_6) {\color{blue}3};
\node[font=\scriptsize,below] at (F3_7) {$v_4$};
\node[font=\scriptsize,above] at (F3_7) {\color{blue}2};
\end{tikzpicture}
\caption{The size of list $L(v_4)$ are 2, and the other vertices have list of size 3.} \label{6-cycle-list-two}
\end{center}
\end{figure}


\subsection{Subgraph $H_2$ is reducible}

In this subsection,  we will prove that subgraph $H_2$ cannot appear in a minimal couterexample to Theorem \ref{main-thm}.  First, we prove that subgraph $J_3$ and $J_4$ do not appear in $G$.
We consider the case when two 4-cycles share two edges.
In $J_3$, two 4-cycles are adjacent to a 6-face, and in $J_4$, two 4-cycles are adjacent to a 7-face (see Figure \ref{No-two-C4}).

\begin{figure}[htbp]
  \begin{center}
\begin{tikzpicture}[
    v2/.style={fill=black,minimum size=4pt,ellipse,inner sep=1pt},
    scale=0.4
]

    \node[v2] (H2_1) at (0, 0) {};
    \node[v2] (H2_2) at (0, 2) {};
    \node[v2] (H2_3) at (1.732, 3) {};
    \node[v2] (H2_4) at (3.464, 2) {};
    \node[v2] (H2_5) at (3.464, 0) {};
    \node[v2] (H2_6) at (1.732, -1) {};

    \node[v2] (H2_7) at (-2, 0) {};
    \node[v2] (H2_8) at (-2, 2) {};

    \draw (H2_1) -- (H2_2) -- (H2_3) -- (H2_4) -- (H2_5) -- (H2_6) -- (H2_1);

    \draw (H2_1) -- (H2_7) -- (H2_8) -- (H2_2);

    \draw
  (-2, 0)
    .. controls (-5,5.5) and (1,3)..
  (1.732, 3);

\node[font=\scriptsize,below] at (H2_7) { $v_4$};
\node[font=\scriptsize,xshift=5pt,yshift=5pt] at (H2_7) {\color{blue}6};
\node[font=\scriptsize,above] at (H2_8) { $v_1$};
\node[font=\scriptsize,xshift=5pt,yshift=-4pt] at (H2_8) {\color{blue}4};
\node[font=\scriptsize,below] at (H2_1) { $v_5$};
\node[font=\scriptsize,xshift=5pt,yshift=5pt] at (H2_1) {\color{blue}6};
\node[font=\scriptsize,above] at (H2_2) { $v_2$};
\node[font=\scriptsize,xshift=5pt,yshift=-4pt] at (H2_2) {\color{blue}6};
\node[font=\scriptsize,above] at (H2_3) { $v_3$};
\node[font=\scriptsize, below] at (H2_3) {\color{blue}6};
\node[font=\scriptsize,above] at (H2_4) { $v_6$};
\node[font=\scriptsize, xshift=-5pt,yshift=-4pt] at (H2_4) {\color{blue}3};
\node[font=\scriptsize,below] at (H2_5) { $v_7$};
\node[font=\scriptsize, xshift=-5pt,yshift=4pt] at (H2_5) {\color{blue}2};
\node[font=\scriptsize,below] at (H2_6) { $v_8$};
\node[font=\scriptsize, above] at (H2_6) {\color{blue}3};
  \node[font=\scriptsize]  at (1, -3) {Graph $J_3$};

\end{tikzpicture}\hspace{3cm}
\begin{tikzpicture}[
  v2/.style={fill=black,minimum size=4pt,ellipse,inner sep=1pt},
  node distance=1.5cm,scale=0.4]

      \node[v2] (H5_1) at (0, 0){};
      \node[v2] (H5_2) at (0,-2){};
      \node[v2] (H5_3) at (1,-3.5){};
      \node[v2] (H5_4) at (2.8,-3.5){};
      \node[v2] (H5_5) at (4,-2){};
      \node[v2] (H5_6) at (4, 0){};
      \node[v2] (H5_7) at (2, 1.2){};
      \node[v2] (H5_9) at (-2, 0){};
      \node[v2] (H5_10) at (-2,-2){};

   \draw (H5_1)--(H5_2)--(H5_3)--(H5_4)--(H5_5)--(H5_6)--(H5_7)--(H5_1);
   \draw (H5_1)--(H5_9)--(H5_10)--(H5_2);
      \draw
  (-2,-2)
    .. controls (-5,2) and (0,2)..
  (2, 1.2);

  \node[font=\scriptsize,above] at (H5_9) {$v_1$};
  \node[font=\scriptsize,xshift=5pt,yshift=-4pt] at (H5_9) {\color{blue}4};
  \node[font=\scriptsize,above] at (H5_1) {$v_2$};
  \node[font=\scriptsize,xshift=5pt,yshift=-4pt] at (H5_1) {\color{blue}6};
  \node[font=\scriptsize,above] at (H5_7) {$v_3$};
  \node[font=\scriptsize,below] at (H5_7) {\color{blue}6};
  \node[font=\scriptsize,above] at (H5_6) {$v_6$};
  \node[font=\scriptsize,xshift=-5pt,yshift=-4pt] at (H5_6) {\color{blue}3};

  \node[font=\scriptsize,below] at (H5_10) {$v_4$};
  \node[font=\scriptsize,xshift=5pt,yshift= 5pt] at (H5_10) {\color{blue}6};
  \node[font=\scriptsize,below=1pt, xshift=-3pt] at (H5_2) {$v_{5}$};
  \node[font=\scriptsize,xshift=5pt,yshift= 4pt] at (H5_2) {\color{blue}6};
  \node[font=\scriptsize,below] at (H5_3) {$v_9$};
  \node[font=\scriptsize, above] at (H5_3) {\color{blue}3};
  \node[font=\scriptsize,below] at (H5_4) {$v_{8}$};
  \node[font=\scriptsize,above] at (H5_4) {\color{blue}2};
  \node[font=\scriptsize,below=1pt, xshift=3pt] at (H5_5) {$v_{7}$};
  \node[font=\scriptsize,left] at (H5_5) {\color{blue}2};

  \node[font=\scriptsize]  at (2,-5.5) {Graph $J_4$};
  \end{tikzpicture}
  \end{center}
\caption{Graphs $J_3$ and $J_4$. The numbers at vertices are the number of available colors.} \label{No-two-C4}
\end{figure}
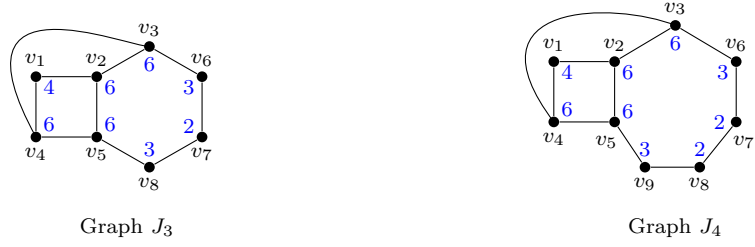

\begin{lemma} \label{C4-share-two-edge}
The subgraphs $J_3$ and $J_4$ in Figure \ref{No-two-C4}
do not appear in $G$.
\end{lemma}
\begin{proof}
Suppose that the subgraphs $J_3$ or $J_4$ in Figure \ref{No-two-C4}  appear in $G$.
Let $L$ be a list assignment with lists of size $7$ for each vertex in $G$.
We will show that $G^2$ has a proper coloring from the list $L$, which is a contradiction for the fact that $G$ is a counterexample to the theorem.

\medskip
\noindent {\bf Case 1:} The subgraph $J_3$.

Let $V(J_3) = \{v_1, v_2, v_3, v_4, v_5, v_6, v_7, v_8\}$ as in Figure \ref{No-two-C4}.
Let $G' = G - V(J_3)$.
Then $G'$ is also a subcubic planar graph and $|V(G')| < |V(G)|$.
Since $G$ is a minimal counterexample to Theorem \ref{main-thm},
the square of $G'$ has a proper coloring $\phi$ such that $\phi(v) \in L(v)$ for each vertex $v \in V(G')$.
Now, for each $v_i \in V(J_3)$, we define \[
L_{J_3}(v_i) = L(v_i) \setminus \{\phi(x) : xv_i \in E(G^2) \mbox{ and } x \notin V(J_3)\}.
\]
Then, we have the following (see the graph $J_3$ in Figure \ref{No-two-C4}).
$$
|L_{J_3}(v_i)| \geq
\begin{cases}
2 & i=7, \\
3 & i=6, 8, \\
4 & i=1, \\
6 & i=2,3, 4, 5.
\end{cases}
$$
Now, we will show that
$J_3^2$ admits a proper coloring from the list $L_{J_3}$.

Since $|L_{J_3}(v_6)| \geq 3$, $|L_{J_3}(v_7)| \geq 2$, $|L_{J_3}(v_8)| \geq 3$, we can color $v_6, v_7, v_8$ by some colors $c_6, c_7, c_8$, respectively, so that $L_{J_3}(v_1) \neq L_{J_3}(v_3) \setminus \{c_6, c_7\}$.
For each $v_i \in \{v_1, v_2, v_3, v_4, v_5\}$, let  $L_{J_3}'(v_i)$ the list of available colors at $v_i$ after coloring  $v_6, v_7, v_8$.  That is,
\begin{eqnarray*}
& & L_{J_3}'(v_1) = L_{J_3}(v_1), \ L_{J_3}'(v_2) = L_{J_3}(v_2) \setminus \{c_6, c_8\},  \
L_{J_3}'(v_3) = L_{J_3}(v_3) \setminus \{c_6, c_7\}, \\
& & L_{J_3}'(v_4) = L_{J_3}(v_4) \setminus \{c_6,  c_8\}, \
\text{and } L_{J_3}'(v_5) = L_{J_3}(v_5) \setminus \{c_7, c_8\}.
\end{eqnarray*}
Then $|L_{J_3}'(v_i)| \geq 4$ for each $v_i \in \{v_1, v_2, v_3, v_4, v_5\}$.
Note that $v_1, v_2, v_3, v_4, v_5$ form a $K_5$ in $J_3^2$ and
$L_{J_3}'(v_1) \neq L_{J_3}(v_3) \setminus \{c_6, c_7\} = L_{J_3}'(v_3)$.
So,  $ v_1, v_2, v_3, v_4, v_5 $ are colored properly from the list $L_{J_3}'$.
Thus $J_3^2$ has an $L_{J_3}$-coloring from the list, and so
$G^2$ has a proper coloring from the list $L$, which is a contradiction.
Hence the subgraph $J_3$ does not appear in $G$.

\medskip
\noindent {\bf Case 2:} The subgraph $J_4$.

Let $V(J_4) = \{v_1, v_2, v_3, v_4, v_5, v_6, v_7, v_8, v_9\}$ as in Figure \ref{No-two-C4}.
Let $G' = G -  V(J_4) $.
Then
the square of $G'$ has a proper coloring $\phi$ such that $\phi(v) \in L(v)$ for each vertex $v \in V(G')$.
Now, for each $v_i \in V(J_4)$, we define \[
L_{J_4}(v_i) = L(v_i) \setminus \{\phi(x) : xv_i \in E(G^2) \mbox{ and } x \notin V(J_4)\}.
\]
Then, we have the following (see the graph $J_4$ in Figure \ref{No-two-C4}).
$$
|L_{J_4}(v_i)| \geq
\begin{cases}
2 & i=7, 8,  \\
3 & i=6, 9, \\
4 & i=1, \\
6 & i=2,3, 4, 5.
\end{cases}
$$
Now, we  show that
$J_4^2$ admits a proper coloring from the list $L_{J_4}$.

Since $|L_{J_4}(v_6)| \geq 3$, $|L_{J_4}(v_7)| \geq 2$, $|L_{J_4}(v_8)| \geq 2$, $|L_{J_4}(v_9)| \geq 3$,  we can color $v_6, v_7, v_8, v_9$ by some colors $c_6, c_7, c_8, c_9$, respectively, so that $L_{J_4}(v_1) \neq L_{J_4}(v_3) \setminus \{c_6, c_7\}$.
By the same way as in Case 1, we obtain that
$J_4^2$ has a proper coloring from the list $L_{J_4}$, and so
$G^2$ has a proper coloring from the list $L$, which is a contradiction.
Hence the subgraph $J_4$ does not appear in $G$.
\end{proof}

Now, we will prove that $H_2$ does not appear in $G$.

\begin{lemma} \label{reducible-H2}
The graph $H_2$ 
does not appear in $G$.
\end{lemma}
\begin{proof}
Suppose  to the contrary that $G$ has $H_2$ as a subgraph, and denote $V(H_2) = \{v_1, \ldots, v_{10}\}$ as in Figure \ref{H2fig}.
Let $L$ be a list assignment with lists of size 7 for each vertex in $G$.
We will show that $G^2$ has a  proper coloring from the list $L$, which is a contradiction for the fact that $G$ is a counterexample to the theorem.

Let $G' = G - V(H_2)$.
Then $G'$ is also a subcubic planar graph and $|V(G')| < |V(G)|$.
Since $G$ is a minimal counterexample to Theorem \ref{main-thm},
the square of $G'$ has a proper coloring $\phi$ such that $\phi(v) \in L(v)$ for each vertex $v \in V(H)$.
For each $v_i \in V(H_2)$, we define
\begin{equation*} \label{list-L}
L_{H_2}(v_i) = L(v_i) \setminus \{\phi(x) : xv_i \in E(G^2) \mbox{ and } x \notin V(H_2)\}.
\end{equation*}

We have the following three cases.
\medskip

\noindent {\bf Case 1:} $H_2^2$ is an induced subgraph of $G^2$. \\

\begin{figure}[htbp]
  \begin{center}
\begin{tikzpicture}[
    v2/.style={fill=black,minimum size=4pt,ellipse,inner sep=1pt},
    scale=0.4
]

    \node[v2] (H2_1) at (0, 0) {};
    \node[v2] (H2_2) at (0, 2) {};
    \node[v2] (H2_3) at (1.732, 3) {};
    \node[v2] (H2_4) at (3.464, 2) {};
    \node[v2] (H2_5) at (3.464, 0) {};
    \node[v2] (H2_6) at (1.732, -1) {};

    \node[v2] (H2_7) at (-2, 0) {};
    \node[v2] (H2_8) at (-2, 2) {};
    \node[v2] (H2_9) at ({1.732 + 1}, {3 + 1.732}) {};
    \node[v2] (H2_10) at ({3.464 + 1}, {2 + 1.732}) {};

    \draw (H2_1) -- (H2_2) -- (H2_3) -- (H2_4) -- (H2_5) -- (H2_6) -- (H2_1);

    \draw (H2_1) -- (H2_7) -- (H2_8) -- (H2_2);

    \draw (H2_3) -- (H2_9) -- (H2_10) -- (H2_4);

    \node[font=\scriptsize,below] at   (H2_7) {$v_5$};
     \node[font=\scriptsize,xshift=5pt,yshift=5pt]  at   (H2_7) {\color{blue}3};
     \node[font=\scriptsize,above] at   (H2_8) {$v_1$};
     \node[font=\scriptsize,xshift=5pt,yshift=-5pt]  at   (H2_8) {\color{blue}3};
          \node[font=\scriptsize,above] at   (H2_2) {$v_2$};
     \node[font=\scriptsize,xshift=5pt,yshift=-5pt]  at   (H2_2) {\color{blue}6};
          \node[font=\scriptsize,xshift=-7pt,yshift=2pt] at   (H2_3) {$v_3$};
     \node[font=\scriptsize,below]  at   (H2_3) {\color{blue}6};
          \node[font=\scriptsize,right] at   (H2_4) {$v_9$};
     \node[font=\scriptsize,xshift=-5pt,yshift=-5pt]  at   (H2_4) {\color{blue}5};
          \node[font=\scriptsize,right] at   (H2_5) {$v_8$};
     \node[font=\scriptsize,xshift=-5pt,yshift=2pt]  at   (H2_5) {\color{blue}3};
          \node[font=\scriptsize,below] at   (H2_6) {$v_7$};
     \node[font=\scriptsize,above]  at   (H2_6) {\color{blue}3};
          \node[font=\scriptsize,below] at   (H2_1) {$v_6$};
     \node[font=\scriptsize,xshift=5pt,yshift=5pt]  at   (H2_1) {\color{blue}5};
          \node[font=\scriptsize,above] at   (H2_9) {$v_4$};
     \node[font=\scriptsize,below]  at   (H2_9) {\color{blue}3};
           \node[font=\scriptsize,right] at   (H2_10) {$v_{10}$};
     \node[font=\scriptsize,xshift=-6pt,yshift=-2pt]  at   (H2_10) {\color{blue}3};

    \node[font=\scriptsize ] at (1.732, -3) {Graph $H_2$};

\end{tikzpicture}\hspace{3cm}
\begin{tikzpicture}[
    v2/.style={fill=black,minimum size=4pt,ellipse,inner sep=1pt},
    scale=0.4
]

    \node[v2] (H2_1) at (0, 0) {};
    \node[v2] (H2_2) at (0, 2) {};
    \node[v2] (H2_3) at (1.732, 3) {};
    \node[v2] (H2_4) at (3.464, 2) {};
    \node[v2] (H2_5) at (3.464, 0) {};
    \node[v2] (H2_6) at (1.732, -1) {};

    \node[v2] (H2_7) at (-2, 0) {};
    \node[v2] (H2_8) at (-2, 2) {};
    \node[v2] (H2_9) at ({1.732 + 1}, {3 + 1.732}) {};
    \node[v2] (H2_10) at ({3.464 + 1}, {2 + 1.732}) {};

    \draw (H2_1) -- (H2_2) -- (H2_3) -- (H2_4) -- (H2_5) -- (H2_6) -- (H2_1);

    \draw (H2_1) -- (H2_7) -- (H2_8) -- (H2_2);

    \draw (H2_3) -- (H2_9) -- (H2_10) -- (H2_4);

    \draw
  (-2, 0)
    .. controls (2, -4) and (7,-2) ..
  ({3.464 + 1}, {2 + 1.732});

     \node[font=\scriptsize,below] at   (H2_7) {$v_5$};
     \node[font=\scriptsize,xshift=5pt,yshift=5pt]  at   (H2_7) {\color{blue}6};
     \node[font=\scriptsize,above] at   (H2_8) {$v_1$};
     \node[font=\scriptsize,xshift=5pt,yshift=-5pt]  at   (H2_8) {\color{blue}4};
          \node[font=\scriptsize,above] at   (H2_2) {$v_2$};
     \node[font=\scriptsize,xshift=5pt,yshift=-5pt]  at   (H2_2) {\color{blue}6};
          \node[font=\scriptsize,xshift=-7pt,yshift=2pt] at   (H2_3) {$v_3$};
     \node[font=\scriptsize,below]  at   (H2_3) {\color{blue}6};
          \node[font=\scriptsize,right] at   (H2_4) {$v_9$};
     \node[font=\scriptsize,xshift=-5pt,yshift=-5pt]  at   (H2_4) {\color{blue}6};
          \node[font=\scriptsize,right] at   (H2_5) {$v_8$};
     \node[font=\scriptsize,xshift=-5pt,yshift=2pt]  at   (H2_5) {\color{blue}3};
          \node[font=\scriptsize,below] at   (H2_6) {$v_7$};
     \node[font=\scriptsize,above]  at   (H2_6) {\color{blue}3};
          \node[font=\scriptsize,below] at   (H2_1) {$v_6$};
     \node[font=\scriptsize,xshift=5pt,yshift=5pt]  at   (H2_1) {\color{blue}6};
          \node[font=\scriptsize,above] at   (H2_9) {$v_4$};
     \node[font=\scriptsize,below]  at   (H2_9) {\color{blue}4};
           \node[font=\scriptsize,right] at   (H2_10) {$v_{10}$};
     \node[font=\scriptsize,xshift=-6pt,yshift=-2pt]  at   (H2_10) {\color{blue}6};

    \node[font=\scriptsize] at (1.732, -3) {Case 2};

\end{tikzpicture}
  \end{center}
\caption{Graph $H_2$ and Case 2 of the proof of Lemma \ref{reducible-H2}. The numbers at vertices are the number of available colors.}
\label{H2fig}
\end{figure}
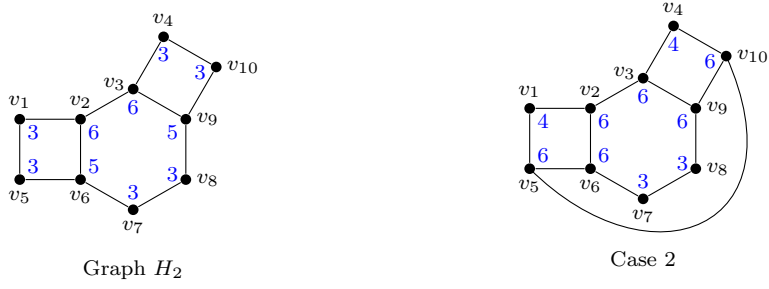

In this case, we have the following (see Figure \ref{H2fig}).
$$
|L_{H_2}(v_i)| \geq
\begin{cases}
3 & i=1,4,5,7,8,10, \\
5 & i=6,9, \\
6 & i=2,3.
\end{cases}
$$
Now, we  show that
$H_2^2$ admits a proper coloring from the list $L_{H_2}$.

First, we color $v_1, v_4, v_5, v_{10}$.  Let $L_{H_2}'(v_i)$ for $i = 2, 3, 6, 7, 8, 9$ be the color list of $v_i$ after coloring  $v_1, v_4, v_5, v_{10}$.
Then we have that $|L_{H_2}'(v_i) | \geq 3$ for $i = 2, 3, 6, 9$ and $|L_{H_2}'(v_j) | \geq 2$ for $j = 7, 8$.  Here note that since $|L_{H_2}(v_5)| \geq 3$ and $|L_{H_2}(v_{10})| \geq 3$, we can color $v_5$ and  $v_{10}$ so that $L_{H_2}'(v_7) \neq L_{H_2}'(v_{8})$.  So, we can complete the coloring of $v_2, v_3, v_6, v_7, v_8, v_9$ by Lemma \ref{cycle-six-original}. This completes the proof of Case 1.
\\


Note that for $X \subset V(G)$, $G[X]$ denotes the subgraph of $G$ induced by $X$.

\medskip

\noindent {\bf Case 2:} $H_2^2$ is not an induced subgraph of $G^2$ and
$E(G[V(H_2)]) - E(H_2) \neq \emptyset$.

\medskip
\noindent {\bf Simplifying cases:}
\begin{itemize}
\item  The vertices in each of the following pairs are nonadjacent
since it makes a 5-cycle:
$\{v_1, v_8\}$, $\{v_1, v_{10}\}$, $\{v_4, v_5\}$, $\{v_4, v_7\}$,
$\{v_5, v_7\}$, and $\{v_8, v_{10}\}$.

\item  The vertices in each of the following pairs are nonadjacent since it makes $H_1$,
which does not exist by Lemma \ref{reducible-H0}:
$\{v_1, v_4\}$, $\{v_5, v_8\}$, and $\{v_7, v_{10}\}$.

\item  The vertices in each of the following pairs are nonadjacent since
 it makes $J_3$, which does not exist by Lemma \ref{C4-share-two-edge}:
$\{v_1, v_7\}$, and $\{v_4, v_{8}\}$.
\end{itemize}
Thus we only need to consider the case when $v_5$ and $v_{10}$ are adjacent.
We have the following (see Case 2 in Figure \ref{H2fig}).
$$
|L_{H_2}(v_i)| \geq
\begin{cases}
3 & i=7, 8, \\
4 & i=1, 4,  \\
6 & i=2, 3, 5, 6, 9, 10.
\end{cases}
$$
Color $v_5$ by a color $c_5 \in L_{H_2}(v_5) \setminus L_{H_2}(v_7)$, and color $v_{10}$ by a color $c_{10} \in L_{H_2}(v_{10}) \setminus L_{H_2}(v_8)$ and $c_{5} \neq c_{10}$.  Next, color $v_1$ and $v_4$ greedily.
Let $L_{H_2}'(v_i)$ be the list of available colors at $v_i \in \{v_2, v_3, v_6, v_7, v_8, v_9\}$ after coloring $v_1, v_4, v_5, v_{10}$.  Then we have the following.
$$
|L_{H_2}'(v_i)| \geq 3 \mbox{ for } i = 2, 3, 6, 7, 8, 9.
$$
Since $ v_2, v_3, v_6, v_7, v_8, v_9 $ form  a 6-cycle,  $ v_2, v_3, v_6, v_7, v_8, v_9 $ can be colored properly from the list $L_{H_2}'$ by Lemma \ref{cycle-six}.  \\

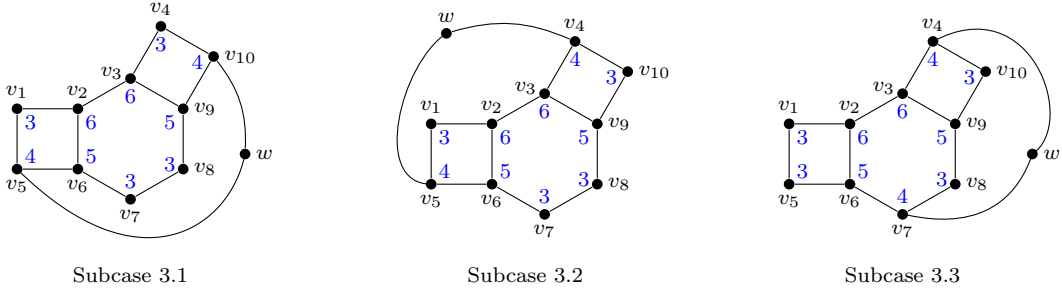
\begin{figure}[htbp]
  \begin{center}
\begin{tikzpicture}[
    v2/.style={fill=black,minimum size=4pt,ellipse,inner sep=1pt},
    scale=0.4
]

    \node[v2] (H2_1) at (0, 0) {};
    \node[v2] (H2_2) at (0, 2) {};
    \node[v2] (H2_3) at (1.732, 3) {};
    \node[v2] (H2_4) at (3.464, 2) {};
    \node[v2] (H2_5) at (3.464, 0) {};
    \node[v2] (H2_6) at (1.732, -1) {};

    \node[v2] (H2_7) at (-2, 0) {};
    \node[v2] (H2_8) at (-2, 2) {};
    \node[v2] (H2_9) at ({1.732 + 1}, {3 + 1.732}) {};
    \node[v2] (H2_10) at ({3.464 + 1}, {2 + 1.732}) {};
    \node[v2] (H2_11) at (5.5,0.5) {};

    \draw (H2_1) -- (H2_2) -- (H2_3) -- (H2_4) -- (H2_5) -- (H2_6) -- (H2_1);

    \draw (H2_1) -- (H2_7) -- (H2_8) -- (H2_2);

    \draw (H2_3) -- (H2_9) -- (H2_10) -- (H2_4);
    \draw (H2_11)to[bend right=20] (H2_10);

     \draw (-2, 0).. controls (2, -4) and (5,-2) ..(5.5,0.5);

    \node[font=\scriptsize,below] at   (H2_7) {$v_5$};
     \node[font=\scriptsize,xshift=5pt,yshift=5pt]  at   (H2_7) {\color{blue}4};
     \node[font=\scriptsize,above] at   (H2_8) {$v_1$};
     \node[font=\scriptsize,xshift=5pt,yshift=-5pt]  at   (H2_8) {\color{blue}3};
          \node[font=\scriptsize,above] at   (H2_2) {$v_2$};
     \node[font=\scriptsize,xshift=5pt,yshift=-5pt]  at   (H2_2) {\color{blue}6};
          \node[font=\scriptsize,xshift=-7pt,yshift=2pt] at   (H2_3) {$v_3$};
     \node[font=\scriptsize,below]  at   (H2_3) {\color{blue}6};
          \node[font=\scriptsize,right] at   (H2_4) {$v_9$};
     \node[font=\scriptsize,xshift=-5pt,yshift=-5pt]  at   (H2_4) {\color{blue}5};
          \node[font=\scriptsize,right] at   (H2_5) {$v_8$};
     \node[font=\scriptsize,xshift=-5pt,yshift=2pt]  at   (H2_5) {\color{blue}3};
          \node[font=\scriptsize,below] at   (H2_6) {$v_7$};
     \node[font=\scriptsize,above]  at   (H2_6) {\color{blue}3};
          \node[font=\scriptsize,below] at   (H2_1) {$v_6$};
     \node[font=\scriptsize,xshift=5pt,yshift=5pt]  at   (H2_1) {\color{blue}5};
          \node[font=\scriptsize,above] at   (H2_9) {$v_4$};
     \node[font=\scriptsize,below]  at   (H2_9) {\color{blue}3};
           \node[font=\scriptsize,right] at   (H2_10) {$v_{10}$};
     \node[font=\scriptsize,xshift=-6pt,yshift=-2pt]  at   (H2_10) {\color{blue}4};
            \node[font=\scriptsize,right] at   (H2_11) {$w$};

    \node[font=\scriptsize ] at (1.732, -3.5) {Subcase 3.1};

\end{tikzpicture}\hspace{1cm}
\begin{tikzpicture}[
    v2/.style={fill=black,minimum size=4pt,ellipse,inner sep=1pt},
    scale=0.4
]

    \node[v2] (H2_1) at (0, 0) {};
    \node[v2] (H2_2) at (0, 2) {};
    \node[v2] (H2_3) at (1.732, 3) {};
    \node[v2] (H2_4) at (3.464, 2) {};
    \node[v2] (H2_5) at (3.464, 0) {};
    \node[v2] (H2_6) at (1.732, -1) {};

    \node[v2] (H2_7) at (-2, 0) {};
    \node[v2] (H2_8) at (-2, 2) {};
    \node[v2] (H2_9) at ({1.732 + 1}, {3 + 1.732}) {};
    \node[v2] (H2_10) at ({3.464 + 1}, {2 + 1.732}) {};

 \node[v2] (H2_11) at (-1.5,5) {};
    \draw (H2_1) -- (H2_2) -- (H2_3) -- (H2_4) -- (H2_5) -- (H2_6) -- (H2_1);

    \draw (H2_1) -- (H2_7) -- (H2_8) -- (H2_2);

    \draw (H2_3) -- (H2_9) -- (H2_10) -- (H2_4);

     \draw (H2_11)to[bend left=20] (H2_9);

     \draw (-1.5,5).. controls (-3,4) and (-4,0) ..(-2, 0);

    \node[font=\scriptsize,below] at   (H2_7) {$v_5$};
     \node[font=\scriptsize,xshift=5pt,yshift=5pt]  at   (H2_7) {\color{blue}4};
     \node[font=\scriptsize,above] at   (H2_8) {$v_1$};
     \node[font=\scriptsize,xshift=5pt,yshift=-5pt]  at   (H2_8) {\color{blue}3};
          \node[font=\scriptsize,above] at   (H2_2) {$v_2$};
     \node[font=\scriptsize,xshift=5pt,yshift=-5pt]  at   (H2_2) {\color{blue}6};
          \node[font=\scriptsize,xshift=-7pt,yshift=2pt] at   (H2_3) {$v_3$};
     \node[font=\scriptsize,below]  at   (H2_3) {\color{blue}6};
          \node[font=\scriptsize,right] at   (H2_4) {$v_9$};
     \node[font=\scriptsize,xshift=-5pt,yshift=-5pt]  at   (H2_4) {\color{blue}5};
          \node[font=\scriptsize,right] at   (H2_5) {$v_8$};
     \node[font=\scriptsize,xshift=-5pt,yshift=2pt]  at   (H2_5) {\color{blue}3};
          \node[font=\scriptsize,below] at   (H2_6) {$v_7$};
     \node[font=\scriptsize,above]  at   (H2_6) {\color{blue}3};
          \node[font=\scriptsize,below] at   (H2_1) {$v_6$};
     \node[font=\scriptsize,xshift=5pt,yshift=5pt]  at   (H2_1) {\color{blue}5};
          \node[font=\scriptsize,above] at   (H2_9) {$v_4$};
     \node[font=\scriptsize,below]  at   (H2_9) {\color{blue}4};
           \node[font=\scriptsize,right] at   (H2_10) {$v_{10}$};
     \node[font=\scriptsize,xshift=-6pt,yshift=-2pt]  at   (H2_10) {\color{blue}3};
     \node[font=\scriptsize,above]  at   (H2_11) {$w$};

    \node[font=\scriptsize] at (1.1, -3) {Subcase 3.2};

\end{tikzpicture}\hspace{1cm}
\begin{tikzpicture}[
    v2/.style={fill=black,minimum size=4pt,ellipse,inner sep=1pt},
    scale=0.4
]

    \node[v2] (H2_1) at (0, 0) {};
    \node[v2] (H2_2) at (0, 2) {};
    \node[v2] (H2_3) at (1.732, 3) {};
    \node[v2] (H2_4) at (3.464, 2) {};
    \node[v2] (H2_5) at (3.464, 0) {};
    \node[v2] (H2_6) at (1.732, -1) {};

    \node[v2] (H2_7) at (-2, 0) {};
    \node[v2] (H2_8) at (-2, 2) {};
    \node[v2] (H2_9) at ({1.732 + 1}, {3 + 1.732}) {};
    \node[v2] (H2_10) at ({3.464 + 1}, {2 + 1.732}) {};
    \node[v2] (H2_11) at (6,1) {};

    \draw (H2_1) -- (H2_2) -- (H2_3) -- (H2_4) -- (H2_5) -- (H2_6) -- (H2_1);

    \draw (H2_1) -- (H2_7) -- (H2_8) -- (H2_2);

    \draw (H2_3) -- (H2_9) -- (H2_10) -- (H2_4);

\draw ({1.732 + 1}, {3 + 1.732}).. controls (6, 6.5) and (7.5,2) ..(6,1);
     \draw (1.732, -1).. controls (2, -1) and (5,-2) ..(6,1);

    \node[font=\scriptsize,below] at   (H2_7) {$v_5$};
     \node[font=\scriptsize,xshift=5pt,yshift=5pt]  at   (H2_7) {\color{blue}3};
     \node[font=\scriptsize,above] at   (H2_8) {$v_1$};
     \node[font=\scriptsize,xshift=5pt,yshift=-5pt]  at   (H2_8) {\color{blue}3};
          \node[font=\scriptsize,above] at   (H2_2) {$v_2$};
     \node[font=\scriptsize,xshift=5pt,yshift=-5pt]  at   (H2_2) {\color{blue}6};
          \node[font=\scriptsize,xshift=-7pt,yshift=2pt] at   (H2_3) {$v_3$};
     \node[font=\scriptsize,below]  at   (H2_3) {\color{blue}6};
          \node[font=\scriptsize,right] at   (H2_4) {$v_9$};
     \node[font=\scriptsize,xshift=-5pt,yshift=-5pt]  at   (H2_4) {\color{blue}5};
          \node[font=\scriptsize,right] at   (H2_5) {$v_8$};
     \node[font=\scriptsize,xshift=-5pt,yshift=2pt]  at   (H2_5) {\color{blue}3};
          \node[font=\scriptsize,below] at   (H2_6) {$v_7$};
     \node[font=\scriptsize,above]  at   (H2_6) {\color{blue}4};
          \node[font=\scriptsize,below] at   (H2_1) {$v_6$};
     \node[font=\scriptsize,xshift=5pt,yshift=5pt]  at   (H2_1) {\color{blue}5};
          \node[font=\scriptsize,above] at   (H2_9) {$v_4$};
     \node[font=\scriptsize,below]  at   (H2_9) {\color{blue}4};
           \node[font=\scriptsize,right] at   (H2_10) {$v_{10}$};
     \node[font=\scriptsize,xshift=-6pt,yshift=-2pt]  at   (H2_10) {\color{blue}3};
            \node[font=\scriptsize,right] at   (H2_11) {$w$};

    \node[font=\scriptsize ] at (1.732, -3) {Subcase 3.3};

\end{tikzpicture}
\end{center}
\caption{Subcases 3.1, 3.2, and 3.3 of the proof of Lemma \ref{reducible-H2}. The numbers at vertices are the number of available colors.} \label{H2-adjacent-nbr}
\end{figure}

\medskip
\noindent {\bf Case 3:} $H_2^2$ is not an induced subgraph of $G^2$ and
 $E(G[V(H_2)]) - E(H_2) =  \emptyset$.

\medskip
\noindent {\bf Simplifying cases:}
\begin{itemize}
\item
 The vertices in each of the following pairs do not have a common neighbor outside $H_2$,
since it makes a 5-cycle:
$\{v_1, v_4\}$, $\{v_1, v_5\}$, $\{v_1, v_7\}$, $\{v_4, v_8\}$, $\{v_4, v_{10}\}$, $\{v_5, v_8\}$, and $\{v_7, v_{10}\}$.

\item  The vertices in each of the following pairs do not have a common neighbor outside $H_2$,
since it makes the subgraph $F_3$ in Figure \ref{key configuration-C3},
which does not exist by Lemma \ref{C3-C6}(b):
$\{v_7,v_8\}$.

\item
The vertices in each of the following pairs do not have a common neighbor outside $H_2$,
since it makes $H_1$ in Figure \ref{key configuration},
which does not exist by Lemma \ref{reducible-H0}:
$\{v_5,v_7\}$, and $\{v_8,v_{10}\}$.
\end{itemize}

\noindent
Considering these,
we only need to deal with the following three subcases.

\medskip
\noindent {\bf Subcase 3.1:} $v_5$ and $v_{10}$ have a common neighbor $w$,
where $w \notin V(H_2)$ (see Subcase 3.1 in Figure \ref{H2-adjacent-nbr}).

Note that
$w$ is not adjacent to any vertex in $V(H_2) \setminus \{v_5, v_{10}\}$ by the argument of  the Simplifying cases.
The number of available colors at vertices of $V(H_2)$ is presented in Subcase 3.1 of Figure \ref{H2-adjacent-nbr}.
Color $v_5$ by a color in $L_{H_2}(v_5) \setminus L_{H_2}(v_7)$ and  greedily color $v_{10}, v_1, v_4$ in order.
Let $L_{H_2}'(v_i)$ be the list of available colors at $v_i \in \{v_2, v_3, v_6, v_7, v_8, v_9\}$ after coloring $v_1, v_4, v_5, v_{10}$. Then we have the following.
$$
|L_{H_2}'(v_i)| \geq 3 \mbox{ for } i = 2, 3, 6, 7, 9, \mbox{ and } |L'(v_8)| \geq 2.
$$
Since $ v_2, v_3, v_6, v_7, v_8, v_9 $ form  a 6-cycle,  $ v_2, v_3, v_6, v_7, v_8, v_9 $ can be colored properly from the list $L_{H_2}'$ by Lemma \ref{cycle-six}.
\\

\noindent {\bf Subcase 3.2:}
$v_4$ and $v_{5}$ (or $v_1$ and $v_{10}$ by symmetry)
have a common neighbor $w$,
where $w \notin V(H_2)$.

Note that
$w$ is not adjacent to any vertex in $V(H_2) \setminus \{v_4, v_5\}$ by the argument of he Simplifying cases.
Then the number of available colors at vertices of $V(H_2)$ is presented in Subcase 3.2 of Figure \ref{H2-adjacent-nbr}.
In this case, we follow the same argument as Subcase 3.1.
Color $v_5$ by a color in $L_{H_2}(v_5) \setminus L_{H_2}(v_7)$ and  greedily color $v_{10}, v_1, v_4$ in order.
Let $L_{H_2}'(v_i)$ be the list of available colors at $v_i \in \{v_2, v_3, v_6, v_7, v_8, v_9\}$ after coloring $v_1, v_4, v_5, v_{10}$.  Then we have the following.
$$
|L_{H_2}'(v_i)| \geq 3 \mbox{ for } i = 2, 3, 6, 7, 9, \mbox{ and } |L_{H_2}'(v_8)| \geq 2.
$$
Since $ v_2, v_3, v_6, v_7, v_8, v_9 $ form a 6-cycle,  $ v_2, v_3, v_6, v_7, v_8, v_9$ can be colored properly from the list $L_{H_2}'$ by Lemma \ref{cycle-six}.
\\

\noindent {\bf Subcase 3.3:} $v_4$ and $v_{7}$
(or $v_1$ and $v_{8}$ by symmetry)
have a common neighbor $w$,
where $w \notin V(H_2)$.

Note that
$w$ is not adjacent to any vertex in $V(H_2) \setminus \{v_4, v_7\}$ by the argument of he Simplifying cases.
Then the number of available colors at vertices of $V(H_2)$ is presented in Subcase 3.3 of Figure \ref{H2-adjacent-nbr}.
In this case, first we color $v_1$ and $v_5$ greedily.
We may assume that $c_5$ is the color assigned at $v_5$.
Next, color $v_4$ by a color $c_4$ so that $|L_{H_2}(v_7) \setminus\{c_4, c_5\}| \geq 3$. This is possible since $|L_{H_2}(v_4)|  \geq 4$.
And then color $v_{10}$.
Let $L_{H_2}'(v_i)$ be the list of available colors at $v_i \in \{v_2, v_3, v_6, v_7, v_8, v_9\}$ after coloring $v_1, v_4, v_5, v_{10}$.  Then we have the following.
$$
|L_{H_2}'(v_i)| \geq 3 \mbox{ for } i = 2, 3, 6, 7, 9, \mbox{ and } |L_{H_2}'(v_8)| \geq 2.
$$
Since $\{v_2, v_3, v_6, v_7, v_8, v_9\}$ form a 6-cycle,  $\{v_2, v_3, v_6, v_7, v_8, v_9\}$ can be colored properly from the list $L_{H_2}'$ by Lemma \ref{cycle-six}.

\medskip
So,
in either case,
the vertices in $H_2$ can be colored from the list $L_{H_2}$
so that we obtain an $L$-coloring in $G^2$.
This is a contradiction for the fact that $G$ is a counterexample.  So, $G$ has no $H_2$.
\end{proof}

\subsection{Subgraph $H_3$ is reducible}

In this subsection, we will prove that  graph $H_3$ in Figure \ref{key configuration} does not appear in a minimal counterexample.

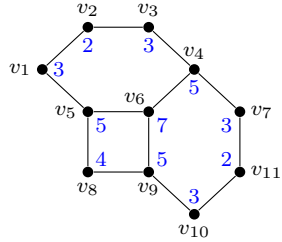
\begin{figure}[htbp]
  \begin{center}

\begin{tikzpicture}[
    v2/.style={fill=black,minimum size=4pt,ellipse,inner sep=1pt},
    scale=0.4
]
    \node[v2] (H3_1) at (0, 0) {};
    \node[v2] (H3_2) at (0, 2) {};
    \node[v2] (H3_3) at (1.5, 3.4) {};
    \node[v2] (H3_4) at (3, 2) {};
    \node[v2] (H3_5) at (3, 0) {};
    \node[v2] (H3_6) at (1.5, -1.4) {};
    \node[v2] (H3_7) at (-2, 0) {};
    \node[v2] (H3_8) at (-2, 2) {};
    \node[v2] (H3_9)  at (0, 4.8) {};
    \node[v2] (H3_10) at (-2,4.8) {};
    \node[v2] (H3_11) at (-3.5, 3.4){};

\draw(H3_1)--(H3_2)--(H3_3)--(H3_4)--(H3_5)--(H3_6)--(H3_1);
\draw (H3_2)--(H3_8)--(H3_7)--(H3_1);
        \draw (H3_3) -- (H3_9) -- (H3_10) -- (H3_11) -- (H3_8);

    \node[font=\scriptsize,below] at (H3_1) {$v_9$};
\node[font=\scriptsize,xshift=5pt,yshift= 5pt] at (H3_1) { \color{blue}5};
     \node[font=\scriptsize,xshift=-5pt,yshift=5pt] at (H3_2) {$v_6$};
\node[font=\scriptsize,xshift=5pt,yshift=-5pt] at (H3_2) { \color{blue}7};
     \node[font=\scriptsize,above] at (H3_3) {$v_4$};
\node[font=\scriptsize,below] at (H3_3) { \color{blue}5};
     \node[font=\scriptsize,right] at (H3_4) {$v_7$};
\node[font=\scriptsize,xshift=-5pt,yshift=-5pt] at (H3_4) { \color{blue}3};
     \node[font=\scriptsize,right] at (H3_5) {$v_{11}$};
\node[font=\scriptsize,xshift=-5pt,yshift= 5pt] at (H3_5) { \color{blue}2};
     \node[font=\scriptsize,below] at (H3_6) {$v_{10}$};
\node[font=\scriptsize,above] at (H3_6) { \color{blue}3};
     \node[font=\scriptsize,below] at (H3_7) {$v_8$};
\node[font=\scriptsize,xshift=5pt,yshift=5pt] at (H3_7) { \color{blue}4};
     \node[font=\scriptsize,left] at (H3_8) {$v_5$};
\node[font=\scriptsize,xshift=5pt,yshift=-5pt] at (H3_8) { \color{blue}5};
     \node[font=\scriptsize,above] at (H3_9) {$v_3$};
\node[font=\scriptsize,below] at (H3_9) { \color{blue}3};
     \node[font=\scriptsize,above] at (H3_10) {$v_2$};
\node[font=\scriptsize,below] at (H3_10) { \color{blue}2};
     \node[font=\scriptsize,left] at (H3_11) {$v_1$};
\node[font=\scriptsize,right] at (H3_11) { \color{blue}3};
\end{tikzpicture}
\end{center}
\caption{Graph $H_3$. The numbers at vertices are the number of available colors.}
\label{H3fig}
\end{figure}

\begin{lemma} \label{reducible-H3}
The graph $H_3$
does not appear in $G$.
\end{lemma}
\begin{proof}
Suppose that $G$ has $H_3$ as a subgraph, and denote $V(H_3) = \{v_1, \ldots, v_{11}\}$ (Figure \ref{H3fig}).
Let $L$ be a list assignment with lists of size 7 for each vertex in $G$.
We will show that $G^2$ has a proper coloring from the list $L$, which is a contradiction for the fact that $G$ is a counterexample to the theorem.

Let $G' = G - V(H_3)$.
Then $G'$ is also a subcubic planar graph and $|V(G')| < |V(G)|$.
Since $G$ is a minimal counterexample to Theorem \ref{main-thm},
the square of $G'$ has a proper coloring $\phi$ such that $\phi(v) \in L(v)$ for each vertex $v \in V(H)$.
For each $v_i \in V(H_3)$, we define \[
L_{H_3}(v_i) = L(v_i) \setminus \{\phi(x) : xv_i \in E(G^2) \mbox{ and } x \notin V(H_3)\}.
\]
We have the following three cases.

\medskip
\noindent {\bf Case 1:} $H_3^2$ is an induced subgraph of $G^2$.

In this case, we have the following (see Figure \ref{H3fig}).
$$
|L_{H_3}(v_i)| \geq
\begin{cases}
2 & i=2,11, \\
3 & i=1,3,7,10, \\
4 & i=8, \\
5 & i=4,5,9, \\
7 & i=6.
\end{cases}
$$
Now, we  show that
$H_3^2$ admits an $L$-coloring from the list $L_{H_3}$.
Observe that $H_3^2$ has 30 edges.  And the graph polynomial for $H_3^2$ is as follows.
\begin{eqnarray*} \label{poly-H3}
P_{H_3^2}(\bm{x})
&=&
(x_1-x_2)(x_1-x_3)(x_1-x_5)(x_1-x_6)(x_1-x_8)(x_2-x_3)(x_2-x_4)(x_2-x_5)
\\
&&
(x_3-x_4)(x_3-x_6)(x_3-x_7)(x_4-x_5)(x_4-x_6)(x_4-x_7)(x_4-x_9)(x_4-x_{11})
\\
&&
(x_5-x_6)(x_5-x_8)(x_5-x_9)(x_6-x_7)(x_6-x_8)(x_6-x_9)(x_6-x_{10})(x_7-x_{10})
\\
&&
(x_7-x_{11})(x_8-x_9)(x_8-x_{10})(x_9-x_{10})(x_9-x_{11})(x_{10}-x_{11}).
\end{eqnarray*}
By the calculation using Mathematica,
the coefficient of
$x_1^2x_2^1x_3^2x_4^4x_5^4x_6^5x_7^2x_8^3x_9^4x_{10}^2x_{11}^1$ is 2, which
is nonzero.
Thus,
by Theorem \ref{cnull},
$H_3^2$ admits an $L$-coloring from its list.
This gives an $L$-coloring for $G^2$.  This is a contradiction for the fact that $G$ is a counterexample.
\\

\noindent {\bf Case 2:} $H_3^2$ is not an induced subgraph of $G^2$ and
 $E(G[V(H_3)]) - E(H_3) \neq \emptyset$.

\medskip
\noindent {\bf Simplifying cases:}
\begin{itemize}
\item  The vertices in each of the following pairs are nonadjacent,
since it makes a 5-cycle:
$\{v_1, v_3\}$, $\{v_1, v_{7}\}$, $\{v_1, v_8\}$, $\{v_1, v_{10}\}$,
$\{v_2, v_{11}\}$, $\{v_3, v_{8}\}$, $\{v_3, v_{10}\}$, $\{v_7, v_{8}\}$,
$\{v_7, v_{10}\}$,
and $\{v_8, v_{10}\}$.

\item The vertices in each of the following pairs are nonadjacent since it makes $H_1$,
which does not exist by Lemma \ref{reducible-H0}:
$\{v_2, v_8\}$, and $\{v_8, v_{11}\}$.

\item $v_3$ and $v_{7}$ cannot be adjacent in $G$ since
 it makes $F_3$, which does not exist by Lemma \ref{C3-C6} (b).

\item $v_3$ and $v_{11}$ cannot be adjacent in $G$ since
 it makes $H_2$,
consisting of the 4-cycles $v_3v_4v_7v_{11}v_3$ and $v_5v_6v_9v_8v_5$, and the 6-cycle $v_1v_2v_3v_4v_6v_5v_1$ (see the case $v_3v_{11} \in E(G)$ in Figure \ref{H3-adjacent}),
a contradiction to Lemma \ref{reducible-H2}.
(In fact, it corresponds to Subcase 3.2 in the proof of Lemma \ref{reducible-H2}, where $v_{10}$ corresponds to the common neighbor $w$.)  So, if $v_3$ and $v_{11}$ are adjacent, then it is reducible by Lemma \ref{reducible-H2}.

Symmetrically,
$v_2$ and $v_7$ cannot be adjacent in $G$.
\end{itemize}

Thus,
either
$v_1$ and $v_{11}$ are adjacent
or
$v_2$ and $v_{10}$ are adjacent.
By symmetry, we only need to consider the latter case.
We have the following
(see the case $v_2v_{10} \in E(G)$ in Figure \ref{H3-adjacent}).
$$
|L_{H_3}(v_i)| \geq
\begin{cases}
3 & i=7, 11, \\
4 & i=1, 3, 8,  \\
5 & i=2, 4,5, \\
6 & i=9, 10, \\
7 & i=6.
\end{cases}
$$
We first color $v_3$ by a color $c_3 \in L_{H_3}(v_3)$
so that $|L_{H_3}(v_{7}) \setminus \{c_3\}| \geq 3$,
and color $v_2$ by a color $c_2 \in L_{H_3}(v_2)\setminus \{c_3\}$
so that $|L_{H_3}(v_{11}) \setminus \{c_2\}| \geq 3$.
Next,
color $v_5$ by a color $c_5 \in L_{H_3}(v_5) \setminus \{c_2\}$
so that $|L_{H_3}(v_4) \setminus \{c_2, c_3, c_5\}| \geq 3$,
and then greedily color $v_1$ and $v_8$ in order.
Let $L_{H_3}'(v_i)$ be the list of available colors
for $i \in \{4,6,7,9,10,11\}$
after coloring $v_1, v_2, v_3, v_5, v_8$.
We have the following.
$$
|L_{H_3}'(v_i)| \geq
\begin{cases}
3 & i=4, 6, 7, 9, 11,  \\
2& i= 10.
\end{cases}
$$
Note that $ v_4, v_6, v_7, v_9, v_{10}, v_{11}$ form  a 6-cycle.  So, $v_4, v_6, v_7, v_9, v_{10}$, and $v_{11}$ can be colored properly in $H_3^2$ from the list by Lemma \ref{cycle-six}.

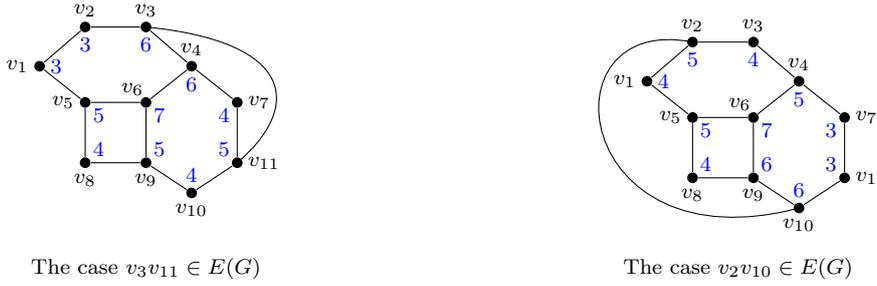
\begin{figure}[htbp]
  \begin{center}

\begin{tikzpicture}[
    v2/.style={fill=black,minimum size=4pt,ellipse,inner sep=1pt},
    scale=0.4
]
    \node[v2] (H3_1) at (0, 0) {};
    \node[v2] (H3_2) at (0, 2) {};
    \node[v2] (H3_3) at (1.5, 3.2) {};
    \node[v2] (H3_4) at (3, 2) {};
    \node[v2] (H3_5) at (3, 0) {};
    \node[v2] (H3_6) at (1.5, -1) {};
    \node[v2] (H3_7) at (-2, 0) {};
    \node[v2] (H3_8) at (-2, 2) {};
    \node[v2] (H3_9)  at (0, 4.5) {};
    \node[v2] (H3_10) at (-2,4.5) {};
    \node[v2] (H3_11) at (-3.5, 3.2){};

\draw(H3_1)--(H3_2)--(H3_3)--(H3_4)--(H3_5)--(H3_6)--(H3_1);
\draw (H3_2)--(H3_8)--(H3_7)--(H3_1);
        \draw (H3_3) -- (H3_9) -- (H3_10) -- (H3_11) -- (H3_8);

        \draw(0, 4.5) .. controls (5,4) and (5.1,1.8) .. (3, 0);

    \node[font=\scriptsize,below] at (H3_1) {$v_9$};
\node[font=\scriptsize,xshift=5pt,yshift= 5pt] at (H3_1) { \color{blue}5};
     \node[font=\scriptsize,xshift=-5pt,yshift=5pt] at (H3_2) {$v_6$};
\node[font=\scriptsize,xshift=5pt,yshift=-5pt] at (H3_2) { \color{blue}7};
     \node[font=\scriptsize,above] at (H3_3) {$v_4$};
\node[font=\scriptsize,below] at (H3_3) { \color{blue}6};
     \node[font=\scriptsize,right] at (H3_4) {$v_7$};
\node[font=\scriptsize,xshift=-5pt,yshift=-5pt] at (H3_4) { \color{blue}4};
     \node[font=\scriptsize,right] at (H3_5) {$v_{11}$};
\node[font=\scriptsize,xshift=-5pt,yshift= 5pt] at (H3_5) { \color{blue}5};
     \node[font=\scriptsize,below] at (H3_6) {$v_{10}$};
\node[font=\scriptsize,above] at (H3_6) { \color{blue}4};
     \node[font=\scriptsize,below] at (H3_7) {$v_8$};
\node[font=\scriptsize,xshift=5pt,yshift=5pt] at (H3_7) { \color{blue}4};
     \node[font=\scriptsize,left] at (H3_8) {$v_5$};
\node[font=\scriptsize,xshift=5pt,yshift=-5pt] at (H3_8) { \color{blue}5};
     \node[font=\scriptsize,above] at (H3_9) {$v_3$};
\node[font=\scriptsize,below] at (H3_9) { \color{blue}6};
     \node[font=\scriptsize,above] at (H3_10) {$v_2$};
\node[font=\scriptsize,below] at (H3_10) { \color{blue}3};
     \node[font=\scriptsize,left] at (H3_11) {$v_1$};
\node[font=\scriptsize,right] at (H3_11) { \color{blue}3};
\node[font=\scriptsize] at (-0,-3.5) {The case $v_3v_{11} \in E(G)$};
\end{tikzpicture}\hspace{3cm}
\begin{tikzpicture}[
    v2/.style={fill=black,minimum size=4pt,ellipse,inner sep=1pt},
    scale=0.4
]
    \node[v2] (H3_1) at (0, 0) {};
    \node[v2] (H3_2) at (0, 2) {};
    \node[v2] (H3_3) at (1.5, 3.2) {};
    \node[v2] (H3_4) at (3, 2) {};
    \node[v2] (H3_5) at (3, 0) {};
    \node[v2] (H3_6) at (1.5, -1) {};
    \node[v2] (H3_7) at (-2, 0) {};
    \node[v2] (H3_8) at (-2, 2) {};
    \node[v2] (H3_9)  at (0, 4.5) {};
    \node[v2] (H3_10) at (-2,4.5) {};
    \node[v2] (H3_11) at (-3.5, 3.2){};

\draw(H3_1)--(H3_2)--(H3_3)--(H3_4)--(H3_5)--(H3_6)--(H3_1);
\draw (H3_2)--(H3_8)--(H3_7)--(H3_1);
        \draw (H3_3) -- (H3_9) -- (H3_10) -- (H3_11) -- (H3_8);
\draw (1.5, -1) ..controls (-6,-3) and (-7,5.5).. (-2,4.5);

    \node[font=\scriptsize,below] at (H3_1) {$v_9$};
\node[font=\scriptsize,xshift=5pt,yshift= 5pt] at (H3_1) { \color{blue}6};
     \node[font=\scriptsize,xshift=-5pt,yshift=5pt] at (H3_2) {$v_6$};
\node[font=\scriptsize,xshift=5pt,yshift=-5pt] at (H3_2) { \color{blue}7};
     \node[font=\scriptsize,above] at (H3_3) {$v_4$};
\node[font=\scriptsize,below] at (H3_3) { \color{blue}5};
     \node[font=\scriptsize,right] at (H3_4) {$v_7$};
\node[font=\scriptsize,xshift=-5pt,yshift=-5pt] at (H3_4) { \color{blue}3};
     \node[font=\scriptsize,right] at (H3_5) {$v_{11}$};
\node[font=\scriptsize,xshift=-5pt,yshift= 5pt] at (H3_5) { \color{blue}3};
     \node[font=\scriptsize,below] at (H3_6) {$v_{10}$};
\node[font=\scriptsize,above] at (H3_6) { \color{blue}6};
     \node[font=\scriptsize,below] at (H3_7) {$v_8$};
\node[font=\scriptsize,xshift=5pt,yshift=5pt] at (H3_7) { \color{blue}4};
     \node[font=\scriptsize,left] at (H3_8) {$v_5$};
\node[font=\scriptsize,xshift=5pt,yshift=-5pt] at (H3_8) { \color{blue}5};
     \node[font=\scriptsize,above] at (H3_9) {$v_3$};
\node[font=\scriptsize,below] at (H3_9) { \color{blue}4};
     \node[font=\scriptsize,above] at (H3_10) {$v_2$};
\node[font=\scriptsize,below] at (H3_10) { \color{blue}5};
     \node[font=\scriptsize,left] at (H3_11) {$v_1$};
\node[font=\scriptsize,right] at (H3_11) { \color{blue}4};
\node[font=\scriptsize] at (-0.5,-3) {The case $v_2v_{10} \in E(G)$};
\end{tikzpicture}
\end{center}
\caption{Two cases in Case 2 of the proof of Lemma \ref{reducible-H3}.
The numbers at vertices are the number of available colors.} \label{H3-adjacent}
\end{figure}

\medskip

\noindent {\bf Case 3:} $H_3^2$ is not an induced subgraph of $G^2$ and
 $E(G[V(H_3)]) - E(H_3) = \emptyset$.
 
\medskip

\noindent {\bf Simplifying cases:}
\begin{itemize}
\item
The vertices in each of the following pairs do not have a common neighbor outside $H_3$,
since it makes a $5$-cycle:
$\{v_2,v_7\}$, $\{v_2,v_8\}$, $\{v_3,v_{11}\}$,
and $\{v_8,v_{11}\}$.

\item
The vertices in each of the following pairs do not have a common neighbor outside $H_3$,
since it makes $F_3$ in Figure \ref{key configuration-C3},
which does not exist by Lemma \ref{reducible-H0}:
$\{v_1, v_2\}$, $\{v_2, v_3\}$, $\{v_7, v_{11}\}$, and $\{v_{10}, v_{11}\}$.

\item
The vertices in each of the following pairs do not have a common neighbor outside $H_3$,
since it makes $H_1$ in Figure \ref{key configuration},
which does not exist by Lemma \ref{reducible-H0}:
$\{v_1,v_8\}$, and $\{v_8,v_{10}\}$.

\item
The vertices in each of the following pairs do not have a common neighbor outside $H_3$,
since it makes $H_2$ in Figure \ref{key configuration},
which does not exist by Lemma \ref{reducible-H2}:
$\{v_3,v_7\}$.

\end{itemize}

Note that $v_1$ and $v_3$ are adjacent in $H_3^2$ through $v_2$,
and hence we do not need to cosider the case $v_1$ and $v_3$ have a common neighbor outside $H_3$.
Similarly, we do not need to cosider the case $v_7$ and $v_{11}$ have a common neighbor outside $H_3$.
Thus, by symmetry,
we only need to consider the following  five subcases in Case 3.

\begin{itemize}
\item[]
\textbf{Subcase 3.1:} $v_1$ and $v_{10}$ have a common neighbor outside $H_3$ in $G$.

\item[]
\textbf{Subcase 3.2:} $v_2$ and $v_{10}$ ($v_1$ and $v_{11}$ by symmetry)
have a common neighbor outside $H_3$ in $G$.

\item[]
\textbf{Subcase 3.3:} $v_3$ and $v_{10}$ ($v_1$ and $v_{7}$ by symmetry)
have a common neighbor outside $H_3$ in $G$.

\item[]
\textbf{Subcase 3.4:} $v_7$ and $v_{8}$ ($v_3$ and $v_{8}$ by symmetry)
have a common neighbor outside $H_3$ in $G$.

\item[]
\textbf{Subcase 3.5:} $v_2$ and $v_{11}$ have a common neighbor outside $H_3$ in $G$.
\end{itemize}

In either case, we follow the same procedure as Case 1,
by using the Combinatorial Nullstellensatz.
Since the argument is repeated in each subcase, we will just write the graph polynomial
and indicated a monomial whose coefficient is not zero.
Recall that $P_{H_3^2}(\bm{x})$ is the graph polynomial for $H_3^2$ when $H_3$ is an induced subgraph in Case 1.
\medskip

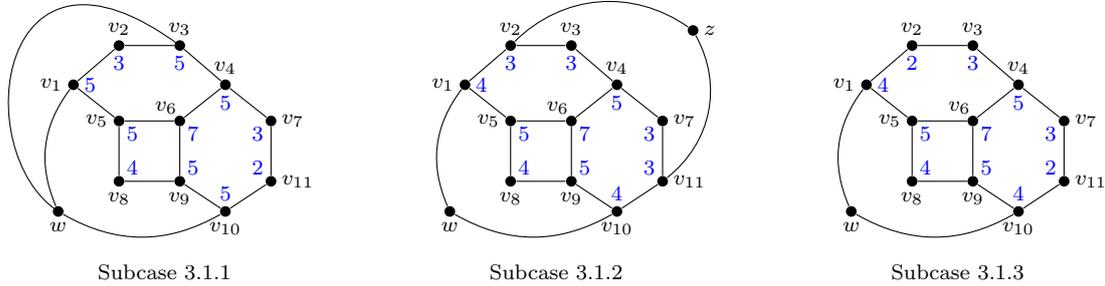
\begin{figure}[htbp]
  \begin{center}

\begin{tikzpicture}[
    v2/.style={fill=black,minimum size=4pt,ellipse,inner sep=1pt},
    scale=0.4
]
    \node[v2] (H3_1) at (0, 0) {};
    \node[v2] (H3_2) at (0, 2) {};
    \node[v2] (H3_3) at (1.5, 3.2) {};
    \node[v2] (H3_4) at (3, 2) {};
    \node[v2] (H3_5) at (3, 0) {};
    \node[v2] (H3_6) at (1.5, -1) {};
    \node[v2] (H3_7) at (-2, 0) {};
    \node[v2] (H3_8) at (-2, 2) {};
    \node[v2] (H3_9)  at (0, 4.5) {};
    \node[v2] (H3_10) at (-2,4.5) {};
    \node[v2] (H3_11) at (-3.5, 3.2){};
 \node[v2] (H3_12) at (-4, -1){};

\draw(H3_1)--(H3_2)--(H3_3)--(H3_4)--(H3_5)--(H3_6)--(H3_1);
\draw (H3_2)--(H3_8)--(H3_7)--(H3_1);
        \draw (H3_3) -- (H3_9) -- (H3_10) -- (H3_11) -- (H3_8);
        \draw (H3_6)to[bend left=30](H3_12);
\draw(H3_12)to[bend left=30](H3_11);
\draw(0, 4.5).. controls (-6,9) and (-7,1).. (-4, -1);

    \node[font=\scriptsize,below] at (H3_1) {$v_9$};
\node[font=\scriptsize,xshift=5pt,yshift= 5pt] at (H3_1) { \color{blue}5};
     \node[font=\scriptsize,xshift=-5pt,yshift=5pt] at (H3_2) {$v_6$};
\node[font=\scriptsize,xshift=5pt,yshift=-5pt] at (H3_2) { \color{blue}7};
     \node[font=\scriptsize,above] at (H3_3) {$v_4$};
\node[font=\scriptsize,below] at (H3_3) { \color{blue}5};
     \node[font=\scriptsize,right] at (H3_4) {$v_7$};
\node[font=\scriptsize,xshift=-5pt,yshift=-5pt] at (H3_4) { \color{blue}3};
     \node[font=\scriptsize,right] at (H3_5) {$v_{11}$};
\node[font=\scriptsize,xshift=-5pt,yshift= 5pt] at (H3_5) { \color{blue}2};
     \node[font=\scriptsize,below] at (H3_6) {$v_{10}$};
\node[font=\scriptsize,above] at (H3_6) { \color{blue}5};
     \node[font=\scriptsize,below] at (H3_7) {$v_8$};
\node[font=\scriptsize,xshift=5pt,yshift=5pt] at (H3_7) { \color{blue}4};
     \node[font=\scriptsize,left] at (H3_8) {$v_5$};
\node[font=\scriptsize,xshift=5pt,yshift=-5pt] at (H3_8) { \color{blue}5};
     \node[font=\scriptsize,above] at (H3_9) {$v_3$};
\node[font=\scriptsize,below] at (H3_9) { \color{blue}5};
     \node[font=\scriptsize,above] at (H3_10) {$v_2$};
\node[font=\scriptsize,below] at (H3_10) { \color{blue}3};
     \node[font=\scriptsize,left] at (H3_11) {$v_1$};
\node[font=\scriptsize,right] at (H3_11) { \color{blue}5};
\node[font=\scriptsize,below] at (H3_12) {$w$};
\node[font=\scriptsize] at (-0.5,-3) {Subcase 3.1.1};
\end{tikzpicture} \hspace{1cm}
\begin{tikzpicture}[
    v2/.style={fill=black,minimum size=4pt,ellipse,inner sep=1pt},
    scale=0.4
]
    \node[v2] (H3_1) at (0, 0) {};
    \node[v2] (H3_2) at (0, 2) {};
    \node[v2] (H3_3) at (1.5, 3.2) {};
    \node[v2] (H3_4) at (3, 2) {};
    \node[v2] (H3_5) at (3, 0) {};
    \node[v2] (H3_6) at (1.5, -1) {};
    \node[v2] (H3_7) at (-2, 0) {};
    \node[v2] (H3_8) at (-2, 2) {};
    \node[v2] (H3_9)  at (0, 4.5) {};
    \node[v2] (H3_10) at (-2,4.5) {};
    \node[v2] (H3_11) at (-3.5, 3.2){};
 \node[v2] (H3_12) at (-4, -1){};
  \node[v2] (H3_13) at (4,5){};

\draw(H3_1)--(H3_2)--(H3_3)--(H3_4)--(H3_5)--(H3_6)--(H3_1);
\draw (H3_2)--(H3_8)--(H3_7)--(H3_1);
        \draw (H3_3) -- (H3_9) -- (H3_10) -- (H3_11) -- (H3_8);
        \draw (H3_6)to[bend left=30](H3_12);
\draw(H3_12)to[bend left=30](H3_11);
\draw(H3_13)to[bend left=40](H3_5);
\draw(H3_10)to[bend left=40](H3_13);

    \node[font=\scriptsize,below] at (H3_1) {$v_9$};
\node[font=\scriptsize,xshift=5pt,yshift= 5pt] at (H3_1) { \color{blue}5};
     \node[font=\scriptsize,xshift=-5pt,yshift=5pt] at (H3_2) {$v_6$};
\node[font=\scriptsize,xshift=5pt,yshift=-5pt] at (H3_2) { \color{blue}7};
     \node[font=\scriptsize,above] at (H3_3) {$v_4$};
\node[font=\scriptsize,below] at (H3_3) { \color{blue}5};
     \node[font=\scriptsize,right] at (H3_4) {$v_7$};
\node[font=\scriptsize,xshift=-5pt,yshift=-5pt] at (H3_4) { \color{blue}3};
     \node[font=\scriptsize,right] at (H3_5) {$v_{11}$};
\node[font=\scriptsize,xshift=-5pt,yshift= 5pt] at (H3_5) { \color{blue}3};
     \node[font=\scriptsize,below] at (H3_6) {$v_{10}$};
\node[font=\scriptsize,above] at (H3_6) { \color{blue}4};
     \node[font=\scriptsize,below] at (H3_7) {$v_8$};
\node[font=\scriptsize,xshift=5pt,yshift=5pt] at (H3_7) { \color{blue}4};
     \node[font=\scriptsize,left] at (H3_8) {$v_5$};
\node[font=\scriptsize,xshift=5pt,yshift=-5pt] at (H3_8) { \color{blue}5};
     \node[font=\scriptsize,above] at (H3_9) {$v_3$};
\node[font=\scriptsize,below] at (H3_9) { \color{blue}3};
     \node[font=\scriptsize,above] at (H3_10) {$v_2$};
\node[font=\scriptsize,below] at (H3_10) { \color{blue}3};
     \node[font=\scriptsize,left] at (H3_11) {$v_1$};
\node[font=\scriptsize,right] at (H3_11) { \color{blue}4};
\node[font=\scriptsize,below] at (H3_12) {$w$};
\node[font=\scriptsize,right] at (H3_13) {$z$};
\node[font=\scriptsize] at (-0.5,-3) {Subcase 3.1.2};
\end{tikzpicture} \hspace{1cm}
\begin{tikzpicture}[
    v2/.style={fill=black,minimum size=4pt,ellipse,inner sep=1pt},
    scale=0.4
]
    \node[v2] (H3_1) at (0, 0) {};
    \node[v2] (H3_2) at (0, 2) {};
    \node[v2] (H3_3) at (1.5, 3.2) {};
    \node[v2] (H3_4) at (3, 2) {};
    \node[v2] (H3_5) at (3, 0) {};
    \node[v2] (H3_6) at (1.5, -1) {};
    \node[v2] (H3_7) at (-2, 0) {};
    \node[v2] (H3_8) at (-2, 2) {};
    \node[v2] (H3_9)  at (0, 4.5) {};
    \node[v2] (H3_10) at (-2,4.5) {};
    \node[v2] (H3_11) at (-3.5, 3.2){};
 \node[v2] (H3_12) at (-4, -1){};

\draw(H3_1)--(H3_2)--(H3_3)--(H3_4)--(H3_5)--(H3_6)--(H3_1);
\draw (H3_2)--(H3_8)--(H3_7)--(H3_1);
        \draw (H3_3) -- (H3_9) -- (H3_10) -- (H3_11) -- (H3_8);
        \draw (H3_6)to[bend left=30](H3_12);
\draw(H3_12)to[bend left=30](H3_11);

    \node[font=\scriptsize,below] at (H3_1) {$v_9$};
\node[font=\scriptsize,xshift=5pt,yshift= 5pt] at (H3_1) { \color{blue}5};
     \node[font=\scriptsize,xshift=-5pt,yshift=5pt] at (H3_2) {$v_6$};
\node[font=\scriptsize,xshift=5pt,yshift=-5pt] at (H3_2) { \color{blue}7};
     \node[font=\scriptsize,above] at (H3_3) {$v_4$};
\node[font=\scriptsize,below] at (H3_3) { \color{blue}5};
     \node[font=\scriptsize,right] at (H3_4) {$v_7$};
\node[font=\scriptsize,xshift=-5pt,yshift=-5pt] at (H3_4) { \color{blue}3};
     \node[font=\scriptsize,right] at (H3_5) {$v_{11}$};
\node[font=\scriptsize,xshift=-5pt,yshift= 5pt] at (H3_5) { \color{blue}2};
     \node[font=\scriptsize,below] at (H3_6) {$v_{10}$};
\node[font=\scriptsize,above] at (H3_6) { \color{blue}4};
     \node[font=\scriptsize,below] at (H3_7) {$v_8$};
\node[font=\scriptsize,xshift=5pt,yshift=5pt] at (H3_7) { \color{blue}4};
     \node[font=\scriptsize,left] at (H3_8) {$v_5$};
\node[font=\scriptsize,xshift=5pt,yshift=-5pt] at (H3_8) { \color{blue}5};
     \node[font=\scriptsize,above] at (H3_9) {$v_3$};
\node[font=\scriptsize,below] at (H3_9) { \color{blue}3};
     \node[font=\scriptsize,above] at (H3_10) {$v_2$};
\node[font=\scriptsize,below] at (H3_10) { \color{blue}2};
     \node[font=\scriptsize,left] at (H3_11) {$v_1$};
\node[font=\scriptsize,right] at (H3_11) { \color{blue}4};
\node[font=\scriptsize,below] at (H3_12) {$w$};
\node[font=\scriptsize] at (-0.5,-3) {Subcase 3.1.3};
\end{tikzpicture}
\end{center}
\caption{Subcases 3.1.1, 3.1.2, and 3.1.3 of the proof of Lemma \ref{reducible-H3}. The numbers at vertices are the number of available colors.}
\label{H3-adjacent-one}
\end{figure}

\medskip
\noindent {\bf Subcase 3.1:} $v_1$ and $v_{10}$ have a common neighbor $w$,
where $w \notin V(H_3)$.

In this subcase,
we further divide the proof into some cases.
\medskip

\noindent {\bf Subcase 3.1.1:} $w$ is adjacent to $v_3$ or $v_7$.

By symmetry, we may assume $w$ is adjacent to $v_3$.
In this subcase, the number of available colors at vertices of $V(H_3)$ is Subcase 3.1.1 in Figure \ref{H3-adjacent-one}.  The graph polynomial for this subcase is
\[
f(\bm{x}) = (x_1 - x_{10})(x_3 - x_{10}) P_{H_3^2}(\bm{x}).
\]
By the calculation using Mathematica,
we see that
the coefficient of
$x_1^3x_2^1x_3^3x_4^4x_5^4x_6^5x_7^2x_8^3x_9^4x_{10}^2x_{11}^1$ is 3, which
is nonzero.

\medskip
\noindent {\bf Subcase 3.1.2:}
$v_2$ and $v_{11}$ have a common neighbor $z$,
where $z \notin V(H_3)$.

The number of available colors at vertices of $V(H_3)$ is Subcase 3.1.2 in Figure \ref{H3-adjacent-one}.  The graph polynomial for this subcase is
\[
f(\bm{x'}) = (x_1 - x_{10})(x_2 - x_{11}) P_{H_3^2}(\bm{x}).
\]
By the calculation using Mathematica,
we see that
the coefficient of
$x_1^3x_2^2x_3^2x_4^4x_5^4x_6^5x_7^2x_8^3x_9^4x_{10}^2x_{11}^1$ is $2$, which
is nonzero.

\medskip
\noindent {\bf Subcase 3.1.3:}
$w$ is adjacent to neither $v_3$ nor $v_7$,
and $v_2$ and $v_{11}$ do not have a common neighbor outside $H_3$.

The number of available colors at vertices of $V(H_3)$ is Subcase 3.1.3 in Figure \ref{H3-adjacent-one}.  The graph polynomial for this subcase is
\[
f(\bm{x'}) = (x_1 - x_{10}) P_{H_3^2}(\bm{x}).
\]
By the calculation using Mathematica,
we see that
the coefficient of
$x_1^3x_2^1x_3^2x_4^4x_5^4x_6^5x_7^2x_8^3x_9^4x_{10}^2x_{11}^1$ is 2, which
is nonzero.

\begin{figure}[htbp]
  \begin{center}

\begin{tikzpicture}[
    v2/.style={fill=black,minimum size=4pt,ellipse,inner sep=1pt},
    scale=0.34
]
    \node[v2] (H3_1) at (0, 0) {};
    \node[v2] (H3_2) at (0, 2) {};
    \node[v2] (H3_3) at (1.5, 3.2) {};
    \node[v2] (H3_4) at (3, 2) {};
    \node[v2] (H3_5) at (3, 0) {};
    \node[v2] (H3_6) at (1.5, -1) {};
    \node[v2] (H3_7) at (-2, 0) {};
    \node[v2] (H3_8) at (-2, 2) {};
    \node[v2] (H3_9)  at (0, 4.5) {};
    \node[v2] (H3_10) at (-2,4.5) {};
    \node[v2] (H3_11) at (-3.5, 3.2){};
 \node[v2] (H3_12) at (-4, -1){};

\draw(H3_1)--(H3_2)--(H3_3)--(H3_4)--(H3_5)--(H3_6)--(H3_1);
\draw (H3_2)--(H3_8)--(H3_7)--(H3_1);
        \draw (H3_3) -- (H3_9) -- (H3_10) -- (H3_11) -- (H3_8);
        \draw (H3_6)to[bend left=30](H3_12);

\draw(-2,4.5).. controls (-6,6) and (-7,1).. (-4, -1);

    \node[font=\scriptsize,below] at (H3_1) {$v_9$};
\node[font=\scriptsize,xshift=5pt,yshift= 5pt] at (H3_1) { \color{blue}5};
     \node[font=\scriptsize,xshift=-5pt,yshift=5pt] at (H3_2) {$v_6$};
\node[font=\scriptsize,xshift=5pt,yshift=-5pt] at (H3_2) { \color{blue}7};
     \node[font=\scriptsize,above] at (H3_3) {$v_4$};
\node[font=\scriptsize,below] at (H3_3) { \color{blue}5};
     \node[font=\scriptsize,right] at (H3_4) {$v_7$};
\node[font=\scriptsize,xshift=-5pt,yshift=-5pt] at (H3_4) { \color{blue}3};
     \node[font=\scriptsize,right] at (H3_5) {$v_{11}$};
\node[font=\scriptsize,xshift=-5pt,yshift= 5pt] at (H3_5) { \color{blue}2};
     \node[font=\scriptsize,below] at (H3_6) {$v_{10}$};
\node[font=\scriptsize,above] at (H3_6) { \color{blue}4};
     \node[font=\scriptsize,below] at (H3_7) {$v_8$};
\node[font=\scriptsize,xshift=5pt,yshift=5pt] at (H3_7) { \color{blue}4};
     \node[font=\scriptsize,left] at (H3_8) {$v_5$};
\node[font=\scriptsize,xshift=5pt,yshift=-5pt] at (H3_8) { \color{blue}5};
     \node[font=\scriptsize,above] at (H3_9) {$v_3$};
\node[font=\scriptsize,below] at (H3_9) { \color{blue}3};
     \node[font=\scriptsize,above] at (H3_10) {$v_2$};
\node[font=\scriptsize,below] at (H3_10) { \color{blue}3};
     \node[font=\scriptsize,left] at (H3_11) {$v_1$};
\node[font=\scriptsize,right] at (H3_11) { \color{blue}3};
\node[font=\scriptsize,below] at (H3_12) {$w$};
\node[font=\scriptsize] at (-0.5,-3) {Subcase 3.2};
\end{tikzpicture}\begin{tikzpicture}[
    v2/.style={fill=black,minimum size=4pt,ellipse,inner sep=1pt},
    scale=0.34
]
    \node[v2] (H3_1) at (0, 0) {};
    \node[v2] (H3_2) at (0, 2) {};
    \node[v2] (H3_3) at (1.5, 3.2) {};
    \node[v2] (H3_4) at (3, 2) {};
    \node[v2] (H3_5) at (3, 0) {};
    \node[v2] (H3_6) at (1.5, -1) {};
    \node[v2] (H3_7) at (-2, 0) {};
    \node[v2] (H3_8) at (-2, 2) {};
    \node[v2] (H3_9)  at (0, 4.5) {};
    \node[v2] (H3_10) at (-2,4.5) {};
    \node[v2] (H3_11) at (-3.5, 3.2){};

  \node[v2] (H3_13) at (4,4){};

\draw(H3_1)--(H3_2)--(H3_3)--(H3_4)--(H3_5)--(H3_6)--(H3_1);
\draw (H3_2)--(H3_8)--(H3_7)--(H3_1);
        \draw (H3_3) -- (H3_9) -- (H3_10) -- (H3_11) -- (H3_8);

\draw(0, 4.5)to[bend left=30] (4,4);
\draw(4,4) ..controls (5,3) and (7,-2)..(1.5, -1);

    \node[font=\scriptsize,below] at (H3_1) {$v_9$};
\node[font=\scriptsize,xshift=5pt,yshift= 5pt] at (H3_1) { \color{blue}5};
     \node[font=\scriptsize,xshift=-5pt,yshift=5pt] at (H3_2) {$v_6$};
\node[font=\scriptsize,xshift=5pt,yshift=-5pt] at (H3_2) { \color{blue}7};
     \node[font=\scriptsize,above] at (H3_3) {$v_4$};
\node[font=\scriptsize,below] at (H3_3) { \color{blue}5};
     \node[font=\scriptsize,right] at (H3_4) {$v_7$};
\node[font=\scriptsize,xshift=-5pt,yshift=-5pt] at (H3_4) { \color{blue}3};
     \node[font=\scriptsize,right] at (H3_5) {$v_{11}$};
\node[font=\scriptsize,xshift=-5pt,yshift= 5pt] at (H3_5) { \color{blue}2};
     \node[font=\scriptsize,below] at (H3_6) {$v_{10}$};
\node[font=\scriptsize,above] at (H3_6) { \color{blue}4};
     \node[font=\scriptsize,below] at (H3_7) {$v_8$};
\node[font=\scriptsize,xshift=5pt,yshift=5pt] at (H3_7) { \color{blue}4};
     \node[font=\scriptsize,left] at (H3_8) {$v_5$};
\node[font=\scriptsize,xshift=5pt,yshift=-5pt] at (H3_8) { \color{blue}5};
     \node[font=\scriptsize,above] at (H3_9) {$v_3$};
\node[font=\scriptsize,below] at (H3_9) { \color{blue}4};
     \node[font=\scriptsize,above] at (H3_10) {$v_2$};
\node[font=\scriptsize,below] at (H3_10) { \color{blue}2};
     \node[font=\scriptsize,left] at (H3_11) {$v_1$};
\node[font=\scriptsize,right] at (H3_11) { \color{blue}3};

\node[font=\scriptsize,right] at (H3_13) {$w$};
\node[font=\scriptsize] at (0.5,-3) {Subcase 3.3};
\end{tikzpicture}\begin{tikzpicture}[
    v2/.style={fill=black,minimum size=4pt,ellipse,inner sep=1pt},
    scale=0.34
]
    \node[v2] (H3_1) at (0, 0) {};
    \node[v2] (H3_2) at (0, 2) {};
    \node[v2] (H3_3) at (1.5, 3.2) {};
    \node[v2] (H3_4) at (3, 2) {};
    \node[v2] (H3_5) at (3, 0) {};
    \node[v2] (H3_6) at (1.5, -1) {};
    \node[v2] (H3_7) at (-2, 0) {};
    \node[v2] (H3_8) at (-2, 2) {};
    \node[v2] (H3_9)  at (0, 4.5) {};
    \node[v2] (H3_10) at (-2,4.5) {};
    \node[v2] (H3_11) at (-3.5, 3.2){};
 \node[v2] (H3_13) at (5, -1.5){};

\draw(H3_1)--(H3_2)--(H3_3)--(H3_4)--(H3_5)--(H3_6)--(H3_1);
\draw (H3_2)--(H3_8)--(H3_7)--(H3_1);
        \draw (H3_3) -- (H3_9) -- (H3_10) -- (H3_11) -- (H3_8);
        \draw(H3_4)to [bend left=30] (H3_13);
\draw(H3_7)to [bend right=40] (H3_13);
    \node[font=\scriptsize,below] at (H3_1) {$v_9$};
\node[font=\scriptsize,xshift=5pt,yshift= 5pt] at (H3_1) { \color{blue}5};
     \node[font=\scriptsize,xshift=-5pt,yshift=5pt] at (H3_2) {$v_6$};
\node[font=\scriptsize,xshift=5pt,yshift=-5pt] at (H3_2) { \color{blue}7};
     \node[font=\scriptsize,above] at (H3_3) {$v_4$};
\node[font=\scriptsize,below] at (H3_3) { \color{blue}5};
     \node[font=\scriptsize,right] at (H3_4) {$v_7$};
\node[font=\scriptsize,xshift=-5pt,yshift=-5pt] at (H3_4) { \color{blue}4};
     \node[font=\scriptsize,right] at (H3_5) {$v_{11}$};
\node[font=\scriptsize,xshift=-5pt,yshift= 5pt] at (H3_5) { \color{blue}2};
     \node[font=\scriptsize,below] at (H3_6) {$v_{10}$};
\node[font=\scriptsize,above] at (H3_6) { \color{blue}3};
     \node[font=\scriptsize,below] at (H3_7) {$v_8$};
\node[font=\scriptsize,xshift=5pt,yshift=5pt] at (H3_7) { \color{blue}5};
     \node[font=\scriptsize,left] at (H3_8) {$v_5$};
\node[font=\scriptsize,xshift=5pt,yshift=-5pt] at (H3_8) { \color{blue}5};
     \node[font=\scriptsize,above] at (H3_9) {$v_3$};
\node[font=\scriptsize,below] at (H3_9) { \color{blue}3};
     \node[font=\scriptsize,above] at (H3_10) {$v_2$};
\node[font=\scriptsize,below] at (H3_10) { \color{blue}2};
     \node[font=\scriptsize,left] at (H3_11) {$v_1$};
\node[font=\scriptsize,right] at (H3_11) { \color{blue}3};
\node[font=\scriptsize,below] at (H3_13) {$w$};
\node[font=\scriptsize] at (0.5,-3) {Subcase 3.4};
\end{tikzpicture}\begin{tikzpicture}[
    v2/.style={fill=black,minimum size=4pt,ellipse,inner sep=1pt},
    scale=0.34
]
    \node[v2] (H3_1) at (0, 0) {};
    \node[v2] (H3_2) at (0, 2) {};
    \node[v2] (H3_3) at (1.5, 3.2) {};
    \node[v2] (H3_4) at (3, 2) {};
    \node[v2] (H3_5) at (3, 0) {};
    \node[v2] (H3_6) at (1.5, -1) {};
    \node[v2] (H3_7) at (-2, 0) {};
    \node[v2] (H3_8) at (-2, 2) {};
    \node[v2] (H3_9)  at (0, 4.5) {};
    \node[v2] (H3_10) at (-2,4.5) {};
    \node[v2] (H3_11) at (-3.5, 3.2){};

  \node[v2] (H3_13) at (4,5){};

\draw(H3_1)--(H3_2)--(H3_3)--(H3_4)--(H3_5)--(H3_6)--(H3_1);
\draw (H3_2)--(H3_8)--(H3_7)--(H3_1);
        \draw (H3_3) -- (H3_9) -- (H3_10) -- (H3_11) -- (H3_8);

\draw(H3_13)to[bend left=40](H3_5);
\draw(H3_10)to[bend left=40](H3_13);

    \node[font=\scriptsize,below] at (H3_1) {$v_9$};
\node[font=\scriptsize,xshift=5pt,yshift= 5pt] at (H3_1) { \color{blue}5};
     \node[font=\scriptsize,xshift=-5pt,yshift=5pt] at (H3_2) {$v_6$};
\node[font=\scriptsize,xshift=5pt,yshift=-5pt] at (H3_2) { \color{blue}7};
     \node[font=\scriptsize,above] at (H3_3) {$v_4$};
\node[font=\scriptsize,below] at (H3_3) { \color{blue}5};
     \node[font=\scriptsize,right] at (H3_4) {$v_7$};
\node[font=\scriptsize,xshift=-5pt,yshift=-5pt] at (H3_4) { \color{blue}3};
     \node[font=\scriptsize,right] at (H3_5) {$v_{11}$};
\node[font=\scriptsize,xshift=-5pt,yshift= 5pt] at (H3_5) { \color{blue}3};
     \node[font=\scriptsize,below] at (H3_6) {$v_{10}$};
\node[font=\scriptsize,above] at (H3_6) { \color{blue}3};
     \node[font=\scriptsize,below] at (H3_7) {$v_8$};
\node[font=\scriptsize,xshift=5pt,yshift=5pt] at (H3_7) { \color{blue}4};
     \node[font=\scriptsize,left] at (H3_8) {$v_5$};
\node[font=\scriptsize,xshift=5pt,yshift=-5pt] at (H3_8) { \color{blue}5};
     \node[font=\scriptsize,above] at (H3_9) {$v_3$};
\node[font=\scriptsize,below] at (H3_9) { \color{blue}3};
     \node[font=\scriptsize,above] at (H3_10) {$v_2$};
\node[font=\scriptsize,below] at (H3_10) { \color{blue}3};
     \node[font=\scriptsize,left] at (H3_11) {$v_1$};
\node[font=\scriptsize,right] at (H3_11) { \color{blue}3};

\node[font=\scriptsize,right] at (H3_13) {$w$};
\node[font=\scriptsize] at (0.5,-3) {Subcase 3.5};
\end{tikzpicture}
\end{center}
\caption{Subcases 3.2, 3.3, 3.4 and 3.5 of the proof of Lemma \ref{reducible-H3}. The numbers at vertices are the number of available colors.} \label{H3-adjacent-two}
\end{figure}
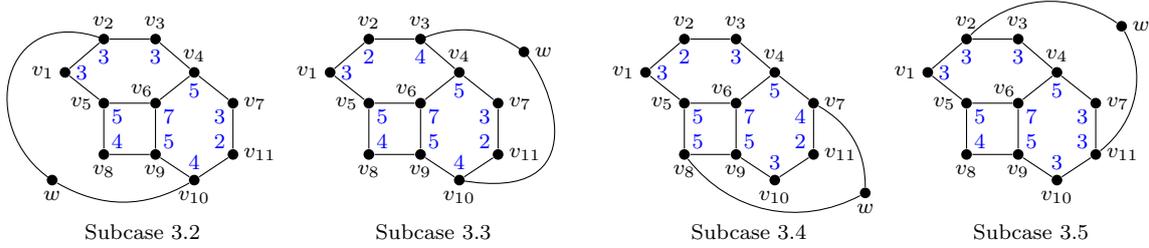

\medskip
\noindent {\bf Subcase 3.2:} $v_2$ and $v_{10}$ have a common neighbor $w$,
where $w \notin V(H_3)$.

The number of available colors at vertices of $V(H_3)$ is Subcase 3.2 in Figure \ref{H3-adjacent-two}.  The graph polynomial for this subcase is
\[
f(\bm{x'}) = (x_2 - x_{10}) P_{H_3^2}(\bm{x}).
\]
By the calculation using Mathematica,
we see that
the coefficient of
$x_1^2x_2^1x_3^2x_4^4x_5^4x_6^5x_7^2x_8^3x_9^4x_{10}^3x_{11}^1$ is $-4$, which
is nonzero.
So, it is reducible.

\medskip
\noindent {\bf Subcase 3.3:} $v_3$ and $v_{10}$ have a common neighbor $w$,
where $w \notin V(H_3)$,
and $w$ is not adjacent to $v_1$.

The number of available colors at vertices of $V(H_3)$ is Subcase 3.3 in Figure \ref{H3-adjacent-two}.  The graph polynomial for this subcase is
\[
f(\bm{x'}) = (x_3 - x_{10}) P_{H_3^2}(\bm{x}).
\]
By the calculation using Mathematica,
we see that
the coefficient of
$x_1^2x_2^1x_3^3x_4^4x_5^4x_6^5x_7^2x_8^3x_9^4x_{10}^2x_{11}^1$ is $2$, which
is nonzero.

\medskip
\noindent {\bf Subcase 3.4:} $v_7$ and $v_{8}$ have a common neighbor $w$,
where $w \notin V(H_3)$.

The number of available colors at vertices of $V(H_3)$ is Subcase 3.4 in Figure \ref{H3-adjacent-two}.  The graph polynomial for this subcase is
\[
f(\bm{x'}) = (x_7 - x_{8}) P_{H_3^2}(\bm{x}).
\]
By the calculation using Mathematica,
we see that
the coefficient of
$x_1^2x_2^1x_3^2x_4^4x_5^4x_6^5x_7^2x_8^4x_9^4x_{10}^2x_{11}^1$ is $-3$, which
is nonzero.


\medskip
\noindent {\bf Subcase 3.5:} $v_2$ and $v_{11}$ have a common neighbor $w$,
where $w \notin V(H_3)$,
and $v_1$ and $v_{10}$ have no common neighbor outside $H_3$.

The number of available colors at vertices of $V(H_3)$ is Subcase 3.5 in Figure \ref{H3-adjacent-two}.  The graph polynomial for this subcase is
\[
f(\bm{x'}) = (x_2 - x_{11}) P_{H_3^2}(\bm{x}).
\]
By the calculation using Mathematica,
we see that
the coefficient of
$x_1^2x_2^2x_3^2x_4^4x_5^4x_6^5x_7^2x_8^3x_9^4x_{10}^2x_{11}^1$ is $2$, which
is nonzero.
\medskip

Thus,
in either case,
it follows from Theorem \ref{cnull} that
the vertices in $H_3$ can be colored from the list $L_{H_3}$
so that we obtain an $L$-coloring in $G^2$.
This is a contradiction for the fact that $G$ is a counterexample.  So, $G$ has no $H_3$.  This completes the proof of Lemma \ref{reducible-H3}.
\end{proof}


\subsection{Subgraph $H_4$ is reducible}
In this subsection, we will prove that
$H_4$ in Figure \ref{key configuration}
does not appear in a minimal counterexample.  Before we prove Lemma \ref{reducible-H4}, we will prove the following lemma which is used in the proof of Lemma \ref{reducible-H4}.

\begin{lemma} \label{H2-type-two-reducible}
The graph $J_5$ in Figure \ref{H2-type-two} does not appear in $G$.
\end{lemma}
\begin{proof}
Suppose that $G$ has $J_5$ as a subgraph.
We denote $V(J_5) = \{v_1, \ldots, v_{10}\}$ as in Figure \ref{H2-type-two}.
Let $L$ be a list assignment with lists of size 7 for each vertex in $G$.
We will show that $G^2$ has a proper coloring from the list $L$, which is a contradiction for the fact that $G$ is a counterexample to the theorem.

 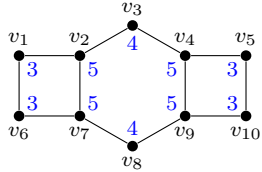
\begin{figure}[htbp]
  \begin{center}

\begin{tikzpicture}[
    v2/.style={fill=black,minimum size=4pt,ellipse,inner sep=1pt},
    scale=0.4
]

    \node[v2] (H2_1) at (0, 0) {};
    \node[v2] (H2_2) at (0, 2) {};
    \node[v2] (H2_3) at (1.732, 3) {};
    \node[v2] (H2_4) at (3.464, 2) {};
    \node[v2] (H2_5) at (3.464, 0) {};
    \node[v2] (H2_6) at (1.732, -1) {};

    \node[v2] (H2_7) at (-2, 0) {};
    \node[v2] (H2_8) at (-2, 2) {};
     \node[v2] (H2_9) at (5.464, 2) {};
    \node[v2] (H2_10) at (5.464, 0) {};

    \draw (H2_1) -- (H2_2) -- (H2_3) -- (H2_4) -- (H2_5) -- (H2_6) -- (H2_1);

    \draw (H2_1) -- (H2_7) -- (H2_8) -- (H2_2);
    \draw (H2_4) -- (H2_9) -- (H2_10) -- (H2_5);

    \node[font=\scriptsize,above] at (H2_8) {$v_1$};
\node[font=\scriptsize,xshift=5pt,yshift=-5pt] at (H2_8) {\color{blue}3};
 \node[font=\scriptsize,above] at (H2_2) {$v_2$};
\node[font=\scriptsize,xshift=5pt,yshift=-5pt] at (H2_2) {\color{blue}5};
 \node[font=\scriptsize,above] at (H2_3) {$v_3$};
\node[font=\scriptsize,below] at (H2_3) {\color{blue}4};
 \node[font=\scriptsize,above] at (H2_4) {$v_4$};
\node[font=\scriptsize,xshift=-5pt,yshift=-5pt] at (H2_4) {\color{blue}5};
 \node[font=\scriptsize,above] at (H2_9) {$v_5$};
\node[font=\scriptsize,xshift=-5pt,yshift=-5pt] at (H2_9) {\color{blue}3};
 \node[font=\scriptsize,below] at (H2_7) {$v_6$};
\node[font=\scriptsize,xshift=5pt,yshift=5pt] at (H2_7) {\color{blue}3};
\node[font=\scriptsize,below] at (H2_1) {$v_7$};
\node[font=\scriptsize,xshift=5pt,yshift=5pt] at (H2_1) {\color{blue}5};

 \node[font=\scriptsize,below] at (H2_6) {$v_8$};
\node[font=\scriptsize,above] at (H2_6) {\color{blue}4};
 \node[font=\scriptsize,below] at (H2_5) {$v_9$};
\node[font=\scriptsize,xshift=-5pt,yshift=5pt] at (H2_5) {\color{blue}5};

 \node[font=\scriptsize,below] at (H2_10) {$v_{10}$};
\node[font=\scriptsize,xshift=-5pt,yshift=5pt] at (H2_10) {\color{blue}3};

\end{tikzpicture}
  \caption{Graph $J_5$. The numbers at vertices are the number of available colors.}
\label{H2-type-two}
\end{center}
\end{figure}

Let $G' = G - V({J_5})$.
Then $G'$ is also a subcubic planar graph and $|V(G')| < |V(G)|$.   Since $G$ is a minimal counterexample to Theorem \ref{main-thm},
the square of $G'$ has a proper coloring $\phi$ such that $\phi(v) \in L(v)$ for each vertex $v \in V(G')$.
For each $v_i \in V(J_5)$, we define
\begin{equation*} \label{list-L}
L_{J_5}(v_i) = L(v_i) \setminus \{\phi(x) : xv_i \in E(G^2) \mbox{ and } x \notin V(J_5)\}.
\end{equation*}
Then, we have the following (see Figure \ref{H2-type-two}).
$$
|L_{J_5}(v_i)| \geq
\begin{cases}
3 & i=1,5,6,10, \\
4 & i=3,8, \\
5 & i=2,4,7,9.
\end{cases}
$$

We now consider the following two cases.

\medskip
\noindent {\bf Case 1:}
$v_1$ and $v_5$ are not adjacent in $G^2$. \medskip

If $L_{J_5}(v_1) \cap L_{J_5}(v_5) \neq \emptyset$, then color $v_1$ and $v_5$ by a color $c \in L_{J_5}(v_1) \cap L_{J_5}(v_5)$.  If $L_{J_5}(v_1) \cap L_{J_5}(v_5) = \emptyset$, then we can color $v_1$ by a color $c_1$ and $v_5$ by a color $c_5$ so that $|L_{J_5}(v_3) \setminus \{c_1, c_5\}| \geq 3$. And then, color $v_6$ and $v_{10}$ greedily.
Let $L_{J_5}'(v_i)$ for $i = 2, 3, 4, 7, 8, 9$ be the color list after coloring $v_1, v_5, v_6, v_{10}$.  Then
\[
|L_{J_5}'(v_i)| \geq 3 \mbox{ for } i = 2, 3, 4, 7, 9, \mbox{ and } |L_{J_5}'(v_8)| \geq 2.
\]
Then, $v_2, v_3, v_4, v_7, v_8, v_9$ are colorable from the list by Lemma \ref{cycle-six}.

\medskip
\noindent {\bf Case 2:}
$v_1$ and $v_5$ are adjacent in $G^2$. \medskip

Note that $v_1$ and $v_5$ cannot be adjacent in $G$ since it makes a 5-cycle.
So, we just need to consider the case when $v_1$ and $v_5$ have a common neighbor $v_{11}$ outside $J_5$.

Suppose first $v_{11} = v_6$ or $v_{10}$, say $v_{11} = v_{10}$ by symmetry.
In this case, we have the following.
$$
|L_{J_5}(v_i)| \geq
\begin{cases}
4 & i=3, 5, 6, 8,\\
5 & i=4, 7, \\
6 & i=1, 2, 9, 10.
\end{cases}
$$

Since $|L_{J_5}(v_1)| \geq 6$ and $|L_{J_5}(v_3)| \geq 4$, there exists a color $c_1 \in L_{J_5}(v_1)$ such that $|L_{J_5}(v_3) \setminus \{c_1\}| \geq 4$.  Color $v_1$ by $c_1$, and then greedily color $v_5, v_6, v_{10}$ in order.
Let $L'(v_i)$ for $i = 2, 3, 4, 7, 8, 9$ be the color list after coloring $v_1, v_5, v_6, v_{10}$.  We have
\[
|L'(v_i)| \geq 3 \mbox{ for } i = 2, 3, 4, 7, 9, \mbox{ and } |L'(v_8)| \geq 2.
\]
Then, $v_2, v_3, v_4, v_7, v_8, v_9$ are colorable from the list by Lemma \ref{cycle-six}.
\medskip

Suppose next $v_{11} \neq v_6, v_{10}$.
In this case, we uncolor $v_{11}$,
and we define for each $v_i \in V(J_5) \cup \{v_{11}\}$,
\begin{equation*} \label{list-L}
L_{J_5}'(v_i) = L(v_i) \setminus \{\phi(x) : xv_i \in E(G^2) \mbox{ and } x \notin V(J_5) \cup \{v_{11}\}\}.
\end{equation*}
Then, we have the following:
$$
|L_{J_5}'(v_i)| \geq
\begin{cases}
4 & i=3,6, 8, 10, 11,\\
5 & i=1, 5, 7, 9, \\
6 & i=2, 4.
\end{cases}
$$

Since $|L_{J_5}(v_1)| \geq 5$ and $|L_{J_5}(v_3)| \geq 4$, there exists a color $c_1 \in L_{J_5}(v_1)$ such that $|L_{J_5}(v_3) \setminus \{c_1\}| \geq 4$.  Color $v_1$ by $c_1$, and then greedily color $v_{11}, v_5, v_6, v_{10}$ in order. Let $L_{J_5}''(v_i)$ for $i = 2, 3, 4, 7, 8, 9$ be the color list after coloring $v_1, v_5, v_6, v_{10}, v_{11}$.  Then
\[
|L_{J_5}''(v_i)| \geq 3 \mbox{ for } i = 2, 3, 4, 7, 9, \mbox{ and } |L_{J_5}''(v_8)| \geq 2.
\]
Then, $v_2, v_3, v_4, v_7, v_8, v_9$ are colorable from the list $L_{J_5}''$ by Lemma \ref{cycle-six}.
\end{proof}

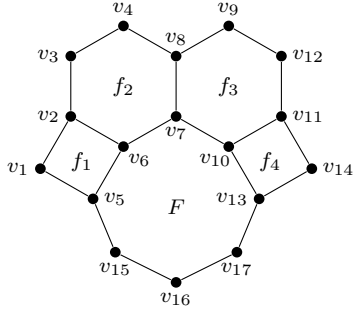
\begin{figure}[htbp]
  \begin{center}

\begin{tikzpicture}[
    v2/.style={fill=black,minimum size=4pt,ellipse,inner sep=1pt},
    scale=0.4
]
\node[v2] (H4_1) at (0, 0){};
    \node[v2] (H4_2) at (0, 2){};
    \node[v2] (H4_3) at (-1.732,-1){};
 \node[v2] (H4_4) at (-2*1.732,0){};
 \node[v2] (H4_5) at (-2*1.732,2){};
 \node[v2] (H4_6) at (-1.732,3){};
 \node[v2] (H4_7) at (1.732,-1){};
 \node[v2] (H4_8) at (2*1.732,0){};
 \node[v2] (H4_9) at (2*1.732,2){};
 \node[v2] (H4_10) at (1.732,3){};
 \node[v2] (H4_11) at (-4.464,-1.732){};
 \node[v2] (H4_12) at (-2.732,-2.732){};
 \node[v2] (H4_13) at (4.464,-1.732){};
 \node[v2] (H4_14) at (2.732,-2.732){};
 \node[v2] (H4_15) at (-2,-4.5){};
 \node[v2] (H4_16) at (0,-5.5){};
 \node[v2] (H4_17) at (2,-4.5){};

   \draw (H4_1) -- (H4_2) -- (H4_6) -- (H4_5) -- (H4_4) -- (H4_3)-- (H4_1);
   \draw (H4_2) -- (H4_10) -- (H4_9) -- (H4_8) -- (H4_7) -- (H4_1);
   \draw (H4_4) -- (H4_11) -- (H4_12) -- (H4_3);
   \draw (H4_8) -- (H4_13) -- (H4_14) -- (H4_7);
   \draw (H4_12) -- (H4_15) -- (H4_16) -- (H4_17) -- (H4_14);

    \node[font=\scriptsize] at (-0, -3) {$F$};
    \node[font=\scriptsize] at (1.732, 1) {$f_3$};
    \node[font=\scriptsize] at (-1.732, 1) {$f_2$};
 \node[font=\scriptsize] at (3.1, -1.5) {$f_4$};
 \node[font=\scriptsize] at (-3.1, -1.5) {$f_1$};

 \node[font=\scriptsize,left] at(H4_11) {$v_1$};
  \node[font=\scriptsize,left] at(H4_4) {$v_2$};
  \node[font=\scriptsize,left] at(H4_5) {$v_3$};
  \node[font=\scriptsize,above] at(H4_6) {$v_4$};
   \node[font=\scriptsize,right] at(H4_12) {$v_5$};
  \node[font=\scriptsize,xshift=6pt,yshift=-5pt] at(H4_3) {$v_6$};
  \node[font=\scriptsize,below] at(H4_1) {$v_7$};
  \node[font=\scriptsize,above] at(H4_2) {$v_8$};
   \node[font=\scriptsize,above] at(H4_10) {$v_9$};
  \node[font=\scriptsize,xshift=-5pt,yshift=-5pt] at(H4_7) {$v_{10}$};
  \node[font=\scriptsize,right] at(H4_8) {$v_{11}$};
  \node[font=\scriptsize,right] at(H4_9) {$v_{12}$};
   \node[font=\scriptsize,left] at(H4_14) {$v_{13}$};
  \node[font=\scriptsize,right] at(H4_13) {$v_{14}$};
  \node[font=\scriptsize,below] at(H4_15) {$v_{15}$};
  \node[font=\scriptsize,below] at(H4_16) {$v_{16}$};
    \node[font=\scriptsize,below] at(H4_17) {$v_{17}$};
\end{tikzpicture}
\end{center}
\caption{Graph $H_4$.  A face $F$ is adjacent to four faces $f_1, f_2, f_3, f_4$ in order.  $F$ is a 8-face, $f_1$ and $f_4$ are 4-faces and $f_2$ and $f_3$ are 6-faces.}
\label{H4fig-on-face}
\end{figure}

Now, we will prove that the subgraph $H_4$ in Figure \ref{H4fig-on-face} does not appear in $G$.
In Figure \ref{H4fig-on-face}, $F$ is a face of length 8 and is adjacent to four faces $f_1, f_2, f_3, f_4$ in order, where $f_1$ and $f_2$ are 4-faces and $f_2$ and $f_3$ are 6-faces.

\begin{lemma} \label{reducible-H4}
The graph $H_4$ in Figure \ref{H4fig-on-face} does not appear in $G$.
\end{lemma}
\begin{proof}
Suppose that $G$ has $H_4$ as a subgraph,
and denote $V(H_4) = \{v_1, v_2, \ldots, v_{17}\}$ as in Figure \ref{H4fig-on-face}.
In addition,
we denote the subgraph  contained in $H_4$ induced by $\{v_1, v_2, \ldots, v_{14}\}$ by $J_6$.  That is,
$V(J_6) = \{v_1, \ldots, v_{14}\}$ (see Figure \ref{H4fig}).
Note that $J_6$ is a subgraph of $H_4$.

Let $L$ be a list assignment with lists of size 7 for each vertex in $G$.
We will show that $G^2$ has  a proper coloring from the list $L$, which is a contradiction for the fact that $G$ is a counterexample to the theorem. \\

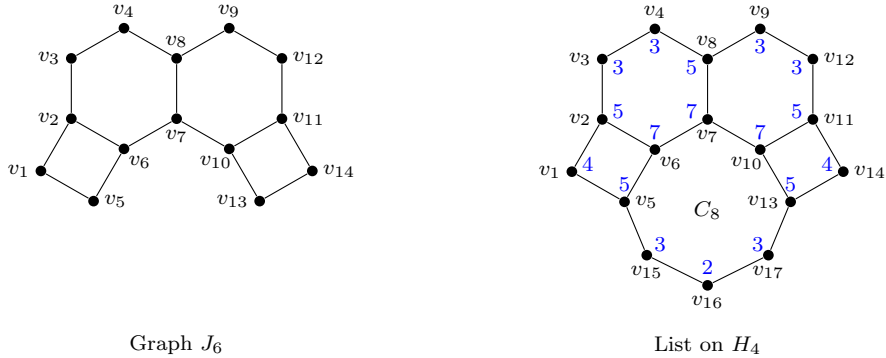
\begin{figure}[htbp]
  \begin{center}

\begin{tikzpicture}[
    v2/.style={fill=black,minimum size=4pt,ellipse,inner sep=1pt},
    scale=0.4
]
\node[v2] (H4_1) at (0, 0){};
    \node[v2] (H4_2) at (0, 2){};
    \node[v2] (H4_3) at (-1.732,-1){};
 \node[v2] (H4_4) at (-2*1.732,0){};
 \node[v2] (H4_5) at (-2*1.732,2){};
 \node[v2] (H4_6) at (-1.732,3){};
 \node[v2] (H4_7) at (1.732,-1){};
 \node[v2] (H4_8) at (2*1.732,0){};
 \node[v2] (H4_9) at (2*1.732,2){};
 \node[v2] (H4_10) at (1.732,3){};
 \node[v2] (H4_11) at (-4.464,-1.732){};
 \node[v2] (H4_12) at (-2.732,-2.732){};
 \node[v2] (H4_13) at (4.464,-1.732){};
 \node[v2] (H4_14) at (2.732,-2.732){};

   \draw (H4_1) -- (H4_2) -- (H4_6) -- (H4_5) -- (H4_4) -- (H4_3)-- (H4_1);
   \draw (H4_2) -- (H4_10) -- (H4_9) -- (H4_8) -- (H4_7) -- (H4_1);
   \draw (H4_4) -- (H4_11) -- (H4_12) -- (H4_3);
   \draw (H4_8) -- (H4_13) -- (H4_14) -- (H4_7);

    \node[font=\scriptsize] at (0, -7.5) { Graph $J_6$};

 \node[font=\scriptsize,left] at(H4_11) {$v_1$};
  \node[font=\scriptsize,left] at(H4_4) {$v_2$};
  \node[font=\scriptsize,left] at(H4_5) {$v_3$};
  \node[font=\scriptsize,above] at(H4_6) {$v_4$};
   \node[font=\scriptsize,right] at(H4_12) {$v_5$};
  \node[font=\scriptsize,xshift=6pt,yshift=-5pt] at(H4_3) {$v_6$};
  \node[font=\scriptsize,below] at(H4_1) {$v_7$};
  \node[font=\scriptsize,above] at(H4_2) {$v_8$};
   \node[font=\scriptsize,above] at(H4_10) {$v_9$};
  \node[font=\scriptsize,xshift=-5pt,yshift=-5pt] at(H4_7) {$v_{10}$};
  \node[font=\scriptsize,right] at(H4_8) {$v_{11}$};
  \node[font=\scriptsize,right] at(H4_9) {$v_{12}$};
   \node[font=\scriptsize,left] at(H4_14) {$v_{13}$};
  \node[font=\scriptsize,right] at(H4_13) {$v_{14}$};
\end{tikzpicture}\hspace{2cm}
\begin{tikzpicture}[
    v2/.style={fill=black,minimum size=4pt,ellipse,inner sep=1pt},
    scale=0.4
]
\node[v2] (H4_1) at (0, 0){};
    \node[v2] (H4_2) at (0, 2){};
    \node[v2] (H4_3) at (-1.732,-1){};
 \node[v2] (H4_4) at (-2*1.732,0){};
 \node[v2] (H4_5) at (-2*1.732,2){};
 \node[v2] (H4_6) at (-1.732,3){};
 \node[v2] (H4_7) at (1.732,-1){};
 \node[v2] (H4_8) at (2*1.732,0){};
 \node[v2] (H4_9) at (2*1.732,2){};
 \node[v2] (H4_10) at (1.732,3){};
 \node[v2] (H4_11) at (-4.464,-1.732){};
 \node[v2] (H4_12) at (-2.732,-2.732){};
 \node[v2] (H4_13) at (4.464,-1.732){};
 \node[v2] (H4_14) at (2.732,-2.732){};
 \node[v2] (H4_15) at (-2,-4.5){};
 \node[v2] (H4_16) at (0,-5.5){};
 \node[v2] (H4_17) at (2,-4.5){};

   \draw (H4_1) -- (H4_2) -- (H4_6) -- (H4_5) -- (H4_4) -- (H4_3)-- (H4_1);
   \draw (H4_2) -- (H4_10) -- (H4_9) -- (H4_8) -- (H4_7) -- (H4_1);
   \draw (H4_4) -- (H4_11) -- (H4_12) -- (H4_3);
   \draw (H4_8) -- (H4_13) -- (H4_14) -- (H4_7);
   \draw (H4_12) -- (H4_15) -- (H4_16) -- (H4_17) -- (H4_14);

    \node[font=\scriptsize] at (-0, -3) {$C_8$};

 \node[font=\scriptsize,left] at(H4_11) {$v_1$};
 \node[font=\scriptsize,xshift=6pt,yshift=3pt] at(H4_11) {\color{blue}4};
  \node[font=\scriptsize,left] at(H4_4) {$v_2$};
   \node[font=\scriptsize,xshift=6pt,yshift=3pt] at(H4_4) {\color{blue}5};
  \node[font=\scriptsize,left] at(H4_5) {$v_3$};
   \node[font=\scriptsize,xshift=6pt,yshift=-3pt] at(H4_5) {\color{blue}3};
  \node[font=\scriptsize,above] at(H4_6) {$v_4$};
   \node[font=\scriptsize,below] at(H4_6) {\color{blue}3};
   \node[font=\scriptsize,right] at(H4_12) {$v_5$};
    \node[font=\scriptsize,above] at(H4_12) {\color{blue}5};
  \node[font=\scriptsize,xshift=6pt,yshift=-5pt] at(H4_3) {$v_6$};
   \node[font=\scriptsize,above] at(H4_3) {\color{blue}7};
  \node[font=\scriptsize,below] at(H4_1) {$v_7$};
   \node[font=\scriptsize,xshift=-6pt,yshift=3pt] at(H4_1) {\color{blue}7};
  \node[font=\scriptsize,above] at(H4_2) {$v_8$};
   \node[font=\scriptsize,xshift=-6pt,yshift=-3pt] at(H4_2) {\color{blue}5};
   \node[font=\scriptsize,above] at(H4_10) {$v_9$};
    \node[font=\scriptsize,below] at(H4_10) {\color{blue}3};
  \node[font=\scriptsize,xshift=-5pt,yshift=-5pt] at(H4_7) {$v_{10}$};
   \node[font=\scriptsize,above] at(H4_7) {\color{blue}7};
  \node[font=\scriptsize,right] at(H4_8) {$v_{11}$};
   \node[font=\scriptsize,xshift=-6pt,yshift=3pt] at(H4_8) {\color{blue}5};
  \node[font=\scriptsize,right] at(H4_9) {$v_{12}$};
   \node[font=\scriptsize,xshift=-6pt,yshift=-3pt] at(H4_9) {\color{blue}3};
   \node[font=\scriptsize,left] at(H4_14) {$v_{13}$};
    \node[font=\scriptsize,above] at(H4_14) {\color{blue}5};
  \node[font=\scriptsize,right] at(H4_13) {$v_{14}$};
   \node[font=\scriptsize,xshift=-6pt,yshift=3pt] at(H4_13) {\color{blue}4};
  \node[font=\scriptsize,below] at(H4_15) {$v_{15}$};
   \node[font=\scriptsize,xshift=5pt,yshift=4pt] at(H4_15) {\color{blue}3};
  \node[font=\scriptsize,below] at(H4_16) {$v_{16}$};
     \node[font=\scriptsize,above] at(H4_16) {\color{blue}2};
    \node[font=\scriptsize,below] at(H4_17) {$v_{17}$};
       \node[font=\scriptsize,xshift=-4pt,yshift=4pt] at(H4_17) {\color{blue}3};
     \node[font=\scriptsize] at (0, -7.5) {List on $H_4$};
\end{tikzpicture}
\end{center}
\caption{Coloring $J_6^2$. The numbers at vertices are the number of available colors.} \label{H4fig}
\end{figure}

\noindent {\bf Case 1:} $J_6^2$ is an induced subgraph of $G^2$. \\

Let $G' = G - V(H_4)$.
Then $G'$ is also a subcubic planar graph and $|V(G')| < |V(G)|$.
Since $G$ is a minimal counterexample to Theorem \ref{main-thm},
the square of $G'$ has a proper coloring $\phi$ such that $\phi(v) \in L(v)$ for each vertex $v \in V(H)$.
For each $v_i \in V(H_4)$, we define \[
L_{H_4}(v_i) = L(v_i) \setminus \{\phi(x) : xv_i \in E(G^2) \mbox{ and } x \notin V(H_4)\}.
\]
Then, we have the following (see Figure \ref{H4fig}).
$$
|L_{H_4}(v_i)| \geq
\begin{cases}
2 & i=16, \\
3 & i=3,4,9, 12, 15, 17, \\
4 & i=1, 14, \\
5 & i=2, 5, 8, 11, 13, \\
7 & i=6, 7, 10.
\end{cases}
$$
We divide Case 1 into two subcases depending on
whether $v_3$ is adjacent to a vertex in $\{ v_{15}, v_{16}, v_{17} \}$ in $G^2$.

\medskip

\noindent {\bf Subcase 1.1:}
$v_3$ is adjacent to none of $v_{15}, v_{16}, v_{17}$ in $G^2$.

Now, we   show that
the vertices in $H_4$ can be colored from the list $L_{H_4}$
so that we obtain an $L$-coloring in $G^2$
by the following three steps.

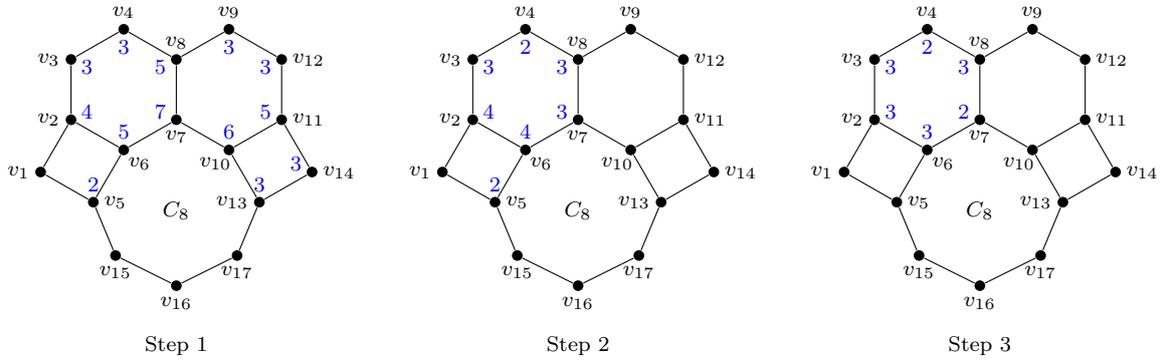
\begin{figure}[htbp]
  \begin{center}

\begin{tikzpicture}[
    v2/.style={fill=black,minimum size=4pt,ellipse,inner sep=1pt},
    scale=0.4
]
\node[v2] (H4_1) at (0, 0){};
    \node[v2] (H4_2) at (0, 2){};
    \node[v2] (H4_3) at (-1.732,-1){};
 \node[v2] (H4_4) at (-2*1.732,0){};
 \node[v2] (H4_5) at (-2*1.732,2){};
 \node[v2] (H4_6) at (-1.732,3){};
 \node[v2] (H4_7) at (1.732,-1){};
 \node[v2] (H4_8) at (2*1.732,0){};
 \node[v2] (H4_9) at (2*1.732,2){};
 \node[v2] (H4_10) at (1.732,3){};
 \node[v2] (H4_11) at (-4.464,-1.732){};
 \node[v2] (H4_12) at (-2.732,-2.732){};
 \node[v2] (H4_13) at (4.464,-1.732){};
 \node[v2] (H4_14) at (2.732,-2.732){};
 \node[v2] (H4_15) at (-2,-4.5){};
 \node[v2] (H4_16) at (0,-5.5){};
 \node[v2] (H4_17) at (2,-4.5){};

   \draw (H4_1) -- (H4_2) -- (H4_6) -- (H4_5) -- (H4_4) -- (H4_3)-- (H4_1);
   \draw (H4_2) -- (H4_10) -- (H4_9) -- (H4_8) -- (H4_7) -- (H4_1);
   \draw (H4_4) -- (H4_11) -- (H4_12) -- (H4_3);
   \draw (H4_8) -- (H4_13) -- (H4_14) -- (H4_7);
   \draw (H4_12) -- (H4_15) -- (H4_16) -- (H4_17) -- (H4_14);

    \node[font=\scriptsize] at (-0, -3) {$C_8$};

 \node[font=\scriptsize,left] at(H4_11) {$v_1$};
  \node[font=\scriptsize,left] at(H4_4) {$v_2$};
   \node[font=\scriptsize,xshift=6pt,yshift=3pt] at(H4_4) {\color{blue}4};
  \node[font=\scriptsize,left] at(H4_5) {$v_3$};
   \node[font=\scriptsize,xshift=6pt,yshift=-3pt] at(H4_5) {\color{blue}3};
  \node[font=\scriptsize,above] at(H4_6) {$v_4$};
   \node[font=\scriptsize,below] at(H4_6) {\color{blue}3};
   \node[font=\scriptsize,right] at(H4_12) {$v_5$};
    \node[font=\scriptsize,above] at(H4_12) {\color{blue}2};
  \node[font=\scriptsize,xshift=6pt,yshift=-5pt] at(H4_3) {$v_6$};
   \node[font=\scriptsize,above] at(H4_3) {\color{blue}5};
  \node[font=\scriptsize,below] at(H4_1) {$v_7$};
   \node[font=\scriptsize,xshift=-6pt,yshift=3pt] at(H4_1) {\color{blue}7};
  \node[font=\scriptsize,above] at(H4_2) {$v_8$};
   \node[font=\scriptsize,xshift=-6pt,yshift=-3pt] at(H4_2) {\color{blue}5};
   \node[font=\scriptsize,above] at(H4_10) {$v_9$};
    \node[font=\scriptsize,below] at(H4_10) {\color{blue}3};
  \node[font=\scriptsize,xshift=-5pt,yshift=-5pt] at(H4_7) {$v_{10}$};
   \node[font=\scriptsize,above] at(H4_7) {\color{blue}6};
  \node[font=\scriptsize,right] at(H4_8) {$v_{11}$};
   \node[font=\scriptsize,xshift=-6pt,yshift=3pt] at(H4_8) {\color{blue}5};
  \node[font=\scriptsize,right] at(H4_9) {$v_{12}$};
   \node[font=\scriptsize,xshift=-6pt,yshift=-3pt] at(H4_9) {\color{blue}3};
   \node[font=\scriptsize,left] at(H4_14) {$v_{13}$};
    \node[font=\scriptsize,above] at(H4_14) {\color{blue}3};
  \node[font=\scriptsize,right] at(H4_13) {$v_{14}$};
   \node[font=\scriptsize,xshift=-6pt,yshift=3pt] at(H4_13) {\color{blue}3};
  \node[font=\scriptsize,below] at(H4_15) {$v_{15}$};
 \node[font=\scriptsize,below] at(H4_16) {$v_{16}$};
   \node[font=\scriptsize,below] at(H4_17) {$v_{17}$};
 \node[font=\scriptsize] at (0, -7.5) {  Step 1};
\end{tikzpicture}\hspace{0.3cm}
\begin{tikzpicture}[
    v2/.style={fill=black,minimum size=4pt,ellipse,inner sep=1pt},
    scale=0.4
]
\node[v2] (H4_1) at (0, 0){};
    \node[v2] (H4_2) at (0, 2){};
    \node[v2] (H4_3) at (-1.732,-1){};
 \node[v2] (H4_4) at (-2*1.732,0){};
 \node[v2] (H4_5) at (-2*1.732,2){};
 \node[v2] (H4_6) at (-1.732,3){};
 \node[v2] (H4_7) at (1.732,-1){};
 \node[v2] (H4_8) at (2*1.732,0){};
 \node[v2] (H4_9) at (2*1.732,2){};
 \node[v2] (H4_10) at (1.732,3){};
 \node[v2] (H4_11) at (-4.464,-1.732){};
 \node[v2] (H4_12) at (-2.732,-2.732){};
 \node[v2] (H4_13) at (4.464,-1.732){};
 \node[v2] (H4_14) at (2.732,-2.732){};
 \node[v2] (H4_15) at (-2,-4.5){};
 \node[v2] (H4_16) at (0,-5.5){};
 \node[v2] (H4_17) at (2,-4.5){};

   \draw (H4_1) -- (H4_2) -- (H4_6) -- (H4_5) -- (H4_4) -- (H4_3)-- (H4_1);
   \draw (H4_2) -- (H4_10) -- (H4_9) -- (H4_8) -- (H4_7) -- (H4_1);
   \draw (H4_4) -- (H4_11) -- (H4_12) -- (H4_3);
   \draw (H4_8) -- (H4_13) -- (H4_14) -- (H4_7);
   \draw (H4_12) -- (H4_15) -- (H4_16) -- (H4_17) -- (H4_14);

    \node[font=\scriptsize] at (-0, -3) {$C_8$};

 \node[font=\scriptsize,left] at(H4_11) {$v_1$};
  \node[font=\scriptsize,left] at(H4_4) {$v_2$};
   \node[font=\scriptsize,xshift=6pt,yshift=3pt] at(H4_4) {\color{blue}4};
  \node[font=\scriptsize,left] at(H4_5) {$v_3$};
   \node[font=\scriptsize,xshift=6pt,yshift=-3pt] at(H4_5) {\color{blue}3};
  \node[font=\scriptsize,above] at(H4_6) {$v_4$};
   \node[font=\scriptsize,below] at(H4_6) {\color{blue}2};
   \node[font=\scriptsize,right] at(H4_12) {$v_5$};
    \node[font=\scriptsize,above] at(H4_12) {\color{blue}2};
  \node[font=\scriptsize,xshift=6pt,yshift=-5pt] at(H4_3) {$v_6$};
   \node[font=\scriptsize,above] at(H4_3) {\color{blue}4};
  \node[font=\scriptsize,below] at(H4_1) {$v_7$};
   \node[font=\scriptsize,xshift=-6pt,yshift=3pt] at(H4_1) {\color{blue}3};
  \node[font=\scriptsize,above] at(H4_2) {$v_8$};
   \node[font=\scriptsize,xshift=-6pt,yshift=-3pt] at(H4_2) {\color{blue}3};
   \node[font=\scriptsize,above] at(H4_10) {$v_9$};
  \node[font=\scriptsize,xshift=-5pt,yshift=-5pt] at(H4_7) {$v_{10}$};
  \node[font=\scriptsize,right] at(H4_8) {$v_{11}$};
  \node[font=\scriptsize,right] at(H4_9) {$v_{12}$};
   \node[font=\scriptsize,left] at(H4_14) {$v_{13}$};
  \node[font=\scriptsize,right] at(H4_13) {$v_{14}$};
  \node[font=\scriptsize,below] at(H4_15) {$v_{15}$};
  \node[font=\scriptsize,below] at(H4_16) {$v_{16}$};
    \node[font=\scriptsize,below] at(H4_17) {$v_{17}$};
     \node[font=\scriptsize] at (0, -7.5) {  Step 2};
\end{tikzpicture}\hspace{0.3cm}
\begin{tikzpicture}[
    v2/.style={fill=black,minimum size=4pt,ellipse,inner sep=1pt},
    scale=0.4
]
\node[v2] (H4_1) at (0, 0){};
    \node[v2] (H4_2) at (0, 2){};
    \node[v2] (H4_3) at (-1.732,-1){};
 \node[v2] (H4_4) at (-2*1.732,0){};
 \node[v2] (H4_5) at (-2*1.732,2){};
 \node[v2] (H4_6) at (-1.732,3){};
 \node[v2] (H4_7) at (1.732,-1){};
 \node[v2] (H4_8) at (2*1.732,0){};
 \node[v2] (H4_9) at (2*1.732,2){};
 \node[v2] (H4_10) at (1.732,3){};
 \node[v2] (H4_11) at (-4.464,-1.732){};
 \node[v2] (H4_12) at (-2.732,-2.732){};
 \node[v2] (H4_13) at (4.464,-1.732){};
 \node[v2] (H4_14) at (2.732,-2.732){};
 \node[v2] (H4_15) at (-2,-4.5){};
 \node[v2] (H4_16) at (0,-5.5){};
 \node[v2] (H4_17) at (2,-4.5){};

   \draw (H4_1) -- (H4_2) -- (H4_6) -- (H4_5) -- (H4_4) -- (H4_3)-- (H4_1);
   \draw (H4_2) -- (H4_10) -- (H4_9) -- (H4_8) -- (H4_7) -- (H4_1);
   \draw (H4_4) -- (H4_11) -- (H4_12) -- (H4_3);
   \draw (H4_8) -- (H4_13) -- (H4_14) -- (H4_7);
   \draw (H4_12) -- (H4_15) -- (H4_16) -- (H4_17) -- (H4_14);

    \node[font=\scriptsize] at (-0, -3) {$C_8$};

 \node[font=\scriptsize,left] at(H4_11) {$v_1$};
  \node[font=\scriptsize,left] at(H4_4) {$v_2$};
   \node[font=\scriptsize,xshift=6pt,yshift=3pt] at(H4_4) {\color{blue}3};
  \node[font=\scriptsize,left] at(H4_5) {$v_3$};
   \node[font=\scriptsize,xshift=6pt,yshift=-3pt] at(H4_5) {\color{blue}3};
  \node[font=\scriptsize,above] at(H4_6) {$v_4$};
   \node[font=\scriptsize,below] at(H4_6) {\color{blue}2};
   \node[font=\scriptsize,right] at(H4_12) {$v_5$};
  \node[font=\scriptsize,xshift=6pt,yshift=-5pt] at(H4_3) {$v_6$};
   \node[font=\scriptsize,above] at(H4_3) {\color{blue}3};
  \node[font=\scriptsize,below] at(H4_1) {$v_7$};
   \node[font=\scriptsize,xshift=-6pt,yshift=3pt] at(H4_1) {\color{blue}2};
  \node[font=\scriptsize,above] at(H4_2) {$v_8$};
   \node[font=\scriptsize,xshift=-6pt,yshift=-3pt] at(H4_2) {\color{blue}3};
   \node[font=\scriptsize,above] at(H4_10) {$v_9$};
  \node[font=\scriptsize,xshift=-5pt,yshift=-5pt] at(H4_7) {$v_{10}$};
  \node[font=\scriptsize,right] at(H4_8) {$v_{11}$};
  \node[font=\scriptsize,right] at(H4_9) {$v_{12}$};
   \node[font=\scriptsize,left] at(H4_14) {$v_{13}$};
  \node[font=\scriptsize,right] at(H4_13) {$v_{14}$};
  \node[font=\scriptsize,below] at(H4_15) {$v_{15}$};
  \node[font=\scriptsize,below] at(H4_16) {$v_{16}$};
    \node[font=\scriptsize,below] at(H4_17) {$v_{17}$};
     \node[font=\scriptsize] at (0, -7.5) { Step 3};
\end{tikzpicture}
\end{center}
  \caption{Coloring $H_4^2$. The numbers at vertices are the number of available colors.} \label{H4fig-step-two}
\end{figure}

\noindent {\bf Step 1:} Since $|L_{H_4}(v_1)| \geq 4$ and $|L_{H_4}(v_3)| \geq 3$, there exists a color $c_1 \in L_{H_4}(v_1)$ such that $|L_{H_4}(v_3) \setminus \{c_1\}| \geq 3$.
Color $v_1$ by $c_1$, and greedily color $v_{15}, v_{16}, v_{17}$ in order.
Let $L_{H_4}'(v_i)$ be the list of available colors at $v_i \in V(H_4) \setminus \{v_1, v_{15}, v_{16}, v_{17}\}$ after Step 1,
where thier sizes are represented in Step 1 in Figure \ref{H4fig-step-two}.

\medskip
\noindent {\bf Step 2:}
Since $|L_{H_4}'(v_{10})| \geq 6$ and $|L_{H_4}'(v_8)| \geq 5$, there exists a color $c_{10} \in L_{H_4}'(v_{10})$ such that $|L_{H_4}'(v_8) \setminus \{c_{10}\}| \geq 5$.
Color $v_{10}$ by $c_{10}$, and greedily color $v_{13}, v_{14}, v_{12}, v_{11}, v_{9}$ in order.
Let $L_{H_4}''(v_i)$ be the list of available colors at $v_i \in \{v_2, v_{3}, v_{4}, v_{5}, v_6, v_7, v_8\}$ after Step 2,
where  their  sizes are represented in Step 2 in Figure \ref{H4fig-step-two}.

\medskip
\noindent {\bf Step 3:}
Since $|L_{H_4}''(v_5)| \geq 2$,
we can color $v_5$ by a color $c_5$ so that $L_{H_4}''(v_4) \neq L_{H_4}''(v_7) \setminus \{c_5\}$.

\medskip
Since $L_{H_4}''(v_4) \neq L_{H_4}''(v_7) \setminus \{c_5\}$,
$v_2, v_3, v_4, v_6, v_7, v_8$ are colorable from the list by Lemma \ref{cycle-six-original-second}.  So,
the vertices in $H_4$ can be colored from the list $L_{H_4}$
so that we obtain an $L$-coloring in $G^2$.
This completes the proof of Subcase 1.1.


\medskip \medskip

\noindent {\bf Subcase 1.2:}
$v_3$ is adjacent to a vertex in $\{v_{15}, v_{16}, v_{17}\}$ in $G^2$.

In this case,
$v_3$ or $v_4$ is adjacent to
a vertex in $\{v_{15}, v_{16}, v_{17}\}$ in $G$,
or $v_3$ has a common neighbor with a vertex in $\{v_{15}, v_{16}, v_{17}\}$ outside $H_4$ in $G$.

In Subcase 1.1, we color $v_1$ first with keeping the size of $L(v_3)$ at least 3 since $|L_{H_4}(v_1)| \geq 4$ and $|L_{H_4}(v_3)| \geq 3$. But,
in this case, we cannot start Step 1 in Case 1.1.  Hence we modify the procedure in this subcase.

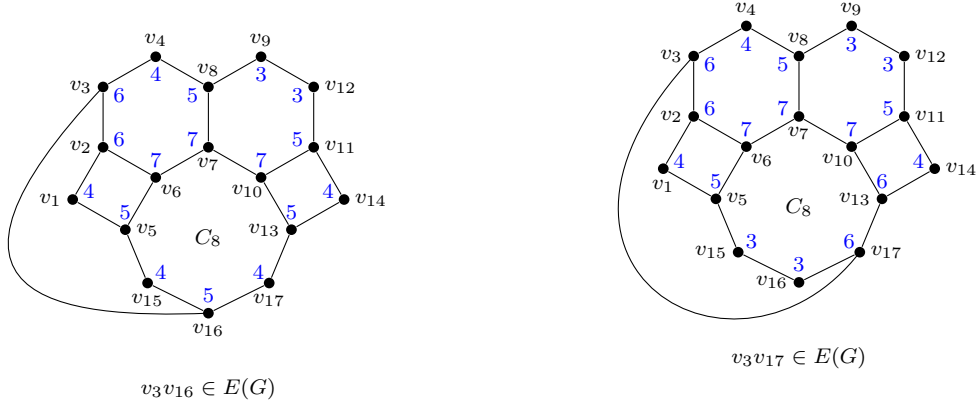
\begin{figure}[htbp]
  \begin{center}
\begin{tikzpicture}[
    v2/.style={fill=black,minimum size=4pt,ellipse,inner sep=1pt},
    scale=0.4
]
\node[v2] (H4_1) at (0, 0){};
    \node[v2] (H4_2) at (0, 2){};
    \node[v2] (H4_3) at (-1.732,-1){};
 \node[v2] (H4_4) at (-2*1.732,0){};
 \node[v2] (H4_5) at (-2*1.732,2){};
 \node[v2] (H4_6) at (-1.732,3){};
 \node[v2] (H4_7) at (1.732,-1){};
 \node[v2] (H4_8) at (2*1.732,0){};
 \node[v2] (H4_9) at (2*1.732,2){};
 \node[v2] (H4_10) at (1.732,3){};
 \node[v2] (H4_11) at (-4.464,-1.732){};
 \node[v2] (H4_12) at (-2.732,-2.732){};
 \node[v2] (H4_13) at (4.464,-1.732){};
 \node[v2] (H4_14) at (2.732,-2.732){};
 \node[v2] (H4_15) at (-2,-4.5){};
 \node[v2] (H4_16) at (0,-5.5){};
 \node[v2] (H4_17) at (2,-4.5){};

   \draw (H4_1) -- (H4_2) -- (H4_6) -- (H4_5) -- (H4_4) -- (H4_3)-- (H4_1);
   \draw (H4_2) -- (H4_10) -- (H4_9) -- (H4_8) -- (H4_7) -- (H4_1);
   \draw (H4_4) -- (H4_11) -- (H4_12) -- (H4_3);
   \draw (H4_8) -- (H4_13) -- (H4_14) -- (H4_7);
   \draw (H4_12) -- (H4_15) -- (H4_16) -- (H4_17) -- (H4_14);
     \draw (-2*1.732,2) ..controls (-10,-5)and (-5,-5.7) .. (0,-5.5);

    \node[font=\scriptsize] at (-0, -3) {$C_8$};

 \node[font=\scriptsize,left] at(H4_11) {$v_1$};
 \node[font=\scriptsize,xshift=6pt,yshift=3pt] at(H4_11) {\color{blue}4};
  \node[font=\scriptsize,left] at(H4_4) {$v_2$};
   \node[font=\scriptsize,xshift=6pt,yshift=3pt] at(H4_4) {\color{blue}6};
  \node[font=\scriptsize,left] at(H4_5) {$v_3$};
   \node[font=\scriptsize,xshift=6pt,yshift=-3pt] at(H4_5) {\color{blue}6};
  \node[font=\scriptsize,above] at(H4_6) {$v_4$};
   \node[font=\scriptsize,below] at(H4_6) {\color{blue}4};
   \node[font=\scriptsize,right] at(H4_12) {$v_5$};
    \node[font=\scriptsize,above] at(H4_12) {\color{blue}5};
  \node[font=\scriptsize,xshift=6pt,yshift=-5pt] at(H4_3) {$v_6$};
   \node[font=\scriptsize,above] at(H4_3) {\color{blue}7};
  \node[font=\scriptsize,below] at(H4_1) {$v_7$};
   \node[font=\scriptsize,xshift=-6pt,yshift=3pt] at(H4_1) {\color{blue}7};
  \node[font=\scriptsize,above] at(H4_2) {$v_8$};
   \node[font=\scriptsize,xshift=-6pt,yshift=-3pt] at(H4_2) {\color{blue}5};
   \node[font=\scriptsize,above] at(H4_10) {$v_9$};
    \node[font=\scriptsize,below] at(H4_10) {\color{blue}3};
  \node[font=\scriptsize,xshift=-5pt,yshift=-5pt] at(H4_7) {$v_{10}$};
   \node[font=\scriptsize,above] at(H4_7) {\color{blue}7};
  \node[font=\scriptsize,right] at(H4_8) {$v_{11}$};
   \node[font=\scriptsize,xshift=-6pt,yshift=3pt] at(H4_8) {\color{blue}5};
  \node[font=\scriptsize,right] at(H4_9) {$v_{12}$};
   \node[font=\scriptsize,xshift=-6pt,yshift=-3pt] at(H4_9) {\color{blue}3};
   \node[font=\scriptsize,left] at(H4_14) {$v_{13}$};
    \node[font=\scriptsize,above] at(H4_14) {\color{blue}5};
  \node[font=\scriptsize,right] at(H4_13) {$v_{14}$};
   \node[font=\scriptsize,xshift=-6pt,yshift=3pt] at(H4_13) {\color{blue}4};
  \node[font=\scriptsize,below] at(H4_15) {$v_{15}$};
   \node[font=\scriptsize,xshift=5pt,yshift=4pt] at(H4_15) {\color{blue}4};
  \node[font=\scriptsize,below] at(H4_16) {$v_{16}$};
     \node[font=\scriptsize,above] at(H4_16) {\color{blue}5};
    \node[font=\scriptsize,below] at(H4_17) {$v_{17}$};
       \node[font=\scriptsize,xshift=-4pt,yshift=4pt] at(H4_17) {\color{blue}4};
     \node[font=\scriptsize] at (0,-8) { $v_3v_{16}\in E(G)$};
\end{tikzpicture}
\hspace{1cm}
\begin{tikzpicture}[
    v2/.style={fill=black,minimum size=4pt,ellipse,inner sep=1pt},
    scale=0.4
]
\node[v2] (H4_1) at (0, 0){};
    \node[v2] (H4_2) at (0, 2){};
    \node[v2] (H4_3) at (-1.732,-1){};
 \node[v2] (H4_4) at (-2*1.732,0){};
 \node[v2] (H4_5) at (-2*1.732,2){};
 \node[v2] (H4_6) at (-1.732,3){};
 \node[v2] (H4_7) at (1.732,-1){};
 \node[v2] (H4_8) at (2*1.732,0){};
 \node[v2] (H4_9) at (2*1.732,2){};
 \node[v2] (H4_10) at (1.732,3){};
 \node[v2] (H4_11) at (-4.464,-1.732){};
 \node[v2] (H4_12) at (-2.732,-2.732){};
 \node[v2] (H4_13) at (4.464,-1.732){};
 \node[v2] (H4_14) at (2.732,-2.732){};
 \node[v2] (H4_15) at (-2,-4.5){};
 \node[v2] (H4_16) at (0,-5.5){};
 \node[v2] (H4_17) at (2,-4.5){};

   \draw (H4_1) -- (H4_2) -- (H4_6) -- (H4_5) -- (H4_4) -- (H4_3)-- (H4_1);
   \draw (H4_2) -- (H4_10) -- (H4_9) -- (H4_8) -- (H4_7) -- (H4_1);
   \draw (H4_4) -- (H4_11) -- (H4_12) -- (H4_3);
   \draw (H4_8) -- (H4_13) -- (H4_14) -- (H4_7);
   \draw (H4_12) -- (H4_15) -- (H4_16) -- (H4_17) -- (H4_14);
     \draw (-2*1.732,2) ..controls (-10,-5)and (-2,-9.7) .. (2,-4.5);

    \node[font=\scriptsize] at (-0, -3) {$C_8$};

 \node[font=\scriptsize,below] at(H4_11) {$v_1$};
 \node[font=\scriptsize,xshift=6pt,yshift=3pt] at(H4_11) {\color{blue}4};
  \node[font=\scriptsize,left] at(H4_4) {$v_2$};
   \node[font=\scriptsize,xshift=6pt,yshift=3pt] at(H4_4) {\color{blue}6};
  \node[font=\scriptsize,left] at(H4_5) {$v_3$};
   \node[font=\scriptsize,xshift=6pt,yshift=-3pt] at(H4_5) {\color{blue}6};
  \node[font=\scriptsize,above] at(H4_6) {$v_4$};
   \node[font=\scriptsize,below] at(H4_6) {\color{blue}4};
   \node[font=\scriptsize,right] at(H4_12) {$v_5$};
    \node[font=\scriptsize,above] at(H4_12) {\color{blue}5};
  \node[font=\scriptsize,xshift=6pt,yshift=-5pt] at(H4_3) {$v_6$};
   \node[font=\scriptsize,above] at(H4_3) {\color{blue}7};
  \node[font=\scriptsize,below] at(H4_1) {$v_7$};
   \node[font=\scriptsize,xshift=-6pt,yshift=3pt] at(H4_1) {\color{blue}7};
  \node[font=\scriptsize,above] at(H4_2) {$v_8$};
   \node[font=\scriptsize,xshift=-6pt,yshift=-3pt] at(H4_2) {\color{blue}5};
   \node[font=\scriptsize,above] at(H4_10) {$v_9$};
    \node[font=\scriptsize,below] at(H4_10) {\color{blue}3};
  \node[font=\scriptsize,xshift=-5pt,yshift=-5pt] at(H4_7) {$v_{10}$};
   \node[font=\scriptsize,above] at(H4_7) {\color{blue}7};
  \node[font=\scriptsize,right] at(H4_8) {$v_{11}$};
   \node[font=\scriptsize,xshift=-6pt,yshift=3pt] at(H4_8) {\color{blue}5};
  \node[font=\scriptsize,right] at(H4_9) {$v_{12}$};
   \node[font=\scriptsize,xshift=-6pt,yshift=-3pt] at(H4_9) {\color{blue}3};
   \node[font=\scriptsize,left] at(H4_14) {$v_{13}$};
    \node[font=\scriptsize,above] at(H4_14) {\color{blue}6};
  \node[font=\scriptsize,right] at(H4_13) {$v_{14}$};
   \node[font=\scriptsize,xshift=-6pt,yshift=3pt] at(H4_13) {\color{blue}4};
  \node[font=\scriptsize,left] at(H4_15) {$v_{15}$};
   \node[font=\scriptsize,xshift=5pt,yshift=4pt] at(H4_15) {\color{blue}3};
  \node[font=\scriptsize,left] at(H4_16) {$v_{16}$};
     \node[font=\scriptsize,above] at(H4_16) {\color{blue}3};
    \node[font=\scriptsize,right] at(H4_17) {$v_{17}$};
       \node[font=\scriptsize,xshift=-4pt,yshift=4pt] at(H4_17) {\color{blue}6};
     \node[font=\scriptsize] at (0,-8) { $v_3v_{17}\in E(G)$};
\end{tikzpicture}
\end{center}

\caption{Subcases 1.2.1. The numbers at vertices are the number of available colors.} \label{H4fig-bad-v3}
\end{figure}

\medskip

\noindent {\bf Subcase 1.2.1:} $v_3$ is adjacent to
a vertex in $\{v_{15}, v_{16}, v_{17}\}$ in $G$.

Since
$G$ has no $5$-cycle,
we only need to consider the case when $v_3$ is adjacent to $v_{16}$ or $v_{17}$.
In this case, the  number of available colors at vertices of $V(H_4)$ is
presented in Figure \ref{H4fig-bad-v3}.
By the planarity,
we see that $v_{12}$ is adjacent to none of $v_{15}, v_{16}, v_{17}$ in $G^2$,
and hence we can perform symmetrically the procedures in Subcase 1.1
using $v_{12}$ instead of $v_3$;
Color $v_{14}, v_{17}, v_{16}, v_{15}, v_6, v_5, v_1, v_3, v_2, v_4, v_{13}$ appropriately,
and then vertices $v_7, v_8, v_9, v_{10}, v_{11}, v_{12}$ are colorable
from the remaining lists by Lemma \ref{cycle-six-original-second}.
\\

\begin{figure}[htbp]
  \begin{center}
\begin{tikzpicture}[
    v2/.style={fill=black,minimum size=4pt,ellipse,inner sep=1pt},
    scale=0.4
]
\node[v2] (H4_1) at (0, 0){};
    \node[v2] (H4_2) at (0, 2){};
    \node[v2] (H4_3) at (-1.732,-1){};
 \node[v2] (H4_4) at (-2*1.732,0){};
 \node[v2] (H4_5) at (-2*1.732,2){};
 \node[v2] (H4_6) at (-1.732,3){};
 \node[v2] (H4_7) at (1.732,-1){};
 \node[v2] (H4_8) at (2*1.732,0){};
 \node[v2] (H4_9) at (2*1.732,2){};
 \node[v2] (H4_10) at (1.732,3){};
 \node[v2] (H4_11) at (-4.464,-1.732){};
 \node[v2] (H4_12) at (-2.732,-2.732){};
 \node[v2] (H4_13) at (4.464,-1.732){};
 \node[v2] (H4_14) at (2.732,-2.732){};
 \node[v2] (H4_15) at (-2,-4.5){};
 \node[v2] (H4_16) at (0,-5.5){};
 \node[v2] (H4_17) at (2,-4.5){};

   \draw (H4_1) -- (H4_2) -- (H4_6) -- (H4_5) -- (H4_4) -- (H4_3)-- (H4_1);
   \draw (H4_2) -- (H4_10) -- (H4_9) -- (H4_8) -- (H4_7) -- (H4_1);
   \draw (H4_4) -- (H4_11) -- (H4_12) -- (H4_3);
   \draw (H4_8) -- (H4_13) -- (H4_14) -- (H4_7);
   \draw (H4_12) -- (H4_15) -- (H4_16) -- (H4_17) -- (H4_14);
    \draw (-1.732,3) ..controls (-12,2)and (-4,-6) .. (-2,-4.5);

    \node[font=\scriptsize] at (-0, -3) {$C_8$};

 \node[font=\scriptsize,left] at(H4_11) {$v_1$};
 \node[font=\scriptsize,xshift=6pt,yshift=3pt] at(H4_11) {\color{blue}4};
  \node[font=\scriptsize,left] at(H4_4) {$v_2$};
   \node[font=\scriptsize,xshift=6pt,yshift=3pt] at(H4_4) {\color{blue}5};
  \node[font=\scriptsize,left] at(H4_5) {$v_3$};
   \node[font=\scriptsize,xshift=6pt,yshift=-3pt] at(H4_5) {\color{blue}4};
  \node[font=\scriptsize,above] at(H4_6) {$v_4$};
   \node[font=\scriptsize,below] at(H4_6) {\color{blue}6};
   \node[font=\scriptsize,right] at(H4_12) {$v_5$};
    \node[font=\scriptsize,above] at(H4_12) {\color{blue}6};
  \node[font=\scriptsize,xshift=6pt,yshift=-5pt] at(H4_3) {$v_6$};
   \node[font=\scriptsize,above] at(H4_3) {\color{blue}7};
  \node[font=\scriptsize,below] at(H4_1) {$v_7$};
   \node[font=\scriptsize,xshift=-6pt,yshift=3pt] at(H4_1) {\color{blue}7};
  \node[font=\scriptsize,above] at(H4_2) {$v_8$};
   \node[font=\scriptsize,xshift=-6pt,yshift=-3pt] at(H4_2) {\color{blue}6};
   \node[font=\scriptsize,above] at(H4_10) {$v_9$};
    \node[font=\scriptsize,below] at(H4_10) {\color{blue}3};
  \node[font=\scriptsize,xshift=-5pt,yshift=-5pt] at(H4_7) {$v_{10}$};
   \node[font=\scriptsize,above] at(H4_7) {\color{blue}7};
  \node[font=\scriptsize,right] at(H4_8) {$v_{11}$};
   \node[font=\scriptsize,xshift=-6pt,yshift=3pt] at(H4_8) {\color{blue}5};
  \node[font=\scriptsize,right] at(H4_9) {$v_{12}$};
   \node[font=\scriptsize,xshift=-6pt,yshift=-3pt] at(H4_9) {\color{blue}3};
   \node[font=\scriptsize,left] at(H4_14) {$v_{13}$};
    \node[font=\scriptsize,above] at(H4_14) {\color{blue}5};
  \node[font=\scriptsize,right] at(H4_13) {$v_{14}$};
   \node[font=\scriptsize,xshift=-6pt,yshift=3pt] at(H4_13) {\color{blue}4};
  \node[font=\scriptsize,below] at(H4_15) {$v_{15}$};
   \node[font=\scriptsize,xshift=5pt,yshift=4pt] at(H4_15) {\color{blue}6};
  \node[font=\scriptsize,below] at(H4_16) {$v_{16}$};
     \node[font=\scriptsize,above] at(H4_16) {\color{blue}3};
    \node[font=\scriptsize,below] at(H4_17) {$v_{17}$};
       \node[font=\scriptsize,xshift=-4pt,yshift=4pt] at(H4_17) {\color{blue}3};
     \node[font=\scriptsize] at (0, -7.5) {$v_4v_{15}\in E(G)$};
\end{tikzpicture}
\hspace{1cm}
\begin{tikzpicture}[
    v2/.style={fill=black,minimum size=4pt,ellipse,inner sep=1pt},
    scale=0.4
]
\node[v2] (H4_1) at (0, 0){};
    \node[v2] (H4_2) at (0, 2){};
    \node[v2] (H4_3) at (-1.732,-1){};
 \node[v2] (H4_4) at (-2*1.732,0){};
 \node[v2] (H4_5) at (-2*1.732,2){};
 \node[v2] (H4_6) at (-1.732,3){};
 \node[v2] (H4_7) at (1.732,-1){};
 \node[v2] (H4_8) at (2*1.732,0){};
 \node[v2] (H4_9) at (2*1.732,2){};
 \node[v2] (H4_10) at (1.732,3){};
 \node[v2] (H4_11) at (-4.464,-1.732){};
 \node[v2] (H4_12) at (-2.732,-2.732){};
 \node[v2] (H4_13) at (4.464,-1.732){};
 \node[v2] (H4_14) at (2.732,-2.732){};
 \node[v2] (H4_15) at (-2,-4.5){};
 \node[v2] (H4_16) at (0,-5.5){};
 \node[v2] (H4_17) at (2,-4.5){};

   \draw (H4_1) -- (H4_2) -- (H4_6) -- (H4_5) -- (H4_4) -- (H4_3)-- (H4_1);
   \draw (H4_2) -- (H4_10) -- (H4_9) -- (H4_8) -- (H4_7) -- (H4_1);
   \draw (H4_4) -- (H4_11) -- (H4_12) -- (H4_3);
   \draw (H4_8) -- (H4_13) -- (H4_14) -- (H4_7);
   \draw (H4_12) -- (H4_15) -- (H4_16) -- (H4_17) -- (H4_14);
    \draw (-1.732,3) ..controls (-12,2)and (-4,-7) .. (0,-5.5);

    \node[font=\scriptsize] at (-0, -3) {$C_8$};

 \node[font=\scriptsize,left] at(H4_11) {$v_1$};
 \node[font=\scriptsize,xshift=6pt,yshift=3pt] at(H4_11) {\color{blue}4};
  \node[font=\scriptsize,left] at(H4_4) {$v_2$};
   \node[font=\scriptsize,xshift=6pt,yshift=3pt] at(H4_4) {\color{blue}5};
  \node[font=\scriptsize,left] at(H4_5) {$v_3$};
   \node[font=\scriptsize,xshift=6pt,yshift=-3pt] at(H4_5) {\color{blue}4};
  \node[font=\scriptsize,above] at(H4_6) {$v_4$};
   \node[font=\scriptsize,below] at(H4_6) {\color{blue}6};
   \node[font=\scriptsize,right] at(H4_12) {$v_5$};
    \node[font=\scriptsize,above] at(H4_12) {\color{blue}5};
  \node[font=\scriptsize,xshift=6pt,yshift=-5pt] at(H4_3) {$v_6$};
   \node[font=\scriptsize,above] at(H4_3) {\color{blue}7};
  \node[font=\scriptsize,below] at(H4_1) {$v_7$};
   \node[font=\scriptsize,xshift=-6pt,yshift=3pt] at(H4_1) {\color{blue}7};
  \node[font=\scriptsize,above] at(H4_2) {$v_8$};
   \node[font=\scriptsize,xshift=-6pt,yshift=-3pt] at(H4_2) {\color{blue}6};
   \node[font=\scriptsize,above] at(H4_10) {$v_9$};
    \node[font=\scriptsize,below] at(H4_10) {\color{blue}3};
  \node[font=\scriptsize,xshift=-5pt,yshift=-5pt] at(H4_7) {$v_{10}$};
   \node[font=\scriptsize,above] at(H4_7) {\color{blue}7};
  \node[font=\scriptsize,right] at(H4_8) {$v_{11}$};
   \node[font=\scriptsize,xshift=-6pt,yshift=3pt] at(H4_8) {\color{blue}5};
  \node[font=\scriptsize,right] at(H4_9) {$v_{12}$};
   \node[font=\scriptsize,xshift=-6pt,yshift=-3pt] at(H4_9) {\color{blue}3};
   \node[font=\scriptsize,left] at(H4_14) {$v_{13}$};
    \node[font=\scriptsize,above] at(H4_14) {\color{blue}5};
  \node[font=\scriptsize,right] at(H4_13) {$v_{14}$};
   \node[font=\scriptsize,xshift=-6pt,yshift=3pt] at(H4_13) {\color{blue}4};
  \node[font=\scriptsize,below] at(H4_15) {$v_{15}$};
   \node[font=\scriptsize,xshift=5pt,yshift=4pt] at(H4_15) {\color{blue}4};
  \node[font=\scriptsize,below] at(H4_16) {$v_{16}$};
     \node[font=\scriptsize,above] at(H4_16) {\color{blue}5};
    \node[font=\scriptsize,below] at(H4_17) {$v_{17}$};
       \node[font=\scriptsize,xshift=-4pt,yshift=4pt] at(H4_17) {\color{blue}4};
     \node[font=\scriptsize] at (0, -7.5) {$v_4v_{16}\in E(G)$};
\end{tikzpicture}
\end{center}

\caption{Subcase 1.2.2. The numbers at vertices are the number of available colors.} \label{H4fig-bad-v3-two}
\end{figure}
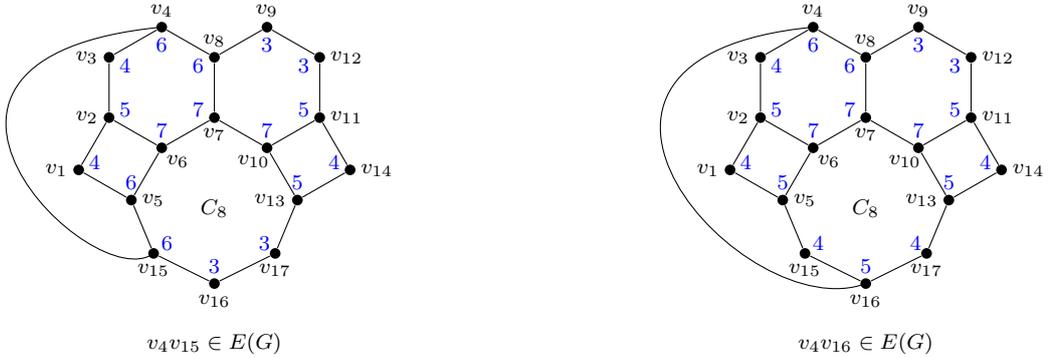


\noindent {\bf Subcase 1.2.2:} $v_4$ is adjacent to
a vertex in $\{v_{15}, v_{16}, v_{17}\}$ in $G$.

For the case when $v_4$ is adjacent to $v_{15}$ or $v_{16}$,
the number of available colors at vertices of $V(H_4)$ is
presented in Figure \ref{H4fig-bad-v3-two}.
At each case, the proof is the same style.  As an example, we provide the proof of the case
when $v_4v_{16} \in E(G)$.
We first color $v_{16}$ by a color $c_{16} \in L_{H_4}(v_{16})$
so that $|L_{H_4}(v_3) \setminus \{c_{16}\}| \geq 4$,
and then
color $v_1$, $v_{15}$, $v_{17}$, $v_{10}$, $v_{13}, v_{14}, v_{12}, v_{11}, v_9, v_5$
in order by the same procedure as in Subcase 1.1.
Then we can show that
the vertices $v_2, v_3, v_4, v_{6}, v_{7}, v_{8}$ are colorable
from the remaining lists by Lemma \ref{cycle-six-original-second}.

If $v_4$ is adjacent to $v_{17}$, then $v_{12}$ is adjacent to none of $v_{15}, v_{16}, v_{17}$ in $G^2$, and hence we can perform symmetrically the procedures in Subcase 1.1
using $v_{12}$ instead of $v_3$.

\begin{figure}[htbp]
  \begin{center}

\begin{tikzpicture}[
    v2/.style={fill=black,minimum size=4pt,ellipse,inner sep=1pt},
    scale=0.4
]
\node[v2] (H4_1) at (0, 0){};
    \node[v2] (H4_2) at (0, 2){};
    \node[v2] (H4_3) at (-1.732,-1){};
 \node[v2] (H4_4) at (-2*1.732,0){};
 \node[v2] (H4_5) at (-2*1.732,2){};
 \node[v2] (H4_6) at (-1.732,3){};
 \node[v2] (H4_7) at (1.732,-1){};
 \node[v2] (H4_8) at (2*1.732,0){};
 \node[v2] (H4_9) at (2*1.732,2){};
 \node[v2] (H4_10) at (1.732,3){};
 \node[v2] (H4_11) at (-4.464,-1.732){};
 \node[v2] (H4_12) at (-2.732,-2.732){};
 \node[v2] (H4_13) at (4.464,-1.732){};
 \node[v2] (H4_14) at (2.732,-2.732){};
 \node[v2] (H4_15) at (-2,-4.5){};
 \node[v2] (H4_16) at (0,-5.5){};
 \node[v2] (H4_17) at (2,-4.5){};
  \node[v2] (H4_18) at (-6,-3){};

   \draw (H4_1) -- (H4_2) -- (H4_6) -- (H4_5) -- (H4_4) -- (H4_3)-- (H4_1);
   \draw (H4_2) -- (H4_10) -- (H4_9) -- (H4_8) -- (H4_7) -- (H4_1);
   \draw (H4_4) -- (H4_11) -- (H4_12) -- (H4_3);
   \draw (H4_8) -- (H4_13) -- (H4_14) -- (H4_7);
   \draw (H4_12) -- (H4_15) -- (H4_16) -- (H4_17) -- (H4_14);

 \draw(H4_5)to[bend right=30](H4_18);
  \draw(H4_18)to[bend right=20](H4_15);

    \node[font=\scriptsize] at (-0, -3) {$C_8$};

 \node[font=\scriptsize,left] at(H4_11) {$v_1$};
 \node[font=\scriptsize,xshift=6pt,yshift=3pt] at(H4_11) {\color{blue}4};
  \node[font=\scriptsize,left] at(H4_4) {$v_2$};
   \node[font=\scriptsize,xshift=6pt,yshift=3pt] at(H4_4) {\color{blue}6};
  \node[font=\scriptsize,left] at(H4_5) {$v_3$};
   \node[font=\scriptsize,xshift=6pt,yshift=-3pt] at(H4_5) {\color{blue}5};
  \node[font=\scriptsize,above] at(H4_6) {$v_4$};
   \node[font=\scriptsize,below] at(H4_6) {\color{blue}4};
   \node[font=\scriptsize,right] at(H4_12) {$v_5$};
    \node[font=\scriptsize,above] at(H4_12) {\color{blue}6};
  \node[font=\scriptsize,xshift=6pt,yshift=-5pt] at(H4_3) {$v_6$};
   \node[font=\scriptsize,above] at(H4_3) {\color{blue}7};
  \node[font=\scriptsize,below] at(H4_1) {$v_7$};
   \node[font=\scriptsize,xshift=-6pt,yshift=3pt] at(H4_1) {\color{blue}7};
  \node[font=\scriptsize,above] at(H4_2) {$v_8$};
   \node[font=\scriptsize,xshift=-6pt,yshift=-3pt] at(H4_2) {\color{blue}5};
   \node[font=\scriptsize,above] at(H4_10) {$v_9$};
    \node[font=\scriptsize,below] at(H4_10) {\color{blue}3};
  \node[font=\scriptsize,xshift=-5pt,yshift=-5pt] at(H4_7) {$v_{10}$};
   \node[font=\scriptsize,above] at(H4_7) {\color{blue}7};
  \node[font=\scriptsize,right] at(H4_8) {$v_{11}$};
   \node[font=\scriptsize,xshift=-6pt,yshift=3pt] at(H4_8) {\color{blue}5};
  \node[font=\scriptsize,right] at(H4_9) {$v_{12}$};
   \node[font=\scriptsize,xshift=-6pt,yshift=-3pt] at(H4_9) {\color{blue}3};
   \node[font=\scriptsize,left] at(H4_14) {$v_{13}$};
    \node[font=\scriptsize,above] at(H4_14) {\color{blue}5};
  \node[font=\scriptsize,right] at(H4_13) {$v_{14}$};
   \node[font=\scriptsize,xshift=-6pt,yshift=3pt] at(H4_13) {\color{blue}4};
  \node[font=\scriptsize,below] at(H4_15) {$v_{15}$};
   \node[font=\scriptsize,xshift=5pt,yshift=4pt] at(H4_15) {\color{blue}5};
  \node[font=\scriptsize,below] at(H4_16) {$v_{16}$};
     \node[font=\scriptsize,above] at(H4_16) {\color{blue}3};
    \node[font=\scriptsize,below] at(H4_17) {$v_{17}$};
       \node[font=\scriptsize,xshift=-4pt,yshift=4pt] at(H4_17) {\color{blue}3};
          \node[font=\scriptsize,below] at(H4_18) {$w$};
       \node[font=\scriptsize,xshift=6pt,yshift=4pt] at(H4_18) {\color{blue}4};
     \node[font=\scriptsize] at (0,-8) { Step (i)};
\end{tikzpicture}\hspace{2cm}
\begin{tikzpicture}[
    v2/.style={fill=black,minimum size=4pt,ellipse,inner sep=1pt},
    scale=0.4
]
\node[v2] (H4_1) at (0, 0){};
    \node[v2] (H4_2) at (0, 2){};
    \node[v2] (H4_3) at (-1.732,-1){};
 \node[v2] (H4_4) at (-2*1.732,0){};
 \node[v2] (H4_5) at (-2*1.732,2){};
 \node[v2] (H4_6) at (-1.732,3){};
 \node[v2] (H4_7) at (1.732,-1){};
 \node[v2] (H4_8) at (2*1.732,0){};
 \node[v2] (H4_9) at (2*1.732,2){};
 \node[v2] (H4_10) at (1.732,3){};
 \node[v2] (H4_11) at (-4.464,-1.732){};
 \node[v2] (H4_12) at (-2.732,-2.732){};
 \node[v2] (H4_13) at (4.464,-1.732){};
 \node[v2] (H4_14) at (2.732,-2.732){};
 \node[v2] (H4_15) at (-2,-4.5){};
 \node[v2] (H4_16) at (0,-5.5){};
 \node[v2] (H4_17) at (2,-4.5){};
  \node[v2] (H4_18) at (-6,-3){};

   \draw (H4_1) -- (H4_2) -- (H4_6) -- (H4_5) -- (H4_4) -- (H4_3)-- (H4_1);
   \draw (H4_2) -- (H4_10) -- (H4_9) -- (H4_8) -- (H4_7) -- (H4_1);
   \draw (H4_4) -- (H4_11) -- (H4_12) -- (H4_3);
   \draw (H4_8) -- (H4_13) -- (H4_14) -- (H4_7);
   \draw (H4_12) -- (H4_15) -- (H4_16) -- (H4_17) -- (H4_14);

 \draw(H4_5)to[bend right=30](H4_18);
  \draw(H4_18)to[bend right=20](H4_15);

    \node[font=\scriptsize] at (-0, -3) {$C_8$};

 \node[font=\scriptsize,left] at(H4_11) {$v_1$};
  \node[font=\scriptsize,left] at(H4_4) {$v_2$};
   \node[font=\scriptsize,xshift=6pt,yshift=3pt] at(H4_4) {\color{blue}4};
  \node[font=\scriptsize,left] at(H4_5) {$v_3$};
   \node[font=\scriptsize,xshift=6pt,yshift=-3pt] at(H4_5) {\color{blue}4};
  \node[font=\scriptsize,above] at(H4_6) {$v_4$};
   \node[font=\scriptsize,below] at(H4_6) {\color{blue}3};
   \node[font=\scriptsize,right] at(H4_12) {$v_5$};
    \node[font=\scriptsize,above] at(H4_12) {\color{blue}4};
  \node[font=\scriptsize,xshift=6pt,yshift=-5pt] at(H4_3) {$v_6$};
   \node[font=\scriptsize,above] at(H4_3) {\color{blue}6};
  \node[font=\scriptsize,below] at(H4_1) {$v_7$};
   \node[font=\scriptsize,xshift=-6pt,yshift=3pt] at(H4_1) {\color{blue}7};
  \node[font=\scriptsize,above] at(H4_2) {$v_8$};
   \node[font=\scriptsize,xshift=-6pt,yshift=-3pt] at(H4_2) {\color{blue}5};
   \node[font=\scriptsize,above] at(H4_10) {$v_9$};
    \node[font=\scriptsize,below] at(H4_10) {\color{blue}3};
  \node[font=\scriptsize,xshift=-5pt,yshift=-5pt] at(H4_7) {$v_{10}$};
   \node[font=\scriptsize,above] at(H4_7) {\color{blue}7};
  \node[font=\scriptsize,right] at(H4_8) {$v_{11}$};
   \node[font=\scriptsize,xshift=-6pt,yshift=3pt] at(H4_8) {\color{blue}5};
  \node[font=\scriptsize,right] at(H4_9) {$v_{12}$};
   \node[font=\scriptsize,xshift=-6pt,yshift=-3pt] at(H4_9) {\color{blue}3};
   \node[font=\scriptsize,left] at(H4_14) {$v_{13}$};
    \node[font=\scriptsize,above] at(H4_14) {\color{blue}5};
  \node[font=\scriptsize,right] at(H4_13) {$v_{14}$};
   \node[font=\scriptsize,xshift=-6pt,yshift=3pt] at(H4_13) {\color{blue}4};
  \node[font=\scriptsize,below] at(H4_15) {$v_{15}$};
   \node[font=\scriptsize,xshift=5pt,yshift=4pt] at(H4_15) {\color{blue}3};
  \node[font=\scriptsize,below] at(H4_16) {$v_{16}$};
     \node[font=\scriptsize,above] at(H4_16) {\color{blue}2};
    \node[font=\scriptsize,below] at(H4_17) {$v_{17}$};
       \node[font=\scriptsize,xshift=-4pt,yshift=4pt] at(H4_17) {\color{blue}3};
          \node[font=\scriptsize,below] at(H4_18) {$w$};
     \node[font=\scriptsize] at (0,-8) { Step (ii)};
\end{tikzpicture}
\end{center}
\caption{Subcase 1.2.3. The numbers at vertices are the number of available colors.} \label{H4fig-step-two-bad-case}
\end{figure}
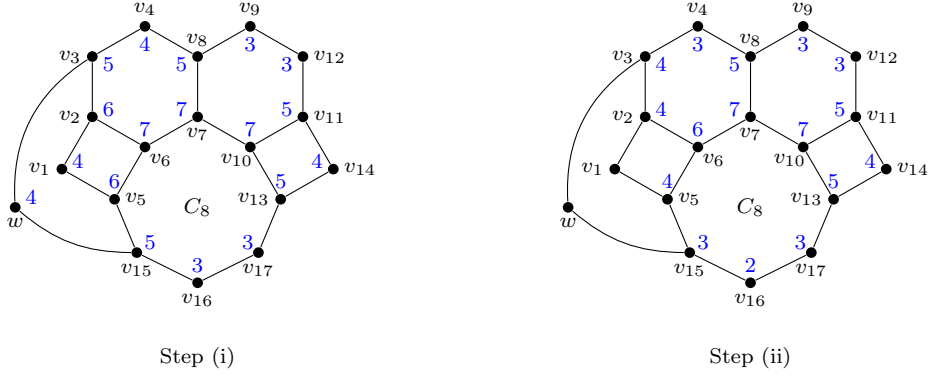

\medskip
\noindent {\bf Subcase 1.2.3:} $v_3$ has a common neighbor $w$
with a vertex in $\{v_{15}, v_{16}, v_{17}\}$ in $G$,
where $w \notin V(H_4)$.
\medskip

In this case, we know that the distance between $v_1$ and $w$ in $G$ is at least 3.  In fact, if
$v_1$ is adjacent to $w$, then $v_1, v_2, v_3, w$ form a 4-cycle,
which forms $H_1$ but contradicts Lemma \ref{reducible-H0}.
So $v_1$ and $w$ cannot be adjacent.  And if $v_1$ and $w$ have a common neighbor, then it makes a 5-cycle, a contradiction.

We uncolor $w$
and we define for each $v_i \in V(H_4) \cup \{w\}$,
\begin{equation*} \label{list-L}
L_{H_4}'(v_i) = L(v_i) \setminus \{\phi(x) : xv_i \in E(G^2) \mbox{ and } x \notin V(H_4) \cup \{w\}\}.
\end{equation*}

As an illustration, we consider the case when $v_3$ has a common neighbor with $v_{15}$,
but the other cases are similarly shown.
In this case, the size of color list is like Step (i) in Figure \ref{H4fig-step-two-bad-case}.

If $L_{H_4}'(v_1) \cap L_{H_4}'(w) \neq \emptyset$, then color $v_1$ and $w$ by a color $c_1 \in L_{H_4}'(v_1) \cap L_{H_4}'(w)$.  If $L_{H_4}'(v_1) \cap L_{H_4}'(w) = \emptyset$, then since $|L_{H_4}(v_1)' \cup L_{H_4}'(w)| \geq 8$ and $|L_{H_4}'(v_3)| \geq 5$, we can color $v_1$ by a color $c_1$ and $w$ by a color $c_w$ so that $|L_{H_4}'(v_3) \setminus \{c_1, c_w\}| \geq 4$.
In either case,
the sizes of list after coloring $v_1$ and $w$
are represented in Step (ii) in Figure \ref{H4fig-step-two-bad-case}.
And then, we follow the same procedure as Subcase 1.1.
We can show that
the vertices in $V(H_4) \cup \{w\}$ can be colored from the list $L_{H_4}'$
so that we obtain an $L$-coloring in $G^2$.

This completes the proof of Subcase 1.2. \\

Before we begin Cases 2 and 3,
we investigate some useful remark on
the structures
about the vertices $v_5$ and $v_{13}$ in $J_6$.

\begin{remark} \label{restrict-case} \rm
When  $J^2_6$  is not an induced subgraph in $G^2$,
the following holds.
\begin{enumerate}[(a)]
\item $v_{13}$ is adjacent to no vertex in $V(J_6)\setminus\{v_{10}, v_{14}\}$.
(By symmetry, $v_{5}$ is adjacent to no vertex in $V(J_6)\setminus\{v_{1}, v_{6}\}$.)
We can see this remark as follows.
\begin{enumerate}[(i)]
\item If $v_{13}$ is adjacent to $v_{1}$, then the face $F$ has length at least 9 since  $v_{5}$ has its third neighbor $v_{15}$ and $G$ has no 1-vertex (see Case (a) in Figure \ref{H4-remark-fig}).  So, $v_{13}$ is  not adjacent to $v_1$.
By the same argument, $v_{13}$ is not adjacent to $v_{3}$.

\item If $v_{13}$ is  adjacent to $v_4$ or $v_9$, then it makes a 5-cycle, which is forbidden.

\item If $v_{13}$ is adjacent to $v_{12}$, then it makes $J_3$,
which does not exist by Lemma \ref{C4-share-two-edge}.
 So $v_{13}$ is not adjacent to $v_{12}$.
\end{enumerate}

\begin{figure}[htbp]
  \begin{center}

\begin{tikzpicture}[
    v2/.style={fill=black,minimum size=4pt,ellipse,inner sep=1pt},invis/.style={
    draw=none, fill=none,minimum size=0pt, inner sep=0pt
},
    scale=0.4
]
\node[v2] (H4_1) at (0, 0){};
    \node[v2] (H4_2) at (0, 2){};
    \node[v2] (H4_3) at (-1.732,-1){};
 \node[v2] (H4_4) at (-2*1.732,0){};
 \node[v2] (H4_5) at (-2*1.732,2){};
 \node[v2] (H4_6) at (-1.732,3){};
 \node[v2] (H4_7) at (1.732,-1){};
 \node[v2] (H4_8) at (2*1.732,0){};
 \node[v2] (H4_9) at (2*1.732,2){};
 \node[v2] (H4_10) at (1.732,3){};
 \node[v2] (H4_11) at (-4.464,-1.732){};
 \node[v2] (H4_12) at (-2.732,-2.732){};
 \node[v2] (H4_13) at (4.464,-1.732){};
 \node[v2] (H4_14) at (2.732,-2.732){};

 \node[v2] (H4_15) at (-1.7,-3.5){};

   \draw (H4_1) -- (H4_2) -- (H4_6) -- (H4_5) -- (H4_4) -- (H4_3)-- (H4_1);
   \draw (H4_2) -- (H4_10) -- (H4_9) -- (H4_8) -- (H4_7) -- (H4_1);
   \draw (H4_4) -- (H4_11) -- (H4_12) -- (H4_3);
   \draw (H4_8) -- (H4_13) -- (H4_14) -- (H4_7);
   \draw(H4_14)to[bend left=90] (H4_11);
  \draw (H4_12) -- (H4_15);
    \node[font=\scriptsize] at (-0, -2.4) {$F$};

 \node[font=\scriptsize,left] at(H4_11) {$v_1$};
  \node[font=\scriptsize,left] at(H4_4) {$v_2$};
  \node[font=\scriptsize,left] at(H4_5) {$v_3$};
  \node[font=\scriptsize,above] at(H4_6) {$v_4$};
   \node[font=\scriptsize,right] at(H4_12) {$v_5$};
  \node[font=\scriptsize,xshift=6pt,yshift=-5pt] at(H4_3) {$v_6$};
  \node[font=\scriptsize,below] at(H4_1) {$v_7$};
  \node[font=\scriptsize,above] at(H4_2) {$v_8$};
   \node[font=\scriptsize,above] at(H4_10) {$v_9$};
  \node[font=\scriptsize,xshift=-5pt,yshift=-5pt] at(H4_7) {$v_{10}$};
  \node[font=\scriptsize,right] at(H4_8) {$v_{11}$};
  \node[font=\scriptsize,right] at(H4_9) {$v_{12}$};
   \node[font=\scriptsize,left] at(H4_14) {$v_{13}$};
  \node[font=\scriptsize,right] at(H4_13) {$v_{14}$};

   \node[font=\scriptsize,right] at(H4_15) {$v_{15}$};
     \node[font=\scriptsize] at (0,-8) { Case (a)};
\end{tikzpicture}\hspace{2cm}
\begin{tikzpicture}[
    v2/.style={fill=black,minimum size=4pt,ellipse,inner sep=1pt},
    scale=0.4
]
\node[v2] (H4_1) at (0, 0){};
    \node[v2] (H4_2) at (0, 2){};
    \node[v2] (H4_3) at (-1.732,-1){};
 \node[v2] (H4_4) at (-2*1.732,0){};
 \node[v2] (H4_5) at (-2*1.732,2){};
 \node[v2] (H4_6) at (-1.732,3){};
 \node[v2] (H4_7) at (1.732,-1){};
 \node[v2] (H4_8) at (2*1.732,0){};
 \node[v2] (H4_9) at (2*1.732,2){};
 \node[v2] (H4_10) at (1.732,3){};
 \node[v2] (H4_11) at (-4.464,-1.732){};
 \node[v2] (H4_12) at (-2.732,-2.732){};
 \node[v2] (H4_13) at (4.464,-1.732){};
 \node[v2] (H4_14) at (2.732,-2.732){};
 \node[v2] (H4_15) at (0,-5.5){};

 \node[v2] (H4_16) at (-0.911, -4.577){};
 \node[v2] (H4_17) at (-1.821, -3.655){};

   \draw (H4_1) -- (H4_2) -- (H4_6) -- (H4_5) -- (H4_4) -- (H4_3)-- (H4_1);
   \draw (H4_2) -- (H4_10) -- (H4_9) -- (H4_8) -- (H4_7) -- (H4_1);
   \draw (H4_4) -- (H4_11) -- (H4_12) -- (H4_3);
   \draw (H4_8) -- (H4_13) -- (H4_14) -- (H4_7);
    \draw (H4_12) -- (H4_17) -- (H4_16) -- (H4_15);

\draw(H4_14)to[bend left=20] (H4_15);
\draw(H4_11)to[bend right=30] (H4_15);
    \node[font=\scriptsize] at (-0, -2.4) {$F$};

 \node[font=\scriptsize,left] at(H4_11) {$v_1$};
  \node[font=\scriptsize,left] at(H4_4) {$v_2$};
  \node[font=\scriptsize,left] at(H4_5) {$v_3$};
  \node[font=\scriptsize,above] at(H4_6) {$v_4$};
   \node[font=\scriptsize,right] at(H4_12) {$v_5$};
  \node[font=\scriptsize,xshift=6pt,yshift=-5pt] at(H4_3) {$v_6$};
  \node[font=\scriptsize,below] at(H4_1) {$v_7$};
  \node[font=\scriptsize,above] at(H4_2) {$v_8$};
   \node[font=\scriptsize,above] at(H4_10) {$v_9$};
  \node[font=\scriptsize,xshift=-5pt,yshift=-5pt] at(H4_7) {$v_{10}$};
  \node[font=\scriptsize,right] at(H4_8) {$v_{11}$};
  \node[font=\scriptsize,right] at(H4_9) {$v_{12}$};
   \node[font=\scriptsize,left] at(H4_14) {$v_{13}$};
  \node[font=\scriptsize,right] at(H4_13) {$v_{14}$};
  \node[font=\scriptsize,below] at(H4_15) {$v_{17}$};
  \node[font=\scriptsize,right] at(H4_16) {$v_{16}$};
    \node[font=\scriptsize,right] at(H4_17) {$v_{15}$};

     \node[font=\scriptsize] at (0,-8) { Case (b)};
\end{tikzpicture}
\end{center}
\caption{$v_1$ and $v_{13}$ are adjacent in Case (a), $v_1$ and $v_{13}$ have a common neighbor $w$ in Case (b).} \label{H4-remark-fig}
\end{figure}
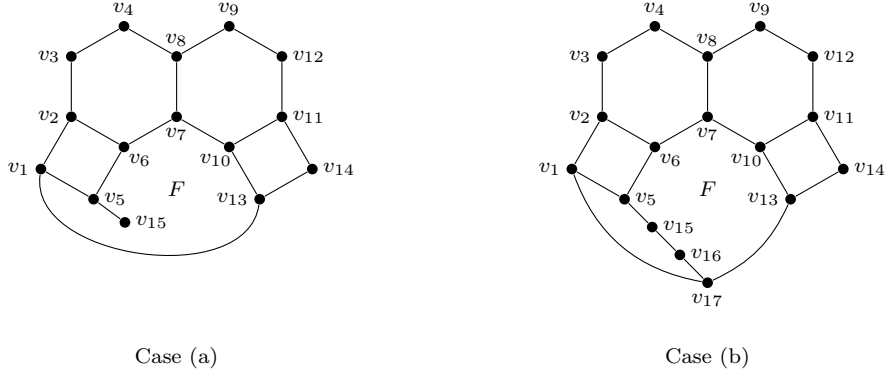

\item $v_1$ and $v_{13}$ cannot have a common neighbor outside $J_6$.
(By symmetry $v_5$ and $v_{14}$ cannot have a common neighbor outside $J_6$.)
We can see this remark as follows.

If $v_1$ and $v_{13}$ have a common neighbor outside $J_6$,
that has to be $v_{17}$,
then there must be a path $v_{17} v_{16} v_{15} v_5$
since the face $F$ has length 8 (see Case (b) in Figure \ref{H4-remark-fig}).
Then
$v_1v_5v_{15} v_{16} v_{17} v_1$
forms a 5-cycle, which is forbidden.  So,  $v_1$ and $v_{13}$ cannot have a common neighbor outside $J_6$.
\end{enumerate}
\end{remark}

\noindent {\bf Case 2:} $J_6^2$ is not an induced subgraph of $G^2$
and  $E(G[V(J_6)]) - E(J_6) \neq \emptyset$.

Let $G' = G - V(J_6)$.
Then $G'$ is also a subcubic planar graph and $|V(G')| < |V(G)|$.
Since $G$ is a minimal counterexample to Theorem \ref{main-thm},
the square of $G'$ has a proper coloring $\phi$ such that $\phi(v) \in L(v)$ for each vertex $v \in V(G')$.
For each $v_i \in V(J_6)$, we define \[
L_{J_6}(v_i) = L(v_i) \setminus \{\phi(x) : xv_i \in E(G^2) \mbox{ and } x \notin V(J_6)\}.
\]
Before we proceed with Case 2,  we simplify cases.


\medskip

\noindent {\bf Simplifying cases:}
\begin{itemize}
\item
As in Remark \ref{restrict-case}(a),
$v_{13}$ is adjacent to no vertex in $V(J_6)\setminus\{v_{10}, v_{14}\}$,
and  $v_{5}$ is adjacent to no vertex in $V(J_6)\setminus\{v_{1}, v_{6}\}$.

\item The vertices in each of the following pairs are nonadjacent
since it makes a 5-cycle:
$\{v_1, v_3\}$, $\{v_3, v_{12}\}$, and $\{v_{12}, v_{14}\}$.

\item The vertices in each of the following pairs are nonadjacent since it makes $F_3$,
which does not exist by Lemma \ref{C3-C6}(b):
$\{v_{4}, v_{9}\}$.

\item The vertices in each of the following pairs are nonadjacent since it makes $H_1$,
which does not exist by Lemma \ref{reducible-H0}:
$\{v_1, v_4\}$, and $\{v_9, v_{14}\}$.

\item $v_3$ and $v_{9}$ cannot be adjacent in $G$ since it makes $J_5$,
consisting of the 4-cycles $v_3v_4v_8v_{9}v_3$ and $v_{10}v_{11}v_{14}v_{13}v_{10}$,
and the 6-cycle $v_7v_8v_9v_{12}v_{11}v_{10}v_7$,
a contradiction to Lemma \ref{H2-type-two-reducible}.

Symmetrically,
$v_4$ and $v_{12}$ cannot be adjacent in $G$.
\end{itemize}


\begin{figure}[htbp]
  \begin{center}
\begin{minipage}{0.24\textwidth}
\centering
 \raisebox{2ex}{
 \begin{tikzpicture}[
    v2/.style={fill=black,minimum size=4pt,ellipse,inner sep=1pt},invis/.style={
    draw=none, fill=none,minimum size=0pt, inner sep=0pt
},
    scale=0.28]
\node[v2] (H4_1) at (0, 0){};
    \node[v2] (H4_2) at (0, 2){};
    \node[v2] (H4_3) at (-1.732,-1){};
 \node[v2] (H4_4) at (-2*1.732,0){};
 \node[v2] (H4_5) at (-2*1.732,2){};
 \node[v2] (H4_6) at (-1.732,3){};
 \node[v2] (H4_7) at (1.732,-1){};
 \node[v2] (H4_8) at (2*1.732,0){};
 \node[v2] (H4_9) at (2*1.732,2){};
 \node[v2] (H4_10) at (1.732,3){};
 \node[v2] (H4_11) at (-4.464,-1.732){};
 \node[v2] (H4_12) at (-2.732,-2.732){};
 \node[v2] (H4_13) at (4.464,-1.732){};
 \node[v2] (H4_14) at (2.732,-2.732){};
\node[v2] (H4_15) at (0,4.6){};

   \draw (H4_1) -- (H4_2) -- (H4_6) -- (H4_5) -- (H4_4) -- (H4_3)-- (H4_1);
   \draw (H4_2) -- (H4_10) -- (H4_9) -- (H4_8) -- (H4_7) -- (H4_1);
   \draw (H4_4) -- (H4_11) -- (H4_12) -- (H4_3);
   \draw (H4_8) -- (H4_13) -- (H4_14) -- (H4_7);
   \draw(H4_11)to[bend right=65] (H4_13);
\draw (H4_5)to[bend left=45] (H4_15);
 \draw (H4_15)to[bend left=40] (H4_9);
 \node[font=\scriptsize,left] at(H4_11) {$v_1$};
 \node[font=\scriptsize,xshift=6pt,yshift=3pt] at(H4_11) {\color{blue}6};
  \node[font=\scriptsize,left] at(H4_4) {$v_2$};
   \node[font=\scriptsize,xshift=6pt,yshift=3pt] at(H4_4) {\color{blue}6};
  \node[font=\scriptsize,left] at(H4_5) {$v_3$};
  \node[font=\scriptsize,xshift=6pt,yshift=-3pt] at(H4_5) {\color{blue}4};
  \node[font=\scriptsize,above] at(H4_6) {$v_4$};
   \node[font=\scriptsize,below] at(H4_6) {\color{blue}3};
   \node[font=\scriptsize,right] at(H4_12) {$v_5$};
    \node[font=\scriptsize,above] at(H4_12) {\color{blue}4};
  \node[font=\scriptsize,xshift=6pt,yshift=-5pt] at(H4_3) {$v_6$};
   \node[font=\scriptsize,above] at(H4_3) {\color{blue}6};
  \node[font=\scriptsize,below] at(H4_1) {$v_7$};
   \node[font=\scriptsize,xshift=-6pt,yshift=3pt] at(H4_1) {\color{blue}7};
  \node[font=\scriptsize,above] at(H4_2) {$v_8$};
   \node[font=\scriptsize,xshift=-6pt,yshift=-3pt] at(H4_2) {\color{blue}5};
   \node[font=\scriptsize,above] at(H4_10) {$v_9$};
    \node[font=\scriptsize,below] at(H4_10) {\color{blue}3};
  \node[font=\scriptsize,xshift=-5pt,yshift=-5pt] at(H4_7) {$v_{10}$};
   \node[font=\scriptsize,above] at(H4_7) {\color{blue}6};
  \node[font=\scriptsize,right] at(H4_8) {$v_{11}$};
   \node[font=\scriptsize,xshift=-6pt,yshift=3pt] at(H4_8) {\color{blue}6};
  \node[font=\scriptsize,right] at(H4_9) {$v_{12}$};
   \node[font=\scriptsize,xshift=-6pt,yshift=-3pt] at(H4_9) {\color{blue}4};
   \node[font=\scriptsize,left] at(H4_14) {$v_{13}$};
    \node[font=\scriptsize,above] at(H4_14) {\color{blue}4};
  \node[font=\scriptsize,right] at(H4_13) {$v_{14}$};
   \node[font=\scriptsize,xshift=-6pt,yshift=3pt] at(H4_13) {\color{blue}6};
  \node[font=\scriptsize,above] at(H4_15) {$w$};

     \node[font=\scriptsize] at (0,-6.2) {(a) $v_1v_{14},v_3w,wv_{12}\in E(G)$};
\end{tikzpicture}}
\end{minipage}
\begin{minipage}{0.24\textwidth}
\centering
\begin{tikzpicture}[
    v2/.style={fill=black,minimum size=4pt,ellipse,inner sep=1pt},invis/.style={
    draw=none, fill=none,minimum size=0pt, inner sep=0pt
},
    scale=0.28]
\node[v2] (H4_1) at (0, 0){};
    \node[v2] (H4_2) at (0, 2){};
    \node[v2] (H4_3) at (-1.732,-1){};
 \node[v2] (H4_4) at (-2*1.732,0){};
 \node[v2] (H4_5) at (-2*1.732,2){};
 \node[v2] (H4_6) at (-1.732,3){};
 \node[v2] (H4_7) at (1.732,-1){};
 \node[v2] (H4_8) at (2*1.732,0){};
 \node[v2] (H4_9) at (2*1.732,2){};
 \node[v2] (H4_10) at (1.732,3){};
 \node[v2] (H4_11) at (-4.464,-1.732){};
 \node[v2] (H4_12) at (-2.732,-2.732){};
 \node[v2] (H4_13) at (4.464,-1.732){};
 \node[v2] (H4_14) at (2.732,-2.732){};

   \draw (H4_1) -- (H4_2) -- (H4_6) -- (H4_5) -- (H4_4) -- (H4_3)-- (H4_1);
   \draw (H4_2) -- (H4_10) -- (H4_9) -- (H4_8) -- (H4_7) -- (H4_1);
   \draw (H4_4) -- (H4_11) -- (H4_12) -- (H4_3);
   \draw (H4_8) -- (H4_13) -- (H4_14) -- (H4_7);
   \draw(H4_11)to[bend right=60] (H4_13);

 \node[font=\scriptsize,left] at(H4_11) {$v_1$};
 \node[font=\scriptsize,xshift=6pt,yshift=3pt] at(H4_11) {\color{blue}6};
  \node[font=\scriptsize,left] at(H4_4) {$v_2$};
   \node[font=\scriptsize,xshift=6pt,yshift=3pt] at(H4_4) {\color{blue}6};
  \node[font=\scriptsize,left] at(H4_5) {$v_3$};
  \node[font=\scriptsize,xshift=6pt,yshift=-3pt] at(H4_5) {\color{blue}3};
  \node[font=\scriptsize,above] at(H4_6) {$v_4$};
   \node[font=\scriptsize,below] at(H4_6) {\color{blue}3};
   \node[font=\scriptsize,right] at(H4_12) {$v_5$};
    \node[font=\scriptsize,above] at(H4_12) {\color{blue}4};
  \node[font=\scriptsize,xshift=6pt,yshift=-5pt] at(H4_3) {$v_6$};
   \node[font=\scriptsize,above] at(H4_3) {\color{blue}6};
  \node[font=\scriptsize,below] at(H4_1) {$v_7$};
   \node[font=\scriptsize,xshift=-6pt,yshift=3pt] at(H4_1) {\color{blue}7};
  \node[font=\scriptsize,above] at(H4_2) {$v_8$};
   \node[font=\scriptsize,xshift=-6pt,yshift=-3pt] at(H4_2) {\color{blue}5};
   \node[font=\scriptsize,above] at(H4_10) {$v_9$};
    \node[font=\scriptsize,below] at(H4_10) {\color{blue}3};
  \node[font=\scriptsize,xshift=-5pt,yshift=-5pt] at(H4_7) {$v_{10}$};
   \node[font=\scriptsize,above] at(H4_7) {\color{blue}6};
  \node[font=\scriptsize,right] at(H4_8) {$v_{11}$};
   \node[font=\scriptsize,xshift=-6pt,yshift=3pt] at(H4_8) {\color{blue}6};
  \node[font=\scriptsize,right] at(H4_9) {$v_{12}$};
   \node[font=\scriptsize,xshift=-6pt,yshift=-3pt] at(H4_9) {\color{blue}3};
   \node[font=\scriptsize,left] at(H4_14) {$v_{13}$};
    \node[font=\scriptsize,above] at(H4_14) {\color{blue}4};
  \node[font=\scriptsize,right] at(H4_13) {$v_{14}$};
   \node[font=\scriptsize,xshift=-6pt,yshift=3pt] at(H4_13) {\color{blue}6};

     \node[font=\scriptsize] at (0,-6.5) {(b) $v_1v_{14}\in E(G)$};
\end{tikzpicture}
\end{minipage}
\begin{minipage}{0.24\textwidth}
\hspace{-1cm}
 \raisebox{4ex}{
\begin{tikzpicture}[
    v2/.style={fill=black,minimum size=4pt,ellipse,inner sep=1pt},invis/.style={
    draw=none, fill=none,minimum size=0pt, inner sep=0pt
},
    scale=0.28
]
\node[v2] (H4_1) at (0, 0){};
    \node[v2] (H4_2) at (0, 2){};
    \node[v2] (H4_3) at (-1.732,-1){};
 \node[v2] (H4_4) at (-2*1.732,0){};
 \node[v2] (H4_5) at (-2*1.732,2){};
 \node[v2] (H4_6) at (-1.732,3){};
 \node[v2] (H4_7) at (1.732,-1){};
 \node[v2] (H4_8) at (2*1.732,0){};
 \node[v2] (H4_9) at (2*1.732,2){};
 \node[v2] (H4_10) at (1.732,3){};
 \node[v2] (H4_11) at (-4.464,-1.732){};
 \node[v2] (H4_12) at (-2.732,-2.732){};
 \node[v2] (H4_13) at (4.464,-1.732){};
 \node[v2] (H4_14) at (2.732,-2.732){};

   \draw (H4_1) -- (H4_2) -- (H4_6) -- (H4_5) -- (H4_4) -- (H4_3)-- (H4_1);
   \draw (H4_2) -- (H4_10) -- (H4_9) -- (H4_8) -- (H4_7) -- (H4_1);
   \draw (H4_4) -- (H4_11) -- (H4_12) -- (H4_3);
   \draw (H4_8) -- (H4_13) -- (H4_14) -- (H4_7);
   \draw(-4.464,-1.732) ..controls (-8.5,4)and (2,8).. (2*1.732,2);

 \node[font=\scriptsize,below] at(H4_11) {$v_1$};
 \node[font=\scriptsize,xshift=6pt,yshift=3pt] at(H4_11) {\color{blue}6};
  \node[font=\scriptsize,left] at(H4_4) {$v_2$};
   \node[font=\scriptsize,xshift=6pt,yshift=3pt] at(H4_4) {\color{blue}6};
  \node[font=\scriptsize,left] at(H4_5) {$v_3$};
  \node[font=\scriptsize,xshift=6pt,yshift=-3pt] at(H4_5) {\color{blue}3};
  \node[font=\scriptsize,above] at(H4_6) {$v_4$};
   \node[font=\scriptsize,below] at(H4_6) {\color{blue}3};
   \node[font=\scriptsize,right] at(H4_12) {$v_5$};
    \node[font=\scriptsize,above] at(H4_12) {\color{blue}4};
  \node[font=\scriptsize,xshift=6pt,yshift=-5pt] at(H4_3) {$v_6$};
   \node[font=\scriptsize,above] at(H4_3) {\color{blue}6};
  \node[font=\scriptsize,below] at(H4_1) {$v_7$};
   \node[font=\scriptsize,xshift=-6pt,yshift=3pt] at(H4_1) {\color{blue}7};
  \node[font=\scriptsize,above] at(H4_2) {$v_8$};
   \node[font=\scriptsize,xshift=-6pt,yshift=-3pt] at(H4_2) {\color{blue}5};
   \node[font=\scriptsize,above] at(H4_10) {$v_9$};
    \node[font=\scriptsize,below] at(H4_10) {\color{blue}4};
  \node[font=\scriptsize,xshift=-5pt,yshift=-5pt] at(H4_7) {$v_{10}$};
   \node[font=\scriptsize,above] at(H4_7) {\color{blue}6};
  \node[font=\scriptsize,right] at(H4_8) {$v_{11}$};
   \node[font=\scriptsize,xshift=-6pt,yshift=3pt] at(H4_8) {\color{blue}6};
  \node[font=\scriptsize,xshift=10pt,yshift=-1pt] at(H4_9) {$v_{12}$};
   \node[font=\scriptsize,xshift=-6pt,yshift=-3pt] at(H4_9) {\color{blue}6};
   \node[font=\scriptsize,left] at(H4_14) {$v_{13}$};
    \node[font=\scriptsize,above] at(H4_14) {\color{blue}3};
  \node[font=\scriptsize,right] at(H4_13) {$v_{14}$};
   \node[font=\scriptsize,xshift=-6pt,yshift=3pt] at(H4_13) {\color{blue}3};

     \node[font=\scriptsize] at (0,-5.8) {(c) $v_1v_{12}\in E(G)$};
\end{tikzpicture}}
\end{minipage}
\begin{minipage}{0.24\textwidth}
\hspace{-1cm}
 \raisebox{4ex}{
 \begin{tikzpicture}[
    v2/.style={fill=black,minimum size=4pt,ellipse,inner sep=1pt},invis/.style={
    draw=none, fill=none,minimum size=0pt, inner sep=0pt
},
    scale=0.28
]
\node[v2] (H4_1) at (0, 0){};
    \node[v2] (H4_2) at (0, 2){};
    \node[v2] (H4_3) at (-1.732,-1){};
 \node[v2] (H4_4) at (-2*1.732,0){};
 \node[v2] (H4_5) at (-2*1.732,2){};
 \node[v2] (H4_6) at (-1.732,3){};
 \node[v2] (H4_7) at (1.732,-1){};
 \node[v2] (H4_8) at (2*1.732,0){};
 \node[v2] (H4_9) at (2*1.732,2){};
 \node[v2] (H4_10) at (1.732,3){};
 \node[v2] (H4_11) at (-4.464,-1.732){};
 \node[v2] (H4_12) at (-2.732,-2.732){};
 \node[v2] (H4_13) at (4.464,-1.732){};
 \node[v2] (H4_14) at (2.732,-2.732){};

   \draw (H4_1) -- (H4_2) -- (H4_6) -- (H4_5) -- (H4_4) -- (H4_3)-- (H4_1);
   \draw (H4_2) -- (H4_10) -- (H4_9) -- (H4_8) -- (H4_7) -- (H4_1);
   \draw (H4_4) -- (H4_11) -- (H4_12) -- (H4_3);
   \draw (H4_8) -- (H4_13) -- (H4_14) -- (H4_7);
   \draw(-4.464,-1.732) ..controls (-8.5,4)and (2,8).. (1.732,3);

 \node[font=\scriptsize,left] at(H4_11) {$v_1$};
 \node[font=\scriptsize,xshift=6pt,yshift=3pt] at(H4_11) {\color{blue}6};
  \node[font=\scriptsize,left] at(H4_4) {$v_2$};
   \node[font=\scriptsize,xshift=6pt,yshift=3pt] at(H4_4) {\color{blue}6};
  \node[font=\scriptsize,left] at(H4_5) {$v_3$};
  \node[font=\scriptsize,xshift=6pt,yshift=-3pt] at(H4_5) {\color{blue}3};
  \node[font=\scriptsize, xshift=-6pt,yshift=5pt] at(H4_6) {$v_4$};
   \node[font=\scriptsize,below] at(H4_6) {\color{blue}3};
   \node[font=\scriptsize,right] at(H4_12) {$v_5$};
    \node[font=\scriptsize,above] at(H4_12) {\color{blue}4};
  \node[font=\scriptsize,xshift=6pt,yshift=-5pt] at(H4_3) {$v_6$};
   \node[font=\scriptsize,above] at(H4_3) {\color{blue}6};
  \node[font=\scriptsize,below] at(H4_1) {$v_7$};
   \node[font=\scriptsize,xshift=-6pt,yshift=3pt] at(H4_1) {\color{blue}7};
  \node[font=\scriptsize,above] at(H4_2) {$v_8$};
   \node[font=\scriptsize,xshift=-6pt,yshift=-3pt] at(H4_2) {\color{blue}6};
   \node[font=\scriptsize,right] at(H4_10) {$v_9$};
    \node[font=\scriptsize,below] at(H4_10) {\color{blue}6};
  \node[font=\scriptsize,xshift=-5pt,yshift=-5pt] at(H4_7) {$v_{10}$};
   \node[font=\scriptsize,above] at(H4_7) {\color{blue}6};
  \node[font=\scriptsize,right] at(H4_8) {$v_{11}$};
   \node[font=\scriptsize,xshift=-6pt,yshift=3pt] at(H4_8) {\color{blue}5};
  \node[font=\scriptsize,right] at(H4_9) {$v_{12}$};
   \node[font=\scriptsize,xshift=-6pt,yshift=-3pt] at(H4_9) {\color{blue}4};
   \node[font=\scriptsize,left] at(H4_14) {$v_{13}$};
    \node[font=\scriptsize,above] at(H4_14) {\color{blue}3};
  \node[font=\scriptsize,right] at(H4_13) {$v_{14}$};
   \node[font=\scriptsize,xshift=-6pt,yshift=3pt] at(H4_13) {\color{blue}3};
     \node[font=\scriptsize] at (0,-5.8) {(d) $v_1v_{9}\in E(G)$};
\end{tikzpicture}}
\end{minipage}
\end{center} 
\caption{Case 2 of the proof of Lemma \ref{reducible-H4}. The numbers at vertices are the number of available colors.} \label{H4-adjacent}
\end{figure}
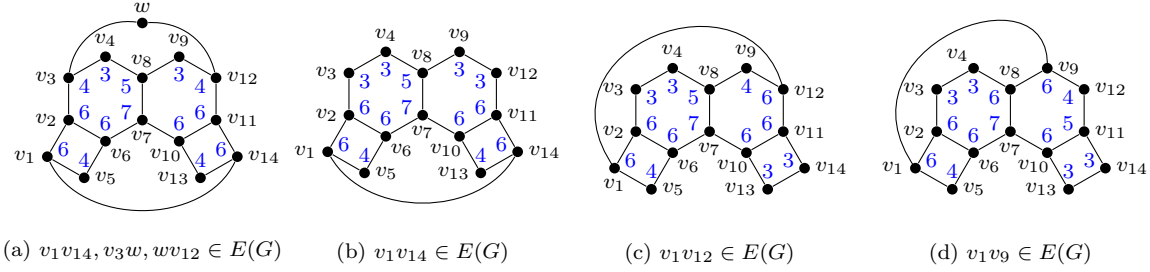

\medskip
Thus, by symmetry, we only need to consider the case where $v_1$ and $v_i$ are adjacent where $i \in \{9,12,14\}$.
The number of available colors at vertices of $V(J_6)$ of each case is as shown in Figure \ref{H4-adjacent}.

In all of these cases, we start by coloring $v_1$ by a color $c \in L_{J_6}(v_1) \setminus L_{J_6}(v_3)$,
color $v_{10},v_{13}, v_{14}, v_{12}, v_{11}, v_{9}, v_5$ in this order
following the same procedure as in Subcase 1.1, and then use Lemma \ref{cycle-six-original-second}.

\medskip
Thus, in Case 2,
the vertices in $V(J_6)$ can be colored from the list $L_{J_6}$
so that we obtain an $L$-coloring in $G^2$.
This is a contradiction for the fact that $G$ is a counterexample.
This completes the proof of Case 2. \\

\medskip

\noindent {\bf Case 3:} $J_6^2$ is not an induced subgraph of $G^2$ and
 $E(G[V(J_6)]) - E(J_6) = \emptyset$.

\medskip
\noindent {\bf Simplifying cases:}
\begin{itemize}
\item
As in Remark \ref{restrict-case}(b),
$v_{1}$ and $v_{13}$ cannot have a common neighbor outside $J_6$,
and
$v_{5}$ and $v_{14}$ cannot have a common neighbor outside $J_6$.

\item
The vertices in each of the following pairs do not have a common neighbor outside $J_6$,
since it makes a $5$-cycle:
$\{v_1,v_4\}$, $\{v_1,v_5\}$, $\{v_3,v_{5}\}$, $\{v_3,v_{9}\}$,
$\{v_4,v_{12}\}$,
$\{v_9,v_{14}\}$,
$\{v_{12},v_{13}\}$,
and $\{v_{13},v_{14}\}$.

\item
The vertices in each of the following pairs do not have a common neighbor outside $J_6$,
since it makes the subgraph $F_3$ in Figure \ref{key configuration-C3},
which does not exist by Lemma \ref{C3-C6}(b):
$\{v_3,v_4\}$, and $\{v_9,v_{12}\}$.

\item
The vertices in each of the following pairs do not have a common neighbor outside $J_6$,
since it makes $H_1$ in Figure \ref{key configuration},
which does not exist by Lemma \ref{reducible-H0}:
$\{v_1,v_3\}$, and $\{v_{12},v_{14}\}$.

\item
The vertices in each of the following pairs do not have a common neighbor outside $J_6$,
since it makes $J_5$,
which does not exist by Lemma \ref{H2-type-two-reducible}:
$\{v_4,v_9\}$.
\end{itemize}

So by Remark \ref{restrict-case} (b) and by symmetry, we only need to consider the following eleven subcases in Case 3.

\medskip
\noindent {\bf Subcase 3.1:} $v_1$ and $v_i$ have a common neighbor $w$,
where $i\in \{9,12,14\}$ and $w \notin V(J_6)$.
(Or by symmetry,
$v_{14}$ and $v_i$ have a common neighbor $w$,
where $i\in \{3,4\}$ and $w \notin V(J_6)$.)

Note that $w$ is not adjacent to any vertex in $V(J_6) \setminus \{v_1, v_{i}\}$ by the argument of Simplifying cases.

Let  $G' = G - V(J_6)$ and we follow the same procedure as in Case 2.
Then the number of available colors at vertices of $V(J_6)$ is represented
in Figure \ref{H4-nbr-one}.
In this case, we start by coloring $v_1$ by a color $c_1 \in L_{J_6}(v_1)$
so that $|L_{J_6}(v_3) \setminus \{v_1\}| \geq 3$, next
color $v_{10}, v_{13}, v_{14}, v_{12}, v_{11}, v_{9}, v_5$ in this order
following the same procedure as in Subcase 1.1, and then use Lemma \ref{cycle-six-original-second}.
Then the vertices in $J_6$ can be colored from the list $L_{J_6}$
so that we obtain an $L$-coloring in $G^2$.

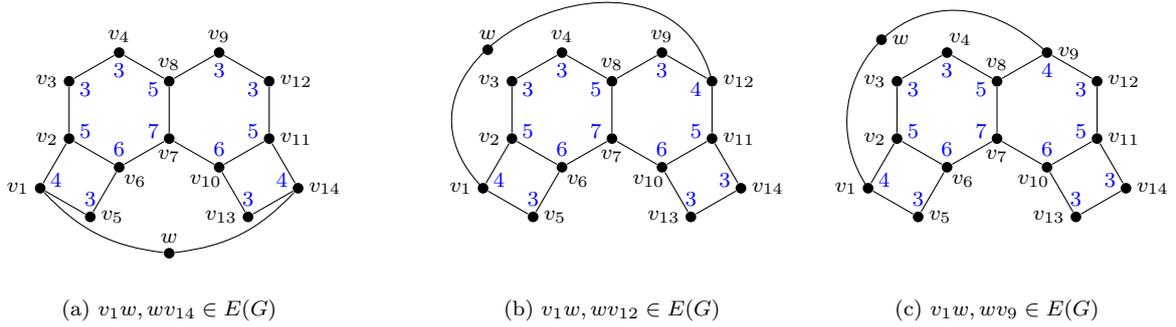
\begin{figure}[htbp]
  \begin{center}
\begin{tikzpicture}[
    v2/.style={fill=black,minimum size=4pt,ellipse,inner sep=1pt},invis/.style={
    draw=none, fill=none,minimum size=0pt, inner sep=0pt
},
    scale=0.38]
\node[v2] (H4_1) at (0, 0){};
    \node[v2] (H4_2) at (0, 2){};
    \node[v2] (H4_3) at (-1.732,-1){};
 \node[v2] (H4_4) at (-2*1.732,0){};
 \node[v2] (H4_5) at (-2*1.732,2){};
 \node[v2] (H4_6) at (-1.732,3){};
 \node[v2] (H4_7) at (1.732,-1){};
 \node[v2] (H4_8) at (2*1.732,0){};
 \node[v2] (H4_9) at (2*1.732,2){};
 \node[v2] (H4_10) at (1.732,3){};
 \node[v2] (H4_11) at (-4.464,-1.732){};
 \node[v2] (H4_12) at (-2.732,-2.732){};
 \node[v2] (H4_13) at (4.464,-1.732){};
 \node[v2] (H4_14) at (2.732,-2.732){};
  \node[v2] (H4_15) at (0,-4){};

   \draw (H4_1) -- (H4_2) -- (H4_6) -- (H4_5) -- (H4_4) -- (H4_3)-- (H4_1);
   \draw (H4_2) -- (H4_10) -- (H4_9) -- (H4_8) -- (H4_7) -- (H4_1);
   \draw (H4_4) -- (H4_11) -- (H4_12) -- (H4_3);
   \draw (H4_8) -- (H4_13) -- (H4_14) -- (H4_7);
   \draw(H4_11)to[bend right=20] (H4_15);
     \draw(H4_15)to[bend right=20] (H4_13);

 \node[font=\scriptsize,left] at(H4_11) {$v_1$};
 \node[font=\scriptsize,xshift=6pt,yshift=3pt] at(H4_11) {\color{blue}4};
  \node[font=\scriptsize,left] at(H4_4) {$v_2$};
   \node[font=\scriptsize,xshift=6pt,yshift=3pt] at(H4_4) {\color{blue}5};
  \node[font=\scriptsize,left] at(H4_5) {$v_3$};
  \node[font=\scriptsize,xshift=6pt,yshift=-3pt] at(H4_5) {\color{blue}3};
  \node[font=\scriptsize,above] at(H4_6) {$v_4$};
   \node[font=\scriptsize,below] at(H4_6) {\color{blue}3};
   \node[font=\scriptsize,right] at(H4_12) {$v_5$};
    \node[font=\scriptsize,above] at(H4_12) {\color{blue}3};
  \node[font=\scriptsize,xshift=6pt,yshift=-5pt] at(H4_3) {$v_6$};
   \node[font=\scriptsize,above] at(H4_3) {\color{blue}6};
  \node[font=\scriptsize,below] at(H4_1) {$v_7$};
   \node[font=\scriptsize,xshift=-6pt,yshift=3pt] at(H4_1) {\color{blue}7};
  \node[font=\scriptsize,above] at(H4_2) {$v_8$};
   \node[font=\scriptsize,xshift=-6pt,yshift=-3pt] at(H4_2) {\color{blue}5};
   \node[font=\scriptsize,above] at(H4_10) {$v_9$};
    \node[font=\scriptsize,below] at(H4_10) {\color{blue}3};
  \node[font=\scriptsize,xshift=-5pt,yshift=-5pt] at(H4_7) {$v_{10}$};
   \node[font=\scriptsize,above] at(H4_7) {\color{blue}6};
  \node[font=\scriptsize,right] at(H4_8) {$v_{11}$};
   \node[font=\scriptsize,xshift=-6pt,yshift=3pt] at(H4_8) {\color{blue}5};
  \node[font=\scriptsize,right] at(H4_9) {$v_{12}$};
   \node[font=\scriptsize,xshift=-6pt,yshift=-3pt] at(H4_9) {\color{blue}3};
   \node[font=\scriptsize,left] at(H4_14) {$v_{13}$};
    \node[font=\scriptsize,above] at(H4_14) {\color{blue}3};
  \node[font=\scriptsize,right] at(H4_13) {$v_{14}$};
   \node[font=\scriptsize,xshift=-6pt,yshift=3pt] at(H4_13) {\color{blue}4};
    \node[font=\scriptsize,above] at(H4_15) {$w$};

     \node[font=\scriptsize] at (0,-6) {(a) $v_1w,wv_{14}\in E(G)$};
\end{tikzpicture}\begin{tikzpicture}[
    v2/.style={fill=black,minimum size=4pt,ellipse,inner sep=1pt},invis/.style={
    draw=none, fill=none,minimum size=0pt, inner sep=0pt
},
    scale=0.38
]
\node[v2] (H4_1) at (0, 0){};
    \node[v2] (H4_2) at (0, 2){};
    \node[v2] (H4_3) at (-1.732,-1){};
 \node[v2] (H4_4) at (-2*1.732,0){};
 \node[v2] (H4_5) at (-2*1.732,2){};
 \node[v2] (H4_6) at (-1.732,3){};
 \node[v2] (H4_7) at (1.732,-1){};
 \node[v2] (H4_8) at (2*1.732,0){};
 \node[v2] (H4_9) at (2*1.732,2){};
 \node[v2] (H4_10) at (1.732,3){};
 \node[v2] (H4_11) at (-4.464,-1.732){};
 \node[v2] (H4_12) at (-2.732,-2.732){};
 \node[v2] (H4_13) at (4.464,-1.732){};
 \node[v2] (H4_14) at (2.732,-2.732){};
\node[v2] (H4_15) at (-4.3,3.1){};

   \draw (H4_1) -- (H4_2) -- (H4_6) -- (H4_5) -- (H4_4) -- (H4_3)-- (H4_1);
   \draw (H4_2) -- (H4_10) -- (H4_9) -- (H4_8) -- (H4_7) -- (H4_1);
   \draw (H4_4) -- (H4_11) -- (H4_12) -- (H4_3);
   \draw (H4_8) -- (H4_13) -- (H4_14) -- (H4_7);
   \draw(2*1.732,2) ..controls (2,8)and (-9,3).. (-4.464,-1.732);

 \node[font=\scriptsize,left] at(H4_11) {$v_1$};
 \node[font=\scriptsize,xshift=6pt,yshift=3pt] at(H4_11) {\color{blue}4};
  \node[font=\scriptsize,left] at(H4_4) {$v_2$};
   \node[font=\scriptsize,xshift=6pt,yshift=3pt] at(H4_4) {\color{blue}5};
  \node[font=\scriptsize,left] at(H4_5) {$v_3$};
  \node[font=\scriptsize,xshift=6pt,yshift=-3pt] at(H4_5) {\color{blue}3};
  \node[font=\scriptsize,above] at(H4_6) {$v_4$};
   \node[font=\scriptsize,below] at(H4_6) {\color{blue}3};
   \node[font=\scriptsize,right] at(H4_12) {$v_5$};
    \node[font=\scriptsize,above] at(H4_12) {\color{blue}3};
  \node[font=\scriptsize,xshift=6pt,yshift=-5pt] at(H4_3) {$v_6$};
   \node[font=\scriptsize,above] at(H4_3) {\color{blue}6};
  \node[font=\scriptsize,below] at(H4_1) {$v_7$};
   \node[font=\scriptsize,xshift=-6pt,yshift=3pt] at(H4_1) {\color{blue}7};
  \node[font=\scriptsize,above] at(H4_2) {$v_8$};
   \node[font=\scriptsize,xshift=-6pt,yshift=-3pt] at(H4_2) {\color{blue}5};
   \node[font=\scriptsize,above] at(H4_10) {$v_9$};
    \node[font=\scriptsize,below] at(H4_10) {\color{blue}3};
  \node[font=\scriptsize,xshift=-5pt,yshift=-5pt] at(H4_7) {$v_{10}$};
   \node[font=\scriptsize,above] at(H4_7) {\color{blue}6};
  \node[font=\scriptsize,right] at(H4_8) {$v_{11}$};
   \node[font=\scriptsize,xshift=-6pt,yshift=3pt] at(H4_8) {\color{blue}5};
  \node[font=\scriptsize,right] at(H4_9) {$v_{12}$};
   \node[font=\scriptsize,xshift=-6pt,yshift=-3pt] at(H4_9) {\color{blue}4};
   \node[font=\scriptsize,left] at(H4_14) {$v_{13}$};
    \node[font=\scriptsize,above] at(H4_14) {\color{blue}3};
  \node[font=\scriptsize,right] at(H4_13) {$v_{14}$};
   \node[font=\scriptsize,xshift=-6pt,yshift=3pt] at(H4_13) {\color{blue}3};
\node[font=\scriptsize,above] at(H4_15) {$w$};

     \node[font=\scriptsize] at (0,-6) {(b) $v_1w,wv_{12}\in E(G)$};
\end{tikzpicture}\begin{tikzpicture}[
    v2/.style={fill=black,minimum size=4pt,ellipse,inner sep=1pt},invis/.style={
    draw=none, fill=none,minimum size=0pt, inner sep=0pt
},
    scale=0.38
]
\node[v2] (H4_1) at (0, 0){};
    \node[v2] (H4_2) at (0, 2){};
    \node[v2] (H4_3) at (-1.732,-1){};
 \node[v2] (H4_4) at (-2*1.732,0){};
 \node[v2] (H4_5) at (-2*1.732,2){};
 \node[v2] (H4_6) at (-1.732,3){};
 \node[v2] (H4_7) at (1.732,-1){};
 \node[v2] (H4_8) at (2*1.732,0){};
 \node[v2] (H4_9) at (2*1.732,2){};
 \node[v2] (H4_10) at (1.732,3){};
 \node[v2] (H4_11) at (-4.464,-1.732){};
 \node[v2] (H4_12) at (-2.732,-2.732){};
 \node[v2] (H4_13) at (4.464,-1.732){};
 \node[v2] (H4_14) at (2.732,-2.732){};
\node[v2] (H4_15) at (-4,3.45){};

   \draw (H4_1) -- (H4_2) -- (H4_6) -- (H4_5) -- (H4_4) -- (H4_3)-- (H4_1);
   \draw (H4_2) -- (H4_10) -- (H4_9) -- (H4_8) -- (H4_7) -- (H4_1);
   \draw (H4_4) -- (H4_11) -- (H4_12) -- (H4_3);
   \draw (H4_8) -- (H4_13) -- (H4_14) -- (H4_7);
   \draw(1.732,3) ..controls (-2.5,7)and (-7,2).. (-4.464,-1.732);

 \node[font=\scriptsize,left] at(H4_11) {$v_1$};
 \node[font=\scriptsize,xshift=6pt,yshift=3pt] at(H4_11) {\color{blue}4};
  \node[font=\scriptsize,left] at(H4_4) {$v_2$};
   \node[font=\scriptsize,xshift=6pt,yshift=3pt] at(H4_4) {\color{blue}5};
  \node[font=\scriptsize,left] at(H4_5) {$v_3$};
  \node[font=\scriptsize,xshift=6pt,yshift=-3pt] at(H4_5) {\color{blue}3};
  \node[font=\scriptsize, xshift=5pt,yshift=5pt] at(H4_6) {$v_4$};
   \node[font=\scriptsize,below] at(H4_6) {\color{blue}3};
   \node[font=\scriptsize,right] at(H4_12) {$v_5$};
    \node[font=\scriptsize,above] at(H4_12) {\color{blue}3};
  \node[font=\scriptsize,xshift=6pt,yshift=-5pt] at(H4_3) {$v_6$};
   \node[font=\scriptsize,above] at(H4_3) {\color{blue}6};
  \node[font=\scriptsize,below] at(H4_1) {$v_7$};
   \node[font=\scriptsize,xshift=-6pt,yshift=3pt] at(H4_1) {\color{blue}7};
  \node[font=\scriptsize,above] at(H4_2) {$v_8$};
   \node[font=\scriptsize,xshift=-6pt,yshift=-3pt] at(H4_2) {\color{blue}5};
   \node[font=\scriptsize,right] at(H4_10) {$v_9$};
    \node[font=\scriptsize,below] at(H4_10) {\color{blue}4};
  \node[font=\scriptsize,xshift=-5pt,yshift=-5pt] at(H4_7) {$v_{10}$};
   \node[font=\scriptsize,above] at(H4_7) {\color{blue}6};
  \node[font=\scriptsize,right] at(H4_8) {$v_{11}$};
   \node[font=\scriptsize,xshift=-6pt,yshift=3pt] at(H4_8) {\color{blue}5};
  \node[font=\scriptsize,right] at(H4_9) {$v_{12}$};
   \node[font=\scriptsize,xshift=-6pt,yshift=-3pt] at(H4_9) {\color{blue}3};
   \node[font=\scriptsize,left] at(H4_14) {$v_{13}$};
    \node[font=\scriptsize,above] at(H4_14) {\color{blue}3};
  \node[font=\scriptsize,right] at(H4_13) {$v_{14}$};
   \node[font=\scriptsize,xshift=-6pt,yshift=3pt] at(H4_13) {\color{blue}3};
 \node[font=\scriptsize,right] at(H4_15) {$w$};

     \node[font=\scriptsize] at (0,-6) {(c) $v_1w,wv_9\in E(G)$};
\end{tikzpicture}
\end{center}
\caption{Subcases 3.1 of the proof of Lemma \ref{reducible-H4}. The numbers at vertices are the number of available colors.} \label{H4-nbr-one}
\end{figure}

\begin{figure}[htbp]
  \begin{center}
  \begin{minipage}{0.24\textwidth}
\centering
\begin{tikzpicture}[
    v2/.style={fill=black,minimum size=4pt,ellipse,inner sep=1pt},
    scale=0.27
]
\node[v2] (H4_1) at (0, 0){};
    \node[v2] (H4_2) at (0, 2){};
    \node[v2] (H4_3) at (-1.732,-1){};
 \node[v2] (H4_4) at (-2*1.732,0){};
 \node[v2] (H4_5) at (-2*1.732,2){};
 \node[v2] (H4_6) at (-1.732,3){};
 \node[v2] (H4_7) at (1.732,-1){};
 \node[v2] (H4_8) at (2*1.732,0){};
 \node[v2] (H4_9) at (2*1.732,2){};
 \node[v2] (H4_10) at (1.732,3){};
 \node[v2] (H4_11) at (-4.464,-1.732){};
 \node[v2] (H4_12) at (-2.732,-2.732){};
 \node[v2] (H4_13) at (4.464,-1.732){};
 \node[v2] (H4_14) at (2.732,-2.732){};
 \node[v2] (H4_15) at (0,4.6){};

   \draw (H4_1) -- (H4_2) -- (H4_6) -- (H4_5) -- (H4_4) -- (H4_3)-- (H4_1);
   \draw (H4_2) -- (H4_10) -- (H4_9) -- (H4_8) -- (H4_7) -- (H4_1);
   \draw (H4_4) -- (H4_11) -- (H4_12) -- (H4_3);
   \draw (H4_8) -- (H4_13) -- (H4_14) -- (H4_7);
\draw (H4_5)to[bend left=40] (H4_15);
 \draw (H4_15)to[bend left=40] (H4_9);

 \node[font=\scriptsize,below] at(H4_11) {$v_1$};
 \node[font=\scriptsize,xshift=6pt,yshift=3pt] at(H4_11) {\color{blue}3};
  \node[font=\scriptsize,left] at(H4_4) {$v_2$};
   \node[font=\scriptsize,xshift=6pt,yshift=3pt] at(H4_4) {\color{blue}6};
  \node[font=\scriptsize,left] at(H4_5) {$v_3$};
   \node[font=\scriptsize,xshift=6pt,yshift=-3pt] at(H4_5) {\color{blue}5};
  \node[font=\scriptsize,above] at(H4_6) {$v_4$};
   \node[font=\scriptsize,below] at(H4_6) {\color{blue}4};
   \node[font=\scriptsize,right] at(H4_12) {$v_5$};
    \node[font=\scriptsize,above] at(H4_12) {\color{blue}3};
  \node[font=\scriptsize,xshift=6pt,yshift=-5pt] at(H4_3) {$v_6$};
   \node[font=\scriptsize,above] at(H4_3) {\color{blue}6};
  \node[font=\scriptsize,below] at(H4_1) {$v_7$};
   \node[font=\scriptsize,xshift=-6pt,yshift=3pt] at(H4_1) {\color{blue}7};
  \node[font=\scriptsize,above] at(H4_2) {$v_8$};
   \node[font=\scriptsize,xshift=-6pt,yshift=-3pt] at(H4_2) {\color{blue}5};
   \node[font=\scriptsize,above] at(H4_10) {$v_9$};
    \node[font=\scriptsize,below] at(H4_10) {\color{blue}4};
  \node[font=\scriptsize,xshift=-5pt,yshift=-5pt] at(H4_7) {$v_{10}$};
   \node[font=\scriptsize,above] at(H4_7) {\color{blue}6};
  \node[font=\scriptsize,right] at(H4_8) {$v_{11}$};
   \node[font=\scriptsize,xshift=-6pt,yshift=3pt] at(H4_8) {\color{blue}6};
  \node[font=\scriptsize,right] at(H4_9) {$v_{12}$};
   \node[font=\scriptsize,xshift=-6pt,yshift=-3pt] at(H4_9) {\color{blue}5};
   \node[font=\scriptsize,left] at(H4_14) {$v_{13}$};
    \node[font=\scriptsize,above] at(H4_14) {\color{blue}3};
  \node[font=\scriptsize,below] at(H4_13) {$v_{14}$};
   \node[font=\scriptsize,xshift=-6pt,yshift=3pt] at(H4_13) {\color{blue}3};
  \node[font=\scriptsize,above] at(H4_15) {$w$};
    \node[font=\scriptsize,below] at(H4_15) {\color{blue}4};

     \node[font=\scriptsize] at (-0.5, -7) {(a) $v_3w,wv_{12}\in E(G)$};
\end{tikzpicture}
\end{minipage}
\begin{minipage}{0.24\textwidth}
\centering
 \hspace{-0.8cm}\begin{tikzpicture}[
    v2/.style={fill=black,minimum size=4pt,ellipse,inner sep=1pt},
    scale=0.27
]
\node[v2] (H4_1) at (0, 0){};
    \node[v2] (H4_2) at (0, 2){};
    \node[v2] (H4_3) at (-1.732,-1){};
 \node[v2] (H4_4) at (-2*1.732,0){};
 \node[v2] (H4_5) at (-2*1.732,2){};
 \node[v2] (H4_6) at (-1.732,3){};
 \node[v2] (H4_7) at (1.732,-1){};
 \node[v2] (H4_8) at (2*1.732,0){};
 \node[v2] (H4_9) at (2*1.732,2){};
 \node[v2] (H4_10) at (1.732,3){};
 \node[v2] (H4_11) at (-4.464,-1.732){};
 \node[v2] (H4_12) at (-2.732,-2.732){};
 \node[v2] (H4_13) at (4.464,-1.732){};
 \node[v2] (H4_14) at (2.732,-2.732){};
 \node[v2] (H4_15) at (0,4.6){};

  \node[v2] (H4_w) at (0,-4){};

   \draw (H4_1) -- (H4_2) -- (H4_6) -- (H4_5) -- (H4_4) -- (H4_3)-- (H4_1);
   \draw (H4_2) -- (H4_10) -- (H4_9) -- (H4_8) -- (H4_7) -- (H4_1);
   \draw (H4_4) -- (H4_11) -- (H4_12) -- (H4_3);
   \draw (H4_8) -- (H4_13) -- (H4_14) -- (H4_7);
  \draw (H4_5)to[bend left=40] (H4_15);
 \draw (H4_15)to[bend left=40] (H4_9);

   \draw (H4_11)to[bend left=-30] (H4_w);
 \draw (H4_w)to[bend left=-30] (H4_13);

 \node[font=\scriptsize,below] at(H4_11) {$v_1$};
 \node[font=\scriptsize,xshift=6pt,yshift=3pt] at(H4_11) {\color{blue}4};
  \node[font=\scriptsize,left] at(H4_4) {$v_2$};
   \node[font=\scriptsize,xshift=6pt,yshift=3pt] at(H4_4) {\color{blue}6};
  \node[font=\scriptsize,left] at(H4_5) {$v_3$};
   \node[font=\scriptsize,xshift=6pt,yshift=-3pt] at(H4_5) {\color{blue}5};
  \node[font=\scriptsize,above] at(H4_6) {$v_4$};
   \node[font=\scriptsize,below] at(H4_6) {\color{blue}4};
   \node[font=\scriptsize,right] at(H4_12) {$v_5$};
    \node[font=\scriptsize,above] at(H4_12) {\color{blue}3};
  \node[font=\scriptsize,xshift=6pt,yshift=-5pt] at(H4_3) {$v_6$};
   \node[font=\scriptsize,above] at(H4_3) {\color{blue}6};
  \node[font=\scriptsize,below] at(H4_1) {$v_7$};
   \node[font=\scriptsize,xshift=-6pt,yshift=3pt] at(H4_1) {\color{blue}7};
  \node[font=\scriptsize,above] at(H4_2) {$v_8$};
   \node[font=\scriptsize,xshift=-6pt,yshift=-3pt] at(H4_2) {\color{blue}5};
   \node[font=\scriptsize,above] at(H4_10) {$v_9$};
    \node[font=\scriptsize,below] at(H4_10) {\color{blue}4};
  \node[font=\scriptsize,xshift=-5pt,yshift=-5pt] at(H4_7) {$v_{10}$};
   \node[font=\scriptsize,above] at(H4_7) {\color{blue}6};
  \node[font=\scriptsize,right] at(H4_8) {$v_{11}$};
   \node[font=\scriptsize,xshift=-6pt,yshift=3pt] at(H4_8) {\color{blue}6};
  \node[font=\scriptsize,right] at(H4_9) {$v_{12}$};
   \node[font=\scriptsize,xshift=-6pt,yshift=-3pt] at(H4_9) {\color{blue}5};
   \node[font=\scriptsize,left] at(H4_14) {$v_{13}$};
    \node[font=\scriptsize,above] at(H4_14) {\color{blue}3};
  \node[font=\scriptsize,right] at(H4_13) {$v_{14}$};
   \node[font=\scriptsize,xshift=-6pt,yshift=3pt] at(H4_13) {\color{blue}4};
  \node[font=\scriptsize,above] at(H4_15) {$w$};
    \node[font=\scriptsize,above] at(H4_w) {$z$};
       \node[font=\scriptsize,below] at(H4_15) {\color{blue}4};

     \node[font=\scriptsize] at (0, -7) {(b) $v_3w,wv_{12},v_1z,zv_{14}\in E(G)$};
\end{tikzpicture}
\end{minipage}
\begin{minipage}{0.24\textwidth}
\centering
 \hspace{-1cm}
 \raisebox{-3.8cm}{
 \begin{tikzpicture}[
    v2/.style={fill=black,minimum size=4pt,ellipse,inner sep=1pt},
    scale=0.27
]
\node[v2] (H4_1) at (0, 0){};
    \node[v2] (H4_2) at (0, 2){};
    \node[v2] (H4_3) at (-1.732,-1){};
 \node[v2] (H4_4) at (-2*1.732,0){};
 \node[v2] (H4_5) at (-2*1.732,2){};
 \node[v2] (H4_6) at (-1.732,3){};
 \node[v2] (H4_7) at (1.732,-1){};
 \node[v2] (H4_8) at (2*1.732,0){};
 \node[v2] (H4_9) at (2*1.732,2){};
 \node[v2] (H4_10) at (1.732,3){};
 \node[v2] (H4_11) at (-4.464,-1.732){};
 \node[v2] (H4_12) at (-2.732,-2.732){};
 \node[v2] (H4_13) at (4.464,-1.732){};
 \node[v2] (H4_14) at (2.732,-2.732){};
 \node[v2] (H4_15) at (-5,-4){};

   \draw (H4_1) -- (H4_2) -- (H4_6) -- (H4_5) -- (H4_4) -- (H4_3)-- (H4_1);
   \draw (H4_2) -- (H4_10) -- (H4_9) -- (H4_8) -- (H4_7) -- (H4_1);
   \draw (H4_4) -- (H4_11) -- (H4_12) -- (H4_3);
   \draw (H4_8) -- (H4_13) -- (H4_14) -- (H4_7);
  \draw (H4_5)to[bend right=45] (H4_15);
 \draw (H4_15)to[bend right=35] (H4_14);

 \node[font=\scriptsize,below] at(H4_11) {$v_1$};
 \node[font=\scriptsize,xshift=6pt,yshift=3pt] at(H4_11) {\color{blue}3};
  \node[font=\scriptsize,left] at(H4_4) {$v_2$};
   \node[font=\scriptsize,xshift=6pt,yshift=3pt] at(H4_4) {\color{blue}6};
  \node[font=\scriptsize,left] at(H4_5) {$v_3$};
   \node[font=\scriptsize,xshift=6pt,yshift=-3pt] at(H4_5) {\color{blue}5};
  \node[font=\scriptsize,above] at(H4_6) {$v_4$};
   \node[font=\scriptsize,below] at(H4_6) {\color{blue}4};
   \node[font=\scriptsize,right] at(H4_12) {$v_5$};
    \node[font=\scriptsize,above] at(H4_12) {\color{blue}3};
  \node[font=\scriptsize,xshift=6pt,yshift=-5pt] at(H4_3) {$v_6$};
   \node[font=\scriptsize,above] at(H4_3) {\color{blue}6};
  \node[font=\scriptsize,below] at(H4_1) {$v_7$};
   \node[font=\scriptsize,xshift=-6pt,yshift=3pt] at(H4_1) {\color{blue}7};
  \node[font=\scriptsize,above] at(H4_2) {$v_8$};
   \node[font=\scriptsize,xshift=-6pt,yshift=-3pt] at(H4_2) {\color{blue}5};
   \node[font=\scriptsize,above] at(H4_10) {$v_9$};
    \node[font=\scriptsize,below] at(H4_10) {\color{blue}3};
  \node[font=\scriptsize,xshift=-5pt,yshift=-5pt] at(H4_7) {$v_{10}$};
   \node[font=\scriptsize,above] at(H4_7) {\color{blue}7};
  \node[font=\scriptsize,right] at(H4_8) {$v_{11}$};
   \node[font=\scriptsize,xshift=-6pt,yshift=3pt] at(H4_8) {\color{blue}5};
  \node[font=\scriptsize,right] at(H4_9) {$v_{12}$};
   \node[font=\scriptsize,xshift=-6pt,yshift=-3pt] at(H4_9) {\color{blue}3};
   \node[font=\scriptsize,left] at(H4_14) {$v_{13}$};
    \node[font=\scriptsize,above] at(H4_14) {\color{blue}5};
  \node[font=\scriptsize,below] at(H4_13) {$v_{14}$};
   \node[font=\scriptsize,xshift=-6pt,yshift=3pt] at(H4_13) {\color{blue}4};
  \node[font=\scriptsize,below] at(H4_15) {$w$};
     \node[font=\scriptsize,above] at(H4_15) {\color{blue}4};

     \node[font=\scriptsize] at (0, -7) {(c) $v_3w,wv_{13}\in E(G)$};
\end{tikzpicture}}
\end{minipage}
\begin{minipage}{0.25\textwidth}
\centering
 \hspace{-1.1cm}\raisebox{2ex}{
 \begin{tikzpicture}[
    v2/.style={fill=black,minimum size=4pt,ellipse,inner sep=1pt},
    scale=0.27
]
\node[v2] (H4_1) at (0, 0){};
    \node[v2] (H4_2) at (0, 2){};
    \node[v2] (H4_3) at (-1.732,-1){};
 \node[v2] (H4_4) at (-2*1.732,0){};
 \node[v2] (H4_5) at (-2*1.732,2){};
 \node[v2] (H4_6) at (-1.732,3){};
 \node[v2] (H4_7) at (1.732,-1){};
 \node[v2] (H4_8) at (2*1.732,0){};
 \node[v2] (H4_9) at (2*1.732,2){};
 \node[v2] (H4_10) at (1.732,3){};
 \node[v2] (H4_11) at (-4.464,-1.732){};
 \node[v2] (H4_12) at (-2.732,-2.732){};
 \node[v2] (H4_13) at (4.464,-1.732){};
 \node[v2] (H4_14) at (2.732,-2.732){};
 \node[v2] (H4_15) at (-5,-4){};
\node[v2] (H4_q) at (4,3.35){};

   \draw (H4_1) -- (H4_2) -- (H4_6) -- (H4_5) -- (H4_4) -- (H4_3)-- (H4_1);
   \draw (H4_2) -- (H4_10) -- (H4_9) -- (H4_8) -- (H4_7) -- (H4_1);
   \draw (H4_4) -- (H4_11) -- (H4_12) -- (H4_3);
   \draw (H4_8) -- (H4_13) -- (H4_14) -- (H4_7);
  \draw (H4_5)to[bend right=45] (H4_15);
 \draw (H4_15)to[bend right=35] (H4_14);

   \draw(-1.732,3) ..controls (-1,7)and (10,2).. (4.464,-1.732);

 \node[font=\scriptsize,below] at(H4_11) {$v_1$};
 \node[font=\scriptsize,xshift=6pt,yshift=3pt] at(H4_11) {\color{blue}3};
  \node[font=\scriptsize,left] at(H4_4) {$v_2$};
   \node[font=\scriptsize,xshift=6pt,yshift=3pt] at(H4_4) {\color{blue}6};
  \node[font=\scriptsize,left] at(H4_5) {$v_3$};
   \node[font=\scriptsize,xshift=6pt,yshift=-3pt] at(H4_5) {\color{blue}5};
  \node[font=\scriptsize, above, left] at(H4_6) {$v_4$};
   \node[font=\scriptsize,below] at(H4_6) {\color{blue}5};
   \node[font=\scriptsize,right] at(H4_12) {$v_5$};
    \node[font=\scriptsize,above] at(H4_12) {\color{blue}3};
  \node[font=\scriptsize,xshift=6pt,yshift=-5pt] at(H4_3) {$v_6$};
   \node[font=\scriptsize,above] at(H4_3) {\color{blue}6};
  \node[font=\scriptsize,below] at(H4_1) {$v_7$};
   \node[font=\scriptsize,xshift=-6pt,yshift=3pt] at(H4_1) {\color{blue}7};
  \node[font=\scriptsize,above] at(H4_2) {$v_8$};
   \node[font=\scriptsize,xshift=-6pt,yshift=-3pt] at(H4_2) {\color{blue}5};
   \node[font=\scriptsize,above] at(H4_10) {$v_9$};
    \node[font=\scriptsize,below] at(H4_10) {\color{blue}3};
  \node[font=\scriptsize,xshift=-5pt,yshift=-5pt] at(H4_7) {$v_{10}$};
   \node[font=\scriptsize,above] at(H4_7) {\color{blue}7};
  \node[font=\scriptsize,right] at(H4_8) {$v_{11}$};
   \node[font=\scriptsize,xshift=-6pt,yshift=3pt] at(H4_8) {\color{blue}5};
  \node[font=\scriptsize,right] at(H4_9) {$v_{12}$};
   \node[font=\scriptsize,xshift=-6pt,yshift=-3pt] at(H4_9) {\color{blue}3};
   \node[font=\scriptsize,left] at(H4_14) {$v_{13}$};
    \node[font=\scriptsize,above] at(H4_14) {\color{blue}5};
  \node[font=\scriptsize,right] at(H4_13) {$v_{14}$};
   \node[font=\scriptsize,xshift=-6pt,yshift=3pt] at(H4_13) {\color{blue}5};
  \node[font=\scriptsize,below] at(H4_15) {$w$};
 \node[font=\scriptsize,right] at(H4_q) {$z$};
    \node[font=\scriptsize,above] at(H4_15) {\color{blue}4};

     \node[font=\scriptsize] at (0, -7.1) {(d) $v_3w,wv_{13},v_4z,zv_{14}\in E(G)$};
\end{tikzpicture}}
\end{minipage}
\end{center}
\caption{Subcases 3.2  of the proof of Lemma \ref{reducible-H4}. The numbers at vertices are the number of available colors.} \label{H4-nbr-two}
\end{figure}
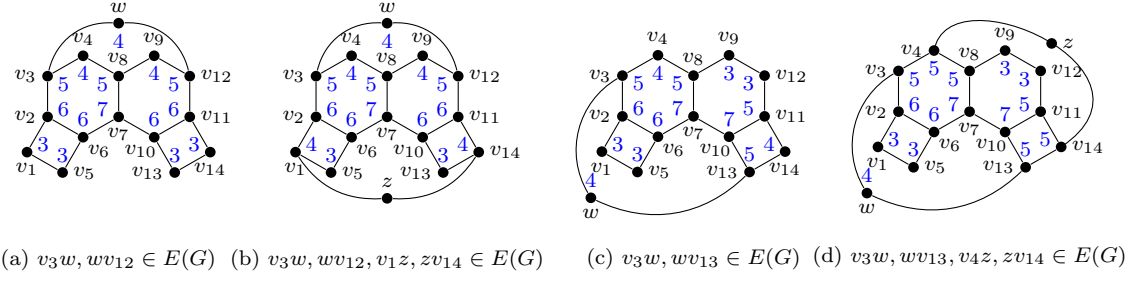

\medskip
\noindent {\bf Subcase 3.2:} $v_3$ and $v_i$ have a common neighbor $w$,
 where $i\in \{12,13\}$ and $w \notin V(J_6)$.
(Or by symmetry,
$v_5$ and $v_{12}$ have a common neighbor $w$, where $w \notin V(J_6)$.)

Note that one may consider the case when $v_3$ and $v_{14}$ have a common neighbor.  But, this is the case of Subcase 3.1 (b). So, we do not need to consider this case. 

Then there are 4 cases needed to consider as shown in Figure~\ref{H4-nbr-two}.
Note that for each case in Figure~\ref{H4-nbr-two}, we have that $w$ is not adjacent to any vertex in $V(J_6) \setminus \{v_3, v_{i}\}$ by the argument of Simplifying cases.

Let $G'  = G - (V(J_6) \cup \{w\})$.
Then $G' $ is also a subcubic planar graph and $|V(G' )| < |V(G)|$.
Since $G$ is a minimal counterexample to Theorem \ref{main-thm},
the square of $G' $ has a proper coloring $\phi$ such that $\phi(v) \in L(v)$ for each vertex $v \in V(H)$.
For each  $v \in V(J_6) \cup \{w\}$ , we define \[
L_{J_6}'(v) = L(v) \setminus \{\phi(x) : xv \in E(G^2) \mbox{ and } x \notin V(J_6)\cup \{w\}\}.
\]
Then the number of available colors at vertices of $V(J_6)\cup \{w\}$ is Subcase 3.2 in Figure \ref{H4-nbr-two} for each case, respectively.

Note that $d_G(v_1, w) \geq 3$ and
$|L_{J_6}'(v_1)| \geq 3$, $|L_{J_6}'(v_3)| \geq 5$, and $|L_{J_6}'(w)| \geq 4$.
So, we can color $v_1$ and $w$ with color $c_1$ and $c_w$, respectively,
so that $|L_{J_6}'(v_3) \setminus \{c_1,c_w\}| \geq 3$. Next, color $v_{10}, v_{13}, v_{14}, v_{12}, v_{11}, v_{9}, v_5$ in this order
following the same procedure as in Subcase 1.1, and then use Lemma \ref{cycle-six-original-second}.
Then
the vertices in $V(J_6) \cup \{w\}$ can be colored from the list $L_{J_6}'$
so that we obtain an $L$-coloring in $G^2$.

\begin{figure}[htbp]
  \begin{center}
  \begin{minipage}{0.33\textwidth}
\centering
  \begin{tikzpicture}[
    v2/.style={fill=black,minimum size=4pt,ellipse,inner sep=1pt},
    scale=0.36
]
\node[v2] (H4_1) at (0, 0){};
    \node[v2] (H4_2) at (0, 2){};
    \node[v2] (H4_3) at (-1.732,-1){};
 \node[v2] (H4_4) at (-2*1.732,0){};
 \node[v2] (H4_5) at (-2*1.732,2){};
 \node[v2] (H4_6) at (-1.732,3){};
 \node[v2] (H4_7) at (1.732,-1){};
 \node[v2] (H4_8) at (2*1.732,0){};
 \node[v2] (H4_9) at (2*1.732,2){};
 \node[v2] (H4_10) at (1.732,3){};
 \node[v2] (H4_11) at (-4.464,-1.732){};
 \node[v2] (H4_12) at (-2.732,-2.732){};
 \node[v2] (H4_13) at (4.464,-1.732){};
 \node[v2] (H4_14) at (2.732,-2.732){};
 \node[v2] (H4_17) at (-6,0.5){};
   \draw (H4_1) -- (H4_2) -- (H4_6) -- (H4_5) -- (H4_4) -- (H4_3)-- (H4_1);
   \draw (H4_2) -- (H4_10) -- (H4_9) -- (H4_8) -- (H4_7) -- (H4_1);
   \draw (H4_4) -- (H4_11) -- (H4_12) -- (H4_3);
   \draw (H4_8) -- (H4_13) -- (H4_14) -- (H4_7);
 \draw (H4_6)to[bend right=45] (H4_17);
 \draw (H4_17)to[bend right=85] (H4_14);
 \node[font=\scriptsize,xshift=-4pt,yshift=5pt] at(H4_11) {$v_1$};
 \node[font=\scriptsize,xshift=6pt,yshift=3pt] at(H4_11) {\color{blue}3};
  \node[font=\scriptsize,left] at(H4_4) {$v_2$};
   \node[font=\scriptsize,xshift=6pt,yshift=3pt] at(H4_4) {\color{blue}5};
  \node[font=\scriptsize,left] at(H4_5) {$v_3$};
   \node[font=\scriptsize,xshift=6pt,yshift=-3pt] at(H4_5) {\color{blue}4};
  \node[font=\scriptsize,above] at(H4_6) {$v_4$};
   \node[font=\scriptsize,below] at(H4_6) {\color{blue}5};
   \node[font=\scriptsize,right] at(H4_12) {$v_5$};
    \node[font=\scriptsize,above] at(H4_12) {\color{blue}3};
  \node[font=\scriptsize,xshift=6pt,yshift=-5pt] at(H4_3) {$v_6$};
   \node[font=\scriptsize,above] at(H4_3) {\color{blue}6};
  \node[font=\scriptsize,below] at(H4_1) {$v_7$};
   \node[font=\scriptsize,xshift=-6pt,yshift=3pt] at(H4_1) {\color{blue}7};
  \node[font=\scriptsize,above] at(H4_2) {$v_8$};
   \node[font=\scriptsize,xshift=-6pt,yshift=-3pt] at(H4_2) {\color{blue}6};
   \node[font=\scriptsize,above] at(H4_10) {$v_9$};
    \node[font=\scriptsize,below] at(H4_10) {\color{blue}3};
  \node[font=\scriptsize,xshift=-5pt,yshift=-5pt] at(H4_7) {$v_{10}$};
   \node[font=\scriptsize,above] at(H4_7) {\color{blue}7};
  \node[font=\scriptsize,right] at(H4_8) {$v_{11}$};
   \node[font=\scriptsize,xshift=-6pt,yshift=3pt] at(H4_8) {\color{blue}5};
  \node[font=\scriptsize,right] at(H4_9) {$v_{12}$};
   \node[font=\scriptsize,xshift=-6pt,yshift=-3pt] at(H4_9) {\color{blue}3};
   \node[font=\scriptsize,left] at(H4_14) {$v_{13}$};
    \node[font=\scriptsize,above] at(H4_14) {\color{blue}5};
  \node[font=\scriptsize,below] at(H4_13) {$v_{14}$};
   \node[font=\scriptsize,xshift=-6pt,yshift=3pt] at(H4_13) {\color{blue}4};
  \node[font=\scriptsize,right] at(H4_17) {$w$};
   \node[font=\scriptsize,left] at(H4_17) {\color{blue}4};
     \node[font=\scriptsize] at (0, -6.5) {(a) $v_4w,wv_{13}\in E(G)$};
\end{tikzpicture}
\end{minipage}
\begin{minipage}{0.31\textwidth}
\centering\begin{tikzpicture}[
    v2/.style={fill=black,minimum size=4pt,ellipse,inner sep=1pt},
    scale=0.36
]
\node[v2] (H4_1) at (0, 0){};
    \node[v2] (H4_2) at (0, 2){};
    \node[v2] (H4_3) at (-1.732,-1){};
 \node[v2] (H4_4) at (-2*1.732,0){};
 \node[v2] (H4_5) at (-2*1.732,2){};
 \node[v2] (H4_6) at (-1.732,3){};
 \node[v2] (H4_7) at (1.732,-1){};
 \node[v2] (H4_8) at (2*1.732,0){};
 \node[v2] (H4_9) at (2*1.732,2){};
 \node[v2] (H4_10) at (1.732,3){};
 \node[v2] (H4_11) at (-4.464,-1.732){};
 \node[v2] (H4_12) at (-2.732,-2.732){};
 \node[v2] (H4_13) at (4.464,-1.732){};
 \node[v2] (H4_14) at (2.732,-2.732){};
 \node[v2] (H4_15) at (0,-4){};
 \node[v2] (H4_16) at (-3,-4.5){};
 \node[v2] (H4_17) at (-6,-0.5){};
   \draw (H4_1) -- (H4_2) -- (H4_6) -- (H4_5) -- (H4_4) -- (H4_3)-- (H4_1);
   \draw (H4_2) -- (H4_10) -- (H4_9) -- (H4_8) -- (H4_7) -- (H4_1);
   \draw (H4_4) -- (H4_11) -- (H4_12) -- (H4_3);
   \draw (H4_8) -- (H4_13) -- (H4_14) -- (H4_7);

  \draw (H4_15)to[bend right=15] (H4_14);
 \draw (H4_15)to[bend left=15] (H4_16);

 \draw (H4_12)to[bend left=55] (H4_17);
 \draw (H4_17)to[bend left=75] (H4_6);
  \draw (H4_17)to[bend right=55] (H4_16);

 \node[font=\scriptsize,below] at(H4_11) {$v_1$};
 \node[font=\scriptsize,xshift=6pt,yshift=3pt] at(H4_11) {\color{blue}4};
  \node[font=\scriptsize,left] at(H4_4) {$v_2$};
   \node[font=\scriptsize,xshift=6pt,yshift=3pt] at(H4_4) {\color{blue}5};
  \node[font=\scriptsize,left] at(H4_5) {$v_3$};
   \node[font=\scriptsize,xshift=6pt,yshift=-3pt] at(H4_5) {\color{blue}4};
  \node[font=\scriptsize,above] at(H4_6) {$v_4$};
   \node[font=\scriptsize,below] at(H4_6) {\color{blue}6};
   \node[font=\scriptsize,right] at(H4_12) {$v_5$};
    \node[font=\scriptsize,above] at(H4_12) {\color{blue}6};
  \node[font=\scriptsize,xshift=6pt,yshift=-5pt] at(H4_3) {$v_6$};
   \node[font=\scriptsize,above] at(H4_3) {\color{blue}7};
  \node[font=\scriptsize,below] at(H4_1) {$v_7$};
   \node[font=\scriptsize,xshift=-6pt,yshift=3pt] at(H4_1) {\color{blue}7};
  \node[font=\scriptsize,above] at(H4_2) {$v_8$};
   \node[font=\scriptsize,xshift=-6pt,yshift=-3pt] at(H4_2) {\color{blue}6};
   \node[font=\scriptsize,above] at(H4_10) {$v_9$};
    \node[font=\scriptsize,below] at(H4_10) {\color{blue}3};
  \node[font=\scriptsize,xshift=-5pt,yshift=-5pt] at(H4_7) {$v_{10}$};
   \node[font=\scriptsize,above] at(H4_7) {\color{blue}7};
  \node[font=\scriptsize,right] at(H4_8) {$v_{11}$};
   \node[font=\scriptsize,xshift=-6pt,yshift=3pt] at(H4_8) {\color{blue}5};
  \node[font=\scriptsize,right] at(H4_9) {$v_{12}$};
   \node[font=\scriptsize,xshift=-6pt,yshift=-3pt] at(H4_9) {\color{blue}3};
   \node[font=\scriptsize,left] at(H4_14) {$v_{13}$};
    \node[font=\scriptsize,above] at(H4_14) {\color{blue}5};
  \node[font=\scriptsize,xshift=6pt,yshift=5pt] at(H4_13) {$v_{14}$};
   \node[font=\scriptsize,xshift=-6pt,yshift=3pt] at(H4_13) {\color{blue}4};
  \node[font=\scriptsize,below] at(H4_15) {$ v_{17}$};
   \node[font=\scriptsize,above] at(H4_15) {\color{blue}3};
  \node[font=\scriptsize,below] at(H4_16) {$v_{16}$};
   \node[font=\scriptsize,above] at(H4_16) {\color{blue}3};
  \node[font=\scriptsize,right] at(H4_17) {$w$};
   \node[font=\scriptsize,left] at(H4_17) {\color{blue}6};
     \node[font=\scriptsize] at (0, -6.3) {(b) $v_4w,wv_{5}\in E(G)$};
\end{tikzpicture}\end{minipage}
\begin{minipage}{0.34\textwidth}
\centering\begin{tikzpicture}[
    v2/.style={fill=black,minimum size=4pt,ellipse,inner sep=1pt},
    scale=0.36
]
\node[v2] (H4_1) at (0, 0){};
    \node[v2] (H4_2) at (0, 2){};
    \node[v2] (H4_3) at (-1.732,-1){};
 \node[v2] (H4_4) at (-2*1.732,0){};
 \node[v2] (H4_5) at (-2*1.732,2){};
 \node[v2] (H4_6) at (-1.732,3){};
 \node[v2] (H4_7) at (1.732,-1){};
 \node[v2] (H4_8) at (2*1.732,0){};
 \node[v2] (H4_9) at (2*1.732,2){};
 \node[v2] (H4_10) at (1.732,3){};
 \node[v2] (H4_11) at (-4.464,-1.732){};
 \node[v2] (H4_12) at (-2.732,-2.732){};
 \node[v2] (H4_13) at (4.464,-1.732){};
 \node[v2] (H4_14) at (2.732,-2.732){};
 \node[v2] (H4_15) at (-6.5,-0){};
 \node[v2] (H4_16) at (0,-4.5){};
 \node[v2] (H4_17) at (6,1){};
   \draw (H4_1) -- (H4_2) -- (H4_6) -- (H4_5) -- (H4_4) -- (H4_3)-- (H4_1);
   \draw (H4_2) -- (H4_10) -- (H4_9) -- (H4_8) -- (H4_7) -- (H4_1);
   \draw (H4_4) -- (H4_11) -- (H4_12) -- (H4_3);
   \draw (H4_8) -- (H4_13) -- (H4_14) -- (H4_7);

  \draw (H4_12)to[bend left=35] (H4_15);
 \draw (H4_15)to[bend left=55] (H4_6);

 \draw (H4_17)to[bend right=35] (H4_10);
 \draw (H4_17)to[bend left=65] (H4_14);

  \draw (H4_17)to[bend left=76] (H4_16);
  \draw (H4_15)to[bend right=35] (H4_16);
 \node[font=\scriptsize,xshift=-6pt,yshift=1pt] at(H4_11) {$v_1$};
 \node[font=\scriptsize,xshift=6pt,yshift=3pt] at(H4_11) {\color{blue}4};
  \node[font=\scriptsize,left] at(H4_4) {$v_2$};
   \node[font=\scriptsize,xshift=6pt,yshift=3pt] at(H4_4) {\color{blue}5};
  \node[font=\scriptsize,left] at(H4_5) {$v_3$};
   \node[font=\scriptsize,xshift=6pt,yshift=-3pt] at(H4_5) {\color{blue}4};
  \node[font=\scriptsize,above] at(H4_6) {$v_4$};
   \node[font=\scriptsize,below] at(H4_6) {\color{blue}6};
   \node[font=\scriptsize,right] at(H4_12) {$v_5$};
    \node[font=\scriptsize,above] at(H4_12) {\color{blue}6};
  \node[font=\scriptsize,xshift=6pt,yshift=-5pt] at(H4_3) {$v_6$};
   \node[font=\scriptsize,above] at(H4_3) {\color{blue}7};
  \node[font=\scriptsize,below] at(H4_1) {$v_7$};
   \node[font=\scriptsize,xshift=-6pt,yshift=3pt] at(H4_1) {\color{blue}7};
  \node[font=\scriptsize,above] at(H4_2) {$v_8$};
   \node[font=\scriptsize,xshift=-6pt,yshift=-3pt] at(H4_2) {\color{blue}7};
   \node[font=\scriptsize,above] at(H4_10) {$v_9$};
    \node[font=\scriptsize,below] at(H4_10) {\color{blue}6};
  \node[font=\scriptsize,xshift=-5pt,yshift=-5pt] at(H4_7) {$v_{10}$};
   \node[font=\scriptsize,above] at(H4_7) {\color{blue}7};
  \node[font=\scriptsize,right] at(H4_8) {$v_{11}$};
   \node[font=\scriptsize,xshift=-6pt,yshift=3pt] at(H4_8) {\color{blue}5};
  \node[font=\scriptsize,right] at(H4_9) {$v_{12}$};
   \node[font=\scriptsize,xshift=-6pt,yshift=-3pt] at(H4_9) {\color{blue}4};
   \node[font=\scriptsize,left] at(H4_14) {$v_{13}$};
    \node[font=\scriptsize,above] at(H4_14) {\color{blue}6};
  \node[font=\scriptsize,xshift=6pt,yshift=5pt] at(H4_13) {$v_{14}$};
   \node[font=\scriptsize,xshift=-6pt,yshift=3pt] at(H4_13) {\color{blue}4};
  \node[font=\scriptsize,left] at(H4_15) {$w$};
   \node[font=\scriptsize,right] at(H4_15) {\color{blue}6};
  \node[font=\scriptsize,below] at(H4_16) {$v_{16}$};
   \node[font=\scriptsize,above] at(H4_16) {\color{blue}4};
  \node[font=\scriptsize,right] at(H4_17) {$z$};
   \node[font=\scriptsize,left] at(H4_17) {\color{blue}6};
     \node[font=\scriptsize] at (0, -6) {(c) $v_4w,wv_5,v_9z,zv_{13}\in E(G)$};
\end{tikzpicture}
\end{minipage}
\end{center}
\caption{Subcases 3.3  of the proof of Lemma \ref{reducible-H4}. The numbers at vertices are the number of available colors.} \label{H4-nbr-three}
\end{figure}
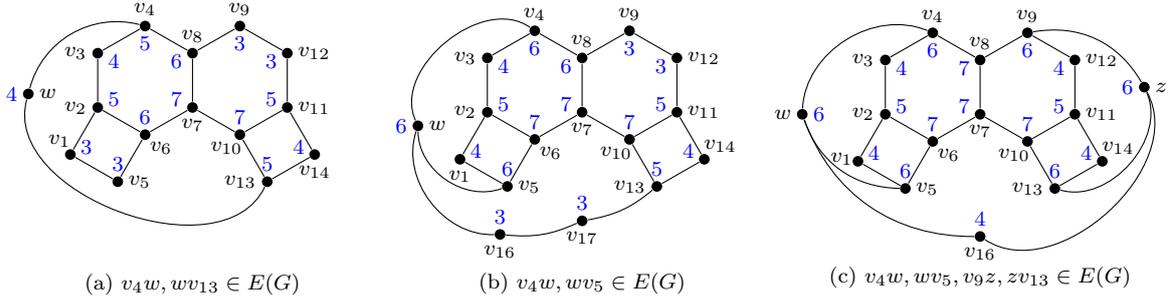

\medskip
\noindent {\bf Subcase 3.3:} $v_4$ and $v_{i}$ have a common neighbor $w$ where $i\in \{5,13\}$
and $w \notin V(J_6)$.
(Or by symmetry,
$v_5$ and $v_{9}$ have a common neighbor $w$,
where $w \notin V(J_6)$.)

Note that one may consider the case when $v_4$ and $v_{14}$ have a common neighbor.  But, this is the case of Subcase 3.1 (c). So, we do not need to consider this case. 

Recall that $v_5$ and $v_{13}$ lies on a 8-face. Then there are 3 cases needed to consider as shown in Figure~\ref{H4-nbr-three}. Furthermore, when $v_4w,wv_{13}\in E(G)$, it holds that  $w$ is not adjacent to any vertex in $V(J_6) \setminus \{v_4, v_{13}\}$ by the argument of Simplifying cases.
Recall further that
$\phi$ is a proper coloring of the square of $G' = G- V(J_6)$
such that $\phi(v) \in L(v)$ for each vertex $v \in V(G')$.
Then we uncolor some vertices as follows:
\begin{itemize}
\item If $v_4w,wv_{13}\in E(G)$, then uncolor $w$;
\item If $v_4w,wv_{5}\in E(G)$, the  uncolor vertices in $\{w,v_{16},v_{17}\}$;
\item If $v_4w,wv_5,v_9z,zv_{13}\in E(G)$, then uncolor vertices in $\{w,v_{16},z\}$.
\end{itemize}
And then follow the same procedure as in Case 2. Let $L_{J_6}'(v)$ be the list at $v$ after uncoloring vertices for each cases. Then the number of available colors at uncolored vertices of  each case is as shown in Figure \ref{H4-nbr-three}, respectively.

If $v_4w,wv_{13}\in E(G)$, then we have that $d_G(v_1, w) \geq 3$ or $d_G(v_1, w) = 2$.
So, we consider the following two subcases.

\medskip
\noindent {\bf Subcase 3.3.1:} $v_4w,wv_{13}\in E(G)$ and  $d_G(v_1, w) \geq 3$.

Color $v_1$ and $w$ with color $c_1$ and $c_w$, respectively, so that $|L_{J_6}'(v_3) \setminus \{c_1,c_w\}| \geq 3$. Color $v_{10}, v_{13}, v_{14}, v_{12}, v_{11}, v_{9}, v_5$ in this order following the same procedure as in Subcase 1.1. 

\medskip
\noindent {\bf Subcase 3.3.2:} $v_4w,wv_{13}\in E(G)$ and  $d_G(v_1, w) = 2$, or 
$v_4w,wv_{5}\in E(G)$,  or \\ $v_4w,wv_5,v_9z,zv_{13}\in E(G)$

If $v_4w,wv_{13}\in E(G)$ and $d_G(v_1, w) = 2$, then $|L_{J_6}'(w)| \geq 5$ and $|L_{J_6}'(v_1)| \geq 4$.
So, at each case of Subcases 3.3.2, we can color $w$ by a color $w_c \in L_{J_6}'(w)$ so that $|L_{J_6}'(v_3) \setminus \{w_c\}| \geq 4$
 since $|L_{J_6}'(w)| \geq 6$ and $|L_{J_6}'(v_3)| \geq 4$. Next, color $v_1$ and then follow  the same procedure as in Subcase 1.1.
Finally, by using  Lemma \ref{cycle-six-original-second}, we can show that each uncolored vertex admits an $L_{J_6}'$-coloring from the list.
This completes the proof of Lemma \ref{reducible-H4}.
\end{proof}


\subsection{7-face is adjacent to at most one 4-face} \label{subsection-H6}

In this subsection, we will show that a 7-face is adjacent to at most one 4-face.
If a 7-face is adjacent to two 4-faces, then by symmetry, we just need to consider  two cases; Figure \ref{subgraph-H6} and Figure \ref{subgraph-H7}.

\begin{figure}[htbp]
  \begin{center}
\begin{tikzpicture}[
  v2/.style={fill=black,minimum size=4pt,ellipse,inner sep=1pt},
  node distance=1.5cm,scale=0.4]

      \node[v2] (H5_1) at (0, 0){};
      \node[v2] (H5_2) at (0,-2){};
      \node[v2] (H5_3) at (1,-3.5){};
      \node[v2] (H5_4) at (2.8,-3.5){};
      \node[v2] (H5_5) at (4,-2){};
      \node[v2] (H5_6) at (4, 0){};
      \node[v2] (H5_7) at (2, 1.2){};
      \node[v2] (H5_9) at (-2, 0){};
      \node[v2] (H5_10) at (-2,-2){};
      \node[v2] (H5_11) at (6, 0){};
      \node[v2] (H5_12) at  (6,-2){};

   \draw (H5_1)--(H5_2)--(H5_3)--(H5_4)--(H5_5)--(H5_6)--(H5_7)--(H5_1);
   \draw (H5_1)--(H5_9)--(H5_10)--(H5_2);
   \draw (H5_5)--(H5_12)--(H5_11)--(H5_6);

  \node[font=\scriptsize,above] at (H5_9) {$v_1$};
  \node[font=\scriptsize,xshift=5pt,yshift=-5pt] at (H5_9) {\color{blue}3};
  \node[font=\scriptsize,above] at (H5_1) {$v_2$};
    \node[font=\scriptsize,xshift=5pt,yshift=-5pt] at (H5_1) {\color{blue}5};
  \node[font=\scriptsize,above] at (H5_7) {$v_3$};
    \node[font=\scriptsize,below] at (H5_7) {\color{blue}4};
  \node[font=\scriptsize,above] at (H5_6) {$v_4$};
    \node[font=\scriptsize,xshift=-5pt,yshift=-5pt] at (H5_6) {\color{blue}5};
  \node[font=\scriptsize,above] at (H5_11) {$v_5$};
    \node[font=\scriptsize,xshift=-5pt,yshift=-5pt] at (H5_11) {\color{blue}3};
  \node[font=\scriptsize,below] at (H5_10) {$v_6$};
    \node[font=\scriptsize,xshift=5pt,yshift=5pt] at (H5_10) {\color{blue}3};
  \node[font=\scriptsize,below=1pt, xshift=-3pt] at (H5_2) {$v_7$};
    \node[font=\scriptsize,xshift=5pt,yshift=5pt] at (H5_2) {\color{blue}5};
  \node[font=\scriptsize,below] at (H5_3) {$v_8$};
    \node[font=\scriptsize,xshift=2pt,yshift=6pt] at (H5_3) {\color{blue}3};
  \node[font=\scriptsize,below] at (H5_4) {$v_{9}$};
    \node[font=\scriptsize,xshift=-3pt,yshift=6pt] at (H5_4) {\color{blue}3};
  \node[font=\scriptsize,below=1pt, xshift=3pt] at (H5_5) {$v_{10}$};
  \node[font=\scriptsize,xshift=-5pt,yshift=5pt] at (H5_5) {\color{blue}5};
  \node[font=\scriptsize,below] at (H5_12) {$v_{11}$};
  \node[font=\scriptsize,xshift=-5pt,yshift=5pt] at (H5_12) {\color{blue}3};

  \end{tikzpicture}
\caption{Graph $H_5$. The numbers at vertices are the number of available colors.}
\label{subgraph-H6}
\end{center}
\end{figure}
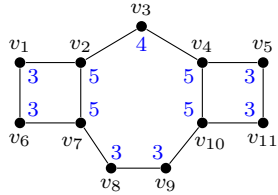


\begin{lemma} \label{reducible-H6}
The graph $H_5$ in Figure \ref{subgraph-H6} does not appear in $G$.
\end{lemma}
\begin{proof}
Suppose that $G$ has $H_5$ as a subgraph, and denote $V(H_5) = \{v_1, \ldots, v_{11}\}$ as in Figure \ref{subgraph-H6}.
Let $L$ be a list assignment with lists of size 7 for each vertex in $G$.
We will show that $G^2$ has a proper coloring from the list $L$, which is a contradiction for the fact that $G$ is a counterexample to the theorem. \\

Let $G' = G - V(H_5)$.
Then $G'$ is also a subcubic planar graph and $|V(G')| < |V(G)|$.   Since $G$ is a minimal counterexample to Theorem \ref{main-thm},
the square of $G'$ has a proper coloring $\phi$ such that $\phi(v) \in L(v)$ for each vertex $v \in V(H)$.
For each $v_i \in V(H_5)$, we define \[
L_{H_5}(v_i) = L(v_i) \setminus \{\phi(x) : xv_i \in E(G^2) \mbox{ and } x \notin V(H_5)\}.
\]
We have the following three cases.
\medskip

\medskip
\noindent {\bf Case 1:} $H_5^2$ is an induced subgraph of $G^2$.

In this case, we have the following (see Figure \ref{subgraph-H6}).
$$
|L_{H_5}(v_i)| \geq
\begin{cases}
3 & i=1,5,6,8,9,11, \\
4 & i=3,\\
5 & i=2,4,7,10.
\end{cases}
$$
Now, we   show that
$H_5^2$ admits an $L$-coloring from the list $L_{H_5}(v_i)$.
Observe that $H_5^2$ has 28 edges.  And the graph polynomial for $H_5^2$ is as follows.
\begin{eqnarray*}
P_{H_5^2}(\bm{x})
&=&
(x_1-x_2)(x_1-x_3)(x_1-x_6)(x_1-x_7)(x_2-x_3)(x_2-x_4)(x_2-x_6)(x_2-x_7)(x_2-x_8)
\\
&&
  (x_3-x_4)(x_3-x_5)(x_3-x_7)(x_3-x_{10})(x_4-x_5)(x_4-x_9)(x_4-x_{10})(x_4-x_{11})
\\
&&
  (x_5-x_{10})(x_5-x_{11})(x_6-x_7)(x_6-x_8)(x_7-x_8)(x_7-x_9)(x_8-x_9)(x_8-x_{10})
\\
&&
  (x_9-x_{10})(x_9-x_{11})(x_{10}-x_{11})
\end{eqnarray*}
By the calculation using Mathematica,
we see that
the coefficient of $x_1^2x_2^4x_3^2x_4^4x_5^2x_6^2x_7^3x_8^2x_9^2x_{10}^3x_{11}^2$
is $-2$.
Thus,
by Theorem \ref{cnull},
$H_5^2$ admits an $L$-coloring from its list.  This gives an $L$-coloring for $G^2$.  This is a contradiction for the fact that $G$ is a counterexample.  So, $G$ has no $H_5$. \\



\medskip

\noindent {\bf Case 2:} $H_5^2$ is not an induced subgraph of $G^2$ and
 $E(G[V(H_5)]) - E(H_5) \neq \emptyset$.

\medskip  
\noindent {\bf Simplifying cases:}
\begin{itemize}
\item The vertices in each of the following pairs are nonadjacent
since it makes a 5-cycle:
$\{v_1, v_3\}$, $\{v_1, v_{5}\}$, $\{v_1, v_{9}\}$,
$\{v_3, v_5\}$, $\{v_3, v_8\}$, $\{v_3, v_9\}$,
$\{v_5, v_8\}$,
$\{v_6, v_8\}$, and $\{v_9, v_{11}\}$.

\item The vertices in each of the following pairs are nonadjacent since it makes $H_1$,
which does not exist by Lemma \ref{reducible-H0}:
$\{v_6, v_9\}$, and $\{v_8, v_{11}\}$.

\item The vertices in each of the following pairs are nonadjacent since it makes $H_2$,
which does not exist by Lemma \ref{reducible-F2}:
$\{v_6, v_{11}\}$.

\item  The vertices in each of the following pairs are nonadjacent since
 it makes $J_4$, which does not exist by Lemma \ref{C4-share-two-edge}:
$\{v_1, v_8\}$, $\{v_3, v_6\}$, $\{v_3, v_{11}\}$, and $\{v_5, v_{9}\}$.
\end{itemize}


Considering these,
we only need to consider the case
where $v_1$ and $v_{11}$ are adjacent in $G$ or $v_5$ and $v_6$ are adjacent in $G$.
By symmetry,
suppose that the latter holds.

\begin{figure}[htbp]
  \begin{center}
\begin{tikzpicture}[
  v2/.style={fill=black,minimum size=4pt,ellipse,inner sep=1pt},
  node distance=1.5cm,scale=0.4]

      \node[v2] (H5_1) at (0, 0){};
      \node[v2] (H5_2) at (0,-2){};
      \node[v2] (H5_3) at (1,-3.5){};
      \node[v2] (H5_4) at (2.8,-3.5){};
      \node[v2] (H5_5) at (4,-2){};
      \node[v2] (H5_6) at (4, 0){};
      \node[v2] (H5_7) at (2, 1.2){};
      \node[v2] (H5_9) at (-2, 0){};
      \node[v2] (H5_10) at (-2,-2){};
      \node[v2] (H5_11) at (6, 0){};
      \node[v2] (H5_12) at  (6,-2){};

   \draw (H5_1)--(H5_2)--(H5_3)--(H5_4)--(H5_5)--(H5_6)--(H5_7)--(H5_1);
   \draw (H5_1)--(H5_9)--(H5_10)--(H5_2);
   \draw (H5_5)--(H5_12)--(H5_11)--(H5_6);
   \draw (6, 0)..controls (10,-4) and (1,-8).. (-2,-2);

  \node[font=\scriptsize,above] at (H5_9) {$v_1$};
  \node[font=\scriptsize,xshift=5pt,yshift=-5pt] at (H5_9) {\color{blue}4};
  \node[font=\scriptsize,above] at (H5_1) {$v_2$};
    \node[font=\scriptsize,xshift=5pt,yshift=-5pt] at (H5_1) {\color{blue}5};
  \node[font=\scriptsize,above] at (H5_7) {$v_3$};
    \node[font=\scriptsize,below] at (H5_7) {\color{blue}4};
  \node[font=\scriptsize,above] at (H5_6) {$v_4$};
    \node[font=\scriptsize,xshift=-5pt,yshift=-5pt] at (H5_6) {\color{blue}6};
  \node[font=\scriptsize,above] at (H5_11) {$v_5$};
    \node[font=\scriptsize,xshift=-5pt,yshift=-5pt] at (H5_11) {\color{blue}6};
  \node[font=\scriptsize,left] at (H5_10) {$v_6$};
    \node[font=\scriptsize,xshift=5pt,yshift=5pt] at (H5_10) {\color{blue}6};
  \node[font=\scriptsize,below=1pt, xshift=-3pt] at (H5_2) {$v_7$};
    \node[font=\scriptsize,xshift=5pt,yshift=5pt] at (H5_2) {\color{blue}6};
  \node[font=\scriptsize,below] at (H5_3) {$v_8$};
    \node[font=\scriptsize,xshift=2pt,yshift=6pt] at (H5_3) {\color{blue}3};
  \node[font=\scriptsize,below] at (H5_4) {$v_{9}$};
    \node[font=\scriptsize,xshift=-3pt,yshift=6pt] at (H5_4) {\color{blue}3};
  \node[font=\scriptsize,below=1pt, xshift=3pt] at (H5_5) {$v_{10}$};
  \node[font=\scriptsize,xshift=-5pt,yshift=5pt] at (H5_5) {\color{blue}5};
  \node[font=\scriptsize,below] at (H5_12) {$v_{11}$};
  \node[font=\scriptsize,xshift=-5pt,yshift=5pt] at (H5_12) {\color{blue}4};
\end{tikzpicture}
\end{center}
  \caption{The case when $v_5$ and $v_6$ are adjacent in the proof of Lemma \ref{reducible-H6}.
The numbers at vertices are the number of available colors.} \label{H6-adjacent}
\end{figure}
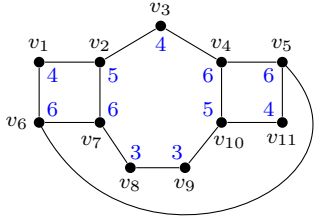

The graph polynomial for this subcase is
\[
f(\bm{x'}) = (x_1 - x_{5})(x_4 - x_{6})(x_5 - x_6)(x_5-x_7)(x_6 - x_{11})P_{H_5^2}(\bm{x}),
\]
where $P_{H_5^2}(\bm{x})$ is the graph polynomial for $H_5^2$ when $H_5$ is an induced subgraph in Case 1.
By the calculation using Mathematica,
the coefficient of
$x_1^3x_2^4x_3^2x_4^4x_5^2x_6^4x_7^3x_8^2x_9^2x_{10}^4x_{11}^3$
is $-5$, which is nonzero.
Thus in Case 2,
by Theorem \ref{cnull},
the vertices in $V(H_5)$ can be colored from the list $L_{H_5}$
so that we obtain an $L$-coloring in $G^2$.
This is a contradiction for the fact that $G$ is a counterexample.  \\

\bigskip
\noindent {\bf Case 3:} $H_5^2$ is not an induced subgraph of $G^2$ and
 $E(G[V(H_5)]) - E(H_5) = \emptyset$.

\medskip
\noindent {\bf Simplifying cases:}
\begin{itemize}
\item
The vertices in each of the following pairs do not have a common neighbor outside $H_5$,
since it makes a 5-cycle:
$\{v_1, v_6\}$, $\{v_1, v_8\}$, $\{v_3, v_6\}$, $\{v_3, v_8\}$,
$\{v_3, v_9\}$, $\{v_3, v_{11}\}$, $\{v_5, v_9\}$,
$\{v_5, v_{11}\}$, $\{v_6, v_9\}$, $\{v_8, v_{11}\}$.

\item
The vertices in each of the following pairs do not have a common neighbor outside $H_5$,
since it makes $F_2$,
which does not exist by Lemma \ref{reducible-F2}:
$\{v_8, v_9\}$.

\item
The vertices in each of the following pairs do not have a common neighbor outside $H_5$,
since it makes $H_1$,
which does not exist by Lemma \ref{reducible-H0}:
$\{v_1, v_3\}$, $\{v_3, v_5\}$, $\{v_6, v_8\}$, $\{v_9, v_{11}\}$.

\item
The vertices in each of the following pairs do not have a common neighbor outside $H_5$,
since it makes $J_5$,
which does not exist by Lemma \ref{H2-type-two-reducible}:
$\{v_1, v_5\}$.
\end{itemize}

Thus,
we only need to consider the following three subcases in Case 3.

\medskip

\noindent {\bf Subcase 3.1:} $v_6$ and $v_{11}$ have a common neighbor $w$,
where $w \notin V(H_5)$.
\medskip

Note that $w$ is not adjacent to any vertex $v_i \in V(H_5) \setminus \{v_6, v_{11}\}$ by the argument of Simplifying cases.
The number of available colors at vertices of $V(H_5)$ is presented in Subcase 3.1 in Figure \ref{H6-nbr-first}.  The graph polynomial for this subcase is
\[
f(\bm{x'}) = (x_6 - x_{11})P_{H_5^2}(\bm{x}),
\]
where $P_{H_5^2}(\bm{x})$ is the graph polynomial for $H_5^2$ when $H_5$ is an induced subgraph in Case 1.
By the calculation using Mathematica,
the coefficient of
$x_1^2x_2^4x_3^2x_4^4x_5^2x_6^2x_7^3x_8^2x_9^2x_{10}^3x_{11}^3$
is $1$, which is nonzero.

\begin{figure}[htbp]
  \begin{center}
\begin{tikzpicture}[
  v2/.style={fill=black,minimum size=4pt,ellipse,inner sep=1pt},
  node distance=1.5cm,scale=0.4]

      \node[v2] (H5_1) at (0, 0){};
      \node[v2] (H5_2) at (0,-2){};
      \node[v2] (H5_3) at (1,-3.5){};
      \node[v2] (H5_4) at (2.8,-3.5){};
      \node[v2] (H5_5) at (4,-2){};
      \node[v2] (H5_6) at (4, 0){};
      \node[v2] (H5_7) at (2, 1.2){};
      \node[v2] (H5_9) at (-2, 0){};
      \node[v2] (H5_10) at (-2,-2){};
      \node[v2] (H5_11) at (6, 0){};
      \node[v2] (H5_12) at  (6,-2){};
     \node[v2] (H5_8) at  (2,-5){};
   \draw (H5_1)--(H5_2)--(H5_3)--(H5_4)--(H5_5)--(H5_6)--(H5_7)--(H5_1);
   \draw (H5_1)--(H5_9)--(H5_10)--(H5_2);
   \draw (H5_5)--(H5_12)--(H5_11)--(H5_6);

    \draw (H5_10)to[bend right=20](H5_8);
     \draw (H5_8)to[bend right=20](H5_12);
  \node[font=\scriptsize,above] at (H5_9) {$v_1$};
  \node[font=\scriptsize,xshift=5pt,yshift=-5pt] at (H5_9) {\color{blue}3};
  \node[font=\scriptsize,above] at (H5_1) {$v_2$};
    \node[font=\scriptsize,xshift=5pt,yshift=-5pt] at (H5_1) {\color{blue}5};
  \node[font=\scriptsize,above] at (H5_7) {$v_3$};
    \node[font=\scriptsize,below] at (H5_7) {\color{blue}4};
  \node[font=\scriptsize,above] at (H5_6) {$v_4$};
    \node[font=\scriptsize,xshift=-5pt,yshift=-5pt] at (H5_6) {\color{blue}5};
  \node[font=\scriptsize,above] at (H5_11) {$v_5$};
    \node[font=\scriptsize,xshift=-5pt,yshift=-5pt] at (H5_11) {\color{blue}3};
  \node[font=\scriptsize,left] at (H5_10) {$v_6$};
    \node[font=\scriptsize,xshift=5pt,yshift=5pt] at (H5_10) {\color{blue}4};
  \node[font=\scriptsize,below=1pt, xshift=-3pt] at (H5_2) {$v_7$};
    \node[font=\scriptsize,xshift=5pt,yshift=5pt] at (H5_2) {\color{blue}5};
  \node[font=\scriptsize,below] at (H5_3) {$v_8$};
    \node[font=\scriptsize,xshift=2pt,yshift=6pt] at (H5_3) {\color{blue}3};
  \node[font=\scriptsize,below] at (H5_4) {$v_{9}$};
    \node[font=\scriptsize,xshift=-3pt,yshift=6pt] at (H5_4) {\color{blue}3};
  \node[font=\scriptsize,below=1pt, xshift=3pt] at (H5_5) {$v_{10}$};
  \node[font=\scriptsize,xshift=-5pt,yshift=5pt] at (H5_5) {\color{blue}5};
  \node[font=\scriptsize,right] at (H5_12) {$v_{11}$};
  \node[font=\scriptsize,xshift=-5pt,yshift=5pt] at (H5_12) {\color{blue}4};
\node[font=\scriptsize,below] at (H5_8) {$w$};
\node[font=\scriptsize] at (2,-7.5) {Subcase 3.1};

  \end{tikzpicture}\hspace{0.5cm}
\begin{tikzpicture}[
  v2/.style={fill=black,minimum size=4pt,ellipse,inner sep=1pt},
  node distance=1.5cm,scale=0.4]

      \node[v2] (H5_1) at (0, 0){};
      \node[v2] (H5_2) at (0,-2){};
      \node[v2] (H5_3) at (1,-3.5){};
      \node[v2] (H5_4) at (2.8,-3.5){};
      \node[v2] (H5_5) at (4,-2){};
      \node[v2] (H5_6) at (4, 0){};
      \node[v2] (H5_7) at (2, 1.2){};
      \node[v2] (H5_9) at (-2, 0){};
      \node[v2] (H5_10) at (-2,-2){};
      \node[v2] (H5_11) at (6, 0){};
      \node[v2] (H5_12) at  (6,-2){};
     \node[v2] (H5_8) at  (5,-5.2){};
   \draw (H5_1)--(H5_2)--(H5_3)--(H5_4)--(H5_5)--(H5_6)--(H5_7)--(H5_1);
   \draw (H5_1)--(H5_9)--(H5_10)--(H5_2);
   \draw (H5_5)--(H5_12)--(H5_11)--(H5_6);

    \draw (H5_10)to[bend right=27](H5_8);
     \draw (H5_8)to[bend right=60](H5_11);
  \node[font=\scriptsize,above] at (H5_9) {$v_1$};
  \node[font=\scriptsize,xshift=5pt,yshift=-5pt] at (H5_9) {\color{blue}3};
  \node[font=\scriptsize,above] at (H5_1) {$v_2$};
    \node[font=\scriptsize,xshift=5pt,yshift=-5pt] at (H5_1) {\color{blue}5};
  \node[font=\scriptsize,above] at (H5_7) {$v_3$};
    \node[font=\scriptsize,below] at (H5_7) {\color{blue}4};
  \node[font=\scriptsize,above] at (H5_6) {$v_4$};
    \node[font=\scriptsize,xshift=-5pt,yshift=-5pt] at (H5_6) {\color{blue}5};
  \node[font=\scriptsize,above] at (H5_11) {$v_5$};
    \node[font=\scriptsize,xshift=-5pt,yshift=-5pt] at (H5_11) {\color{blue}4};
  \node[font=\scriptsize,below] at (H5_10) {$v_6$};
    \node[font=\scriptsize,xshift=5pt,yshift=5pt] at (H5_10) {\color{blue}4};
  \node[font=\scriptsize,below=1pt, xshift=-3pt] at (H5_2) {$v_7$};
    \node[font=\scriptsize,xshift=5pt,yshift=5pt] at (H5_2) {\color{blue}5};
  \node[font=\scriptsize,below] at (H5_3) {$v_8$};
    \node[font=\scriptsize,xshift=2pt,yshift=6pt] at (H5_3) {\color{blue}3};
  \node[font=\scriptsize,below] at (H5_4) {$v_{9}$};
    \node[font=\scriptsize,xshift=-3pt,yshift=6pt] at (H5_4) {\color{blue}3};
  \node[font=\scriptsize,below=1pt, xshift=3pt] at (H5_5) {$v_{10}$};
  \node[font=\scriptsize,xshift=-5pt,yshift=5pt] at (H5_5) {\color{blue}5};
  \node[font=\scriptsize,below] at (H5_12) {$v_{11}$};
  \node[font=\scriptsize,xshift=-5pt,yshift=5pt] at (H5_12) {\color{blue}3};
\node[font=\scriptsize,below] at (H5_8) {$w$};
\node[font=\scriptsize] at (2,-7.5) {Subcase 3.2};

  \end{tikzpicture}\hspace{0.5cm}
  \begin{tikzpicture}[
  v2/.style={fill=black,minimum size=4pt,ellipse,inner sep=1pt},
  node distance=1.5cm,scale=0.4]

      \node[v2] (H5_1) at (0, 0){};
      \node[v2] (H5_2) at (0,-2){};
      \node[v2] (H5_3) at (1,-3.5){};
      \node[v2] (H5_4) at (2.8,-3.5){};
      \node[v2] (H5_5) at (4,-2){};
      \node[v2] (H5_6) at (4, 0){};
      \node[v2] (H5_7) at (2, 1.2){};
      \node[v2] (H5_9) at (-2, 0){};
      \node[v2] (H5_10) at (-2,-2){};
      \node[v2] (H5_11) at (6, 0){};
      \node[v2] (H5_12) at  (6,-2){};
     \node[v2] (H5_8) at  (7,-4){};
   \draw (H5_1)--(H5_2)--(H5_3)--(H5_4)--(H5_5)--(H5_6)--(H5_7)--(H5_1);
   \draw (H5_1)--(H5_9)--(H5_10)--(H5_2);
   \draw (H5_5)--(H5_12)--(H5_11)--(H5_6);

    \draw (H5_3)to[bend right=40](H5_8);
     \draw (H5_8)to[bend right=30](H5_11);
  \node[font=\scriptsize,above] at (H5_9) {$v_1$};
  \node[font=\scriptsize,xshift=5pt,yshift=-5pt] at (H5_9) {\color{blue}3};
  \node[font=\scriptsize,above] at (H5_1) {$v_2$};
    \node[font=\scriptsize,xshift=5pt,yshift=-5pt] at (H5_1) {\color{blue}5};
  \node[font=\scriptsize,above] at (H5_7) {$v_3$};
    \node[font=\scriptsize,below] at (H5_7) {\color{blue}4};
  \node[font=\scriptsize,above] at (H5_6) {$v_4$};
    \node[font=\scriptsize,xshift=-5pt,yshift=-5pt] at (H5_6) {\color{blue}5};
  \node[font=\scriptsize,above] at (H5_11) {$v_5$};
    \node[font=\scriptsize,xshift=-5pt,yshift=-5pt] at (H5_11) {\color{blue}4};
  \node[font=\scriptsize,below] at (H5_10) {$v_6$};
    \node[font=\scriptsize,xshift=5pt,yshift=5pt] at (H5_10) {\color{blue}3};
  \node[font=\scriptsize,below=1pt, xshift=-3pt] at (H5_2) {$v_7$};
    \node[font=\scriptsize,xshift=5pt,yshift=5pt] at (H5_2) {\color{blue}5};
  \node[font=\scriptsize,below] at (H5_3) {$v_8$};
    \node[font=\scriptsize,xshift=2pt,yshift=6pt] at (H5_3) {\color{blue}4};
  \node[font=\scriptsize,below] at (H5_4) {$v_{9}$};
    \node[font=\scriptsize,xshift=-3pt,yshift=6pt] at (H5_4) {\color{blue}3};
  \node[font=\scriptsize,below=1pt, xshift=3pt] at (H5_5) {$v_{10}$};
  \node[font=\scriptsize,xshift=-5pt,yshift=5pt] at (H5_5) {\color{blue}5};
  \node[font=\scriptsize,below] at (H5_12) {$v_{11}$};
  \node[font=\scriptsize,xshift=-5pt,yshift=5pt] at (H5_12) {\color{blue}3};
\node[font=\scriptsize,below] at (H5_8) {$w$};
\node[font=\scriptsize] at (2,-7.5) {Subcase 3.3};

  \end{tikzpicture}
  \end{center}
\caption{Subcase 3.1, 3.2, and 3.3 of the proof of Lemma \ref{reducible-H6}. The numbers at vertices are the number of available colors.} \label{H6-nbr-first}
\end{figure}
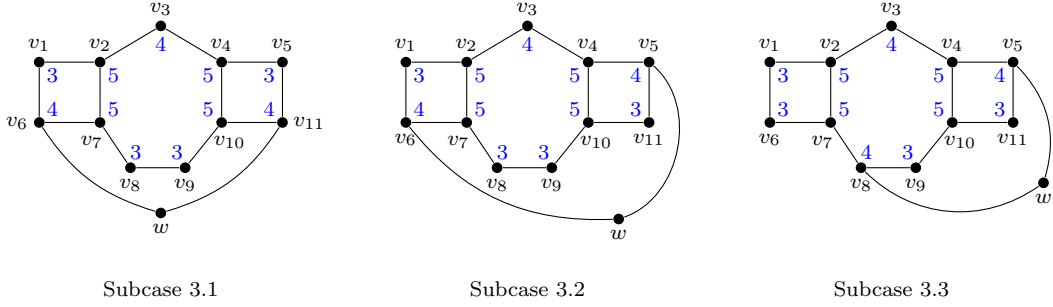

\medskip
\noindent {\bf Subcase 3.2:}
$v_5$ and $v_{6}$ (or $v_1$ and $v_{11}$ by symmetry)
have a common neighbor $w$,
where $w \notin V(H_5)$.
\medskip

Note that $w$ is not adjacent to any vertex $v_i \in V(H_5) \setminus \{v_5, v_6\}$ by the argument of Simplifying cases.
The number of available colors at vertices of $V(H_5)$ is presented in Subcase 3.2 in Figure \ref{H6-nbr-first}.  The graph polynomial for this subcase is
\[
f(\bm{x'}) = (x_5 - x_{6})P_{H_5^2}(\bm{x}).
\]
By the calculation using Mathematica,
the coefficient of
$x_1^2x_2^4x_3^2x_4^4x_5^3x_6^2x_7^3x_8^2x_9^2x_{10}^3x_{11}^2$
is $-1$, which is nonzero.

\medskip
\noindent {\bf Subcase 3.3:}
$v_5$ and $v_{8}$ (or $v_1$ and $v_{9}$ by symmetry)
have a common neighbor $w$,
where $w \notin V(H_5)$.
\medskip

Note that $w$ is not adjacent to any vertex $v_i \in V(H_5) \setminus \{v_5, v_8\}$ by the argument of Simplifying cases.
The number of available colors at vertices of $V(H_5)$ is presented in Subcase 3.3 in Figure \ref{H6-nbr-first}.  The graph polynomial for this subcase is
\[
f(\bm{x'}) = (x_5 - x_{8})P_{H_5^2}(\bm{x}).
\]
By the calculation using Mathematica,
we see that
the coefficient of
$x_1^2x_2^4x_3^2x_4^4x_5^3x_6^2x_7^3x_8^2x_9^2x_{10}^3x_{11}^2$
is $-3$, which is nonzero.

Thus in Case 3,
by Theorem \ref{cnull},
the vertices in $H_5$ can be colored from the list $L_{H_5}$
so that we obtain an $L$-coloring in $G^2$.
This gives an $L$-coloring for $G^2$.
This is a contradiction for the fact that $G$ is a counterexample.  This completes the proof of
Lemma \ref{reducible-H6}.
\end{proof}


Next we will prove a few properties which will be used in the proof of Lemma \ref{reducible-H7}.

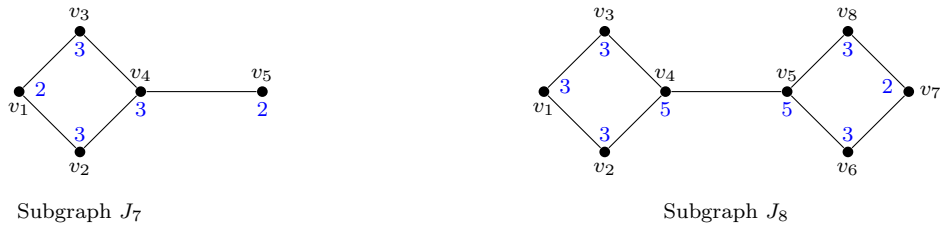
\begin{figure}[htbp]
  \begin{center}
  \begin{tikzpicture}[
  v2/.style={fill=black,minimum size=4pt,ellipse,inner sep=1pt},
  node distance=1.5cm,scale=0.8]

 \node[v2] (F2_1) at (0, 1) {};
    \node[v2] (F2_2) at (-1, 2) {};
    \node[v2] (F2_3) at (-1, 0) {};

    \node[v2] (F2_4) at (2, 1) {};

    \node[v2] (F2_8) at (-2,1) {};

    \draw (F2_1) -- (F2_2)  -- (F2_8)-- (F2_3) -- (F2_1);

    \draw (F2_1) -- (F2_4);

    \node[font=\scriptsize,above] at (F2_1) {$v_4$};
    \node[font=\scriptsize,below] at (F2_1) {\color{blue}3};
    \node[font=\scriptsize,above] at (F2_4) {$v_5$};
    \node[font=\scriptsize,below] at (F2_4) {\color{blue}2};
    \node[font=\scriptsize,above] at (F2_2) {$v_3$};
    \node[font=\scriptsize,below] at (F2_2) {\color{blue}3};
     \node[font=\scriptsize,below] at (F2_8) {$v_1$};
    \node[font=\scriptsize,xshift=8pt,yshift=1pt] at (F2_8) {\color{blue}2};
     \node[font=\scriptsize,below] at (F2_3) {$v_2$};
    \node[font=\scriptsize,above] at (F2_3) {\color{blue}3};
     \node[font=\scriptsize] at (-1,-1) { Subgraph $J_7$};
\end{tikzpicture}\hspace{3cm}
\begin{tikzpicture}[
  v2/.style={fill=black,minimum size=4pt,ellipse,inner sep=1pt},
  node distance=1.5cm,scale=0.8]

 \node[v2] (F2_1) at (0, 1) {};
    \node[v2] (F2_2) at (-1, 2) {};
    \node[v2] (F2_3) at (-1, 0) {};

    \node[v2] (F2_4) at (2, 1) {};
    \node[v2] (F2_5) at (3, 2) {};
    \node[v2] (F2_6) at (4, 1) {};
    \node[v2] (F2_7) at (3, 0) {};
    \node[v2] (F2_8) at (-2,1) {};

    \draw (F2_1) -- (F2_2)  -- (F2_8)-- (F2_3) -- (F2_1);

    \draw (F2_4) -- (F2_5) -- (F2_6) -- (F2_7) -- (F2_4);

    \draw (F2_1) -- (F2_4);

    \node[font=\scriptsize,above] at (F2_2) {$v_3$};
    \node[font=\scriptsize,below] at (F2_2) {\color{blue}3};
    \node[font=\scriptsize,above] at (F2_1) {$v_4$};
    \node[font=\scriptsize,below] at (F2_1) {\color{blue}5};
    \node[font=\scriptsize,below] at (F2_3) {$v_2$};
    \node[font=\scriptsize,above] at (F2_3) {\color{blue}3};
     \node[font=\scriptsize,below] at (F2_8) {$v_1$};
    \node[font=\scriptsize,xshift=8pt,yshift=2pt] at (F2_8) {\color{blue}3};
    \node[font=\scriptsize,above] at (F2_4) {$v_5$};
    \node[font=\scriptsize,below] at (F2_4) {\color{blue}5};
    \node[font=\scriptsize,below] at (F2_7) {$v_6$};
    \node[font=\scriptsize,above] at (F2_7) {\color{blue}3};
      \node[font=\scriptsize,right] at (F2_6) {$v_7$};
    \node[font=\scriptsize,xshift=-8pt,yshift=2pt] at (F2_6) {\color{blue}2};

      \node[font=\scriptsize,above] at (F2_5) {$v_8$};
    \node[font=\scriptsize,below] at (F2_5) {\color{blue}3};
      \node[font=\scriptsize] at (1,-1) { Subgraph $J_8$};
\end{tikzpicture}
\end{center}
\caption{Subgraphs $J_7$ and $J_8$.  The numbers at vertices are the number of available colors.}\label{4cycle-pendent}
\end{figure}

\begin{lemma} \label{lem-4cycle-pendent}
For a graph $J_7$ with $V(J_7) = \{v_1, v_2, v_3, v_4, v_5\}$ as in Figure \ref{4cycle-pendent},
suppose that each vertex $v_i$ has a list $L(v_i)$ with
$|L(v_i)|\geq 2$ for $i=1,5$,
$|L(v_i)| \geq 3$ for $i = 2, 3, 4$,
and $|L(v_2) \cup L(v_3) \cup L(v_4)| \geq 4$.
In addition, suppose further that
$|L(v_1)|\geq 3$,
or
$|V(v_1) \cup L(v_3)| \geq 5$.
Then $J_7^2$ has a proper coloring from the list.
\end{lemma}
\begin{proof}
We consider two cases.

\medskip
\noindent {\bf Case 1:} $L(v_1) \cap L(v_5) \neq \emptyset$. \\
Color $v_1$ and $v_5$ by a color $c \in L(v_1) \cap L(v_5)$.  For $i = 2, 3, 4$, let $L' (v_i) = L(v_i) \setminus \{c\}$.  Then $|L'(v_i)| \geq 2$ for $i = 2, 3, 4$, and $|L'(v_2) \cup L'(v_3) \cup L'(v_4)| \geq 3$ since $|L(v_2) \cup L(v_3) \cup L(v_4)| \geq 4$. So, $v_2$, $v_3$, $v_4$ are colorable from the list.

\medskip
\noindent {\bf Case 2:} $L(v_1) \cap L(v_5) = \emptyset$. \\
Let $X = \{v_1, v_2, v_3, v_4, v_5\}$, and we
 define a bipartite graph $W$ as follows.
\begin{itemize}
\item $V(W) = X \cup Y$ where $Y = \bigcup_{v_i \in X} L(v_i)$.

\item For each  $v_i$, $v_i \alpha \in E(W)$ if $\alpha \in L(v_i)$.
\end{itemize}
Note that $|N_W(v_2) \cup N_W(v_3) \cup N_W(v_4) | \geq 4$,
and in addition
$|N_W(v_1)| \geq 3$
or
$|N_W(v_1) \cup N_W(v_3)| \geq 5$.
Thus, we can easily show that $|N_W(S)| \geq |S|$ for every $S \subset X$.
Hence by Hall's theorem,
the bipartite graph $W$ has a matching $M$ that contains all vertices of $\{v_1, v_2, v_3, v_4, v_5\}$.  Hence $J_7^2$ can be colored properly from the list.
\end{proof}

\begin{lemma} \label{lem-two-4cycle}
For a graph $J_8$ with $V(J_8) = \{v_1, v_2, v_3, v_4, v_5, v_6, v_7, v_8\}$ as in Figure \ref{4cycle-pendent},
if each vertex $v_i$ has a list $L(v_i)$ with $|L(v_i)| = 3$ for $i = 1, 2, 3, 6,8$, $|L(v_4)|=|L(v_5)|=5$, and $|L(v_7)|=2$,
then $J_8^2$ has a proper coloring from the list.
\end{lemma}
\begin{proof}
In subgraph $J_8$ in Figure \ref{4cycle-pendent}, we first greedily color $v_6, v_7, v_8$ by colors $c_6, v_7, v_8$, respectively.  If $L(v_2) = L(v_3) = L(v_4) \setminus \{c_6, c_8\}$, then we recolor $v_6$ by a color $c \in L(v_6) \setminus \{c_6, c_8\}$.  Then $L(v_2) \neq L(v_4) \setminus \{c, c_8\}$.  This satisfies the condition of $J_7$ as in Lemma \ref{lem-4cycle-pendent}, so $J_8^2$ can be colored properly from the list.
\end{proof}

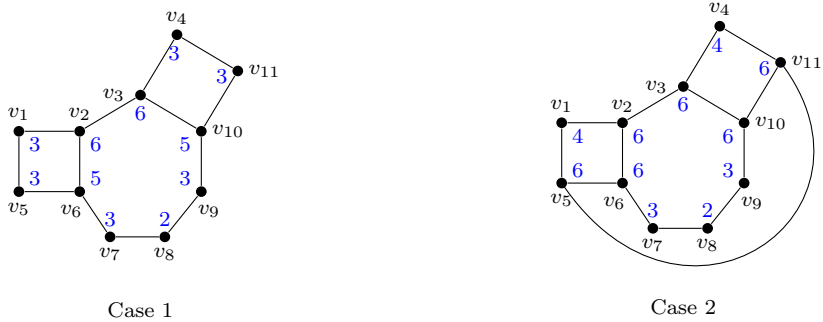
\begin{figure}[htbp]
  \begin{center}
  \begin{tikzpicture}[
  v2/.style={fill=black,minimum size=4pt,ellipse,inner sep=1pt},
  node distance=1.5cm,scale=0.4]

      \node[v2] (H6_1) at (0, 0){};
      \node[v2] (H6_2) at (0,-2){};
      \node[v2] (H6_3) at (1,-3.5){};
      \node[v2] (H6_4) at (2.8,-3.5){};
      \node[v2] (H6_5) at (4,-2){};
      \node[v2] (H6_6) at (4, 0){};
      \node[v2] (H6_7) at (2, 1.2){};
      \node[v2] (H6_9) at (-2, 0){};
      \node[v2] (H6_10) at (-2,-2){};
      \node[v2] (H6_11) at  (3.2,3.2){};
      \node[v2] (H6_12) at  (5.2,2){};

   \draw (H6_1)--(H6_2)--(H6_3)--(H6_4)--(H6_5)--(H6_6)--(H6_7)--(H6_1);
   \draw (H6_1)--(H6_9)--(H6_10)--(H6_2);
   \draw (H6_6)--(H6_12)--(H6_11)--(H6_7);

  \node[font=\scriptsize,above] at (H6_9) {$v_1$};
  \node[font=\scriptsize,xshift=6pt,yshift=-5pt] at (H6_9) {\color{blue}3};
  \node[font=\scriptsize,above] at (H6_1) {$v_2$};
    \node[font=\scriptsize,xshift=6pt,yshift=-5pt] at (H6_1) {\color{blue}6};
  \node[font=\scriptsize,xshift=-10pt] at (H6_7) {$v_3$};
    \node[font=\scriptsize,below] at (H6_7) {\color{blue}6};
 \node[font=\scriptsize,above] at (H6_11) {$v_4$};
     \node[font=\scriptsize,xshift=-1pt,yshift=-7pt] at (H6_11) {\color{blue}3};
  \node[font=\scriptsize,below] at (H6_10) {$v_5$};
      \node[font=\scriptsize,xshift=6pt,yshift=5pt] at (H6_10) {\color{blue}3};
  \node[font=\scriptsize,below=1pt, xshift=-3pt] at (H6_2) {$v_6$};
      \node[font=\scriptsize,xshift=6pt,yshift=5pt] at (H6_2) {\color{blue}5};
  \node[font=\scriptsize,below] at (H6_3) {$v_7$};
      \node[font=\scriptsize,above] at (H6_3) {\color{blue}3};
  \node[font=\scriptsize,below] at (H6_4) {$v_{8}$};
      \node[font=\scriptsize,above] at (H6_4) {\color{blue}2};
  \node[font=\scriptsize,below=1pt, xshift=3pt] at (H6_5) {$v_{9}$};
      \node[font=\scriptsize,xshift=-6pt,yshift=5pt] at (H6_5) {\color{blue}3};
   \node[font=\scriptsize,right] at (H6_6) {$v_{10}$};
       \node[font=\scriptsize,xshift=-6pt,yshift=-5pt] at (H6_6) {\color{blue}5};
  \node[font=\scriptsize,right] at (H6_12) {$v_{11}$};
      \node[font=\scriptsize,xshift=-6pt,yshift=-2pt] at (H6_12) {\color{blue}3};

  \node[font=\scriptsize] at (2,-5.9) {Case 1};

\end{tikzpicture}\hspace{3cm}
\raisebox{-0.5cm}{
  \begin{tikzpicture}[
  v2/.style={fill=black,minimum size=4pt,ellipse,inner sep=1pt},
  node distance=1.5cm,scale=0.4]

      \node[v2] (H6_1) at (0, 0){};
      \node[v2] (H6_2) at (0,-2){};
      \node[v2] (H6_3) at (1,-3.5){};
      \node[v2] (H6_4) at (2.8,-3.5){};
      \node[v2] (H6_5) at (4,-2){};
      \node[v2] (H6_6) at (4, 0){};
      \node[v2] (H6_7) at (2, 1.2){};
      \node[v2] (H6_9) at (-2, 0){};
      \node[v2] (H6_10) at (-2,-2){};
      \node[v2] (H6_11) at  (3.2,3.2){};
      \node[v2] (H6_12) at  (5.2,2){};

   \draw (H6_1)--(H6_2)--(H6_3)--(H6_4)--(H6_5)--(H6_6)--(H6_7)--(H6_1);
   \draw (H6_1)--(H6_9)--(H6_10)--(H6_2);
   \draw (H6_6)--(H6_12)--(H6_11)--(H6_7);
   \draw(-2,-2)..controls (2,-8)and (9,-3)..(5.2,2);

  \node[font=\scriptsize,above] at (H6_9) {$v_1$};
  \node[font=\scriptsize,xshift=6pt,yshift=-5pt] at (H6_9) {\color{blue}4};
  \node[font=\scriptsize,above] at (H6_1) {$v_2$};
    \node[font=\scriptsize,xshift=6pt,yshift=-5pt] at (H6_1) {\color{blue}6};
  \node[font=\scriptsize,xshift=-10pt] at (H6_7) {$v_3$};
    \node[font=\scriptsize,below] at (H6_7) {\color{blue}6};
 \node[font=\scriptsize,above] at (H6_11) {$v_4$};
     \node[font=\scriptsize,xshift=-1pt,yshift=-7pt] at (H6_11) {\color{blue}4};
  \node[font=\scriptsize,below] at (H6_10) {$v_5$};
      \node[font=\scriptsize,xshift=6pt,yshift=5pt] at (H6_10) {\color{blue}6};
  \node[font=\scriptsize,below=1pt, xshift=-3pt] at (H6_2) {$v_6$};
      \node[font=\scriptsize,xshift=6pt,yshift=5pt] at (H6_2) {\color{blue}6};
  \node[font=\scriptsize,below] at (H6_3) {$v_7$};
      \node[font=\scriptsize,above] at (H6_3) {\color{blue}3};
  \node[font=\scriptsize,below] at (H6_4) {$v_{8}$};
      \node[font=\scriptsize,above] at (H6_4) {\color{blue}2};
  \node[font=\scriptsize,below=1pt, xshift=3pt] at (H6_5) {$v_{9}$};
      \node[font=\scriptsize,xshift=-6pt,yshift=5pt] at (H6_5) {\color{blue}3};
   \node[font=\scriptsize,right] at (H6_6) {$v_{10}$};
       \node[font=\scriptsize,xshift=-6pt,yshift=-5pt] at (H6_6) {\color{blue}6};
  \node[font=\scriptsize,right] at (H6_12) {$v_{11}$};
      \node[font=\scriptsize,xshift=-6pt,yshift=-2pt] at (H6_12) {\color{blue}6};

  \node[font=\scriptsize] at (2,-6.1) {Case 2};

\end{tikzpicture}}
\end{center}
\caption{Graph $H_6$. The numbers at vertices are the number of available colors.} \label{subgraph-H7}
\end{figure}
\begin{lemma} \label{reducible-H7}
The graph $H_6$ in Figure \ref{subgraph-H7} does not appear in $G$.
\end{lemma}
\begin{proof}
Suppose that $G$ has $H_6$ as a subgraph, and denote $V(H_6) = \{v_1, \ldots, v_{11}\}$ (Figure \ref{subgraph-H7}).
Let $L$ be a list assignment with lists of size 7 for each vertex in $G$.
We will show that $G^2$ has a proper coloring from the list $L$, which is a contradiction for the fact that $G$ is a counterexample to the theorem.

Let $G' = G - V(H_6)$.
Then $G'$ is also a subcubic planar graph and $|V(G')| < |V(G)|$.   Since $G$ is a minimal counterexample to Theorem \ref{main-thm},
the square of $G'$ has a proper coloring $\phi$ such that $\phi(v) \in L(v)$ for each vertex $v \in V(H)$.
For each $v_i \in V(H_6)$, we define \[
L_{H_6}(v_i) = L(v_i) \setminus \{\phi(x) : xv_i \in E(G^2) \mbox{ and } x \notin V(H_6)\}.
\]
We now consider three cases.

\medskip
\noindent {\bf Case 1:} $H_6^2$ is an induced subgraph of $G^2$.

In this case, we have the following (see Case 1 in Figure \ref{subgraph-H7}).
$$
|L_{H_6}(v_i)| \geq
\begin{cases}
2 & i=8, \\
3 & i=1,4,5,7,9,11, \\
5 & i=6, 10, \\
6 & i=2, 3.
\end{cases}
$$
Now, we  show that
$H_6^2$ admits an $L$-coloring from the list $L_{H_6}(v_i)$.

\medskip
\noindent {\bf Subcase 1.1:} $L_{H_6}(v_1) \cap L_{H_6}(v_7) \neq \emptyset$.

Color $v_1$ and $v_7$ by a color $c \in L_{H_6}(v_1) \cap L_{H_6}(v_7)$, and color $v_8$ and $v_9$ in order.  Let $c_8$ and $c_9$ be the colors at $v_8$ and $v_9$, repectively.
For $v_i \in \{v_2, v_3, v_4, v_5, v_6, v_{10}, v_{11}\}$, let $L_{H_6}'(v_i)$ be the list after coloring $v_1, v_7, v_8, v_9$.  Then
\[\begin{aligned}
|L_{H_6}'(v_2)|& \geq 5, \quad |L_{H_6}'(v_3)| \geq 4,\quad~|L_{H_6}'(v_4)| \geq 3, \quad |L_{H_6}'(v_5)|  \geq 2,\\
 |L_{H_6}'(v_6)|& \geq 3,\quad |L_{H_6}'(v_{10})| \geq 3,\quad |L_{H_6}'(v_{11})| \geq 2.
\end{aligned}
\]
Here since $|L_{H_6}'(v_5)| \geq 2$ and $|L_{H_6}'(v_6)| \geq 3$, we can color $v_5$ and $v_6$ by $c_5, c_6$, respectively, so that $L_{H_6}'(v_3) \setminus \{c_6\} \neq  L_{H_6}'(v_4)$.  For $v_i \in \{v_2, v_3, v_4, v_{10}, v_{11}\}$,
$L_{H_6}''(v)$ be the list after coloring $v_5$ and $v_6$.  That is
\[\begin{aligned}
& L_{H_6}''(v_2) = L_{H_6}'(v_2) \setminus \{c_5, c_6\}, \quad 
 L_{H_6}''(v_3) = L_{H_6}'(v_3) \setminus \{ c_6\},\quad L_{H_6}''(v_4) = L_{H_6}'(v_4),\\
& L_{H_6}''(v_{10}) = L_{H_6}'(v_{10}),\quad\quad\quad \quad~L_{H_6}''(v_{11}) = L_{H_6}'(v_{11}).
\end{aligned}
\]
Then $|L_{H_6}''(v_2)| \geq 3, |L_{H_6}''(v_3)| \geq 3, |L_{H_6}''(v_4)| \geq 3, |L_{H_6}''(v_{10})| \geq 3, |L_{H_6}''(v_{11})| \geq 2$, and $|L_{H_6}''(v_3) \cup L''(v_4) \cup L_{H_6}''(v_{10})| \geq 4$ since $L_{H_6}''(v_3) \neq L_{H_6}''(v_4)$.

Now $v_2, v_3, v_4, v_{10}, v_{11}$ are colorable by Lemma \ref{lem-4cycle-pendent} from the list $L_{H_6}''$.  So, $H_6^2$ admits an $L$-coloring from its list.

\medskip

\noindent {\bf Subcase 1.2:} $L_{H_6}(v_7) \neq L_{H_6}(v_5)$.

Color $v_7$ by a color $c_7 \in L_{H_6}(v_7)$ so that $|L_{H_6}(v_5) \setminus \{c_7\}| \geq 3$,
and greedily color $v_8$ and $v_9$ in order.
For $v_i \in \{v_1, v_2, v_3, v_4, v_5, v_6, v_{10}, v_{11}\}$, let $L_{H_6}'(v_i)$ be the list after coloring $v_7, v_8, v_9$.  Then
\begin{eqnarray*}
|L_{H_6}'(v_i)| \geq
\begin{cases}
2 & i= 11, \\
3 & i= 1, 4, 5, 6, 10, \\
5 & i=2,3.
\end{cases}
\end{eqnarray*}
Then $\{v_1, v_2, v_3, v_4, v_5, v_6, v_{10}, v_{11}\}$ are colorable from the list $L_{H_6}'(v)$ by Lemma \ref{lem-two-4cycle}.
So, $H_6^2$ admits an $L$-coloring from its list.

\medskip

\noindent {\bf Subcase 1.3:} $L_{H_6}(v_1) \cap L_{H_6}(v_7) = \emptyset$ and $L_{H_6}(v_7) = L_{H_6}(v_5)$.

In this case, we have that $L_{H_6}(v_1) \cap L_{H_6}(v_5) = \emptyset$.
Now color greedily $v_7, v_8, v_9$ in order.
For $v_i \in \{v_1, v_2, v_3, v_4,  v_5, v_6, v_{10}, v_{11}\}$, let $L_{H_6}'(v_i)$ be the list after coloring
$v_7, v_8, v_9$.  Then we have the following (see Figure \ref{Case 3-H7}).
$$
|L_{H_6}'(v_i)| \geq
\begin{cases}
2 & i=5, 11, \\
3 & i=1, 4, 6, 10, \\
5 & i=2,3.
\end{cases}
$$
Since $|L_{H_6}'(v_4)| \geq 3$,  $|L_{H_6}'(v_{10})| \geq 3$, and $|L_{H_6}'(v_{11})| \geq 2$,  we can color $v_4, v_{10}, v_{11}$ by $c_4, c_{10}, c_{11}$, respectively, so that $L_{H_6}'(v_1) \neq L_{H_6}'(v_2) \setminus \{c_4, c_{10}\}$.

For $v_i \in \{v_1, v_2, v_3,  v_5, v_6\}$, let $L_{H_6}''(v_i)$ be the list after coloring $v_4, v_{10}, v_{11}$.  Then we have the following (see Figure \ref{Case 3-H7}).
\[
|L_{H_6}''(v_i)| \geq
\begin{cases}
2 & i=3,5, \\
3 & i=1, 2, 6.
\end{cases}
\]
Note that $|L_{H_6}''(v_1) \cup L_{H_6}''(v_2) \cup L_{H_6}''(v_6) | \geq 4$ since $L_{H_6}'(v_1) \neq L_{H_6}'(v_2) \setminus \{c_4, c_{10}\}$, and $|L_{H_6}''(v_1) \cup L_{H_6}''(v_5)| \geq 5$ since $L_{H_6}(v_1) \cap L_{H_6}(v_5) = \emptyset$.
By Lemma \ref{lem-4cycle-pendent}, $v_1, v_2, v_3, v_5, v_6$ are colorable from the list $L_{H_6}''(v)$.
Hence $H_6^2$ can be colored properly from the list.

\begin{figure}[htbp]
  \begin{center}
  \begin{tikzpicture}[
  v2/.style={fill=black,minimum size=4pt,ellipse,inner sep=1pt},
  node distance=1.5cm,scale=0.8]

 \node[v2] (F2_1) at (0, 1) {};
    \node[v2] (F2_2) at (-1, 2) {};
    \node[v2] (F2_3) at (-1, 0) {};

    \node[v2] (F2_4) at (2, 1) {};
    \node[v2] (F2_5) at (3, 2) {};
    \node[v2] (F2_6) at (4, 1) {};
    \node[v2] (F2_7) at (3, 0) {};
    \node[v2] (F2_8) at (-2,1) {};

    \draw (F2_1) -- (F2_2)  -- (F2_8)-- (F2_3) -- (F2_1);

    \draw (F2_4) -- (F2_5) -- (F2_6) -- (F2_7) -- (F2_4);

    \draw (F2_1) -- (F2_4);

    \node[font=\scriptsize,above] at (F2_2) {$v_1$};
    \node[font=\scriptsize,below] at (F2_2) {\color{blue}3};
    \node[font=\scriptsize,above] at (F2_1) {$v_2$};
    \node[font=\scriptsize,xshift=-8pt,yshift=1pt] at (F2_1) {\color{blue}5};
    \node[font=\scriptsize,below] at (F2_3) {$v_6$};
    \node[font=\scriptsize,above] at (F2_3) {\color{blue}3};
     \node[font=\scriptsize,below] at (F2_8) {$v_5$};
    \node[font=\scriptsize,xshift=8pt,yshift=2pt] at (F2_8) {\color{blue}2};
    \node[font=\scriptsize,above] at (F2_4) {$v_3$};
    \node[font=\scriptsize,xshift=8pt,yshift=1pt] at (F2_4) {\color{blue}5};
    \node[font=\scriptsize,below] at (F2_7) {$v_{10}$};
    \node[font=\scriptsize,above] at (F2_7) {\color{blue}3};
      \node[font=\scriptsize,right] at (F2_6) {$v_{11}$};
    \node[font=\scriptsize,xshift=-8pt,yshift=2pt] at (F2_6) {\color{blue}2};

      \node[font=\scriptsize,above] at (F2_5) {$v_4$};
    \node[font=\scriptsize,below] at (F2_5) {\color{blue}3};
      \node[font=\scriptsize] at (1,-1) { Color list $L_{H_6}'(v)$};
\end{tikzpicture}\hspace{3cm}
\begin{tikzpicture}[
  v2/.style={fill=black,minimum size=4pt,ellipse,inner sep=1pt},
  node distance=1.5cm,scale=0.8]

 \node[v2] (F2_1) at (0, 1) {};
    \node[v2] (F2_2) at (-1, 2) {};
    \node[v2] (F2_3) at (-1, 0) {};

    \node[v2] (F2_4) at (2, 1) {};

    \node[v2] (F2_8) at (-2,1) {};

    \draw (F2_1) -- (F2_2)  -- (F2_8)-- (F2_3) -- (F2_1);

    \draw (F2_1) -- (F2_4);

    \node[font=\scriptsize,above] at (F2_1) {$v_2$};
    \node[font=\scriptsize, xshift=-8pt,yshift=2pt] at (F2_1) {\color{blue}3};
    \node[font=\scriptsize,above] at (F2_4) {$v_3$};
    \node[font=\scriptsize,right] at (F2_4) {\color{blue}2};
    \node[font=\scriptsize,above] at (F2_2) {$v_1$};
    \node[font=\scriptsize,below] at (F2_2) {\color{blue}3};
     \node[font=\scriptsize,below] at (F2_8) {$v_5$};
    \node[font=\scriptsize,xshift=8pt,yshift=1pt] at (F2_8) {\color{blue}2};
     \node[font=\scriptsize,below] at (F2_3) {$v_6$};
    \node[font=\scriptsize,above] at (F2_3) {\color{blue}3};
     \node[font=\scriptsize] at (0,-1) {Color list $L_{H_6}''(v)$};
\end{tikzpicture}
\end{center}
\caption{Subcase 1.3.  The numbers at vertices are the number of available colors.}\label{Case 3-H7}
\end{figure}
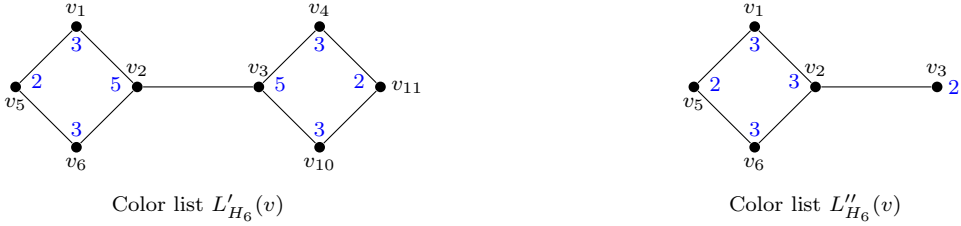

\bigskip
\noindent {\bf Case 2:} $H_6^2$ is not an induced subgraph of $G^2$ and
 $E(G[V(H_6)]) - E(H_6) \neq \emptyset$.

\medskip

\noindent {\bf Simplifying cases:}
\begin{itemize}
\item
The vertices in each of the following pairs are nonadjacent since it makes a $5$-cycle:
$\{v_1,v_8\}$, $\{v_1,v_{9}\}$, $\{v_1,v_{11}\}$,
$\{v_4,v_5\}$, $\{v_4,v_{7}\}$, $\{v_4, v_{8}\}$,
$\{v_5, v_{7}\}$, $\{v_5, v_{9}\}$, $\{v_7, v_{11}\}$, and $\{v_9, v_{11}\}$.

\item
The vertices in each of the following pairs are nonadjacent since it makes $F_2$,
which does not exist by Lemma \ref{reducible-F2}:
$\{v_{7}, v_{9}\}$.

\item
The vertices in each of the following pairs are nonadjacent since it makes $H_1$
in Figure \ref{key configuration},
which does not exist by Lemma \ref{reducible-H0}:
$\{v_1,v_{4}\}$, $\{v_5,v_{8}\}$, and $\{v_8,v_{11}\}$.

\item
The vertices in each of the following pairs are nonadjacent since it makes $J_4$,
which does not exist by Lemma \ref{C4-share-two-edge}:
$\{v_1,v_{7}\}$, and $\{v_4,v_{9}\}$.
\end{itemize}

Thus,
we only need to consider the case when $v_5$ is adjacent to $v_{11}$ (see Case 2 in Figure \ref{subgraph-H7}).
In this case, color greedily $v_1, v_4, v_7, v_8, v_9$ in order.  For $v_i \in \{v_2, v_3, v_5, v_6, v_{10}, v_{11}\}$, let $L_{H_6}'(v_i)$ be the list after coloring $v_1, v_4, v_7, v_8, v_9$.  Then $|L_{H_6}'(v_i)| \geq 3$ for $i = 2, 3, 5, 6, 10, 11$ and $v_2, v_3, v_5, v_6, v_{10}, v_{11}$ form a 6-cycle.  So, $v_2, v_3, v_5, v_6, v_{10}, v_{11}$ are colorable from the list $L'(v)$ by Lemma \ref{cycle-six}.
Hence
the vertices in $H_6$ can be colored from the list $L_{H_6}$
so that we obtain an $L$-coloring in $G^2$.


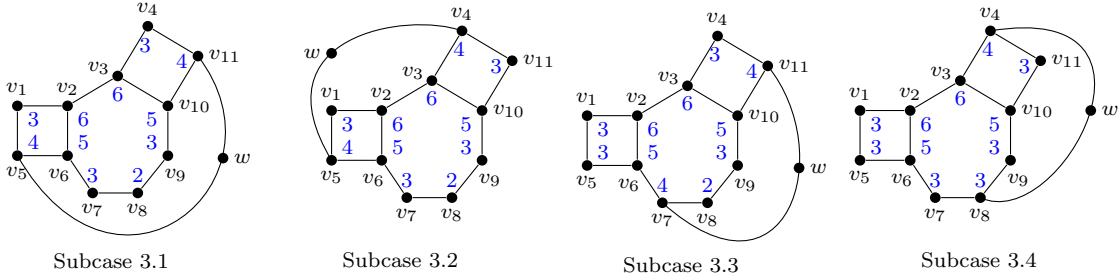
\begin{figure}[htbp]
  \begin{center}
 \raisebox{-0.7cm}{ \begin{tikzpicture}[
  v2/.style={fill=black,minimum size=4pt,ellipse,inner sep=1pt},
  node distance=1.5cm,scale=0.33]
      \node[v2] (H6_1) at (0, 0){};
      \node[v2] (H6_2) at (0,-2){};
      \node[v2] (H6_3) at (1,-3.5){};
      \node[v2] (H6_4) at (2.8,-3.5){};
      \node[v2] (H6_5) at (4,-2){};
      \node[v2] (H6_6) at (4, 0){};
      \node[v2] (H6_7) at (2, 1.2){};
      \node[v2] (H6_9) at (-2, 0){};
      \node[v2] (H6_10) at (-2,-2){};
      \node[v2] (H6_11) at  (3.2,3.2){};
      \node[v2] (H6_12) at  (5.2,2){};
  \node[v2] (H6_13) at  (6.2,-2.1){};
   \draw (H6_1)--(H6_2)--(H6_3)--(H6_4)--(H6_5)--(H6_6)--(H6_7)--(H6_1);
   \draw (H6_1)--(H6_9)--(H6_10)--(H6_2);
   \draw (H6_6)--(H6_12)--(H6_11)--(H6_7);
    \draw(-2,-2)..controls (2,-9)and (9,-3)..(5.2,2);

  \node[font=\scriptsize,above] at (H6_9) {$v_1$};
  \node[font=\scriptsize,xshift=6pt,yshift=-5pt] at (H6_9) {\color{blue}3};
  \node[font=\scriptsize,above] at (H6_1) {$v_2$};
    \node[font=\scriptsize,xshift=6pt,yshift=-5pt] at (H6_1) {\color{blue}6};
  \node[font=\scriptsize,xshift=-7pt,yshift=3pt] at (H6_7) {$v_3$};
    \node[font=\scriptsize,below] at (H6_7) {\color{blue}6};
 \node[font=\scriptsize,above] at (H6_11) {$v_4$};
     \node[font=\scriptsize,xshift=-1pt,yshift=-7pt] at (H6_11) {\color{blue}3};
  \node[font=\scriptsize,below] at (H6_10) {$v_5$};
      \node[font=\scriptsize,xshift=6pt,yshift=5pt] at (H6_10) {\color{blue}4};
  \node[font=\scriptsize,below=1pt, xshift=-3pt] at (H6_2) {$v_6$};
      \node[font=\scriptsize,xshift=6pt,yshift=5pt] at (H6_2) {\color{blue}5};
  \node[font=\scriptsize,below] at (H6_3) {$v_7$};
      \node[font=\scriptsize,above] at (H6_3) {\color{blue}3};
  \node[font=\scriptsize,below] at (H6_4) {$v_{8}$};
      \node[font=\scriptsize,above] at (H6_4) {\color{blue}2};
  \node[font=\scriptsize,below=1pt, xshift=3pt] at (H6_5) {$v_{9}$};
      \node[font=\scriptsize,xshift=-6pt,yshift=5pt] at (H6_5) {\color{blue}3};
   \node[font=\scriptsize,right] at (H6_6) {$v_{10}$};
       \node[font=\scriptsize,xshift=-6pt,yshift=-5pt] at (H6_6) {\color{blue}5};
  \node[font=\scriptsize,right] at (H6_12) {$v_{11}$};
      \node[font=\scriptsize,xshift=-6pt,yshift=-2pt] at (H6_12) {\color{blue}4};
      \node[font=\scriptsize,right] at (H6_13) {$w$};
       \node[font=\scriptsize,right] at (-1,-6.2) {Subcase 3.1};
\end{tikzpicture}}\hspace{-0.5cm}\begin{tikzpicture}[
  v2/.style={fill=black,minimum size=4pt,ellipse,inner sep=1pt},
  node distance=1.5cm,scale=0.33]
      \node[v2] (H6_1) at (0, 0){};
      \node[v2] (H6_2) at (0,-2){};
      \node[v2] (H6_3) at (1,-3.5){};
      \node[v2] (H6_4) at (2.8,-3.5){};
      \node[v2] (H6_5) at (4,-2){};
      \node[v2] (H6_6) at (4, 0){};
      \node[v2] (H6_7) at (2, 1.2){};
      \node[v2] (H6_9) at (-2, 0){};
      \node[v2] (H6_10) at (-2,-2){};
      \node[v2] (H6_11) at  (3.2,3.2){};
      \node[v2] (H6_12) at  (5.2,2){};
  \node[v2] (H6_13) at  (-2,2.3){};
   \draw (H6_1)--(H6_2)--(H6_3)--(H6_4)--(H6_5)--(H6_6)--(H6_7)--(H6_1);
   \draw (H6_1)--(H6_9)--(H6_10)--(H6_2);
   \draw (H6_6)--(H6_12)--(H6_11)--(H6_7);
    \draw(-2,-2)..controls (-5,3)and (1,4)..(3.2,3.2);

  \node[font=\scriptsize,above] at (H6_9) {$v_1$};
  \node[font=\scriptsize,xshift=6pt,yshift=-5pt] at (H6_9) {\color{blue}3};
  \node[font=\scriptsize,above] at (H6_1) {$v_2$};
    \node[font=\scriptsize,xshift=6pt,yshift=-5pt] at (H6_1) {\color{blue}6};
  \node[font=\scriptsize,xshift=-7pt,yshift=3pt] at (H6_7) {$v_3$};
    \node[font=\scriptsize,below] at (H6_7) {\color{blue}6};
 \node[font=\scriptsize,above] at (H6_11) {$v_4$};
     \node[font=\scriptsize,xshift=-1pt,yshift=-7pt] at (H6_11) {\color{blue}4};
  \node[font=\scriptsize,below] at (H6_10) {$v_5$};
      \node[font=\scriptsize,xshift=6pt,yshift=5pt] at (H6_10) {\color{blue}4};
  \node[font=\scriptsize,below=1pt, xshift=-3pt] at (H6_2) {$v_6$};
      \node[font=\scriptsize,xshift=6pt,yshift=5pt] at (H6_2) {\color{blue}5};
  \node[font=\scriptsize,below] at (H6_3) {$v_7$};
      \node[font=\scriptsize,above] at (H6_3) {\color{blue}3};
  \node[font=\scriptsize,below] at (H6_4) {$v_{8}$};
      \node[font=\scriptsize,above] at (H6_4) {\color{blue}2};
  \node[font=\scriptsize,below=1pt, xshift=3pt] at (H6_5) {$v_{9}$};
      \node[font=\scriptsize,xshift=-6pt,yshift=5pt] at (H6_5) {\color{blue}3};
   \node[font=\scriptsize,right] at (H6_6) {$v_{10}$};
       \node[font=\scriptsize,xshift=-6pt,yshift=-5pt] at (H6_6) {\color{blue}5};
  \node[font=\scriptsize,right] at (H6_12) {$v_{11}$};
      \node[font=\scriptsize,xshift=-6pt,yshift=-2pt] at (H6_12) {\color{blue}3};
      \node[font=\scriptsize,left] at (H6_13) {$w$};
    \node[font=\scriptsize,right] at (-2,-6) {Subcase 3.2};
\end{tikzpicture}\raisebox{-0.5cm}{\begin{tikzpicture}[
  v2/.style={fill=black,minimum size=4pt,ellipse,inner sep=1pt},
  node distance=1.5cm,scale=0.33]

      \node[v2] (H6_1) at (0, 0){};
      \node[v2] (H6_2) at (0,-2){};
      \node[v2] (H6_3) at (1,-3.5){};
      \node[v2] (H6_4) at (2.8,-3.5){};
      \node[v2] (H6_5) at (4,-2){};
      \node[v2] (H6_6) at (4, 0){};
      \node[v2] (H6_7) at (2, 1.2){};
      \node[v2] (H6_9) at (-2, 0){};
      \node[v2] (H6_10) at (-2,-2){};
      \node[v2] (H6_11) at  (3.2,3.2){};
      \node[v2] (H6_12) at  (5.2,2){};
  \node[v2] (H6_13) at  (6.45,-2.1){};
   \draw (H6_1)--(H6_2)--(H6_3)--(H6_4)--(H6_5)--(H6_6)--(H6_7)--(H6_1);
   \draw (H6_1)--(H6_9)--(H6_10)--(H6_2);
   \draw (H6_6)--(H6_12)--(H6_11)--(H6_7);
    \draw(1,-3.5)..controls (6,-8)and (8,-1.5)..(5.2,2);

  \node[font=\scriptsize,above] at (H6_9) {$v_1$};
  \node[font=\scriptsize,xshift=6pt,yshift=-5pt] at (H6_9) {\color{blue}3};
  \node[font=\scriptsize,above] at (H6_1) {$v_2$};
    \node[font=\scriptsize,xshift=6pt,yshift=-5pt] at (H6_1) {\color{blue}6};
  \node[font=\scriptsize,xshift=-7pt,yshift=3pt] at (H6_7) {$v_3$};
    \node[font=\scriptsize,below] at (H6_7) {\color{blue}6};
 \node[font=\scriptsize,above] at (H6_11) {$v_4$};
     \node[font=\scriptsize,xshift=-1pt,yshift=-7pt] at (H6_11) {\color{blue}3};
  \node[font=\scriptsize,below] at (H6_10) {$v_5$};
      \node[font=\scriptsize,xshift=6pt,yshift=5pt] at (H6_10) {\color{blue}3};
  \node[font=\scriptsize,below=1pt, xshift=-3pt] at (H6_2) {$v_6$};
      \node[font=\scriptsize,xshift=6pt,yshift=5pt] at (H6_2) {\color{blue}5};
  \node[font=\scriptsize,below] at (H6_3) {$v_7$};
      \node[font=\scriptsize,above] at (H6_3) {\color{blue}4};
  \node[font=\scriptsize,below] at (H6_4) {$v_{8}$};
      \node[font=\scriptsize,above] at (H6_4) {\color{blue}2};
  \node[font=\scriptsize,below=1pt, xshift=3pt] at (H6_5) {$v_{9}$};
      \node[font=\scriptsize,xshift=-6pt,yshift=5pt] at (H6_5) {\color{blue}3};
   \node[font=\scriptsize,right] at (H6_6) {$v_{10}$};
       \node[font=\scriptsize,xshift=-6pt,yshift=-5pt] at (H6_6) {\color{blue}5};
  \node[font=\scriptsize,right] at (H6_12) {$v_{11}$};
      \node[font=\scriptsize,xshift=-6pt,yshift=-2pt] at (H6_12) {\color{blue}4};
      \node[font=\scriptsize,right] at (H6_13) {$w$};
       \node[font=\scriptsize,right] at (-1,-6) {Subcase 3.3};
\end{tikzpicture}}\begin{tikzpicture}[
  v2/.style={fill=black,minimum size=4pt,ellipse,inner sep=1pt},
  node distance=1.5cm,scale=0.33]
      \node[v2] (H6_1) at (0, 0){};
      \node[v2] (H6_2) at (0,-2){};
      \node[v2] (H6_3) at (1,-3.5){};
      \node[v2] (H6_4) at (2.8,-3.5){};
      \node[v2] (H6_5) at (4,-2){};
      \node[v2] (H6_6) at (4, 0){};
      \node[v2] (H6_7) at (2, 1.2){};
      \node[v2] (H6_9) at (-2, 0){};
      \node[v2] (H6_10) at (-2,-2){};
      \node[v2] (H6_11) at  (3.2,3.2){};
      \node[v2] (H6_12) at  (5.2,2){};
  \node[v2] (H6_13) at  (7.2,0){};
   \draw (H6_1)--(H6_2)--(H6_3)--(H6_4)--(H6_5)--(H6_6)--(H6_7)--(H6_1);
   \draw (H6_1)--(H6_9)--(H6_10)--(H6_2);
   \draw (H6_6)--(H6_12)--(H6_11)--(H6_7);
    \draw(3.2,3.2)..controls (11,4.5)and (6,-5)..(2.8,-3.5);

  \node[font=\scriptsize,above] at (H6_9) {$v_1$};
  \node[font=\scriptsize,xshift=6pt,yshift=-5pt] at (H6_9) {\color{blue}3};
  \node[font=\scriptsize,above] at (H6_1) {$v_2$};
    \node[font=\scriptsize,xshift=6pt,yshift=-5pt] at (H6_1) {\color{blue}6};
  \node[font=\scriptsize,xshift=-7pt,yshift=3pt] at (H6_7) {$v_3$};
    \node[font=\scriptsize,below] at (H6_7) {\color{blue}6};
 \node[font=\scriptsize,above] at (H6_11) {$v_4$};
     \node[font=\scriptsize,xshift=-1pt,yshift=-7pt] at (H6_11) {\color{blue}4};
  \node[font=\scriptsize,below] at (H6_10) {$v_5$};
      \node[font=\scriptsize,xshift=6pt,yshift=5pt] at (H6_10) {\color{blue}3};
  \node[font=\scriptsize,below=1pt, xshift=-3pt] at (H6_2) {$v_6$};
      \node[font=\scriptsize,xshift=6pt,yshift=5pt] at (H6_2) {\color{blue}5};
  \node[font=\scriptsize,below] at (H6_3) {$v_7$};
      \node[font=\scriptsize,above] at (H6_3) {\color{blue}3};
  \node[font=\scriptsize,below] at (H6_4) {$v_{8}$};
      \node[font=\scriptsize,above] at (H6_4) {\color{blue}3};
  \node[font=\scriptsize,below=1pt, xshift=3pt] at (H6_5) {$v_{9}$};
      \node[font=\scriptsize,xshift=-6pt,yshift=5pt] at (H6_5) {\color{blue}3};
   \node[font=\scriptsize,right] at (H6_6) {$v_{10}$};
       \node[font=\scriptsize,xshift=-6pt,yshift=-5pt] at (H6_6) {\color{blue}5};
  \node[font=\scriptsize,right] at (H6_12) {$v_{11}$};
      \node[font=\scriptsize,xshift=-6pt,yshift=-2pt] at (H6_12) {\color{blue}3};
      \node[font=\scriptsize,right] at (H6_13) {$w$};
    \node[font=\scriptsize,right] at (0,-6) {Subcase 3.4};
\end{tikzpicture}
\end{center}
\caption{Subcases 3.1--3.4 of the proof of Lemma \ref{reducible-H7}. The numbers at vertices are the number of available colors.} \label{H7-nbr-first}
\end{figure}

\bigskip
\noindent {\bf Case 3:} $H_6^2$ is not an induced subgraph of $G^2$ and
 $E(G[V(H_6)]) - E(H_6) = \emptyset$.

\medskip
\noindent {\bf Simplifying cases:}
\begin{itemize}
\item The vertices in each of the following pairs do not have a common neighbor outside $H_6$,
since it makes a 5-cycle:
$\{v_1, v_4\}$, $\{v_1, v_5\}$, $\{v_1, v_7\}$,
$\{v_4, v_9\}$, $\{v_4, v_{11}\}$, $\{v_5, v_8\}$, and $\{v_8, v_{11}\}$.

\item The vertices in each of the following pairs do not have a common neighbor outside $H_6$,
since it makes $F_2$,
which does not exist by  Lemma \ref{reducible-F2}:
$\{v_{7}, v_{8}\}$, and $\{v_{8}, v_{9}\}$.

\item The vertices in each of the following pairs do not have a common neighbor outside $H_6$,
since it makes $H_1$,
which does not exist by Lemma \ref{reducible-H0}:
$\{v_5, v_7\}$, and $\{v_9, v_{11}\}$.

\item If $v_4$ and $v_7$ have a common neighbor $w$ where $w \notin V(H_6)$,
then the $6$-cycle $w, v_4, v_3, v_2,v_6,v_7$ together with two $4$-cycles $v_1,v_2,v_6,v_5$ and $v_3,v_4,v_{11},v_{10}$ forms $H_2$, contradicting Lemma \ref{reducible-H2}.
Thus, $v_4$ and $v_7$ do not have a common neighbor outside $H_6$.

By symmetry,
$v_1$ and $v_9$ do not have a common neighbor outside $H_6$.
\end{itemize}

Note that $v_7$ and $v_9$ are adjacent in $H_6^2$ through $v_8$,
and hence we do not need to cosider the case $v_7$ and $v_9$ have a common neighbor outside $H_6$.
So, by symmetry we only need to consider the following four subcases in Case 3.
For each subcase, the size of list $L_{H_6}$ is displayed in Figure \ref{H7-nbr-first}.

\medskip

\noindent {\bf Subcase 3.1:} $v_5$ and $v_{11}$ have a common neighbor $w$,
where $w \notin V(H_6)$.

Note that $w$ is not adjacent to any vertex $v_i \in V(H_6) \setminus \{v_5, v_{11}\}$ by the argument of Simplifying cases.

In this case, we uncolor $w$.
For each $v \in V(H_6) \cup \{w\}$, we define \[
L_{H_6}'(v) = L(v) \setminus \{\phi(x) : xv \in E(G^2) \mbox{ and } x \notin V(H_6) \cup \{w\}\}.
\]
The  number of available colors at vertices are in Figure \ref{Case 3-1-H7}.  Note that $d_G(v_7, w) \geq 3$.  First, we color $v_7$ and $w$ so that we can save a color at $v_5$.
If $L_{H_6}'(v_7) \cap L_{H_6}'(w) \neq \emptyset$, then color $v_7$ and $w$ by a color $c \in L_{H_6}'(v_7) \cap L_{H_6}'(w)$.  If $L_{H_6}'(v_7) \cap L_{H_6}'(w) = \emptyset$, then we can color $v_7$ and $w$ by colors $c_7$ and $c_w$, respectively, so that $|L_{H_6}'(v_5) \setminus \{c_7, c_w\}| \geq 4$.
And then color greedily  $v_8$ and $v_9$.

For $v_i \in V(H_6) \setminus \{v_7, v_8, v_9\}$, let $L_{H_6}''(v_i)$ be the list after coloring $v_7, v_8, v_9, w$.  Then the list of $L_{H_6}''(v)$ is like (a) in Figure \ref{Case 3-1-H7}, where $v_5$ and $v_{11}$ have the common neighbor $w$.
Next, we color $v_4, v_{10}, v_{11}$ by colors $c_4, c_{10}, c_{11}$, respectively, so that $L_{H_6}''(v_1) \neq L_{H_6}''(v_2) \setminus\{c_4, c_{10}\}$.  Then, the list of the remaining vertices are presented in Figure \ref{Case 3-1-H7} (b).  For $v_i \in \{v_1, v_2, v_3, v_5, v_6\}$, let $L_{H_6}'''(v_i)$ be the list after coloring $v_4, v_{10}, v_{11}$.

Note that $|L_{H_6}'''(v_1) \cup L_{H_6}'''(v_2) \cup L_{H_6}'''(v_6)| \geq 4$ since $L_{H_6}''(v_1) \neq L_{H_6}''(v_2) \setminus\{c_4, c_{10}\}$.  Then $v_1, v_2, v_3, v_5, v_6$ are colorable from the list $L_{H_6}'''(v)$
by Lemma \ref{lem-4cycle-pendent}.  Hence $H_6^2$ can be colored properly from the list.


\begin{figure}[htbp]
  \begin{center}
  \begin{tikzpicture}[
  v2/.style={fill=black,minimum size=4pt,ellipse,inner sep=1pt},
  node distance=1.5cm,scale=0.4]

      \node[v2] (H6_1) at (0, 0){};
      \node[v2] (H6_2) at (0,-2){};
      \node[v2] (H6_3) at (1,-3.5){};
      \node[v2] (H6_4) at (2.8,-3.5){};
      \node[v2] (H6_5) at (4,-2){};
      \node[v2] (H6_6) at (4, 0){};
      \node[v2] (H6_7) at (2, 1.2){};
      \node[v2] (H6_9) at (-2, 0){};
      \node[v2] (H6_10) at (-2,-2){};
      \node[v2] (H6_11) at  (3.2,3.2){};
      \node[v2] (H6_12) at  (5.2,2){};
  \node[v2] (H6_13) at  (6.1,-2.1){};
   \draw (H6_1)--(H6_2)--(H6_3)--(H6_4)--(H6_5)--(H6_6)--(H6_7)--(H6_1);
   \draw (H6_1)--(H6_9)--(H6_10)--(H6_2);
   \draw (H6_6)--(H6_12)--(H6_11)--(H6_7);
    \draw(-2,-2)..controls (2,-8)and (9,-3)..(5.2,2);

  \node[font=\scriptsize,above] at (H6_9) {$v_1$};
  \node[font=\scriptsize,xshift=6pt,yshift=-5pt] at (H6_9) {\color{blue}4};
  \node[font=\scriptsize,above] at (H6_1) {$v_2$};
    \node[font=\scriptsize,xshift=6pt,yshift=-5pt] at (H6_1) {\color{blue}6};
  \node[font=\scriptsize,xshift=-10pt] at (H6_7) {$v_3$};
    \node[font=\scriptsize,below] at (H6_7) {\color{blue}6};
 \node[font=\scriptsize,above] at (H6_11) {$v_4$};
     \node[font=\scriptsize,xshift=-1pt,yshift=-7pt] at (H6_11) {\color{blue}4};
  \node[font=\scriptsize,below] at (H6_10) {$v_5$};
      \node[font=\scriptsize,xshift=6pt,yshift=5pt] at (H6_10) {\color{blue}5};
  \node[font=\scriptsize,below=1pt, xshift=-3pt] at (H6_2) {$v_6$};
      \node[font=\scriptsize,xshift=6pt,yshift=5pt] at (H6_2) {\color{blue}6};
  \node[font=\scriptsize,below] at (H6_3) {$v_7$};
      \node[font=\scriptsize,above] at (H6_3) {\color{blue}3};
  \node[font=\scriptsize,below] at (H6_4) {$v_{8}$};
      \node[font=\scriptsize,above] at (H6_4) {\color{blue}2};
  \node[font=\scriptsize,below=1pt, xshift=3pt] at (H6_5) {$v_{9}$};
      \node[font=\scriptsize,xshift=-6pt,yshift=5pt] at (H6_5) {\color{blue}3};
   \node[font=\scriptsize,right] at (H6_6) {$v_{10}$};
       \node[font=\scriptsize,xshift=-6pt,yshift=-5pt] at (H6_6) {\color{blue}6};
  \node[font=\scriptsize,right] at (H6_12) {$v_{11}$};
      \node[font=\scriptsize,xshift=-6pt,yshift=-2pt] at (H6_12) {\color{blue}5};
      \node[font=\scriptsize,right] at (H6_13) {$w$};
        \node[font=\scriptsize,left] at (H6_13) {\color{blue}4};
\end{tikzpicture}\hspace{0.5cm}
 \begin{tikzpicture}[
  v2/.style={fill=black,minimum size=4pt,ellipse,inner sep=1pt},
  node distance=1.5cm,scale=0.8]

 \node[v2] (F2_1) at (0, 1) {};
    \node[v2] (F2_2) at (-1, 2) {};
    \node[v2] (F2_3) at (-1, 0) {};

    \node[v2] (F2_4) at (2, 1) {};
    \node[v2] (F2_5) at (3, 2) {};
    \node[v2] (F2_6) at (4, 1) {};
    \node[v2] (F2_7) at (3, 0) {};
    \node[v2] (F2_8) at (-2,1) {};

    \draw (F2_1) -- (F2_2)  -- (F2_8)-- (F2_3) -- (F2_1);

    \draw (F2_4) -- (F2_5) -- (F2_6) -- (F2_7) -- (F2_4);

    \draw (F2_1) -- (F2_4);

    \node[font=\scriptsize,above] at (F2_2) {$v_1$};
    \node[font=\scriptsize,below] at (F2_2) {\color{blue}3};
    \node[font=\scriptsize,above] at (F2_1) {$v_2$};
    \node[font=\scriptsize,xshift=-8pt,yshift=1pt] at (F2_1) {\color{blue}5};
    \node[font=\scriptsize,below] at (F2_3) {$v_6$};
    \node[font=\scriptsize,above] at (F2_3) {\color{blue}3};
     \node[font=\scriptsize,below] at (F2_8) {$v_5$};
    \node[font=\scriptsize,xshift=8pt,yshift=2pt] at (F2_8) {\color{blue}4};
    \node[font=\scriptsize,above] at (F2_4) {$v_3$};
    \node[font=\scriptsize,xshift=8pt,yshift=1pt] at (F2_4) {\color{blue}5};
    \node[font=\scriptsize,below] at (F2_7) {$v_{10}$};
    \node[font=\scriptsize,above] at (F2_7) {\color{blue}3};
      \node[font=\scriptsize,right] at (F2_6) {$v_{11}$};
    \node[font=\scriptsize,xshift=-8pt,yshift=2pt] at (F2_6) {\color{blue}3};

      \node[font=\scriptsize,above] at (F2_5) {$v_4$};
    \node[font=\scriptsize,below] at (F2_5) {\color{blue}3};
      \node[font=\scriptsize] at (1,-2) { (a) List $L_{H_6}''(v)$};
\end{tikzpicture}\hspace{0.3cm}
\begin{tikzpicture}[
  v2/.style={fill=black,minimum size=4pt,ellipse,inner sep=1pt},
  node distance=1.5cm,scale=0.8]

 \node[v2] (F2_1) at (0, 1) {};
    \node[v2] (F2_2) at (-1, 2) {};
    \node[v2] (F2_3) at (-1, 0) {};

    \node[v2] (F2_4) at (2, 1) {};

    \node[v2] (F2_8) at (-2,1) {};

    \draw (F2_1) -- (F2_2)  -- (F2_8)-- (F2_3) -- (F2_1);

    \draw (F2_1) -- (F2_4);

    \node[font=\scriptsize,above] at (F2_1) {$v_2$};
    \node[font=\scriptsize, xshift=-8pt,yshift=2pt] at (F2_1) {\color{blue}3};
    \node[font=\scriptsize,above] at (F2_4) {$v_3$};
    \node[font=\scriptsize,right] at (F2_4) {\color{blue}2};
    \node[font=\scriptsize,above] at (F2_2) {$v_1$};
    \node[font=\scriptsize,below] at (F2_2) {\color{blue}3};
     \node[font=\scriptsize,below] at (F2_8) {$v_5$};
    \node[font=\scriptsize,xshift=8pt,yshift=1pt] at (F2_8) {\color{blue}3};
     \node[font=\scriptsize,below] at (F2_3) {$v_6$};
    \node[font=\scriptsize,above] at (F2_3) {\color{blue}3};
     \node[font=\scriptsize] at (0,-2) {(b) List $L_{H_6}'''(v)$};
\end{tikzpicture}
\end{center}
 \caption{Subcase 3.1.  The numbers at vertices are the number of available colors.}\label{Case 3-1-H7}
\end{figure}
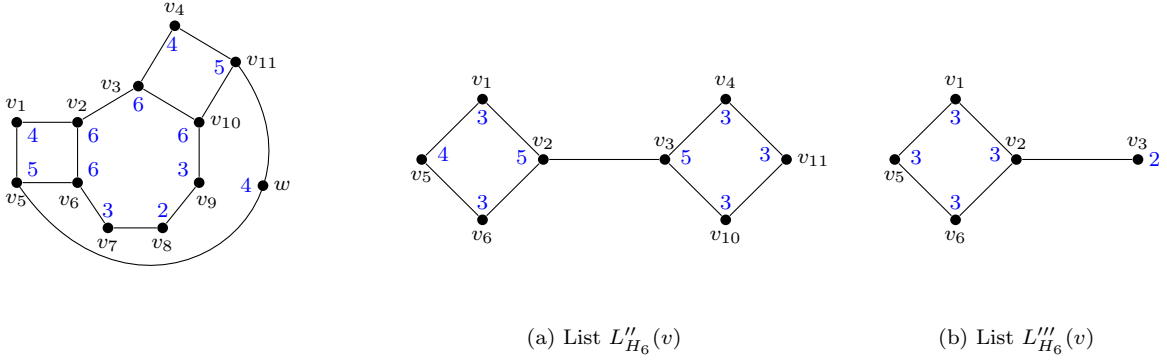

\medskip
\noindent {\bf Subcase 3.2:}
$v_4$ and $v_{5}$ (or $v_1$ and $v_{11}$ by symmetry)
have a common neighbor $w$, where $w \notin V(H_6)$.

Note that $w$ is not adjacent to any vertex $v_i \in V(H_6) \setminus \{v_4, v_5\}$ by the argument of Simplifying cases.
If $v_7$ and $v_{11}$ have a common neighbor $z$ with $z \notin V(H_6)$,
then $w, v_4, v_{11}, z, v_7, v_6, v_5$ form $H_5$, contradicting Lemma \ref{reducible-H6}.

Uncolor $w$ and then color $v_7$ and $w$ to save a color at $v_5$.
Next, follow the same procedure as Case 3.1.
Then
the vertices in $H_6$ can be colored from the list $L_{H_6}$
so that we obtain an $L$-coloring in $G^2$.

\medskip
\noindent {\bf Subcase 3.3:}
$v_7$ and $v_{11}$ (or $v_5$ and $v_{9}$ by symmetry)
have a common neighbor $w$, where $w \notin V(H_6)$.

Note that $w$ is not adjacent to any vertex $v_i \in V(H_6) \setminus \{v_7, v_{11}\}$ by the argument of Simplifying cases.
Color $v_7$ by a color $c_7 \in L_{H_6}(v_7)$
so that $|L_{H_6}(v_5) \setminus \{c_7\}| \geq 3$,
and follow the same procedure as Subcase 1.2.  Then
the vertices in $H_6$ can be colored from the list $L_{H_6}$
so that we obtain an $L$-coloring in $G^2$.

\medskip

\noindent {\bf Subcase 3.4:}
$v_4$ and $v_{8}$ (or $v_1$ and $v_{8}$ by symmetry)
have a common neighbor $w$, where $w \notin V(H_6)$.

Note that $w$ is not adjacent to any vertex $v_i \in V(H_6) \setminus \{v_4, v_8\}$ by the argument of Simplifying cases.
Follow the same procedure as Case 1 from the right hand side with $v_4$, $v_9$, and $v_{11}$.
The roles of $v_4, v_9, v_{11}$ are $v_1, v_7, v_5$, respectively.  Then
we can show that
the vertices in $H_6$ can be colored from the list $L_{H_6}$
so that we obtain an $L$-coloring in $G^2$.

Thus in Case 3,
we obtain a contradiction for the fact that $G$ is a counterexample.
This completes the proof of
Lemma \ref{reducible-H7}.
\end{proof}


\section*{Acknowledgments}
Seog-Jin Kim is supported by the National Research Foundation of Korea(NRF) grant funded by the Korea government(MSIT)(No.NRF-2021R1A2C1005785).
Xiaopan Lian is supported by National Natural Science Foundation of China No. 12371351.
Kenta Ozeki is supported by JSPS KAKENHI Grant Numbers 22K19773 and 23K03195.


\end{document}